%% file: farey_recursion.tex
\documentclass[12pt]{amsart}
\usepackage{fullpage,graphicx,psfrag,amsmath,amsfonts,verbatim, amssymb, hyperref, caption, subcaption, tikz,enumitem,longtable}

\input defs.tex

\title{Integer superharmonic matrices on the $F$-lattice }
\author{Ahmed Bou-Rabee}

\begin{document}

	\begin{abstract}
	We prove that the set of quadratic growths achievable by integer superharmonic functions on the $F$-lattice, a periodic subgraph of the square lattice with oriented edges, has the structure of an overlapping circle packing. The proof recursively constructs a distinct pair of recurrent functions for each rational point on a hyperbola. This proves a conjecture of Smart (2013) and completely describes the scaling limit of the Abelian sandpile on the $F$-lattice.

	\end{abstract}
	\maketitle

	\section{Introduction}
	\begin{figure}
		\input{tikz/flattice_2.tikz}
		\caption{A $5 \times 5$ section of the $F$-lattice}
	\end{figure}
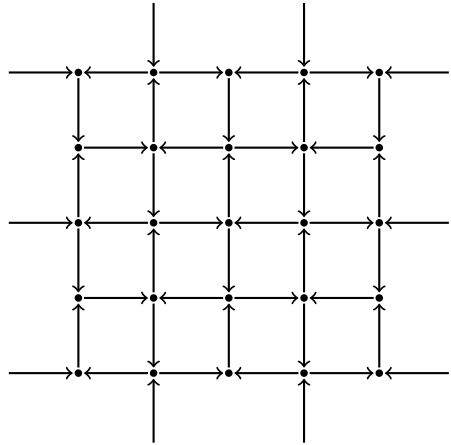
	
	The $F$-lattice is a directed periodic planar graph $(\Z^2, E)$, where
	\[
	\begin{cases}
	(x \pm e_1,x) \in E &\mbox{ if $x_1 + x_2  \equiv 0 \pmod 2$} \\
	(x \pm e_2,x) \in E &\mbox{ otherwise},
	\end{cases}
	\]
	and $e_1,e_2$ are the standard basis vectors in $\Z^2$. A function $g:\Z^2 \to \Z$ is {\it integer superharmonic} if 
	\begin{equation} \label{eq:superharmonic}
	\Delta g(x) := \sum_{(y,x) \in E} (g(y) - g(x)) \leq 0,
	\end{equation}
	for all $x \in \Z^2$.  When it exists, the {\it quadratic growth} of $g$ 
	is specified by a $2 \times 2$ symmetric matrix $A \in S_2$, 
	\begin{equation} \label{eq:quadratic_growth}
	g(x) = \frac{1}{2} x^T A x + o(|x|^2).
	\end{equation}
	When $g$ is integer superharmonic and has quadratic growth $A$, we say that it is an {\it integer superharmonic representative} of $A$
	and $A$ is an {\it integer superharmonic matrix}. Moreover, $g$ is {\it recurrent} if whenever $f:\Z^2 \to \Z$  is integer superharmonic and $X \subset \Z^2$ is finite and strongly connected (with respect to $E$), 
	\begin{equation}
	\sup_{X} (g - f) \leq \sup_{\partial X} (g - f),
	\end{equation}
	where $\partial X = \{ y \in \Z^2 \setminus X : \mbox{there is $x \in X$ with $(y,x) \in E$}\}$. 
	We call an integer superharmonic representative of $A$ which is recurrent an {\it odometer} 
	for $A$.

	In this article we demonstrate an explicit characterization of integer superharmonic matrices on the $F$-lattice
	via a recursive construction of their odometers.

	\subsection{Background}
		A {\it periodic Euclidean graph} is a graph embedded into $\R^d$ such that there exists a basis of $\R^d$
		whose translations leave the graph unchanged.  Any such graph, $(V, E)$, defines a set of integer superharmonic matrices.   
	The study of these matrices was initiated by Pegden and Smart \cite{pegden2013convergence} in the context of the Abelian sandpile model of Bak, Tang and Wiesenfeld and Dhar  \cite{bak1987self, dhar1990self}.
	We briefly describe the model, 	referring the interested reader to the surveys \cite{redig2005mathematical, holroyd2008chip, jarai2018sandpile} and books
	\cite{klivans2018mathematics, corry2018divisors}.
	
	 The Abelian sandpile is a deterministic diffusion process on $(V,E)$, of which the following, the {\it single-source sandpile}, is a canonical example. Start with $n$ chips at the origin (or the closest point to the origin) in $V$. When a vertex has at least as many chips as outgoing edges, it topples, sending one chip along each outgoing edge.  When $n$ is large, the final configuration of chips, $s_n: V \to \Z$, displays fascinating fractal structure.  Pegden-Smart made it possible to study this structure by showing that $s_n$ converges weakly-* to a limiting $s:\R^d \to \R$ which is described by the solution to a certain nonlinear partial differentiable equation, later called the {\it sandpile PDE}. 
	
	The sandpile PDE is characterized by the set of integer superharmonic matrices on $(V,E)$; in particular, the fractal structure of 
	large sandpiles is dependent on the graph upon which the sandpile is run.  In a tour de force, Levine, Pegden, and Smart showed that the set of integer superharmonic matrices on the square lattice, $\Z^2$ with nearest neighbor edges, is the downwards closure of an Apollonian circle packing \cite{levine2017apollonian}. This led to an understanding of the fractal patterns appearing in sandpile experiments, \cite{levine2016apollonian, pegden2020stability}, something which had evaded physicists and mathematicians for decades \cite{liu1990geometry,le2002identity, ostojic2003patterns}.
	
	Levine-Pegden-Smart's proof in \cite{levine2017apollonian} involved explicitly constructing an odometer
	for each circle in an Apollonian band packing.  Their construction mirrored the Soddy recursive generation 
	of Apollonian circle packings - it pieced together later odometers from earlier ones. In this article, we also recursively construct
	odometers, but the recursion follows rational points on a hyperbola rather than curvatures in an Apollonian packing.
	Our choice of lattice also highlights several other coincidences which occur for $\Z^2$ and forces us to develop 
	 new proof techniques which may generalize.  We discuss these possible generalizations in Section \ref{subsec:fk_lattice}  and provide a detailed proof overview in Section \ref{sec:proof_overview}.

	The patterns which appear in $s_n$ on the $F$-lattice have also been investigated by mathematical physicists
	with notable contributions made by Caracciolo, Paoletti, Sportiello \cite{caracciolo2008explicit, paoletti2012deterministic} and Dhar, Sadhu, Chandra \cite{dhar2009pattern, dhar2013sandpile, sadhu2010pattern, sadhu2011effect}. This article provides a new perspective on their results. For example, the patterns which appear in their experiments correspond empirically to the Laplacians of our constructed odometers. In fact, an immediate consequence of Theorems \ref{theorem:circle_packing} and \ref{theorem:hyperbola} is that the weak-* limit 
	of the sandpile identity on ellipsoidal domains is constant \cite{melchionna2020sandpile}. We leave open, but expect that these results can also be used to construct more elaborate sandpile fractals as in \cite{levine2016apollonian}. Moreover, it is a difficult open problem 
	to construct the weak-* limit of the single-source sandpile on the square lattice. It would be interesting to see if the relatively simple structure of the sandpile PDE here can be used to make progress on this for the $F$-lattice.
	 
	\subsection{Main results}	
	Our primary result is that the set of integer superharmonic matrices on the $F$-lattice is the downwards closure of an overlapping circle packing. 
\begin{theorem} \label{theorem:circle_packing}
	$A \in S_2$ is integer superharmonic if and only if the difference
	\[
	\frac{1}{2} \begin{bmatrix} s -t & s+t \\ s + t & t-s \end{bmatrix} - A
	\]
	is positive semidefinite for some $s, t \in \Z$. 
\end{theorem}
  	We explain the connection to circles.  
	Denote the set of integer superharmonic matrices on the $F$-lattice by $\Gamma_F$. 
	The boundary of $\Gamma_F$ may be viewed as a surface by taking the parameterization $M: \R^3 \to S_2$,
	\[
	M(a,b,c) := \frac{1}{2} \begin{bmatrix} c+a & b \\ b & c-a \end{bmatrix}.
	\]
 	In particular, Theorem \ref{theorem:circle_packing} may be restated as 
	\[
	\partial \Gamma_F = \{ M(a,b, \gamma_F(a,b)): (a,b) \in \R^2\}, 
	\]
	where 
	\begin{equation}
	\gamma_F(x) := \max_{s,t \in \Z} -|x-(s-t,s+t)|.
	\end{equation}
	Viewed from above, $\partial \Gamma_F$ is the union of identical slope-1 cones whose bases are the overlapping circle packing displayed in Figure \ref{fig:overlapping_packing}.
	
	\begin{figure}
		\input{tikz/flattice_packing.tikz}
		\caption{A few periods of the bases of cones in $\partial \Gamma_F$.} \label{fig:overlapping_packing}
	\end{figure}
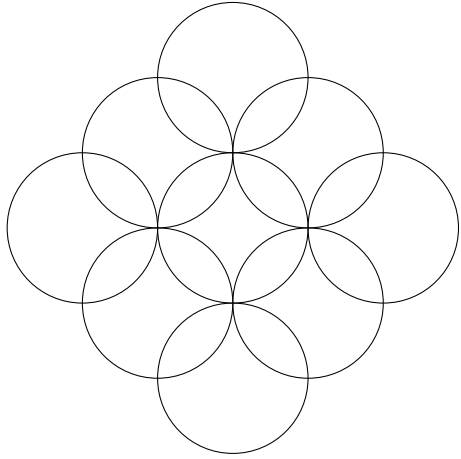
	
	One may check that the matrices, 
	$M(s-t, s+t, 0)$ lie on $\partial \Gamma_F$ for all $s,t \in \Z$ (see Section \ref{subsec:fk_lattice} for the data to do so in a more general setting). This together with the downwards closure of $\Gamma_F$ reduces the proof of Theorem \ref{theorem:circle_packing} to verifying that the intersection curve of each pair of overlapping cones is in $\partial \Gamma_F$. Moreover, by
	symmetry, it suffices to check only one such hyperbola.  Smart made these observations in \cite{smart2013aim} and then conjectured the following, 
	which we prove.
	\begin{theorem} \label{theorem:hyperbola}
		For each $0 \leq t \leq 1$, $M(t,1-t, -\sqrt{t^2 + (1-t)^2})$ lies on the boundary of $\Gamma_F$. 
	\end{theorem}

	The set $\Gamma_F$ is closed (Lemma 3.4 in \cite{levine2017apollonian}), therefore, it suffices to prove Theorem \ref{theorem:hyperbola} for all rational $0 \leq t \leq 1$ 
	along the bottom branch of the hyperbola $\mathcal{H} := \{(t,c) \in \R \times \R^-: t^2 + (1-t)^2 = c^2\}$. We do this recursively.
	We start with explicit formulae for the odometers for $(0,-1)$ and $(1,-1)$ and then use those 
	to construct odometers for all other rational points in between.  Surprisingly, the recursion 
	requires building not just one odometer for each such rational $t$, but {\it two} distinct odometers. This is a significant difference between the square lattice case which builds one odometer at a time; the square lattice odometers
	were also later shown to have a strong uniqueness property \cite{pegden2020stability}. 
	
	Another new challenge is in identifying the correct recursive structure. There is a well-known secant line sweep algorithm 
	which produces (and parameterizes) the rational points on $\mathcal{H}$ given 
	a single rational point on $\mathcal{H}$ (and generally any elliptic curve - see \eg,  \cite{tan1996group}). For example, since $(0,-1) \in \mathcal{H}$, all other rational points can be enumerated by varying the rational slope of a secant line between $(0,-1)$ and $\mathcal{H}$.  
	Unfortunately, the odometers lying on $\mathcal{H}$ under this labeling do not have an apparent recursive structure. 
	
	The parameterization which we adopt in this article utilizes the geometry of two adjacent cones.
	Each rational point on $\mathcal{H}$ is an intersection of two unique lines of rational slope starting at the apexes
	of the cones. These intersections are dense in $\mathcal{H}$ so we may identify each such point by its rational slope. See Figure \ref{fig:f_recursion}.

	\begin{figure}
		\includegraphics[width=0.5\textwidth]{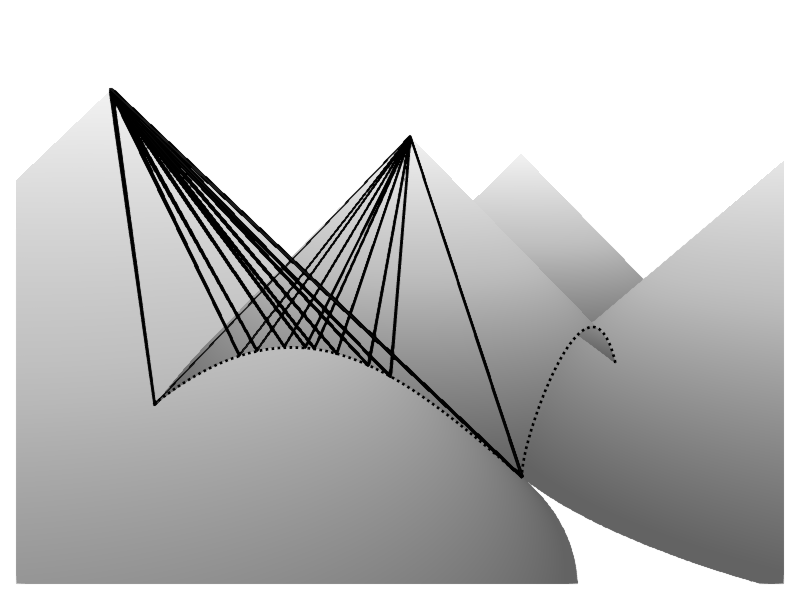}
		\caption{The first three iterations of the hyperbola recursion defined in Section \ref{sec:hyperbola}. The two visible hyperbolas are outlined by dashed lines.} \label{fig:f_recursion}
	\end{figure}

	Specifically, each point in $\Q^2 \cap \mathcal{H}$ may be labeled by a reduced fraction $0 \leq n/d \leq 1$
	with corresponding matrix
	\begin{equation}
	M(n,d) := \frac{1}{(d^2 + 2 d n - n^2)} \begin{bmatrix} -n^2 & d n \\ d n & -d^2 \end{bmatrix}.
	\end{equation}
	We construct odometers for each $M(n,d)$ which grow along the lattice of the matrix, 
	\begin{equation}
	L(n,d) = \{ x \in \Z^2 : M(n,d) x \in \Z^2 \},
	\end{equation}
	and which have periodic Laplacians. However, the $F$-lattice is not transitive.  In particular, if $h:\Z^2 \to \Z$ is $L(n,d)$ periodic, then $\Delta h$ may not be $L(n,d)$ periodic unless its period is even. To circumvent this, we must pass to a sub-lattice by doubling along the kernel of $M(n,d)$. 
	We show in Section \ref{sec:hyperbola} that $L(n,d)$ is equal to the integer span of 
	\begin{equation}
	\mathfrak{v}_{n/d, 1} := \begin{bmatrix} d  \\ n \end{bmatrix}
	\qquad
	\mathfrak{v}_{n/d,2}  := \begin{bmatrix} n -d \\ n + d \end{bmatrix},
	\end{equation}
	and $\mathfrak{v}_{n/d,1}$ generates the kernel of $M(n,d)$. 
	Our modified lattice is 
	\begin{equation}
	L'(n,d) = 
	\begin{cases}
	L(n,d) &\mbox{if $(n+d)$ is even } \\
	2 \Z \mathfrak{v}_{n/d,1} + \Z \mathfrak{v}_{n/d,2}  &\mbox{if $(n+d)$ is odd}.
	\end{cases}
	\end{equation}
	We then derive Theorem \ref{theorem:hyperbola} from the following. 
	\begin{theorem} \label{theorem:odometers}
	For each reduced fraction $0 < n/d < 1$ there exists two distinct odometers 
	$g_{n,d}, \hat{g}_{n,d}$ with quadratic growth $M(n,d)$ both of which satisfy the periodicity condition
	\begin{equation}\label{eq:periodicity}
	g(x + v) = g(x) + x^T M(n,d) v + c_v
	\end{equation}
	for all $v \in L'(n,d)$ where $c_v$ is a constant depending on $v$.
	\end{theorem}
	As in \cite{levine2017apollonian}, the periodicity condition \eqref{eq:periodicity} implies that $g_{n,d}$ and $\hat{g}_{n,d}$ 
	are each integer superharmonic representatives for $M(n,d)$. Moreover, integer superharmonic matrices with odometers are on $\partial \Gamma_F$. Indeed, if $g$ were recurrent but not on the boundary of $\Gamma_F$ there would exist an integer superharmonic $f \geq g + \delta |x|^2$ for some $\delta > 0$. However, on the boundary of 
	a lattice ball of radius $n$, $B_n$, $\sup_{\partial B_n} (g-f) \leq -n \delta$ for all $n$ sufficiently large, contradicting the definition of recurrent as $\sup_{B_n}(g - f) \geq g(0) - f(0)$, a constant. 
		
	\begin{figure}
		\includegraphics[width=0.35\textwidth]{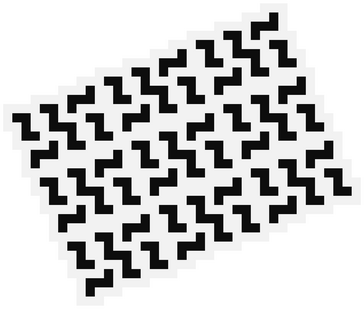}
		\includegraphics[width=0.35\textwidth]{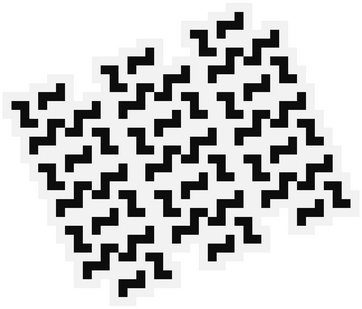}
		\caption{One $L'(2,5)$ period of $\Delta g_{2,5}$ and $\Delta \hat{g}_{2,5}$. White and black are values of 0 and 1 respectively.
		These are used to construct the odometers seen in Figure \ref{fig:need_two_odometers}. }
		\label{fig:standard_and_alternate}
	\end{figure}

	\bibliographystyle{alpha}

\subsection{Kleinian bugs} \label{subsec:fk_lattice}
We briefly mention a connection and possible extensions of this work.  The overlapping circle packing in Figure \ref{fig:overlapping_packing} is an object known as a {\it Kleinian bug} \cite{kapovich2021superintegral}. Kleinian bugs were recently introduced by Kapovich-Kontorovich \cite{kapovich2021superintegral} and generalize Apollonian circle packings.
An important aspect of \cite{levine2017apollonian} is an analogue of Descartes' rule \cite{graham2005apollonian, stange2016sensual}
for integer superharmonic functions --- Kleinian bugs share a similar rule. 

The symmetry group of the Kleinian bug for the $F$-lattice is trivial (the difficult aspect of the argument in this manuscript is in accounting for the intersections between adjacent cones). However, numerical evidence suggests that the set of integer superharmonic matrices on other planar lattices may also be described by nontrivial symmetries of Kleinian bugs. 

Levine-Pegden-Smart have derived a numerical algorithm which can determine the set of integer superharmonic matrices on periodic graphs up to arbitrary precision \cite{levine2016apollonian} (see \cite{pegdensandpile} for some high resolution outputs of this algorithm). 
We ran the Levine-Pegden-Smart algorithm on a family of lattices which generalize the $F$-lattice, what we call the $F^{(k)}$ lattices.
For each $k \geq 2$, the $F^{(k)}$-lattice is a directed, periodic, planar graph $(\Z^2, E^{(k)})$, where
\[
\begin{cases}
(x \pm e_1,x) \in E &\mbox{ if $x_1 + x_2  \equiv 0 \pmod k$} \\
(x \pm e_2,x) \in E &\mbox{ otherwise. }
\end{cases}
\]
Computed sets of $\partial \Gamma_k$, the boundary of the set of integer superharmonic matrices for the the $F^{(k)}$ lattice, are in Figure \ref{fig:fk_lattices}.

Some basic structure of these sets for all $k \geq 2$ may be understood after verifying that
\begin{equation} \label{eq:turan_odometers}
\begin{aligned}
 \Delta r_1(x,y) = 1\{(x_1+x_2) \not \equiv 0 \mbox{ (mod $k$)} \} &\quad \mbox{ for } \quad  r_1(x_1,x_2) = \frac{x_2(x_2+1)}{2} \\
\Delta r_2 = \Delta r_1 &\quad \mbox{ for } \quad  r_2(x_1,x_2) =  q_k(x_1+x_2) \\
\Delta r_3 = 1  &\quad \mbox{ for } \quad  r_3(x_1,x_2) = \frac{x_1(x_1+1) + x_2 (x_2+1)}{2} \\
\end{aligned}
\end{equation}
where
\[
q_k(n) = \frac{(k-1)}{2 k} (n^2 - s^2) + \frac{ s (s-1)}{2} \qquad \mbox{where $s := n \pmod k$ },
\]
(note the Laplacian is that of the $F^{(k)}$ lattice). In particular, $h_1 := r_1 - r_2$ is integer valued and harmonic, $\Delta h_1 = 0$. 
The function $h_2(x_1,x_2) := x_1 x_2$ is also harmonic.  This together with \eqref{eq:turan_odometers} and the standard argument in
Lemma \ref{lemma:checkerboard_recurrence} below can be used to show that $r_i + s h_1 + t h_2 - r_3$ are odometers for all $s, t \in \Z$ and $i \in \{1,2\}$. 
These odometers lie on the hyperbolas between the largest cones in Figure \ref{fig:fk_lattices} and the harmonic functions explain the apparent 
periodicity of $\Gamma_k$. 

When $k=2$ the odometers $s'h_1 + t' h_2$ with $s' = 2(t-s), t' = 2t$ correspond to the peak matrices $M(s-t,s+t,0)$ of Theorem \ref{theorem:circle_packing}.

\begin{remark}
Interestingly, the function $q_k(n)$ also counts the number of edges in a $k$-partite Turan graph of order $n$.
We note that $q_k(n)$ has a simple closed form when $k$ is small,
\[
q_k(n) = \lfloor \frac{(k-1)}{2k} n^2 \rfloor \quad \mbox{only for $k \leq 7$},
\]
but this is false for $k \geq 8$. 
%% indeed 
%%
%%
\end{remark}
\begin{figure}
	\input{tikz/flattice_3.tikz}
	\input{tikz/flattice_4.tikz}
	\input{tikz/flattice_5.tikz} \\
	\includegraphics[width=0.18\textwidth]{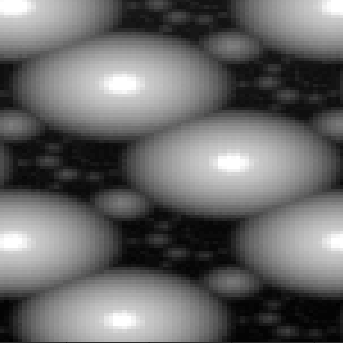}
	\includegraphics[width=0.18\textwidth]{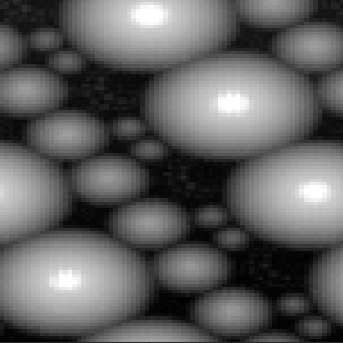}
	\includegraphics[width=0.18\textwidth]{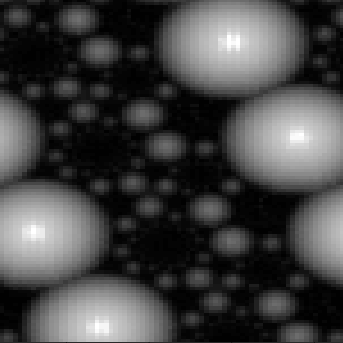}
	\caption{The $F^{(k)}$ lattices for $k=3,4,5$ and computed $\partial \Gamma_k$. 	
	%White is 0 and black is -1 under the 
	%matrix paramterization $M_k(a,b,c) = \frac{1}{2(k-1)} \begin{bmatrix}
	%(k-1)   (c+a) &  b \\ 
	%b & (c-a)   
	%\end{bmatrix}
	%$.
} \label{fig:fk_lattices}
\end{figure}

The general characterization of $\Gamma_k$ seems to require both a recursive construction of the odometers
for all circles in a Kleinian bug as in \cite{levine2017apollonian} and all rational points on an infinite family of inequivalent hyperbolas. 
For example, we have explicitly computed in Figure \ref{fig:f3_lattice} odometers for some of the largest circles appearing in $\partial \Gamma_3$.  Each pair of overlapping circles generates a new hyperbola which we must check contains a dense family of odometers.

We leave the possibility of more detailed investigations of $\Gamma_k$ for future work. From here onwards, we focus solely on $\Gamma_2$
and revert to writing $\Gamma_F$. 
\begin{figure}
	\includegraphics[width=0.32\textwidth]{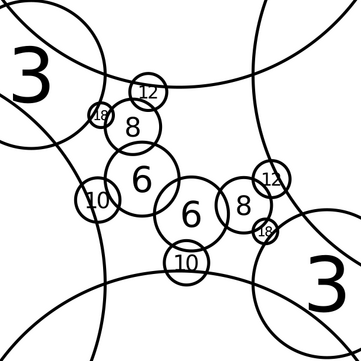} 
		\includegraphics[width=0.45\textwidth]{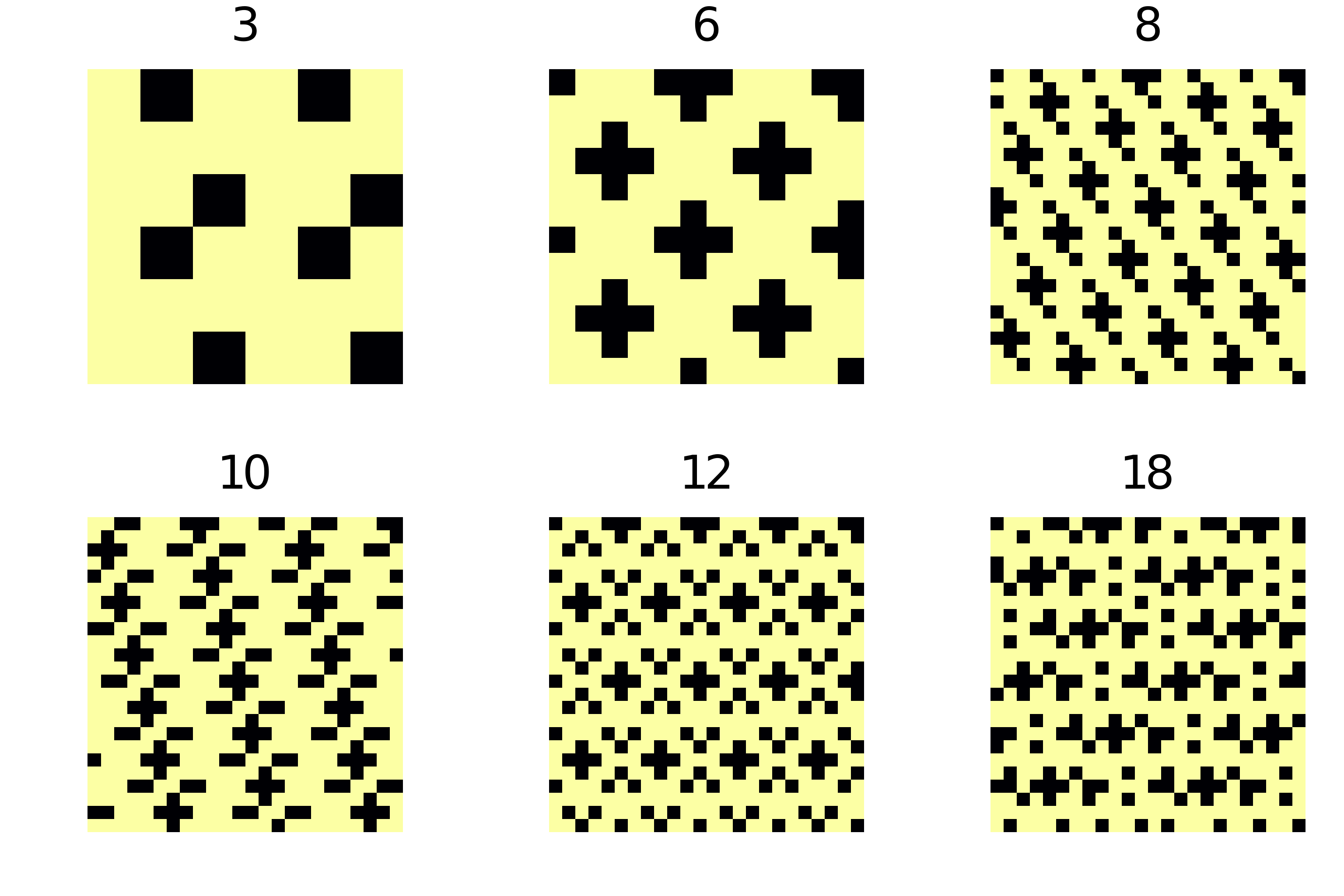}
	\caption{The seven largest circles in a period of $\partial \Gamma_3$. Periods of the Laplacians of odometers for the indicated circles on the left are displayed on the right, black is -1 and yellow is 0. Note that the four bordering largest circles have Laplacian identically 0 and correspond to harmonic functions built from \eqref{eq:turan_odometers}.} \label{fig:f3_lattice}
\end{figure}

\subsection{Code}
This paper presents a recursive algorithm to compute standard and alternate tile odometers on the $F$-lattice. A Julia implementation of this algorithm is included in the arXiv upload. 

\subsection*{Acknowledgments}
Thank you to an anonymous referee for careful, detailed comments on a previous version of this manuscript. Thank you to Charles K. Smart for valuable feedback throughout this project. A.B. was partially supported by NSF grant DMS-2202940 and Charles K. Smart's NSF grant DMS-2137909.

\section{Proof outline and comparison to previous work} \label{sec:proof_overview}
\begin{figure}
	\includegraphics[scale=0.35]{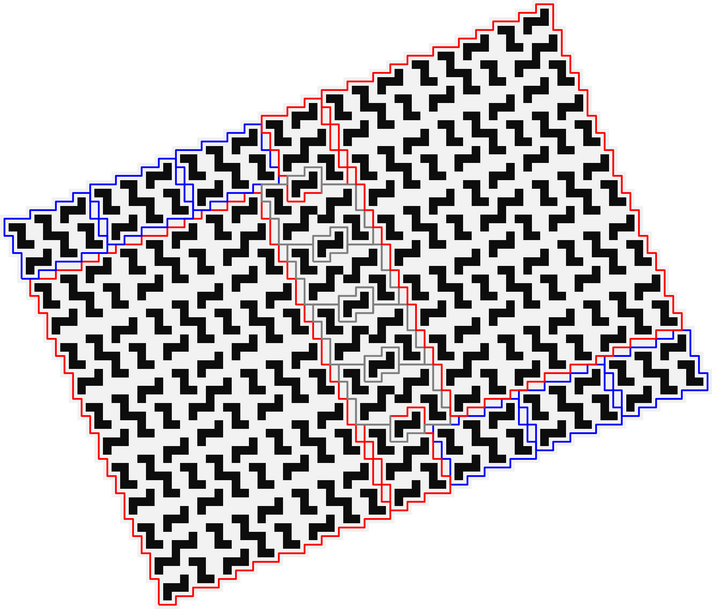}
	\includegraphics[scale=0.31]{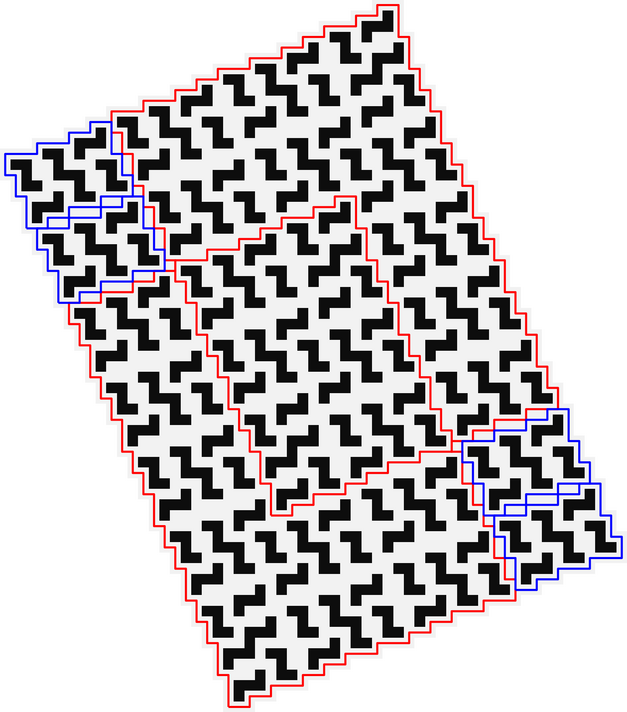}
	\caption{The Laplacian of two standard tile odometers corresponding 
		to the Farey pair $(13/32, 15/37)$; the left and right displays are the odd and even child respectively. Gray is 0 and black is -1. The Laplacian of the standard odd-even ancestor odometers and alternate odometers are outlined in blue, red, and gray respectively. The even odometer decomposes perfectly into four-two copies of the odd-even standard odometers for the parent Farey pair, $(2/5,11/27)$. In particular, two copies of the even parent overlap perfectly on a copy of an even grandparent.
The odd odometer does not have a perfect decomposition into parents or grandparents; the decomposition requires multiple copies of the standard and alternate odometers of the distant ancestor pair $(2/5, 3/7)$.  } \label{fig:need_two_odometers}
\end{figure}

Our method at a high level follows the program of \cite{levine2017apollonian}: 
the proof recursively constructs odometers which then identify $\Gamma_F$. The implementation of this program,
however, requires several new ideas, the most significant being the recursive algorithm itself. Moreover,  
our techniques - in particular the zero-one boundary string construction -  may extend to other lattices. 

In order to make the comparison, we briefly recall Levine-Pegden-Smart's construction in \cite{levine2017apollonian}. 
On the square lattice, odometers were built by first specifying a {\it tile odometer}, a
function with a finite domain, and then extending that function via a periodicity condition like \eqref{eq:periodicity} above. 
Levine-Pegden-Smart's construction associates tile odometers to circles in an an Apollonian band packing. Recall that
Apollonian packings can be drawn by starting with a triple of mutually tangent circles and then recursively filling in Soddy circles \cite{graham2005apollonian}. Each circle in a packing is then part of a {\it Descartes quadruple} 
of pair-wise mutually tangent circles - thus every circle (other than the initial three) has a unique triple 
of parent circles. Levine-Pegden-Smart build tile odometers following this - the recursion starts with a simple formula for the largest circles in a band packing and then builds each child odometer by gluing together two copies each of the three parent odometers in a specified way. 

In our setting, the Apollonian band packing is replaced by reduced rationals $0 \leq t \leq 1$ lying on a hyperbola $\mathcal{H} = \{ (t,c) \in \R \times \R^- : t^2 + (1-t)^2 = c^2\}$. The 
rational recursion is Farey-like but parity aware. That is, all {\it odd} and {\it even} 
reduced rationals - those whose numerator and denominator sum to an odd and even integer respectively - are grouped 
together into unique odd-even {\it Farey pairs}.  The initial Farey pair is $(0/1, 1/1)$ and subsequent pairs are produced via a modification
of the mediant operation and parent-child rotation; the rational recursion produces a ternary tree of unique {\it Farey quadruples},
a grouping of child and parent Farey pairs. We use this tree structure to recursively produce tile odometers. 

A major difference beyond this is that we build for each reduced rational in a Farey pair not one but {\it two} distinct odometers. 
If the recursive algorithm attempted to use only one of the two odometers, it would get stuck -  see Figure \ref{fig:need_two_odometers}.
(This can be thought of as coupling one odometer to each of the two intersecting downwards paths in Figure \ref{fig:f_recursion}.)
The construction also requires ancestor odometers which are arbitrarily far up the recursive tree. 
Moreover, although the function domains, the {\it tiles}, constructed are 180-degree symmetric, 
the tile odometers are not even centrally symmetric, leading to a blow-up in the the number of cases the algorithm must consider. 

For these reasons and more, proving correctness of the recursive algorithm presents new technical challenges. 
A notable one being distant ancestor dependence precludes a finite step inductive proof. 
We address this by augmenting the recursion and associating a binary {\it boundary string} to each odometer.
These strings encapsulate certain compatibility properties across the recursive tree and show it is possible to glue 
distinct tile odometers together in a well-defined way.   These strings allow us to run, in some sense, an analogue of the Euclidean algorithm.

Our proof that the functions which we construct are recurrent also differs from the 
corresponding proof on the square lattice. There, the odometers were shown to be {\it maximal}, a property
strictly stronger than recurrent. Roughly, an integer superharmonic function is maximal if no other integer superharmonic function grows faster than it. Levine-Pegden-Smart showed that their constructed odometers were maximal using the fact that their Laplacians have a `web of 0s', an infinite connected subgraph of 0s. In our case, there is no such web (which uses $F$-lattice edges)
and no hyperbola odometer is maximal. Another technical difference is that the tiles which we construct do not tile $\Z^2$ - they `almost' do but this is fortunately sufficient for our arguments. 

To summarize, our proof proceeds as follows. 
\begin{enumerate}
\item Identify a Farey-like recursion on reduced fractions $t = n/d$ which is dense on a hyperbola and tracks the parity of $(n+d)$. 
\item Pair each reduced fraction with a binary word which records how it was generated.
\item Associate to each such word a {\it boundary string} which carries additional function and domain data.
\item Augment the rational recursion to produce two distinct tile odometers, a standard and an alternate
by piecing together combinations of earlier standard and alternate odometers.
\item Show the recursion is well-defined by reducing every interface
into a pair of boundary strings. 
\item Prove that the functions constructed are recurrent and have the correct growth.  
\end{enumerate}
We start in Section  \ref{sec:hyperbola} by precisely defining the modified Farey recursion on the hyperbola. 
We then prove a technical `almost' tiling lemma in Section \ref{sec:almost_tilings}; this is 
later used to show that tile odometers extend periodically to cover space. 
Then in Section \ref{sec:zero-one_strings}, we introduce and analyze a recursion on binary words
which supplements the hyperbola recursion.  There we also associate degenerate function and tile data, boundary strings, to each such word.

Then, in Section \ref{sec:degenerate_cases} we prove Theorem \ref{theorem:odometers} for a special family of reduced fractions. 
In particular, this family is simple enough that we are able to provide explicit formulae for the tile odometers. 
This forms the base case for the general recursion. 
In Section \ref{sec:odometers} we then introduce a weak form of the recursion which essentially builds only the boundary of tile odometers. We show that these boundaries consist of exactly the boundary strings from Section \ref{sec:zero-one_strings}. The full recursion is completed
in Section \ref{sec:l_correction} where we show the interior of tile odometers can be filled in either by immediate parents or by a chain of distant ancestors.  We conclude in Section \ref{sec:global_odometers} by showing that both standard and alternate tile odometers can be extended in a way that give the desired growth and recurrence.

\section{Hyperbola recursion} \label{sec:hyperbola}
We specify a modified Farey recursion for rational matrices lying on the hyperbola $\mathcal{H} := \{(t,c) \in [0,1] \times \R^-: t^2 + (1-t)^2 = c^2\}$. We also prove that the recursion is invariant with respect to a certain rotation of matrix space. As is later shown, this rotational invariance is maintained in the general recursion and can be leveraged to simplify the proofs of correctness. 

\subsection{Matrix and lattice parameterization} \label{subsec:matrix_param}
Recall the map $M:S_2 \to \R^3$ 
\begin{equation} \label{eq:matrix_param}
M(a,b,c) = \frac{1}{2} \begin{bmatrix} c + a & b \\
b & c-a 
\end{bmatrix}
\end{equation}
and the hyperbola matrices in the statement of Theorem \ref{theorem:hyperbola}, $M(t, 1-t, -\sqrt{t^2 + (1-t)^2})$. 
By solving for the intersection point of rank 1 perturbations of two adjacent cones and then subtracting a matrix corresponding to the quadratic growth of a harmonic polynomial, we can label $(t,c) \in \Q^2 \cap \mathcal{H}$ by 
\begin{equation} \label{eq:hyperbola_param}
f(n,d) := \frac{1}{T(n,d)} ((d^2-n^2), -(d^2+n^2)),
\end{equation}
which has corresponding matrix 
\[
M(n,d) := \frac{1}{T(n,d)} \begin{bmatrix} -n^2 & d n \\ d n & -d^2 \end{bmatrix}, 
\]
where $T(n,d) := (d^2 + 2 d n - n^2)$. Another computation shows that $(n,d) \to(d-n,n+d)$ is a rotation of $S_2$ by: $(a,b) \to (b,a)$. We return to these rotations in Section \ref{subsec:rotation} once we have defined the rational recursion.

As indicated in the introduction, we consider the lattice 
\begin{equation} \label{eq:modified_lattice}
L'(n,d) = 
\begin{cases}
\Z \mathfrak{v}_{n/d,1} + \Z \mathfrak{v}_{n/d,2} &\mbox{if $(n+d)$ is even } \\
2 \Z \mathfrak{v}_{n/d,1} + \Z \mathfrak{v}_{n/d,2}  &\mbox{if $(n+d)$ is odd}.
\end{cases}
\end{equation}
where
\begin{equation} \label{eq:translation_terms}
\mathfrak{v}_{n/d, 1} := \begin{bmatrix} d  \\ n \end{bmatrix}
\qquad
\mathfrak{v}_{n/d,2}  := \begin{bmatrix} n -d \\ n + d \end{bmatrix}.
\end{equation}
Setting $a_{n/d,i} := M(n,d) \mathfrak{v}_{n/d,i}$, we note
\begin{equation} \label{eq:affine_terms}
a_{n/d,1} :=  \begin{bmatrix} 0  \\ 0 \end{bmatrix}
\qquad
a_{n/d,2} := \begin{bmatrix} n \\ -d \end{bmatrix} \, . 
\qquad
\end{equation}
We first observe that $\mathfrak{v}_{n/d,1}$ and $\mathfrak{v}_{n/d,2}$ generate the lattice of the matrix $M(n,d)$. 
\begin{lemma}
	For each reduced fraction $t = n/d$,  
	\[
	L(n,d) := \{ v \in \Z^2 :  M(n,d) v \in \Z^2 \} = \Z  \mathfrak{v}_{n/d,1} + \Z \mathfrak{v}_{n/d,2}.
	\]
\end{lemma}
\begin{proof}
	Suppose $M(n,d) x = y$ for $x, y \in \Z^2$.  For convenience, write $\mathfrak{v}_{n/d, i} =: v_i$. Since $\R^2 = \R v_1 + \R v_2$, we may write 
	\[
	x = c v_1 + c' v_2
	\]
	for $c,c' \in \Q$. We show that $c,c'$ must be in $\Z$, starting with $c'$. By \eqref{eq:affine_terms}, 
	\[
	M(n,d) x = c M(n,d) v_1 + c' M(n,d) v_2 = c' a_{n/d,2}
	\]
	where by supposition
	\[
	c' a_{n/d,2} := \begin{bmatrix} z_1 \\ z_2 \end{bmatrix},
	\]
	for integers $z_1,z_2$. Since $\gcd(n,-d) = 1$, by Bezout's identity, there exists $w_1, w_2 \in \Z$ so that
	\[
	w_1 n - w_2 d = 1.
	\]
	Multiplying the above expression by $c'$, 
	\[
	w_1 z_1 + w_2 z_2 = c',
	\]
	in particular, since the left-hand side is integer-valued, $c' \in \Z$. The exact same argument
	then shows that $c \in \Z$ once we observe $c v_1 = x - c' v_2$ is integer valued. 
\end{proof}

We then check that the map in \eqref{eq:hyperbola_param} is indeed dense in $\mathcal{H}$
by noting it is dense in the first output. 
\begin{lemma}
	 $\frac{d^2-n^2}{d^2 +2 d n - n^2}$ is dense in $[0,1]$ for reduced fractions $0 \leq 	n/d \leq 1$.
\end{lemma}
\begin{proof}
	Suppose $0 < n/d < 1$ and rewrite
	\[
	\frac{d^2-n^2}{d^2 +2 d n - n^2} = 1 - \frac{1}{1 + \frac{1}{2}(\frac{d}{n} - \frac{n}{d})}.
	\]
	Conclude after observing that $\frac{d}{n} - \frac{n}{d}$ is dense in $[0,\infty)$. 
\end{proof}

\subsection{Modified Farey recursion} \label{subsec:modified_farey_recursion}
As evident from \eqref{eq:modified_lattice}, the recursion which we specify must be parity-aware. To that end, 
we say a reduced fraction $n/d$ is {\it even} if $n + d$ is even and otherwise is {\it odd}. We exhibit a modified Farey recursion which generates all rationals in $[0,1]$ and associates to each rational a unique set of odd-even parents and a sibling of the opposite parity. 

An odd reduced fraction $p = o_n/o_d$ and an even reduced fraction $q = e_n/e_d$ produce an odd-even child pair by
\begin{equation} \label{eq:child_operator}
\mathcal{C}(p,q) := \left( \frac{e_n+o_n}{e_d+o_d}, \frac{2 o_n + e_n}{2 o_d + e_d} \right). 
\end{equation}
A quadruple of reduced rationals, $(p_1, q_1, p_2, q_2)$ is a {\it Farey quadruple} if $p_1,q_1 = \mathcal{C}(p_2,q_2)$, $p_2$ is odd, and $q_2$ is even. Each odd-even pair in a Farey quadruple is a {\it Farey pair},  the second pair are the {\it Farey parents} 
of each child in the first pair. A Farey quadruple $(p_1,q_1, p_2,q_2)$
produces three children
\begin{equation} \label{eq:child_farey_quad}
\begin{aligned}
\mbox{Type 1: $\mathcal{C}_1$} \qquad & \left(\mathcal{C}(p_1,q_1), p_1, q_1 \right)  \\
\mbox{Type 2: $\mathcal{C}_2$} \qquad & \left(\mathcal{C}(p_1,q_2), p_1, q_2 \right)  \\
\mbox{Type 3: $\mathcal{C}_3$} \qquad  & \left(\mathcal{C}(p_2,q_1), p_2, q_1\right). 
\end{aligned}	
\end{equation}
The {\it modified Farey recursion} begins with the base quadruple
\begin{equation} \label{eq:base_quadruple}
\q_{()} = \left( \frac{1}{2}, \frac{1}{3}, \frac{0}{1}, \frac{1}{1} \right) 
\end{equation}
and generates descendants which are labeled by {\it recursion words} in the free monoid $F_3^{*}$ generated by $\{1,2,3\}$. The empty word $\{\}$ corresponds to the base quadruple. Each letter in a recursion word corresponds to the type of children chosen in each step. 
For example $\q_{(12)}$ refers to the resulting quadruple after taking the Type 1 children of the root, then the Type 2 children. 

We will also use regex notation: $w = *w'$ for $w' \in F_3^*$ refers to any recursion word $w \in F_3^*$ which ends in $w'$.  
The notation $s^k$ refers to $s \in F_3^*$ concatenated $k$ times, \eg, $3^2 = 3*3$. 

Here is the connection to the usual, vanilla Farey recursion. Recall that the vanilla Farey sequence of order $n$ consists of all 
reduced fractions of denominator at most $n$ between 0 and 1. If $a/b$ and $c/d$ are neighboring terms
in a vanilla Farey sequence of order $n$, then the first term which appears between them in a later sequence
of order $m > n$ is the mediant, $p = \frac{a + c}{b + d}$.  We refer to ($a/b$, $c/d$) as the {\it vanilla Farey 
parents} of $p$ while $p$ is the {\it vanilla Farey child} of {\it vanilla Farey neighbors} ($a/b$,  $c/d$). 
We then observe that \eqref{eq:child_operator} is simply two steps of the vanilla Farey recursion.

%The modified Farey recursion associates to each reduced fraction $0 < n/d < 1$ two unique Farey parents 
%$a/b < p < c/d$ which are neighbors in some Farey sequence.This is immediate after observing 
%that \eqref{eq:child_operator} is simply two steps of the vanilla Farey recursion which ensure that each child has an even parent and
%an odd one. 

\begin{lemma} \label{lemma:recursion-unique} 
	The modified Farey recursion generates unique Farey quadruples in reduced form 
	\[
	\q = (p_1, q_1, p_2,q_2) =  \left( \frac{e_n+o_n}{e_d+o_d}, \frac{2 o_n + e_n}{2 o_d + e_d}, \frac{o_n}{o_d}, \frac{e_n}{e_d} \right), 
	\]
	in particular, $p_1,p_2$ are odd, $q_1, q_2$ are even and $\q$ is a Farey quadruple. 
\end{lemma}
\begin{proof}
	This follows once we inductively check that $(p_2, q_2)$ are vanilla Farey neighbors with vanilla Farey child $p_1$ and 
	$(p_1, p_2)$ are vanilla Farey neighbors with vanilla Farey child $q_1$. That is, by induction,
	$(p_1, q_1)$, $(p_1, q_2)$, and $(p_2, q_1)$ are each pairs of neighbors in some vanilla Farey sequence
	and thus each child has a unique set of Farey parents. 
\end{proof}

Lemma \ref{lemma:recursion-unique} shows that the recursion defines a ternary tree of Farey quadruples. 
Each node in the tree has 3 outgoing edges corresponding to the three types of children. For later reference let $\mathcal{T}_n$, 
denote the set of all Farey quadruples associated to words of length exactly $n$ and denote the full tree by 
\begin{equation} \label{eq:full_farey_tree}
\mathcal{T} = \bigcup \mathcal{T}_n .
\end{equation}

\subsection{Rotational symmetry reduction}  \label{subsec:rotation}
As noted previously in Section \ref{subsec:matrix_param}, the following operator
\begin{equation} \label{eq:rotation_tuples}
\mathcal{R}\left(n,d\right) = (d-n,n+d)
\end{equation}
rotates $\partial \Gamma_F \cap \mathbb{Q}^2 \cap \mathcal{H}$. The goal of this section is to show that an extension of $\mathcal{R}$
to Farey quadruples preserves the depth of the modified Farey recursion.  We start by observing a parity flipping property of $\mathcal{R}$. 
\begin{lemma} \label{lemma:rot_parity-laplace}
	If $0 \leq n/d \leq 1$ is an even reduced fraction then $\gcd( (d-n)/2,(n+d)/2) = 1$, otherwise 
	$\gcd(d-n,n+d) = 1$. Therefore, in the even case, the reduction of $\frac{d-n}{n+d}$ is odd and vice versa. 
\end{lemma}
\begin{proof}
	We split the proof into two steps. 
	
	{\it Step 1.} We check the first claim.
	By the Euclidean algorithm, 
	\[
	\gcd(n,d) = \gcd(d-n,n) =1,
	\]
	and, 
	\begin{align*}
	\gcd( n+d, d-n) &= \gcd( (n+d) - (d-n), d -n) \\
	&= \gcd( 2 n, d - n).
	\end{align*}
	If $(n+d)$ is even or odd, then $(d-n)$ is respectively even or odd. 
	By Bezout's identity, there exist integers $a_i, b_i$ so that
	\begin{align*}
	a_1 (d-n) + b_1 2 &= c  \\ 
	a_2 (d-n) + b_2 n &= 1
	\end{align*}
	where $c$ is 1 if $(d-n)$ is odd and 2 otherwise. Multiplying the above two expressions together shows
	\[
	a' (d-n) + b' ( 2 n ) = c, 
	\]
	for integers $a',b'$. If $c = 1$, this implies $\gcd(d-n, 2n) = 1$. Otherwise, since both $(d-n)$ and $2n$ are even, $\gcd(n+d,d-n) = \gcd(d-n, 2n) = 2$,
	concluding this step. 
		
	{\it Step 2.}
	If $(n+d)$ is odd, then Step 1 shows $ \frac{d-n}{d+n} $ is in reduced form and therefore is even. 
	Otherwise, reduce $\frac{d-n}{d+n} = \frac{ (d-n)/2}{ (d+n)/2}$ and note $( d-n + d + n)/2 = d$. Since $(n+d)$ is even and $\gcd(n,d) = 1$, both $n$ and $d$ must be odd, concluding the proof.
\end{proof}
In light of Lemma \ref{lemma:rot_parity-laplace}, we extend $\mathcal{R}$ to act on reduced fractions $n/d$ by: 
\begin{equation} \label{eq:rotation_reduced_tuples}
\mathcal{R}(n,d) = 
\begin{cases}
(d-n,d+n) &\mbox{ if $n+d$ is odd} \\
(\frac{d-n}{2},\frac{n+d}{2}) &\mbox{otherwise.}
\end{cases}
\end{equation}
In an abuse of notation, we sometimes write $\mathcal{R}(n/d) = n'/d'$ instead.
We extend $\mathcal{R}$ to Farey pairs by $\mathcal{R}(p,q) = (\mathcal{R}(q), \mathcal{R}(p))$ and then component-wise to Farey quadruples. 
Our next two lemmas verify that this is well-defined. 
\begin{lemma} \label{lemma:rot_child_commute}
	If $(p,q)$ is a Farey pair,   $\mathcal{R}(\mathcal{C}(p,q)) = \mathcal{C}(\mathcal{R}(p,q))$.
\end{lemma}
\begin{proof}
	This is a direct computation. 
%	Observe that 
%	\begin{align*}
%	\mathcal{R}(\mathcal{C}(p,q)) &= \left(\mathcal{R}(\frac{2o_n + e_n}{2o_d + e_d} ),  \mathcal{R} (\frac{e_n+o_n}{e_d+o_d}) \right) \\
%	&= \left(\frac{ (2o_d + e_d -2o_n - e_n)/2 }{(2o_d + e_d +2o_n + e_n)/2 },  \frac{e_d+o_d-e_n-o_n}{o_n+e_n+e_d+o_d}  \right) \\
%	&=  \left(\frac{ (o_d - o_n) + (e_d-e_n)/2}{ (o_d + o_n) + (e_d+e_n)/2}, \frac{e_d+o_d-e_n-o_n}{o_n+e_n+e_d+o_d} \right)
%	\end{align*}
%	and
%	\begin{align*}
%	\mathcal{C}(\mathcal{R}(p,q)) &= \mathcal{C}(\mathcal{R}(q), \mathcal{R}(p)) \\
%	&= \mathcal{C}\left(\frac{(e_d-e_n)/2}{(e_d+e_n)/2}, \frac{o_d-o_n}{o_d+o_n}\right) \\
%	&= \left( \frac{ (e_d-e_n)/2 + (o_d - o_n)}{ (e_d+e_n)/2 + (o_d + o_n)}, \frac{ e_d - e_n + o_d - o_n}{o_d + o_n + e_d + e_n} \right)
%	\end{align*}	
\end{proof}

We then show $\mathcal{R}(p)$ is a parent preserving bijection of the recursive tree $\mathcal{T}_n$. 
\begin{lemma} \label{lemma:rot_invariance}
	The following holds for each word of length $n \geq 0$,   $\q_{(w)} = (p_1, q_1, p_2, q_2) \in \mathcal{T}_n$.
	\begin{enumerate}
		\item Rotations flip Type 2 and Type 3 children and preserve Type 1 children, 
		\[
		\mathcal{R}(\q_{(w1)}) =\mathcal{R}(\q_{(w)})_{(1)} \qquad 
		\mathcal{R}(\q_{(w2)}) =\mathcal{R}(\q_{(w)})_{(3)} \qquad
		\mathcal{R}(\q_{(w3)}) =\mathcal{R}(\q_{(w)})_{(2)}.
		\]
		In particular, 
		\[
		\mathcal{R} \circ \mathcal{C}_1 = \mathcal{C}_1 \circ \mathcal{R} \qquad 		\mathcal{R} \circ \mathcal{C}_2 = \mathcal{C}_3 \circ \mathcal{R} \qquad 		\mathcal{R} \circ \mathcal{C}_3 = \mathcal{C}_2 \circ \mathcal{R}.
		\]

	%	\item $p_1$ is the Farey child of $(p_2, q_2)$ and $\mathcal{R}(q_1)$ is the Farey child of $\mathcal{R}(p_2,q_2)$
		\item The rotation preserves depth  $\mathcal{R}(\mathcal{T}_n) = \mathcal{T}_n$.
		
	\end{enumerate}

\end{lemma}
\begin{proof}
	We prove the claims by induction on $n$, the depth of the tree; the base case $n = 0$ can be checked directly. 
	
	{\it Proof of (1).} 
	By definition
	\begin{align*}
	\q_{(w1)} &= ( \mathcal{C}(p_1, q_1), p_1, q_1) \\
	\q_{(w2)} &= ( \mathcal{C}(p_1, q_2), p_1, q_2) \\
	\q_{(w3)} &= ( \mathcal{C}(p_2, q_1), p_2, q_1).
	\end{align*}
	By Lemma \ref{lemma:rot_child_commute}, 
	\begin{align*}
	\mathcal{R}(\q_{(w)})_{(1)} &= ( \mathcal{C} \circ \mathcal{R}(p_1, q_1), \mathcal{R}(p_1, q_1))  \\
	&= (  \mathcal{R} \circ \mathcal{C} (p_1, q_1), \mathcal{R}(p_1, q_1)) \\
	&= \mathcal{R}(\q_{(w1)}).
	\end{align*}
	
	For the other cases, we also use the induction hypothesis. 
	Recall
	\[
	\mathcal{R}(\q_{(w)}) = (\mathcal{R}(q_1), \mathcal{R}(p_1), \mathcal{R}(q_2), \mathcal{R}(p_2)),
	\]
	hence
	\begin{align*}
	\mathcal{R}(\q_{(w)})_{(2)} &=  ( \mathcal{C} (\mathcal{R}(q_1),\mathcal{R}(p_2)), \mathcal{R}(q_1),\mathcal{R}(p_2)) \\
	&= ( \mathcal{C} \circ \mathcal{R}(p_2, q_1), \mathcal{R}(p_2, q_1))  \\
	&= (  \mathcal{R} \circ \mathcal{C} (p_2, q_1), \mathcal{R}(p_2, q_1)) \\
	&= \mathcal{R}(\q_{(w3)}).
	\end{align*}
	The other case is symmetric.	
	
%	{\it Proof of (2).} 
%	By Lemma \ref{lemma:recursion-unique} $(p_2, q_2)$ are vanilla Farey neighbors and $p_1$ is the vanilla Farey child
%	by definition. Also, by (6) and (3) $\mathcal{R}(p_2,q_2)$  are also vanilla Farey neighbors. 
%	Moreover, by Lemma \ref{lemma:rot_child_commute}, $\mathcal{C} \circ \mathcal{R}(p_2, q_2) = \mathcal{R}(p_1, q_1) = (\mathcal{R}(q_1), \mathcal{R}(p_1))$,
%	meaning $\mathcal{R}(q_1)$ is also a Farey child. 
%	
	
	{\it Proof of (2).}
	By the inductive hypothesis, $\mathcal{R}(T_{n-1}) = T_{n-1}$ and by definition, $T_n = \bigcup_{i=1}^3 \mathcal{C}_i \circ T_{n-1} $. 
	Hence, by part (1), 
	\[
	\mathcal{R} \circ T_n = \bigcup_{i=1}^3 \mathcal{R} \circ \mathcal{C}_i \circ T_{n-1} \\
	=  \bigcup_{i=1}^3  \mathcal{C}_i \circ \mathcal{R} \circ T_{n-1} \\
	=  \bigcup_{i=1}^3  \mathcal{C}_i \circ T_{n-1} \\
	= T_n,
	\]
	concluding the proof. 
\end{proof}

We conclude the section by observing that $\mathcal{R}$ is also a rotation of the lattice vectors given by \eqref{eq:translation_terms}.
By identifying $\Z^2$ with $\Z[\I]$ we may write
\begin{equation}
\mathfrak{v}_{n/d,1} = d  + n \I \qquad \mathfrak{v}_{n/d, 2} = (n-d) + (n+d) \I \, . 
\end{equation}

\begin{lemma} \label{lemma:lattice_rotation}
	For odd $p$ and even $q$, 
	\begin{equation}
	\mathfrak{v}_{\mathcal{R}(p), 1} =  - \I \mathfrak{v}_{p,2} \qquad 	\mathfrak{v}_{\mathcal{R}(p),2} = \I 2 \mathfrak{v}_{p,1} 
	\end{equation}
	and
	\begin{equation}
	2 \mathfrak{v}_{\mathcal{R}(q),1} = -\I \mathfrak{v}_{q,2} \qquad \mathfrak{v}_{\mathcal{R}(q),2} = \I \mathfrak{v}_{q,1}
	\end{equation}
\end{lemma}

\begin{proof}
	Note that if $n/d$ is odd, then
	\begin{align*}
	\mathfrak{v}_{\mathcal{R}(n/d), 1} &= (n+d) + (d-n) \I \\
	&=  - \I \mathfrak{v}_{n/d,2} 
	\end{align*}
	and
	\begin{align*}
	\mathfrak{v}_{\mathcal{R}(n/d), 2} &= (d-n) - (n+d) + ( d- n + n + d) \I  \\
	&= 2( -n + d \I) \\
	&= \I 2 \mathfrak{v}_{n/d,1}.
	\end{align*}
	The equation for $n/d$ even follows once we recall $\mathcal{R}$ is an involution,   $\mathcal{R} \circ \mathcal{R}(q) = q$.
\end{proof}

For the remainder of the paper write 
\begin{equation}
v_{n/d,1} = \begin{cases}
2 \mathfrak{v}_{n/d,1} &\mbox{ if $n/d$ is odd} \\
\mathfrak{v}_{n/d,1} &\mbox{ otherwise}
\end{cases}
\end{equation}
and $v_{n/d,2} = \mathfrak{v}_{n/d,2}$.

\section{Almost pseudo-square tilings} \label{sec:almost_tilings}
In this section we prove a technical tiling lemma which will allow us to show that the tiles which we define in the subsequent sections
cover $\Z^2$ periodically. 

We identify $\Z^2$ with $\Z[\I]$. A {\it cell} is a unit square $s_x = \{x, x+1, x+\I, x+1 + \I\} \subset \Z[\I]$ and a {\it tile} 
is a union of cells, which, when viewed as union squares in the plane, is simply connected union of cells and has a boundary which is a simple closed curve. The {\it vertices} of a tile 
are the Gaussian integers on its boundary. Let $(F_2, *)$ be the free group generated by $\{1, \I \}$.
For $w \in F_2$, let $\hat{w}$ denote its involution, \ie, $\hat{w}*w = \{\}$. Let $\rev(w)$ denote the reversal of $w \in F_2$, $w[j]$ the $j$-th letter 
of $w$, and $|w|$ the number of letters in $w$. The {\it boundary word} of a tile is a word $w \in F_2$ which represents a vertex walk around the boundary of the tile in counterclockwise order. In particular, $\sum w = 0$ and $\sum w' \not= 0$ for any non-empty sub-word $w'$ of $w$, where $\sum$ denotes the abelianization of $F_2$.

A {\it tiling} of the plane is an infinite set of translations of a tile $T$ where every cell is contained in exactly 
one copy of $T$. A tiling of $T$ is $(v_1,v_2)$-{\it regular} if every tile $T'$ in the tiling can be expressed as $T + k v_1 + k' v_2$
for $k, k' \in \Z$ and $v_1,v_2 \in \Z[\I]$. That is, the translations of $T$ by $(v_1,v_2)$ {\it generate} the tiling.

\begin{figure}[!b]
	\includegraphics[width=0.6\textwidth]{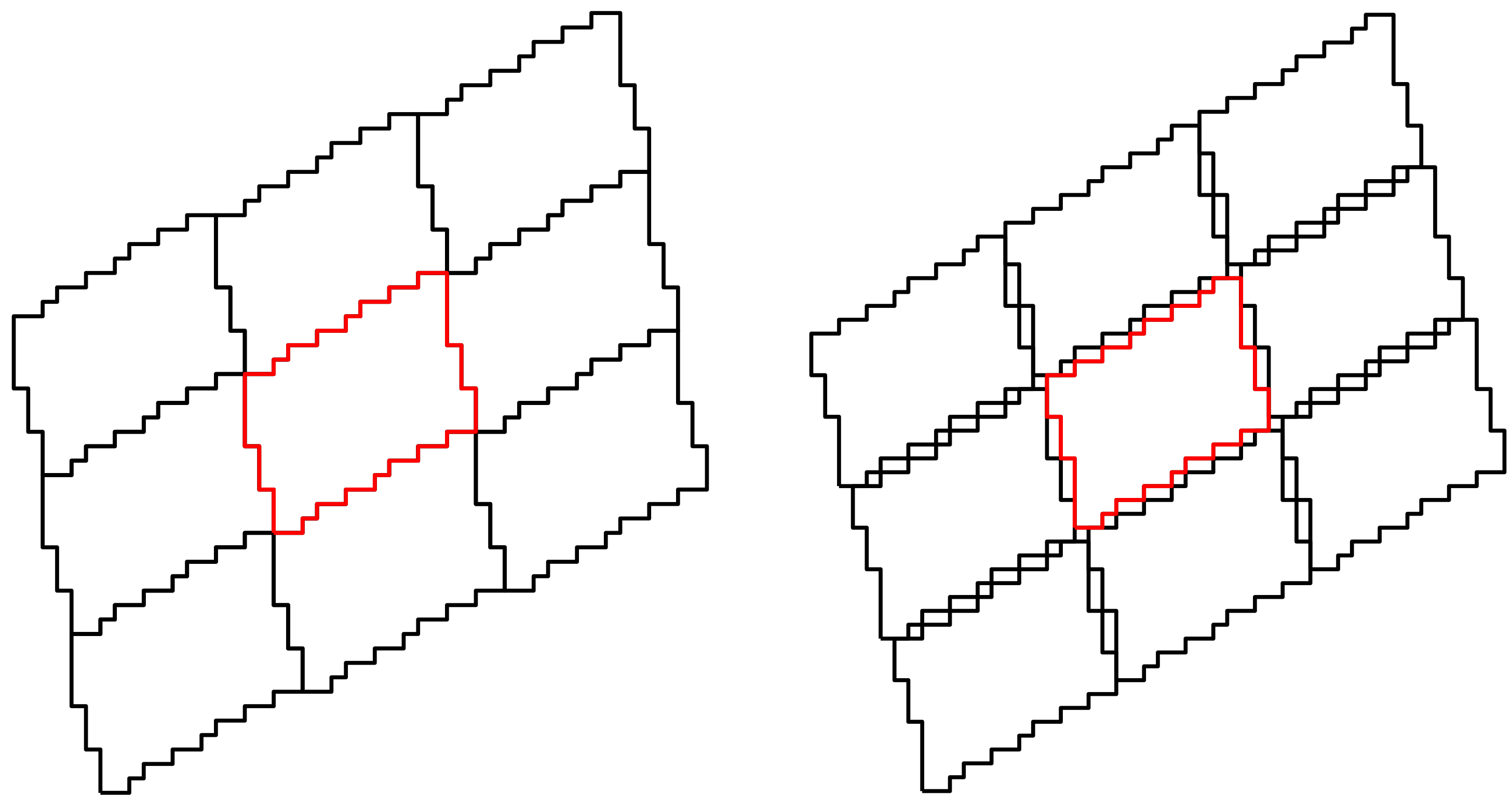}
	\caption{A surrounding of a pseudo-square tiling 
		%	specified by  $w_1 = \{1, 1,i,1,i,1,1,i,1\}$ and $w_2 = \{i,i,i,-1,-1,i,i,i\}$
		and an almost pseudo-square tiling as defined in Proposition \ref{prop:tiling_criteria} and Lemma \ref{lemma:almost_square_tiling} respectively.} \label{fig:almost_tiling}
\end{figure}
Beauquier-Nivat \cite{beauquier1991translating} have a simple criteria for determining if a tile generates a regular tiling. 
Their criteria is expressed in terms of the boundary words of a tile, but can be interpreted geometrically as: a tile generates a regular tiling if it can be perfectly surrounded by copies of itself. We refer to a tiling satisfying the conditions in Proposition \ref{prop:tiling_criteria} as a {\it pseudo-square tiling}. 
\begin{prop}[\cite{beauquier1991translating}] \label{prop:tiling_criteria}
	If the boundary word of a tile, $w \in F_2$, can be expressed as
	\[
	w = w_1 * w_2 *  \hat{w}_1 * \hat{w}_2,
	\]
	then the tile generates a $(\sum w_1, \sum w_2)$-regular tiling.
\end{prop}

In our main argument, we require a technical modification of the notion of tiling in which bounded gaps are allowed. That is, 
we cannot use Proposition \ref{prop:tiling_criteria} and are thus forced to modify it.
An {\it almost} tiling $\mathcal{T}$, is an infinite set of translations of a tile $T$ where every cell is contained 
in at most one tile and every $x \in \Z^2$ is a vertex of a cell in $\mathcal{T}$, \ie,  there is an $s_y \in \text{tile} \in \mathcal{T}$ such that $x \in s_y$. The notion of regular with respect to a lattice is also extended to almost tilings. 

\iffalse
Recall that the dual $T^*$ of a tile $T$ are the centers of the cells $s_x \subset T$ in the dual lattice $(\Z^2)^*$ of $\Z^2$. 
We observe that almost tilings on the dual lattice translate into coverings with overlaps only on the boundary 
\begin{lemma}
	Let $T$ be a tile and suppose $T^*$ generates a $(v_1,v_2)$-regular almost tiling for $v_1, v_2 \in \Z[\I]$. 
	Then, for every $x \in \Z[\I]$ there is at least one and at most two pairs of $k,k' \in \Z$ and $s_c \in T$ so that 
	$x \in s_c + k v_1 + k' v_2$. 
\end{lemma}

\begin{proof}
	This is by definition.
\end{proof}
\fi

We now give a sufficient condition for generating almost tilings. 
Roughly, this condition allows for slight gaps between cells in the surrounding of a tile. 
We will refer to the almost tiling from Lemma \ref{lemma:almost_square_tiling} as an {\it almost pseudo-square tiling}. See Figure \ref{fig:almost_tiling}
for an illustration of this. 
\begin{lemma}\label{lemma:almost_square_tiling}
	Suppose	$w' = w_1 * w_2 * \widehat{\rev(w_1)} * \widehat{\rev(w_2)}$ is the boundary word of a tile $T$. 
	Further suppose the following conditions on $w \in \{ w_1, -\I w_2\}$.
	\begin{enumerate}
		\item 
		Monotonicity: $\{-1, -\I\} \not \in w$
		and $w[1] = w[|w|] = 1$
		\item At least one of the following three cases concerning the form of $w$ and its reversal is satisfied:
		\begin{enumerate}
			\item $w$ is a palindrome, $w = \rev(w)$.			
			
			\item $w = (1*1*\I) * \tilde{w} * 1$, where $\tilde{w}$ is a palindrome.
			Moreover, every $\I$ in $w$ is followed by at least one 1.
			
			\item $w = 1*\tilde{w}*(1*1*1)$ where $\tilde{w}$ is a palindrome. 
			Moreover, every $\I$ in $w$ is followed by at least three 1s.
			
			%old version which doesn't close
			%\item $w = q*\tilde{w}*p$ where $\tilde{w}$ is a palindrome, 
			%\[
			%q = (1*1)^k*(1*\I) \qquad p = (1 *1)^{k+1}*1
			%\]
			%for $k > 0$. Moreover, every $\I$ in $w$ is followed by at least three 1s.
		\end{enumerate}
	\end{enumerate}
	%where 
	%Further suppose $w_1[1] = w_1[|w_1|] = 1$ and $w_2[1] = w_2[|w_2|] = i$. 
	
	Then, $T$ generates a $(\sum w_1+\I, \sum w_2-1)$-regular almost tiling. Moreover, the only tiles in the tiling which share 
	edges with $T$ are $T \pm (\sum w_1+\I)$ and $T \pm (\sum w_2 - 1)$. 
\end{lemma}

\begin{proof}
	Let $(v_1,v_2) = (\sum w_1, \sum w_2$). To show that $T$ generates a $(v_1 + \I, v_2 -1)$-regular almost tiling, by periodicity, 
	it suffices to analyze one surrounding of $T$, 
	\[
	S :=  \bigcup_{|k_1| \leq 1, |k_2| \leq 1} \{ T + k_1 (v_1+ \I) + k_2 (v_2-1)\},
	\]
	see Figure \ref{fig:almost_tiling}. 
	Specifically we show that the closure, 
	\[
	\bar{S} = \{ s_x : x \in S \cap \Z^2\},
	\]
	where each cell $s_x$ is viewed as a unit square in the plane, is simply connected and no two cells in the decomposition of $S$ overlap. 
	
	Observe that the boundary word of $T$ implies it is 180-degree symmetric. Hence, $S$ is 180-degree symmetric and we may reduce to analyzing the interfaces between $T$ and its lower, right, and lower-right neighbors,  
	\[
	T_h := T + v_1 + \I
	\]
	\[
	T_v := T - v_2 + 1
	\]
	\[
	T_d := T + v_1 - v_2 + 1 + \I.
	\]
	We show that the conditions imply no two pairs of edges cross and that every gap in the interface borders a cell of $T$.

	{\it Step 1: The bottom interface} \\
	We start with the bottom interface, $T$ and $T_{v}$. 
	Designate the origin as the lower-left vertex of $T$ so that cells along the bottom of
	$T$ can be labeled by a $w_1$ walk. By the definition and translation offsets, 
	vertices along the top edge of $T_{v}$ can then be labeled by $\rev(w_1)+1$.
	For $j \leq |w_1|$, let $x_j = \sum w_1[1:j]$ and $y_j = 1 + \sum (\rev(w_1)[1:j])$, 
	where $w[1:j]$ represents the first $j$ letters of $w$. In particular, $x_0 =0$ and $y_0 = 1$. 
	
	We now split the argument into three cases depending on the form of $w_1$ as dictated by condition (2). 
	
	{\it Case (a):  $w_1 = \rev(w_1)$} \\
	In this case, $y_j = 1 + \sum w_1[1:j]$ and so
	\begin{equation} \label{eq:bounded_gap}
	x_j = y_j - 1.
	\end{equation}
	Therefore, any vertex $y_j$ along the top edge of $T_{v}$ is distance at most one from $x_j$, the lower left-corner of a cell in $T$. 
	
	To see that the top edge of $T_v$ does not cross above the bottom edge of $T$, we use monotonicity.
	Suppose for sake of contradiction a crossing occurs. Since $w_1 = 1$, $x_1 = y_0$ and therefore there is a first, in the lexicographical order, $j,j' \geq 1$ at which
	$y_j = x_{j'}$ and $y_{j+1} = x_{j'} + \I$. By \eqref{eq:bounded_gap}, $y_j = x_{j} + 1$ and so by monotonicity, $j' = (j+1)$.
	However, by  \eqref{eq:bounded_gap} $y_{j+1} = x_{j+1} + 1 \neq x_{j+1} +  \I$, a contradiction.

	{\it Case (b):} \\
	In this case 
	\begin{align*}
	w_1  &= (1*1*\I)*\tilde{w}*1 \\ 
	\rev(w_1) &= 1*\tilde{w}*(\I*1*1)
	\end{align*}
	for a palindrome $\tilde{w} \in F_2$.
	Therefore, (after remembering the offset of $T_v$)
	\begin{align*}
	x_0 &= 0  \quad x_1 =1 \quad x_2 = 2 \qquad\qquad x_3 = 2 + \I \\
	y_0 &= 1 \quad y_1 = 2 \quad  y_2 = 2 + \tilde{w}[1] \quad y_3 = 2 + \tilde{w}[1] + \tilde{w}[2]
	\end{align*}
	and
	\begin{align*}
	x_{3 + |\tilde{w}|} &= (2 + \I) + \sum \tilde{w}  \quad x_{4 + |\tilde{w}|} = (3 + \I) + \sum \tilde{w} \\
	y_{1+ |\tilde{w}|} &= (2) + \sum \tilde{w} \qquad y_{2+ |\tilde{w}|} = (2+ \I) + \sum \tilde{w} \\
	y_{3+ |\tilde{w}|} &= (3+ \I) + \sum \tilde{w} \quad y_{4+ |\tilde{w}|} = (4+ \I) + \sum \tilde{w}.
	\end{align*}
	Also, by the moreover clause, $y_2 = 3$. 
	Thus, it suffices to consider $1 \leq j \leq 1 + |\tilde{w}|$
	for which the above computations show 
	\begin{equation} \label{eq:bounded_gap_a}
	x_{j+2} = y_{j} + \I.
	\end{equation}
	It remains to show this implies there are no crossings. Suppose for contradiction $y_{j} = x_{j'}$ and $y_{j+1} = y_j + \I$
	but $x_{j'+1} = x_{j'} + 1$. 
	for some $1 \leq j \leq 1 + |\tilde{w}|$. By \eqref{eq:bounded_gap_a}, $y_{j+1} = x_{j+2} = x_{j'} + \I$. By monotonicity, $j' = j+1$
	and so $x_{j'+1} = x_{j'}+ \I$, a contradiction.  
	
	{\it Case (c):} \\
	In this case, 
	\begin{align*}
	w_1  &= 1*\tilde{w}*(1*1*1) \\ 
	\rev(w_1) &= (1*1*1)*\tilde{w}*1
	\end{align*}
	for a palindrome $\tilde{w} \in F_2$.
	Thus, 
	\begin{align*}
	x_0 &= 0  \quad x_1 = 1  \quad x_2 = 1+\tilde{w}[1]\\
	y_0 &= 1 \quad y_1 = 2 \quad  y_2 = 3 \quad y_3 = 4
	\end{align*}
	and 
	\begin{align*}
	x_{1+|\tilde{w}|} &= 1 + \sum \tilde{w}  \quad x_{1+|\tilde{w}|+z} = (1+z) + \sum \tilde{w} \qquad \mbox{ for $z \leq 3$}\\
	y_{3+|\tilde{w}|} &= 4 + \sum \tilde{w}  \quad y_{4 + |\tilde{w}|} = 5 + \sum \tilde{w}.
	\end{align*}
	We note that for all $2 \leq j \leq |\tilde{w}| + 3$, 
	\begin{equation} \label{eq:bounded_gap_b}
	y_{j} = x_{j-2} + 3.
	\end{equation}
	Indeed, $x_{1+z} = 1 + \sum \tilde{w}[1:z]$ and $y_{3 + z} = 4 + \tilde{w}[1:z]$ for $z \leq |\tilde{w}|$. 
	
	We claim that this together with the moreover clause implies no gaps of size larger than 1. Indeed, if $y_j = x_{j-2} + 3$, then 
	$x_{j} = x_{j-2} + 1 + (1 \mbox{ or } i)$. In the first case, we are done. In the second case, $x_{j+2} = x_{j-2} + 1 + \I + 2$. 
	
	The relation \eqref{eq:bounded_gap_b} also implies no crossings. Indeed, suppose for contradiction $y_{j} = x_{j'}$ and $y_{j+1} = y_j + \I$
	but $x_{j'+1} = x_{j'} + 1$. 
	for some $2 \leq j \leq 2 + |\tilde{w}|$. By \eqref{eq:bounded_gap_b}, $y_{j} = x_{j-2} + 3$ and $y_{j+1} = x_{j-1} + 3$. This 
	implies $x_{j-1} = x_{j-2} + \I$ and hence the moreover clause implies 
	\[
	x_{j+2} = x_{j-1} + 3 = x_{j-2} + 3 + \I = y_{j} + \I = y_{j+1}.
	\]
	Monotonicity then implies $x_{j+1} = x_{j'} = y_j$, but this then contradicts $x_{j'+1} = x_{j'} + 1$.

	{\it Step 2: Conclude} \\
	After rotating, the arguments in Step 1 apply to the interface between $T$ and $T_h$.  We then check $T$ and $T_d$. Let $z_0 = \sum w_1$ and note that the top left vertex of $T_{d}$ is $z_0 + 1 + \I$. By the assumption on the first and last letter of $w_2$, the next vertices on the top and left edges of $T_{d}$ are $z_0 + 2 + \I$ and $z_0 + 1$ respectively while the next vertex on the right edge of $T$ is $z_0 + \I$. 
	This implies that no cell of $T$ overlaps a cell of $T_d$ and that the gap between the two tiles is of unit size. 
	
	Finally, by monotonicity, for any other pair of cells in $S$ to overlap, there must first be a crossing on the horizontal or vertical edges which we have just shown to be impossible. 	
\end{proof}

\begin{remark}
	Our usage of Lemma \ref{lemma:almost_square_tiling} is not strictly necessary and may be replaced by an appropriate application of Lemma \ref{lemma:zero-one-fixed-offsets}
	below. We included it as we believe it makes the overall proof easier to follow. It may also be of independent interest.
\end{remark}

%Observe that the boundary word condition on $T$ supposes the tile is 180-degree symmetric. 

\section{Zero-one boundary strings} \label{sec:zero-one_strings}
In this section we begin to associate additional data to the hyperbola recursion.

\subsection{A recursion on binary words}

We associate to each reduced fraction in the modified Farey recursion a binary word
and expose some basic properties. Specifically, given any initial Farey pair
$(p,q)$ we associate each descendant to a {\it binary word}, a word in the alphabet generated by the two letters, $\{p,q\} \in F_2^*$,
by augmenting the recursion. 

Given a recursion word $w \in F_3^*$ and two binary words $p_t, q_t \in F_2^*$ we extend the child operator in \eqref{eq:child_operator}
to pairs of binary words by 
\begin{equation}
\mathcal{C}_{(w)}(p_t, q_t) = 
\begin{cases}
(q_t q_t p_t, q_t p_t) &\mbox{ if $\sum 1\{w_j = 1\}$ is even } \\
(p_t q_t q_t, p_t q_t) &\mbox{ otherwise}.
\end{cases}
\end{equation}

Let $w_0 \in F_3^*$ be the first word for which the Farey pair $(p,q)$ appear as Farey children and let $\mathbf{Q}_{(w_0)} = ( \mathcal{C}_{w_0}(p,q), p, q)$
be the initial binary word quadruple with each term in $F_2^*$. Then, recursively, given $w \in F_3^*$ and $\mathbf{Q}_{(w)} = (p_{t+1}, q_{t+1}, p_t, q_t)$,
each child binary word quadruple is defined by 
\begin{equation} \label{eq:child_farey_quad_continuedfrac}
\begin{aligned}
\mathbf{Q}_{(w*1)}  &= \left(\mathcal{C}_{(w*1)}(p_{t+1},q_{t+1}),  p_{t+1}, q_{t+1} \right)  \\
\mathbf{Q}_{(w*2)}  &= \left(\mathcal{C}_{(w*2)}(p_{t+1},q_t), p_{t+1}, q_t \right)  \\
\mathbf{Q}_{(w*3)}  &= \left(\mathcal{C}_{(w*3)}(p_t,q_{t+1}), p_t, q_{t+1} \right).
\end{aligned}	
\end{equation}
In particular, $\mathbf{Q}_{(w)}$ is only defined for words which are extensions of $w_0$.

Recall that a palindrome $\tilde{w} \in F_2^*$ is a word that is equal to its reversal, $\tilde{w} = \rev(\tilde{w})$. 
An {\it almost palindrome} is a word $w = s_1 * \tilde{w} * s_2$, where $s_1,s_2 \in \{p,q\}$ are two letters
and $\tilde{w}$ is a palindrome.  Write $w[a:b]$ for the subword starting at the $a$-th letter of $w$ 
and ending at the $b$-th letter.

\begin{lemma} \label{lemma:almost_palindrome}
	For every $w_0 \in F_3^*$ and $t \geq 1$, the following holds for every pair of binary words, $(p_t,q_t)$, produced by \eqref{eq:child_farey_quad_continuedfrac}.
	\begin{enumerate}
		\item Both $p_t$ and $q_t$ are almost palindromes.
		\item If $\sum 1\{ {w_{0}}_j = 1\}$ is even then $p_t$ and $q_t$ begin with $q$ and end with $p$
		and otherwise begin with $p$ and end with $q$. 
		\item Let $n = \min(|p_t|, |q_t|)$, $m = \min(|p_t|, 2|q_t|)$, and $w$ denote the current word. Then, 
		if $\sum 1\{ w_j = 1\}$ is even,
		\[
		p_t[2:n] = \rev(q_t)[2:n] \qquad p_t[2:m] = \rev(q_tq_t)[2:m],
		\]
		otherwise,
		\[
		\rev(p_t)[2:n] = q_t[2:n] \qquad \rev(p_t)[2:m] = (q_tq_t)[2:m].
		\]

	\end{enumerate}
\end{lemma}

\begin{proof}
	We may suppose without loss of generality that $\sum 1\{w_{{0}_j}=1\}$ is even, otherwise, reverse the subsequent statements. 
	
	Let $\mathbf{Q}_{(w)} = (p_{t+1}, q_{t+1}, p_{t}, q_{t})$ be given and we will verify claims (1) and (2) for 
	the child Farey pair and claim (3) for the parent Farey pair  in the quadruple
	\[
	\mathbf{Q}_{(w')} = (p_{t+2}, q_{t+2}, p'_{t+1}, q'_{t+1})
	\]
	defined by \eqref{eq:child_farey_quad_continuedfrac}.  %Suppose $\sum w_j$ is even, otherwise reverse each of the expressions below. 
	To do so, we must eliminate the degenerate cases $p'_{t+1} = p$ or $q'_{t+1} = q$. Fortunately, 
	these can only occur if $w' = 3^k$ or $w' = 2^k$ for $k \geq 0 $ respectively --- in which case $w_0 = ()$. An induction shows that
	\begin{equation} \label{eq:words_degenerate_cases}
	\begin{aligned}
	\mathbf{Q}_{(3^k)} &= (q p^k q p^{k+1}, q p^{k+1}, p, q p^k) \\
	\mathbf{Q}_{(2^k)} &= (q^{2(k+1)} p, q^{2k + 1} p, q ^{2 k} p , q),
	\end{aligned}
	\end{equation}
	and we can verify the claim directly in these cases by inspection. We can then use \eqref{eq:words_degenerate_cases} to also handle the cases $w' = 3^k*\{1 \mbox{ or } 2\}$ or $w' = 2^k*\{1 \mbox{ or }3\}$. Indeed, we compute, using  \eqref{eq:words_degenerate_cases}, that 
	\[
	\begin{aligned}
	\mathbf{Q}_{(3^k*1)}(1) &= q (p^k q p^{k+1} q p^{k+1} q p^{k} )p \\
	\mathbf{Q}_{(3^k*1)}(2) &= q (p^k q p^{k+1} q p^{k}) p 
	\end{aligned}
	\]
	and
	\[
	\begin{aligned}
	\mathbf{Q}_{(3^k*2)}(1) &= q (p^k q p^k q p^k q p^k ) p  \\
	\mathbf{Q}_{(3^k*2)}(2) &= q (p^k q p^k q p^k ) p 
	\end{aligned}
	\]
	also
	\[
	\begin{aligned}
	\mathbf{Q}_{(2^k*1)}(1) &=  q ( q^{2 k + 1} p q^{2 k + 1} p q^{2 k + 1} ) p \\
	\mathbf{Q}_{(2^k*1)}(2) &=  q  (q^{2k+1} p q^{2k+1} ) p 
	\end{aligned}
	\]
	and
	\[
	\begin{aligned}
	\mathbf{Q}_{(2^k*3)}(1) &=  q ( q^{2k} p q^{2 k + 1} p q^{2k}) p \\
	\mathbf{Q}_{(2^k*3)}(2) &= q (q^{2 k} p q^{2k}) p.
	\end{aligned}
	\]
	Hence, we may assume none of $p_t,q_t, p_{t+1}, q_{t+1}$ are singletons, that is the induction hypotheses hold for each of them.
	We also suppose $\sum 1\{w_j = 1\}$ is even, the odd case having symmetric arguments. By the induction hypotheses
	\[
	p_t = q w^{1} p \quad q_t = q w^2 p  
	\]
	for palindromes $w^1$ and $w^2$ and so
	\begin{equation} \label{eq:almost_palindrome1}
	\begin{aligned}
	p_{t+1} &=  q w^2 p q w^2 p q w^1 p  \\ %q_t q_t p_t 
	q_{t+1} &=  q w^2 p q w^1 p.  %q_t p_t 
	\end{aligned}
	\end{equation}
	Since $p_{t+1}$ and $q_{t+1}$ are almost palindromes and $w^i = \rev(w^i)$ we have the reversal relations 
	\begin{equation} \label{eq:reversal_relations}
	\begin{aligned}
	w^2 p q w^1 &= w^1 q p w^2 \\%=\rev(w^1) q p \rev(w^2) 
	w^2 p q w^2 p q w^1 &=  w^1 q p w^2 q p w^2. %\rev(w^1) q p \rev(w^2) q p \rev(w^2) 
	\end{aligned}
	\end{equation}
	This implies claim (3),
	\begin{align*}
	\rev(q_{t+1}) &= p w^1 q p w^2 q \\
		p_{t+1} &= q w^1 q p w^2 q p w^2 p \\
		\rev(q_{t+1} q_{t+1}) &= p w^1 q p w^2 q p w^2 p q w^1 q.
	\end{align*}
	For claims (1) and (2),  the possible decompositions of $(p_{t+2}, q_{t+2})$ are 
	\begin{align*}
	( p_{t+1} q_{t+1} q_{t+1}, p_{t+1} q_{t+1})  \qquad &\mbox{ Type 1} \\
	(q_{t} q_{t} p_{t+1}, q_{t} p_{t+1})  \qquad &\mbox{ Type 2} \\
	(q_{t+1} q_{t+1} p_{t}, q_{t+1} p_{t})  \qquad &\mbox{ Type 3}.
	\end{align*}
	The reversal relations \eqref{eq:reversal_relations} together with \eqref{eq:almost_palindrome1} 
	show that each of the decompositions are almost palindromes. We show only the odd Type 1 case 
	as the rest are similar. First, write using \eqref{eq:almost_palindrome1} 
	\[
	p_{t+1} q_{t+1} q_{t+1} =  q w^2 p q w^2 p q w^1 p q w^2 p q w^1 p q w^2 p q w^1 p,
	\]
	and then use \eqref{eq:reversal_relations} to check 
	\begin{align*}
	\rev(w^2 p q w^2 p q w^1 p q w^2 p q w^1 p q w^2 p q w^1) &= [w^1 q p w^2] q p [w^1 q p w^2] q p [w^1 q p w^2 q p w^2] \\
	&=   	[w^2 p q (w^1] q p 	[w^2) p q (w^1 q p [ w^2) p q w^2 p q w^1] \\
	&=  w^2 p q (w^2 p q w^1) p q (w^2 p q  w^1) p q w^2 p q w^1.
	%see july 20.tex for the proof of the rest
	\end{align*}
\end{proof}

\subsection{Basic definitions} \label{subsec:basic_definitions}
We next associate tile and function data to the binary word recursion. But in order to do so, we must recall and modify
some definitions from \cite{levine2017apollonian}. A {\it tile} $T$ is now, depending on the context, either a finite subset of $\Z^2$ or a finite union of simply connected cells. Let $c(T)$ denote 
the lower-left vertex of $T$, specifically, the vertex of $T$ which has the smallest imaginary coordinate followed by the smallest real coordinate. A {\it partial odometer} is a function $h:T \to \Z$. The domain of $h$ is $T(h)$ and $s(h) \in \C$ is the {\it slope} of $T$,
the average of
\begin{equation}
\begin{aligned}
\frac 1 2 \left(h(x+1)-h(x)+h(x+1+\I)-h(x+\I)\right)+\\
\frac{\I}{2}\left(h(x+\I)-h(x)+h(x+1+\I)-h(x+1)\right)
\end{aligned}
\end{equation}
over squares $\{x, x+1, x + \I, x + 1 + \I\} \subset T$. The slope is not defined when $T$ is a singleton. Two partial odometers $o_{1}$ and $o_2$ are {\it translations} of one another if 
\begin{equation}
T(o_{1}) = T(o_2) + v \quad \mbox{ and } \quad o_{1}(x) = o_{2}(x+v) + a^T x + b 
\end{equation}
for some $v,a \in \Z^2$ and $b \in \Z$.  

Partial odometers $o_{1}$ and $o_2$ are {\it compatible} if 
$o_{1} - o_2 = c$ on $T(o_{1}) \cap T(o_2)$ for some {\it offset constant} $c \in \Z$.  As in \cite{levine2017apollonian}, if the offset constant is
0 or the tiles do not overlap then $o_{1} \cup o_2$ is the common extension to $T(o_{1}) \cup T(o_2)$. We recall for later reference the following lemma 
which will allow us to construct global odometers from pairwise compatible partial odometers. 
\begin{lemma}[Lemma 9.2 in \cite{levine2017apollonian}] \label{lemma:pairwise_compatibility}
	If $\mathcal{S} = \{o_i\}$ is a collection of pairwise compatible partial odometers such that $\{T(o_i)\}$ forms an almost pseudo-square
	tiling then there is a function $g:\Z^2 \to \Z$ unique up to adding a constant that is compatible with every $o_i \in \mathcal{S}$. 
\end{lemma}
We remark that Lemma 9.2 in \cite{levine2017apollonian} is stated for a different notion of tiling (hexagonal tiling rather than 
pseudo-square tilings) however the proof carries over verbatim to this case.

\subsection{Even-odd boundary strings}
We now associate additional data to each binary word constructed in the first subsection. The result in this subsection
will form a key tool in verifying correctness of the subsequent tile and odometer recursion. 

We first adapt the notion of {\it boundary string} from \cite{levine2017apollonian} to our setting.
Suppose $T_p$ and $T_q$ are tiles which generate $(v_{p,1},v_{p,2})$ and $(v_{q,1}, v_{q,2})$ 
regular almost pseudo-square tilings respectively. A $q$-$p$ {\it boundary string}
is a collection of tiles $T_i \in \{T_q,T_p\}$ such that
\begin{equation} \label{eq:boundary_string_left_to_right}
c(T_i) - c(T_{i-1}) = \begin{cases}
v_{p,j} &\mbox{ if $T_{i-1} = T_p$} \\
v_{q,j} &\mbox{ if $T_{i-1} = T_q$},
\end{cases}
\end{equation}
for fixed $j \in \{1, 2\}$. A $q$-$p$ {\it reversed boundary-string} is also a collection of tiles $T_i \in \{T_q,T_p\}$ 
but with different offsets:
\begin{equation} \label{eq:boundary_string_left_to_right_reversed}
c(T_i) - c(T_{i-1}) = \begin{cases}
v_{p,j} + (v_{p,j'} - v_{q,j'}) &\mbox{ if $pq$} \\
v_{p,j} &\mbox{ if $pp$} \\
v_{q,j} + (v_{q,j'} - v_{p,j'}) &\mbox{ if $qp$} \\
v_{q,j} &\mbox{ if $qq$},
\end{cases}
\end{equation}
where $(j,j') \in \{ (1,2), (2,1)\}$ is fixed and the right column denotes the tile tuple, \eg, the first row is $(T_{i-1}, T_i) = (T_p, T_q)$. When $j=1$, a boundary string is {\it horizontal} and otherwise is {\it vertical}. We label 
a boundary string $\mathcal{B}_{w}$ by a binary word $w \in F_2^*$ where a superscript $r$ indicates it is reversed. 

A horizontal or vertical {\it stacked boundary string} for $w \in F_2$ is a union of $\{T_i^+\} := \mathcal{B}_w$ and $\{T_i^-\} := \mathcal{B}^r_{\rev(w)}$ both oriented in the same direction. The first tiles $T_1^+$ and $T_1^-$ in each string and the shared direction dictate the relative positions,
\begin{equation} \label{eq:stacked_boundary_string}
c(T_1^+) - c(T_1^-) = v_{n'/d', j'}
\end{equation}
where $j' \in \{1,2\}$ is the perpendicular direction and $n'/d' \in \{p,q\}$ is the type of $T_1^-$. See Figure \ref{fig:stacked_boundary_string}.

We now observe that tile offsets between perpendicular adjacent tiles in a stacked boundary string
are given by a simple formula if the binary word describing the string is an almost palindrome. 
\begin{lemma} \label{lemma:almost_palindrome_offset}
	If $w$ is an almost palindrome, then for all $1 < i \leq |w|$
	\[
	c(T_i^{+}) - c(T_i^{-}) = v_{n'/d', j'} +  (v_{a,j} - v_{b,j})
	\]
	where $j' \in \{1,2\}$ is the perpendicular direction, $n'/d', a,b \in \{p,q\}$ is the type of $T_i^-$,
	$T_1^+$ and $T_1^-$ respectively.  
	%and 
	%\[
	%c(T_{|w|}^+) - c(T_{|w|}^-)
	%\]
\end{lemma}
\begin{proof}
	For concreteness and since $w$ is an almost palindrome, take $j = 1$, $T_1^+ = T_p$ and $T_1^- = T_q$.
	If $T_2^-$ and $T_2^+$ are both of type $p$, then 
	\begin{align*}
	c(T_2^+) - c(T_2^-) &=	(c(T_2^+) - c(T_1^+)) + (c(T_1^+) - c(T_1^-)) + (c(T_1^-) - c(T_2^-)) \\
	&= v_{p,1} + v_{q,2} - (v_{q,1} + (v_{q,2} - v_{p,2})) \\
	&= v_{p,2} +  (v_{p,1} - v_{q,1}).
	\end{align*}
	If $T_2^-$ and $T_2^+$ are both of type $q$, then 
	\[
	c(T_2^+) - c(T_2^-) = v_{p,1} + v_{q,2} - v_{q,1}.
	\]
	Conclude by similar computations together with an induction on $1 < i \leq |w|$. 
	\iffalse
	For concreteness and since $w$ is an almost palindrome, take $j = 1$, $T_1^+ = T_p$ and $T_1^- = T_q$.
	If $T_2^-$ and $T_2^+$ are both of type $p$, then 
	\begin{align*}
	c(T_2^+) - c(T_2^-) &=	(c(T_2^+) - c(T_1^+)) + (c(T_1^+) - c(T_1^-)) + (c(T_1^-) - c(T_2^-) \\
	&= v_{p,1} + v_{q,2} - (v_{q,1} + (v_{q,2} - v_{p,2}) \\
	&= (v_{p,1} - v_{q,1}) + v_{p,2}.
	\end{align*}
	If $T_2^-$ and $T_2^+$ are both of type $q$, then 
	\[
	c(T_2^+) - c(T_2^-) = v_{p,1} + v_{q,2} - v_{q,1}.
	\]
	If $T_2^-$ and $T_2^+$ are types $p$ and $q$ respectively, \ie, $|w| = 2$, then
	\begin{align*}
	c(T_2^+) - c(T_2^-) &= v_{p,1} + v_{q,2} - (v_{q,1} + (v_{q,2} - v_{p,2}) \\
	&= (v_{p,1} - v_{q,1}) + v_{p,2}.
	\end{align*}

	Now, assuming 	$1 < i < |w|$, 
	\begin{align*}
	c(T_i^+) - c(T_i^-) &=	(c(T_i^+) - c(T_{i-1}^+)) + (c(T_{i-1}^+)-c(T_{i-1}^-)) + (c(T_{i-1}^-) - c(T_{i}^-)) \\
	&= (v_{i-1,1} +  v_{i-1,2}) + (v_{p,1} - v_{q,1}) -(c(T_{i}^-) - c(T_{i-1}^-))
	\end{align*}
	If both $i$ and $i-1$ are of the same type, the expression becomes
	\[
	(v_{i-1,1} +  v_{i-1,2}) + (v_{p,1} - v_{q,1}) - v_{i-1,1} = v_{i-1,2} + v_{p,1} - v_{q,1},
	\]
	otherwise
	\[
	(v_{i-1,1} +  v_{i-1,2}) + (v_{p,1} - v_{q,1}) - (v_{i-1,1}+v_{i-1,2}-v_{i,2}) = v_{i,2} + v_{p,1} - v_{q,1},
	\]
	\fi
\end{proof}

\iffalse
If $j = 1$, 
\[
c(T^{+,1})-c(T^{-,1}) =  v_{n'/d', j'} 
\]
where $n'/d' \in \{p,q\}$ is the type of $T^{-,2}$
and if $j = 2$, 
\[
c(T^{-,1}) - c(T^{+,1}) = (-1)^{j} v_{n/d, j'} 
\]
where $j' \in \{1,2\}$ is the perpendicular string direction and $n/d \in \{p,q\}$ is the type of $T^{+,1}$;
\fi

\begin{figure}
	\includegraphics[scale=0.55]{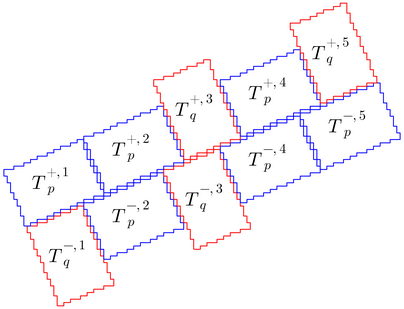}
	\includegraphics[scale=0.55]{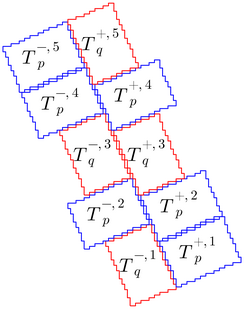}
	\caption{Stacked horizontal and vertical boundary strings. The superscripts $+,-$ denote the non-reversed and reversed strings respectively and, \eg, $T_p^{+,i}$ refers to $T_i^+$ and indicates that $T_i^+$ is type $p$. Here the tiles are outlined in the dual lattice. 
	} \label{fig:stacked_boundary_string}
	
\end{figure}

\subsection{A degenerate boundary string} \label{subsec:zero-one-tiles}
We now examine a degenerate boundary string which we later show completely describes the odometer recursion. 
Due to the degenerate nature of the tiles in the string, the offsets in the definition of boundary
string must be modified slightly. Let $p,q = 0/1,1/1$ and the lattice vectors be as defined in Section \ref{sec:hyperbola}: 
\[
\begin{aligned}
v_{p,1} &= 2 \qquad v_{p,2} = -1+ \I \\
v_{q,1} &= 1+ \I \qquad v_{q,2} = 2 \I.
\end{aligned}
\]
%\textcolor{red}{(Note that we do include the factor of 2 here.)} 
%Strictly speaking the objects which we will define will not be even-odd boundary strings due to the degenerate nature of the offsets
%so we are forced to redefine them. 

The {\it zero-tile} is $T_{0/1} = \{0, \I, 2 \I, 1, 1 + \I, 1 + 2\I \}$ and the {\it one-tile} is 
$T_{1/1} = \{0, \I, 1 , 1+ \I\}$.  A {\it zero-one}  {\it horizontal} boundary string is a collection of tiles $T_i \in \{T_{0/1}, T_{1/1}\}$ 
with offsets given by 
\begin{equation} \label{eq:zero-one-bs}
c(T_i) - c(T_{i-1}) = \begin{cases}
v_{q,1} &\mbox{ if $T_{i} = T_q$} \\
v_{p,1} &\mbox{ if $T_{i} = T_p$},
\end{cases}
\end{equation}
and in the reversed case 
\begin{equation} \label{eq:zero-one-bs-reversed}
c(T_i) - c(T_{i-1}) = \begin{cases}
v_{q,1}+1 &\mbox{ if $pq$} \\
v_{p,1} &\mbox{ if $pp$} \\
v_{p,1}-1 &\mbox{ if $qp$} \\
v_{q,1} &\mbox{ if $qq$},
\end{cases}
\end{equation}
where the right column denotes the tile tuple.
We further impose that a (resp. reversed) horizontal zero-one boundary string begins with (resp. $T_{0/1}$) $T_{1/1}$.  
We also label horizontal zero-one boundary strings by binary words. See Table \ref{table:zero-one-single-overlaps} for an illustration 
of tiles associated to boundary strings of length two.

A {\it zero-one stacked horizontal} boundary string is a union of a horizontal zero-one boundary string $\{T_i^+\}$ and its reversal $\{T_i^-\}$ where
\begin{equation}\label{eq:stacked_boundary_string_even_odd_horizontal}
c(T^{+,1}) - c(T^{-,1}) = v_{p,2} + \I.
\end{equation}
We again label the zero-one stacked horizontal boundary string by the non-reversed binary word. 
Unfortunately, in this case the stacked boundary strings may leave gaps which are too large. 
This occurs for exactly one particular interface $qq$, which we have displayed in Figure \ref{fig:zero_one_gap}. 
To fix this, we fill the gap by requiring that whenever $T_{1/1}$ follows a $T_{1/1}$, 
the subsequent tile is replaced by an enlarged version: 
\begin{equation}
T_{1/1}^d = T_{1/1} \cup \{\I -1, 1 -\I\},
\end{equation}
but there are no other changes, \ie, we impose $c(T_{1/1}^d) = c(T_{1/1})$. See Figure \ref{fig:zero_one_gap_fixed}.

\begin{figure}[t]
	\includegraphics[width=0.3\textwidth]{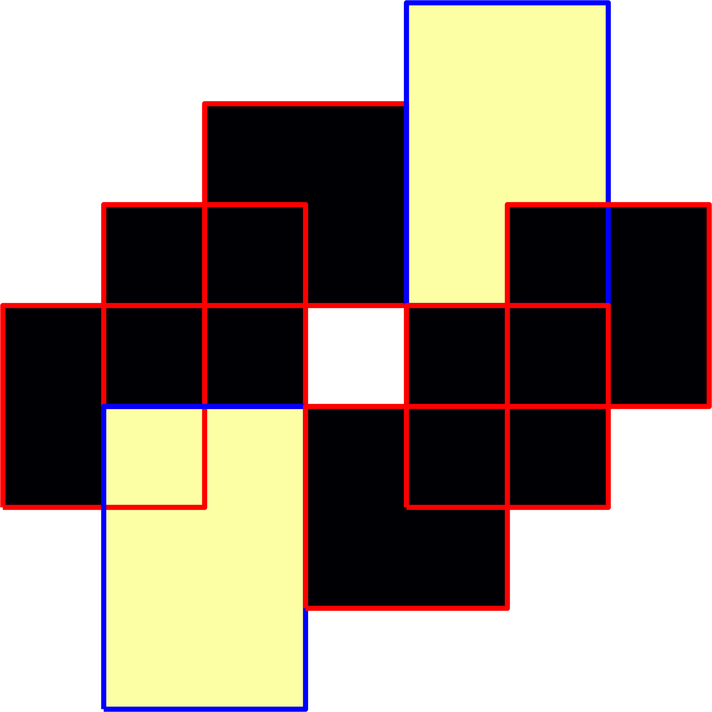}
	\includegraphics[width=0.3\textwidth]{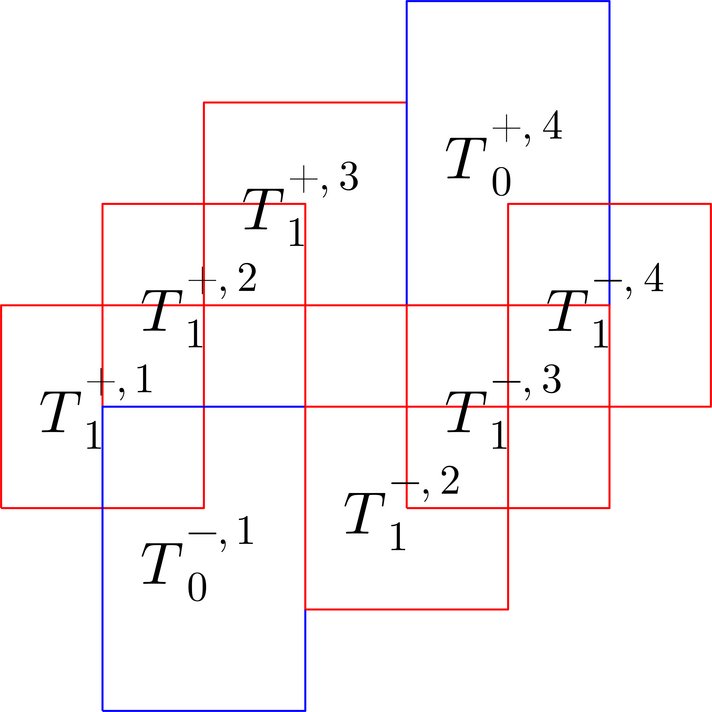}
	\caption{A gap inside a zero-one horizontal stacked boundary string corresponding to the word $qqqp$. 
		Here we are outlining tiles on a square grid where each $x \in \Z[\I]$
		is in the center of a square. On the left, points in the stacked string 
		are filled in with either black $(T_{0/1})$ or yellow $(T_{1/1})$. On the right the outlines and annotations are displayed and the labeling is as in Figure \ref{fig:stacked_boundary_string}. 
		For brevity, we write $0$ and $1$ for $0/1$ and $1/1$ respectively.  
	} \label{fig:zero_one_gap}
\end{figure}

\begin{figure}[t]
	\includegraphics[width=0.3\textwidth]{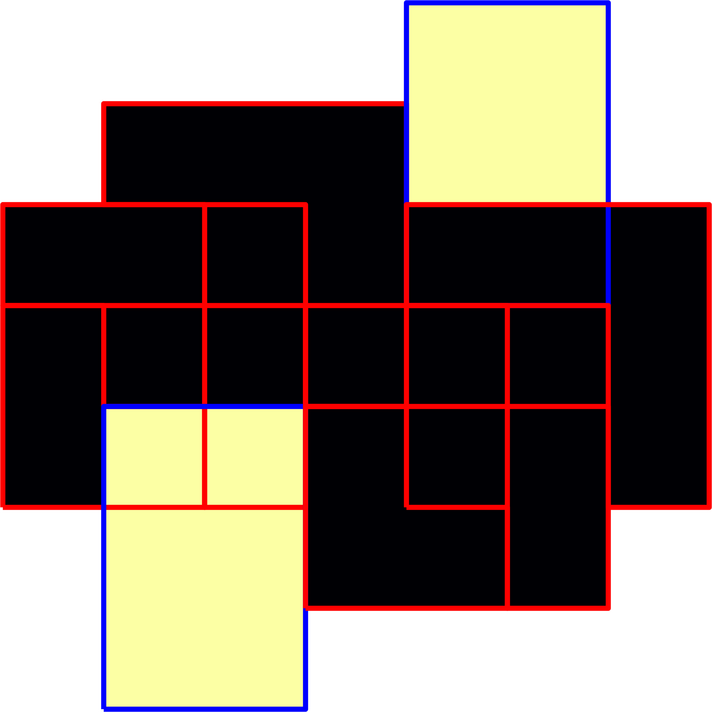}
	\includegraphics[width=0.3\textwidth]{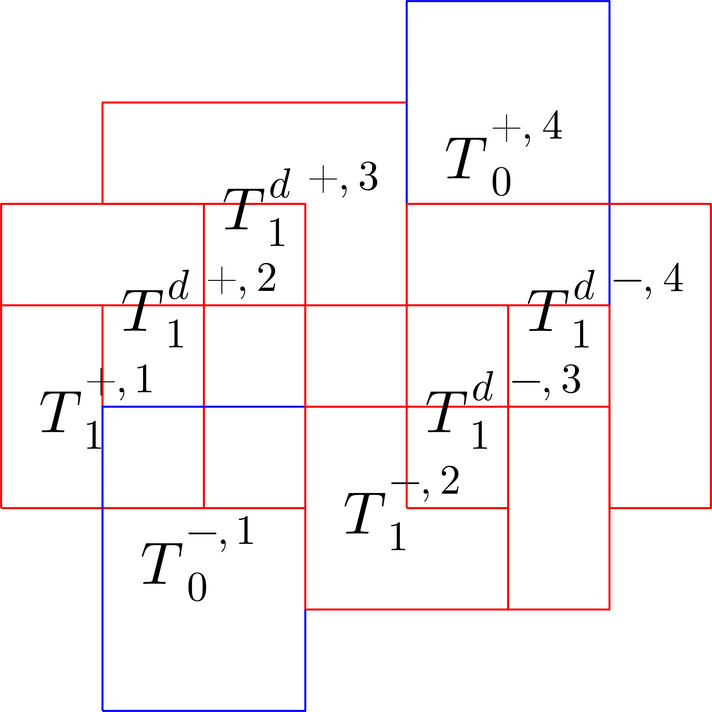}
	\caption{As Figure \ref{fig:zero_one_gap} but with gap fixed by $T_{1/1}^d$.}  
	\label{fig:zero_one_gap_fixed}
\end{figure}

\begin{figure}
	\includegraphics[width=0.2\textwidth]{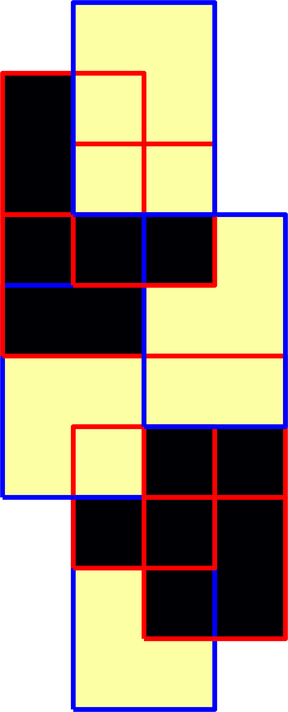}
	\includegraphics[width=0.2\textwidth]{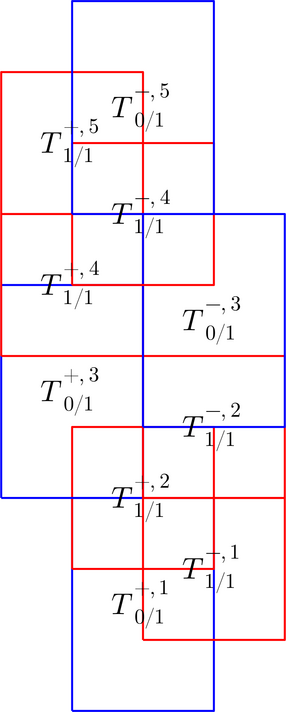}
	\caption{A vertical stacked zero-one boundary string corresponding to the word $pqpqq$ with the same labeling scheme as Figure \ref{fig:zero_one_gap}.} \label{fig:zero_one_stacked_vertical}
\end{figure}

To define vertical strings, we roughly rotate horizontal tiles by 90 degrees but exclude the doubled tiles. To be specific, 
a {\it zero-one vertical} boundary string is also a collection of tiles $T_i \in \{T_{0/1}, T_{1/1}\}$ 
but with rotated offsets: 
\begin{equation} \label{eq:zero-one-bs-vertical}
c(T_i) - c(T_{i-1}) = \begin{cases}
v_{q,2} &\mbox{ if $T_i=T_q$ } \\%$*q$
v_{p,2} &\mbox{ if $T_i = T_p$} %$*p$ 
\end{cases}
\end{equation}
and in the reversed case 
\begin{equation} \label{eq:zero-one-bs-reversed-vertical}
c(T_i) - c(T_{i-1}) = \begin{cases}
v_{p,2}+ \I &\mbox{ if $pq$} \\
v_{p,2} &\mbox{ if $pp$} \\
v_{q,2}-\I &\mbox{ if $qp$} \\
v_{q,2} &\mbox{ if $qq$}.
\end{cases}
\end{equation}
A {\it zero-one stacked vertical} boundary string is a union of a vertical zero-one boundary string $\{T_i^+\}$ and its reversal $\{T_i^-\}$ where
\begin{equation}\label{eq:stacked_boundary_string_even_odd_vertical}
c(T^{+,1}) - c(T^{-,1}) = -v_{q,1}.
\end{equation}
In the vertical case, we do not use doubled tiles and we further impose that every (resp. reversed) vertical zero-one boundary string begins with $T_{0/1}$ (resp. $T_{1/1}$).  See an example of a stacked vertical zero-one boundary string in 
Figure \ref{fig:zero_one_stacked_vertical}. We again label the zero-one vertical boundary string by a binary word and the stacked string by the non-reversed word.

We conclude with a similar counterpart to Lemma \ref{lemma:almost_palindrome_offset}. 
\begin{lemma} \label{lemma:zero-one-fixed-offsets}
	If $w$ is an almost palindrome the offsets 
	between perpendicular tiles in the zero-one stacked string are fixed: in the horizontal case, if $w$ starts with $q$ and ends with $p$,
	\begin{align*}
	c(T_1^{+,1}) - c(T_1^{-,1})  &=   -1 + 2 \I  = v_{p,2} + \I \qquad \mbox{--- $qp$} \\
	c(T_{1/1}^+) - c(T_{1/1}^-) &=  -2 + 2 \I = v_{p,2} + \I-1  \qquad \mbox{--- $qq$} \\
	c(T_{0/1}^+) - c(T_{0/1}^-) &= -1 + 2 \I = v_{p,2} + \I \qquad \mbox{--- $pp$} \\
	c(T_{|w|}^+) - c(T_{|w|}^-) &= -1 + \I   = (v_{p,1} -v_{q,1}) + (v_{p,2} + \I -1) \qquad \mbox{--- $pq$}
	\end{align*}
	and in the vertical case, if $w$ starts with $p$ and ends with $q$, 
	\begin{align*}
	c(T_1^{+,1})-c(T_1^{-,1}) &=   -1 - \I = -v_{q,1} \qquad \mbox{--- $pq$}  \\
	c(T_{1/1}^+) - c(T_{1/1}^-) &= -1 - \I  = -v_{q,1} \qquad \mbox{--- $qq$} \\
	c(T_{0/1}^+) - c(T_{0/1}^-) &= -2 - \I = v_{p,2} - v_{q,1} - (v_{q,2}-\I)  \quad \mbox{--- $pp$} \\
	c(T_{|w|}^+) - c(T_{|w|}^-) &= -1 = -v_{q,1} + \I \quad \mbox{--- $qp$},
	\end{align*}
	where the right column denotes the tile tuple.

	In particular, there are finitely many types of pairwise intersecting tiles in such stacked strings. See  Figures \ref{fig:horizontal-zero-one-overlaps} and \ref{fig:vertical-zero-one-overlaps} respectively. 
	This finite check implies every stacked zero-one boundary string is simply connected. 
\end{lemma}
\qed

\subsection{Function data} \label{subsec:zero-one-functions}

\begin{table}
	\begin{tabular}{cccc}
		Horizontal non-reversed & Horizontal reversed & Vertical non-reversed & Vertical reversed   \\
		\includegraphics[width=0.2\textwidth]{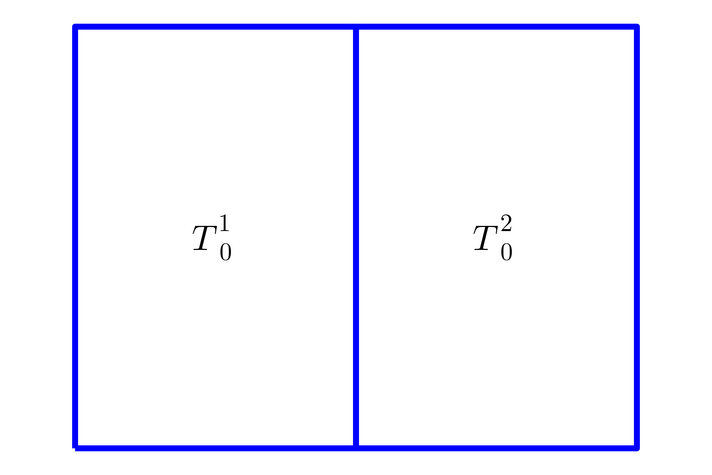} &
		\includegraphics[width=0.2\textwidth]{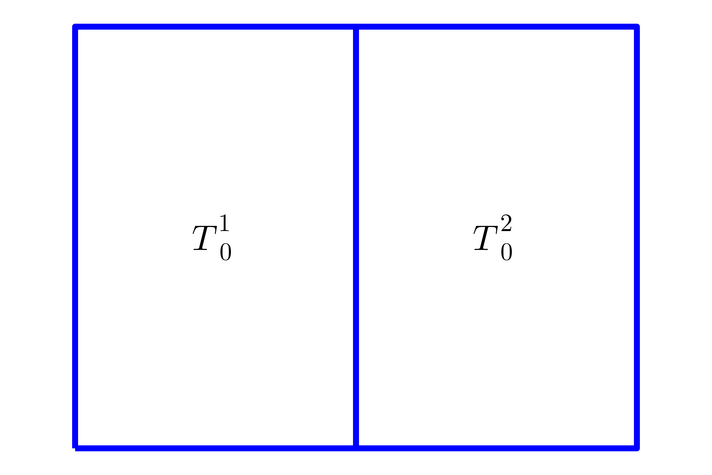} &
		\includegraphics[width=0.2\textwidth]{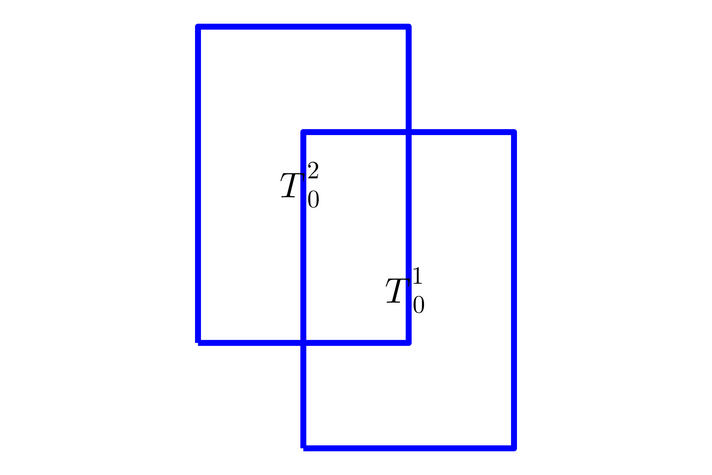} &
		\includegraphics[width=0.2\textwidth]{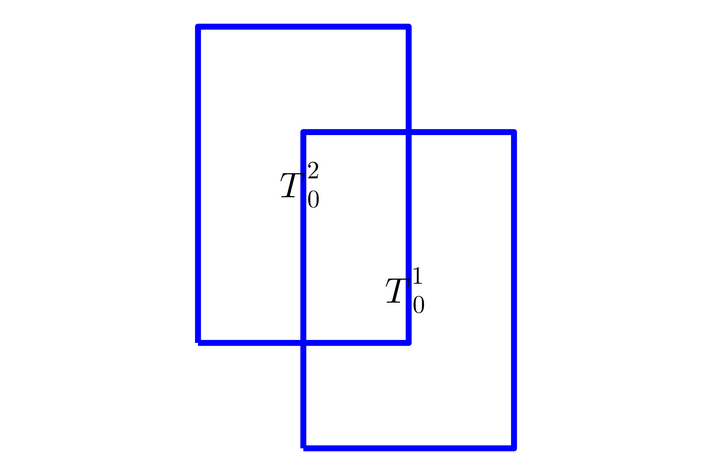} \\
		
		\includegraphics[width=0.2\textwidth]{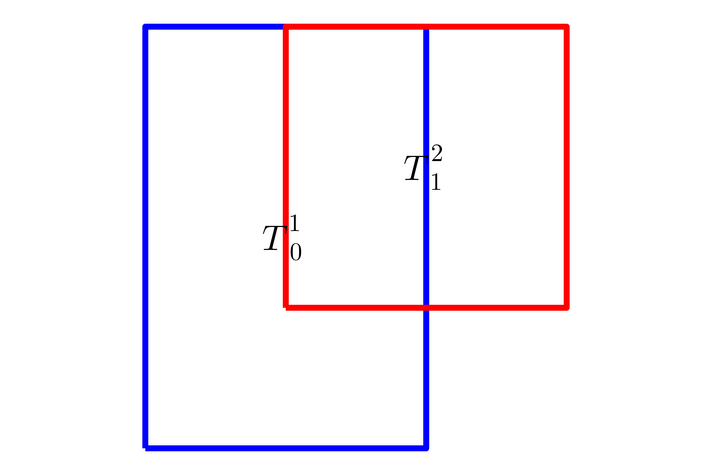} &
		\includegraphics[width=0.2\textwidth]{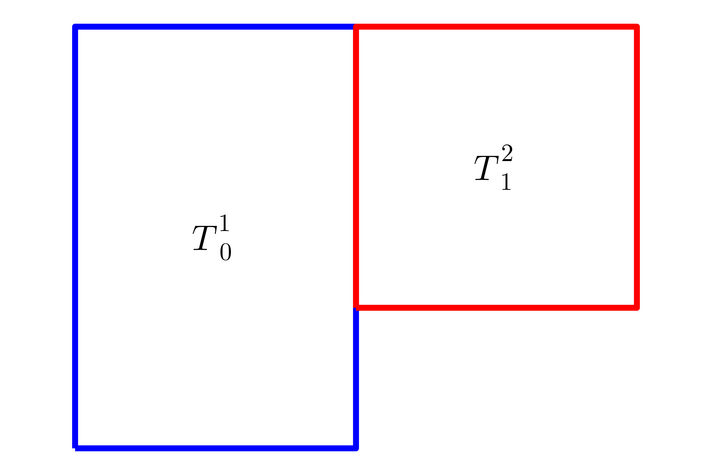} &
		\includegraphics[width=0.2\textwidth]{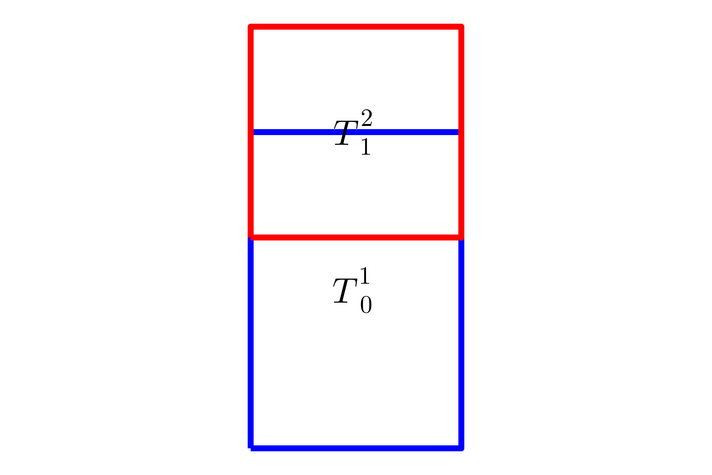} &
		\includegraphics[width=0.2\textwidth]{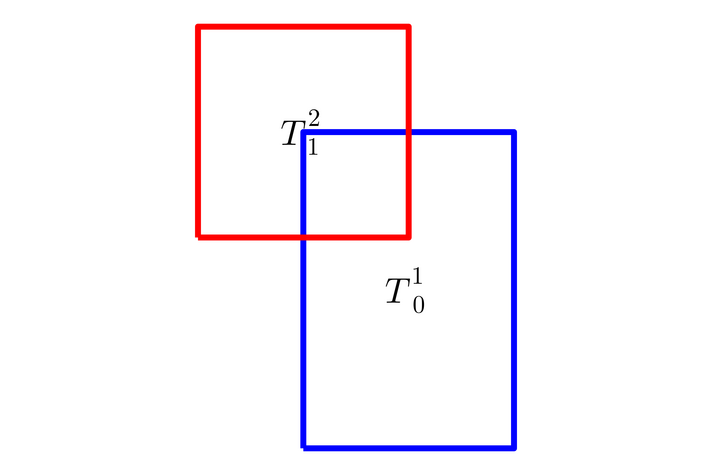}  \\
		
		\includegraphics[width=0.2\textwidth]{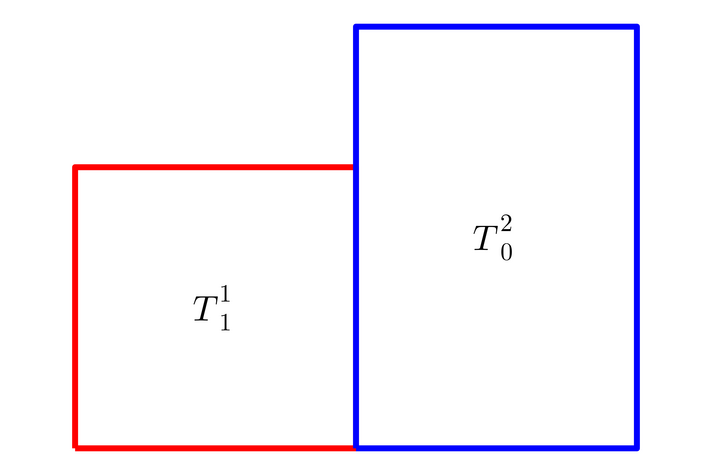} &
		\includegraphics[width=0.2\textwidth]{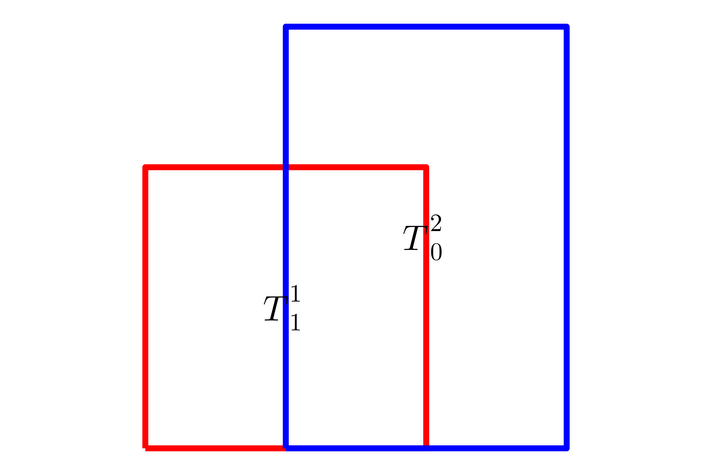} &
		\includegraphics[width=0.2\textwidth]{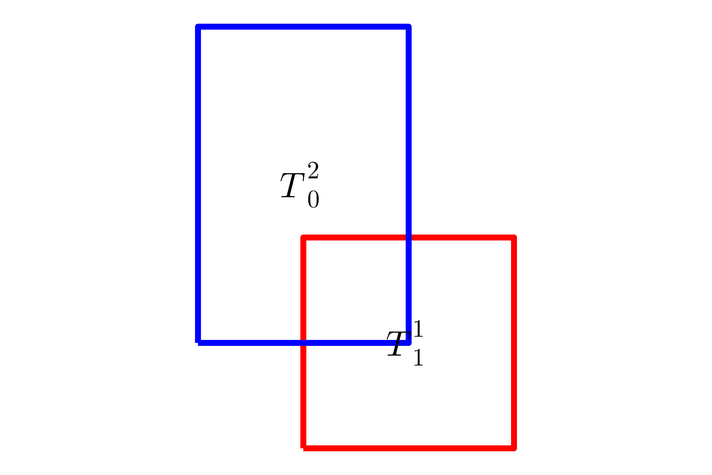} &
		\includegraphics[width=0.2\textwidth]{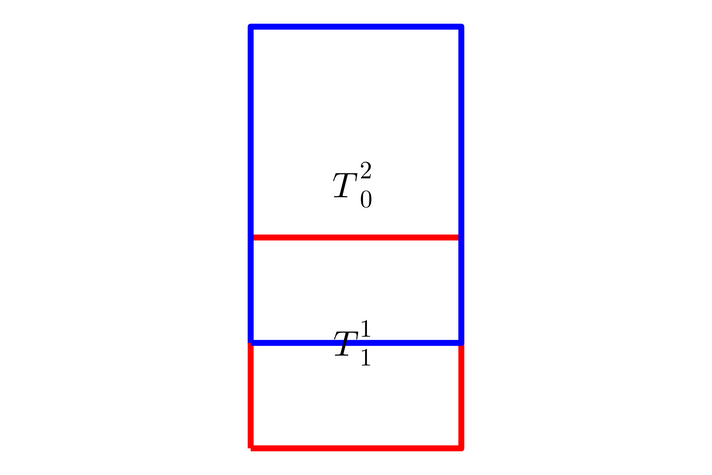} \\
		
		\includegraphics[width=0.2\textwidth]{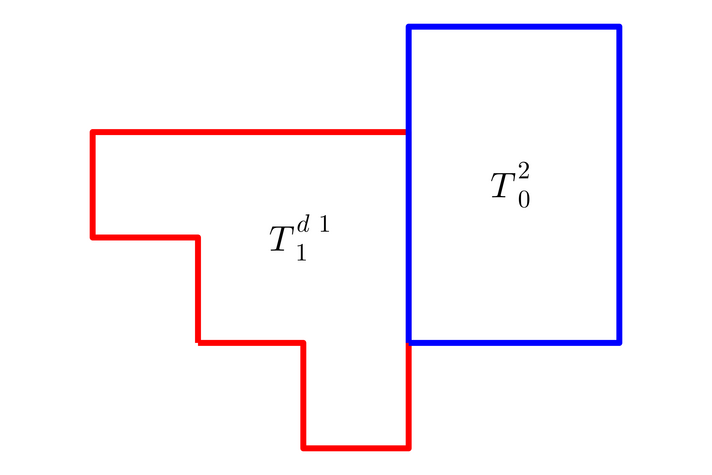} &
		\includegraphics[width=0.2\textwidth]{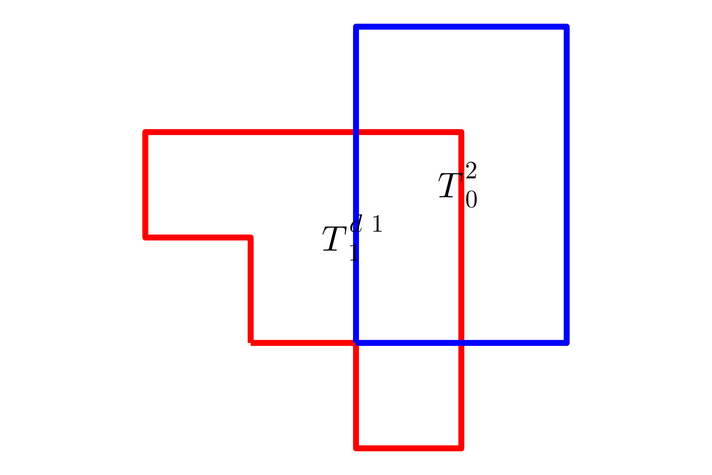} &
		$-$ & $-$  \\		

		\includegraphics[width=0.2\textwidth]{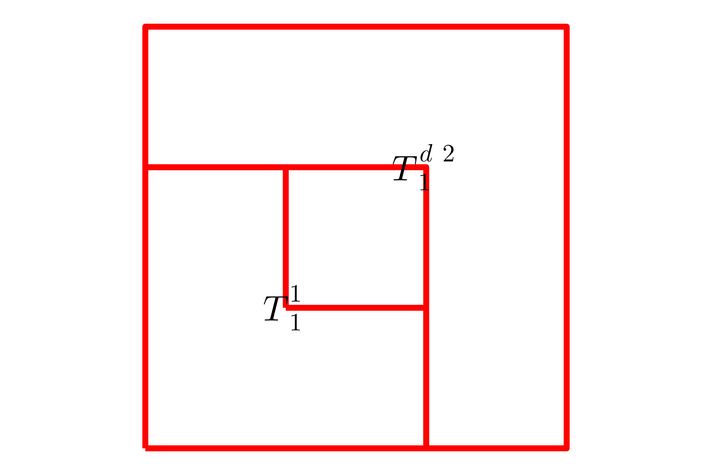} &
		\includegraphics[width=0.2\textwidth]{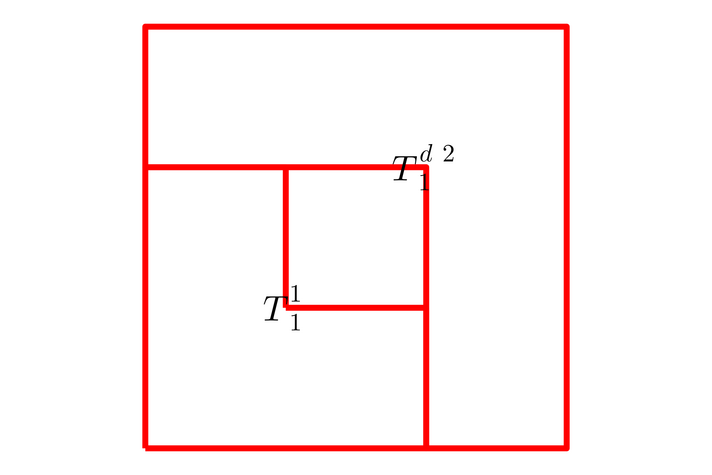} &
		$-$ &
		$-$  \\
		
		\includegraphics[width=0.2\textwidth]{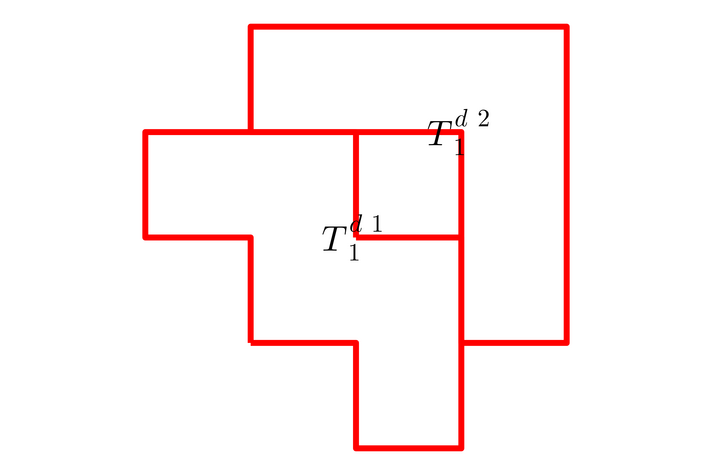} &
		\includegraphics[width=0.2\textwidth]{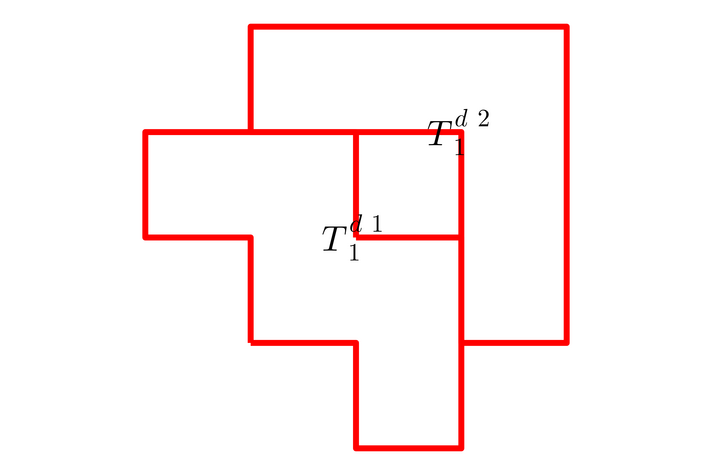} &
		\includegraphics[width=0.2\textwidth]{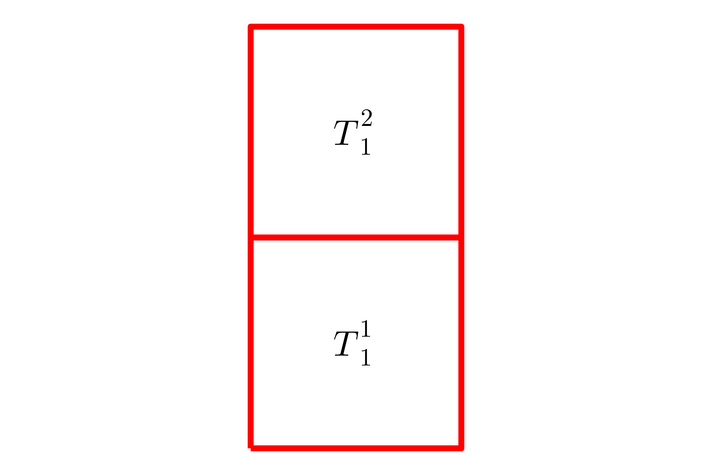} &
		\includegraphics[width=0.2\textwidth]{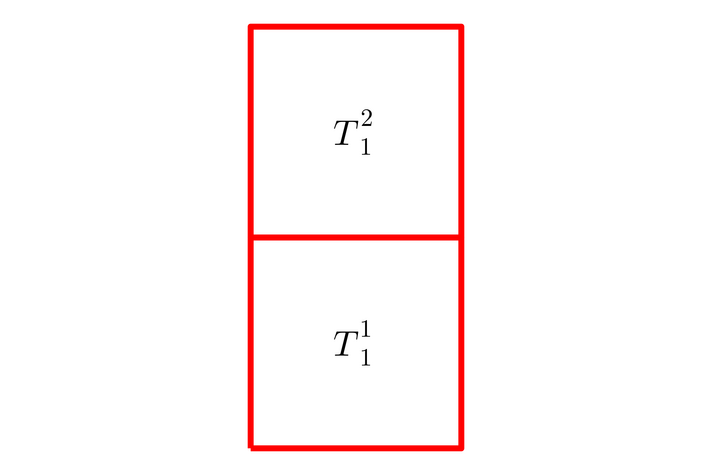} 
	\end{tabular}
	\caption{Possible pairwise overlaps in a zero-one boundary string with labeling as Figure \ref{fig:zero_one_gap}.}
	\label{table:zero-one-single-overlaps}
\end{table}

\begin{table}
	\begin{tabular}{cccc}
		Horizontal non-reversed & Horizontal reversed & Vertical non-reversed & Vertical reversed   \\
		\includegraphics[width=15mm,height=15mm]{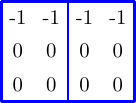} &
		\includegraphics[width=15mm,height=15mm]{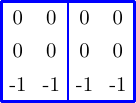} &
		\includegraphics[width=8mm,height=15mm]{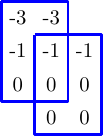} &
		\includegraphics[width=8mm,height=15mm]{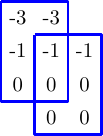} \\
		
		\includegraphics[width=15mm,height=15mm]{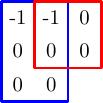} &
		\includegraphics[width=15mm,height=15mm]{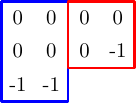} &
		\includegraphics[width=8mm,height=15mm]{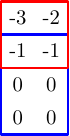} &
		\includegraphics[width=8mm,height=15mm]{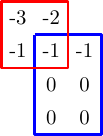}  \\
		
		\includegraphics[width=15mm,height=15mm]{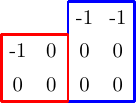} &
		\includegraphics[width=15mm,height=15mm]{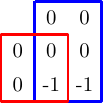} &
		\includegraphics[width=8mm,height=15mm]{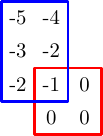} &
		\includegraphics[width=8mm,height=15mm]{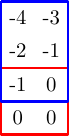} \\
		
		\includegraphics[width=15mm,height=15mm]{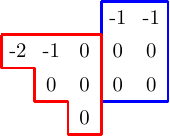} &
		\includegraphics[width=15mm,height=15mm]{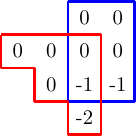} &
		$-$ & $-$  \\		
		
		\includegraphics[width=15mm,height=15mm]{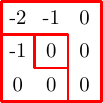} &
		\includegraphics[width=15mm,height=15mm]{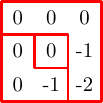} &
		$-$ &
		$-$  \\
		
		\includegraphics[width=15mm,height=15mm]{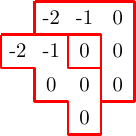} &
		\includegraphics[width=15mm,height=15mm]{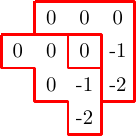} &
		\includegraphics[width=8mm,height=15mm]{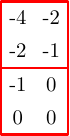} &
		\includegraphics[width=8mm,height=15mm]{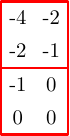} 
	\end{tabular}
	\caption{As Table \ref{table:zero-one-single-overlaps} except with the values of the odometer.}
	\label{table:zero-one-single-overlaps-odometer}
\end{table}

\begin{figure}
	\fbox{\includegraphics[width=0.25\textwidth]{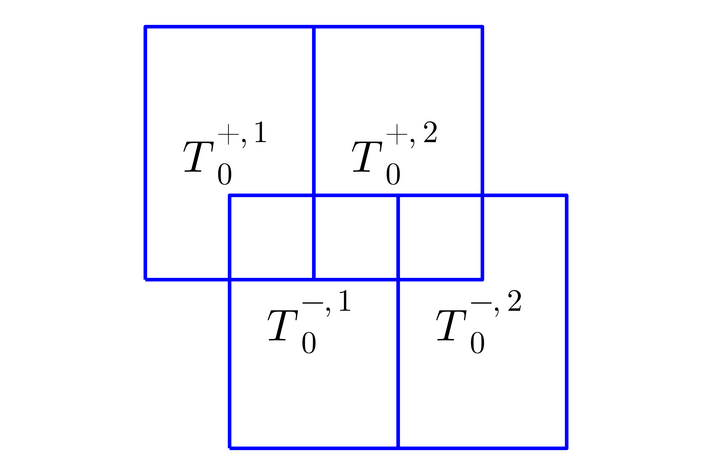}}
	\fbox{\includegraphics[width=0.25\textwidth]{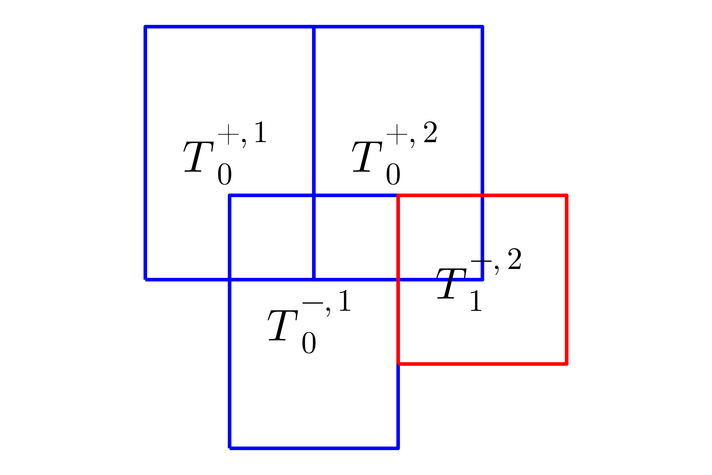}}
	\fbox{\includegraphics[width=0.25\textwidth]{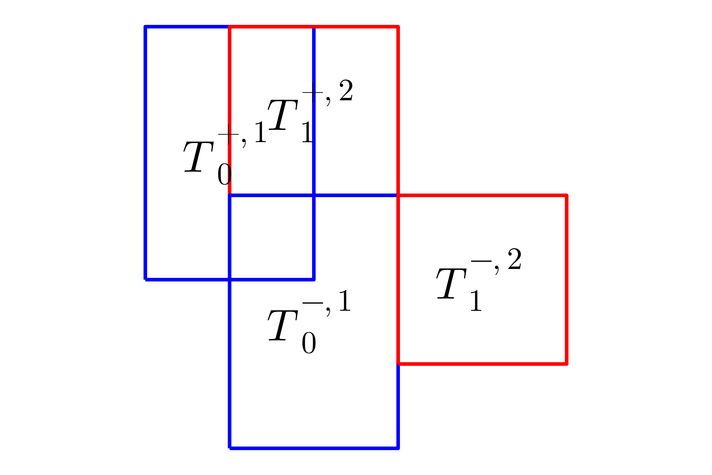}}	\fbox{\includegraphics[width=0.25\textwidth]{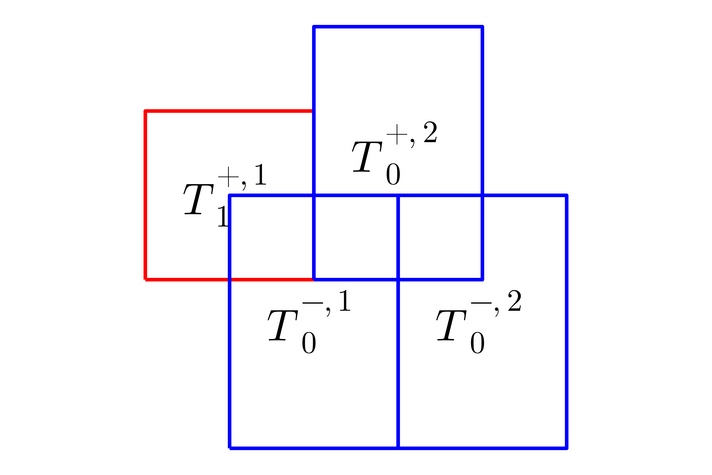}}	\fbox{\includegraphics[width=0.25\textwidth]{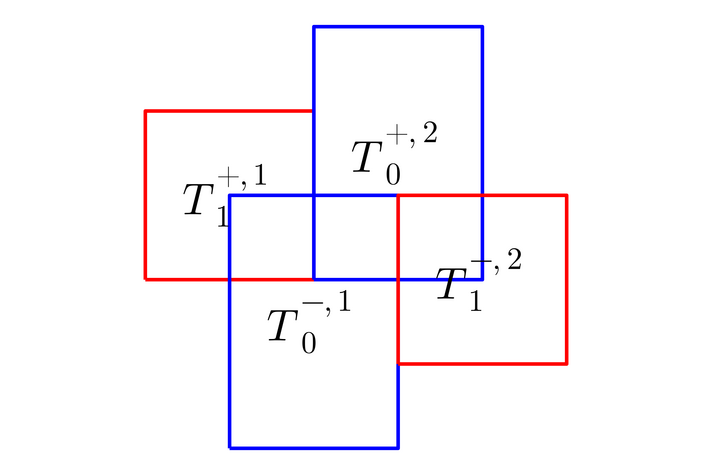}}	\fbox{\includegraphics[width=0.25\textwidth]{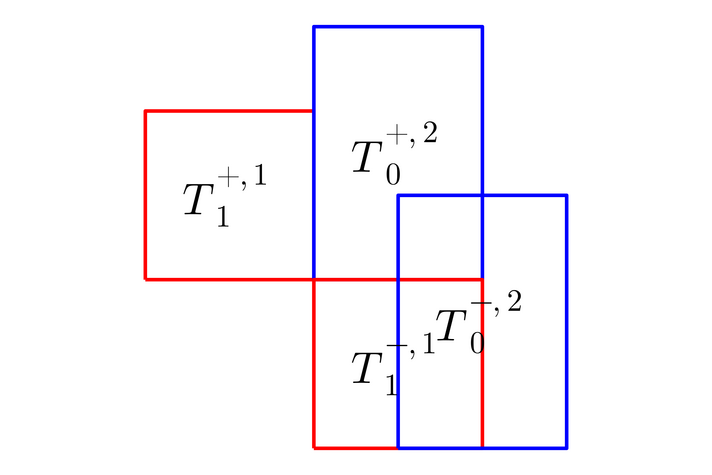}}	\fbox{\includegraphics[width=0.25\textwidth]{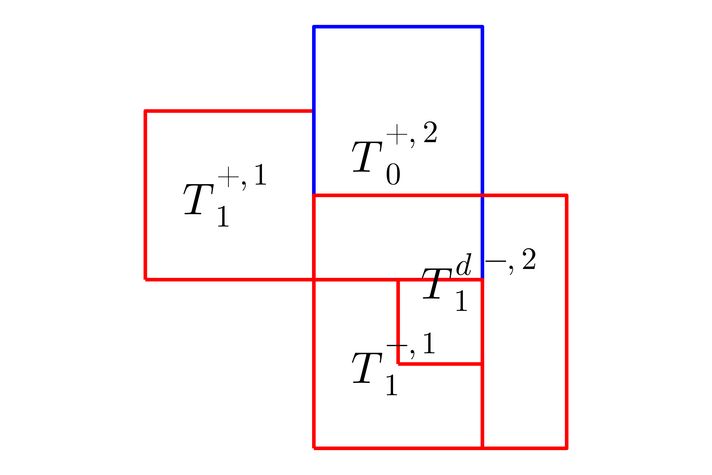}}	\fbox{\includegraphics[width=0.25\textwidth]{Figures/bdry_strings_base_compatibilities/h/base_cases/qp_pp.png}}
	\fbox{\includegraphics[width=0.25\textwidth]{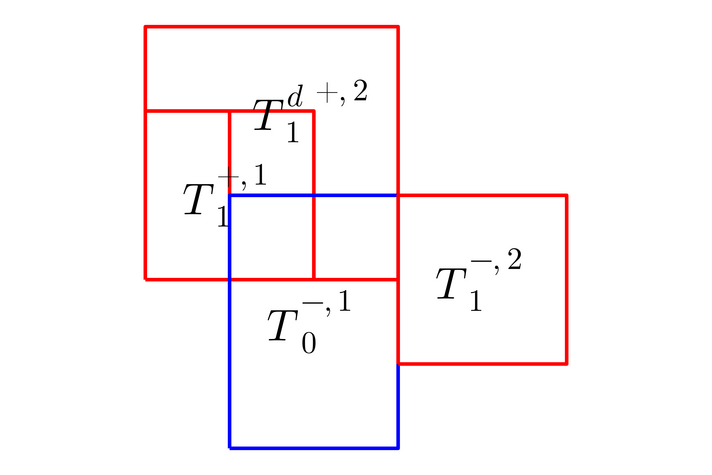}}
	\fbox{\includegraphics[width=0.25\textwidth]{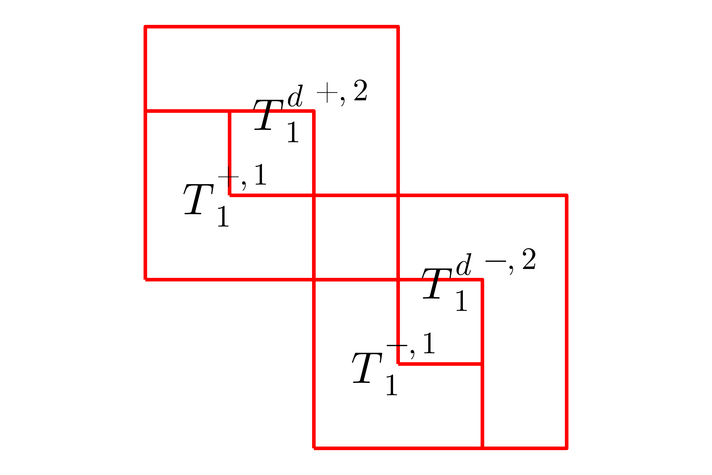}}
	\caption{Possible overlaps in a stacked almost palindrome zero-one horizontal boundary string with labeling as Figure \ref{fig:zero_one_gap}. Note that the first $T_{1/1}$ tile may be a $T_{1/1}^d$ tile.}
	\label{fig:horizontal-zero-one-overlaps}
\end{figure}

\begin{figure}
	\fbox{\includegraphics[width=0.25\textwidth]{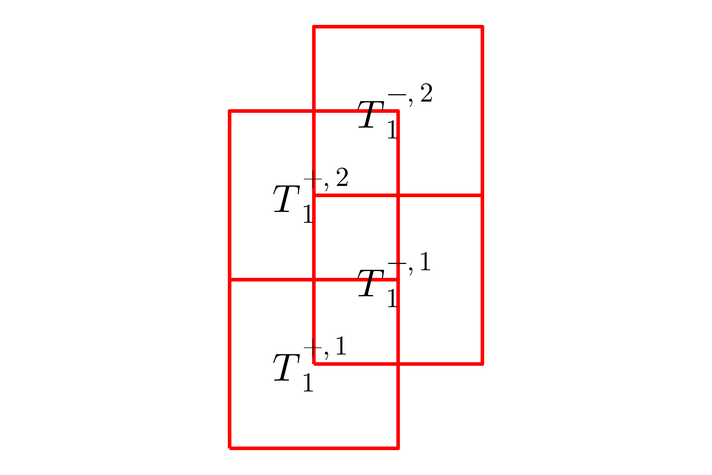}}
	\fbox{\includegraphics[width=0.25\textwidth]{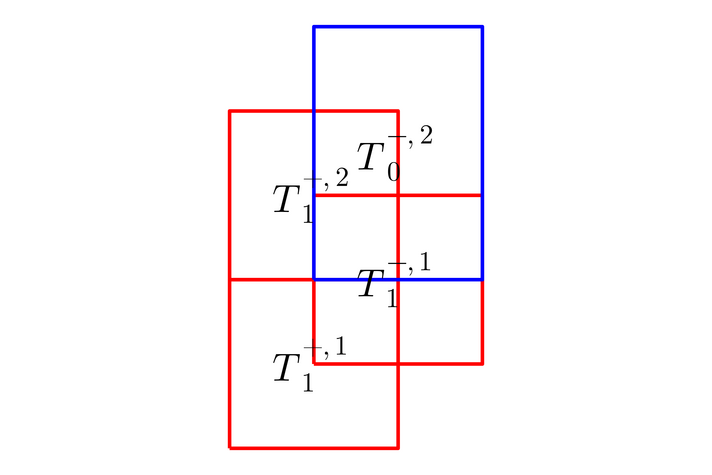}}
	\fbox{\includegraphics[width=0.25\textwidth]{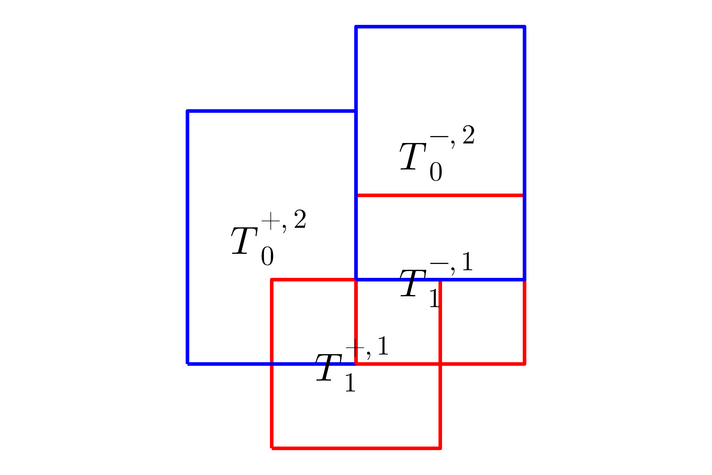}}	\fbox{\includegraphics[width=0.25\textwidth]{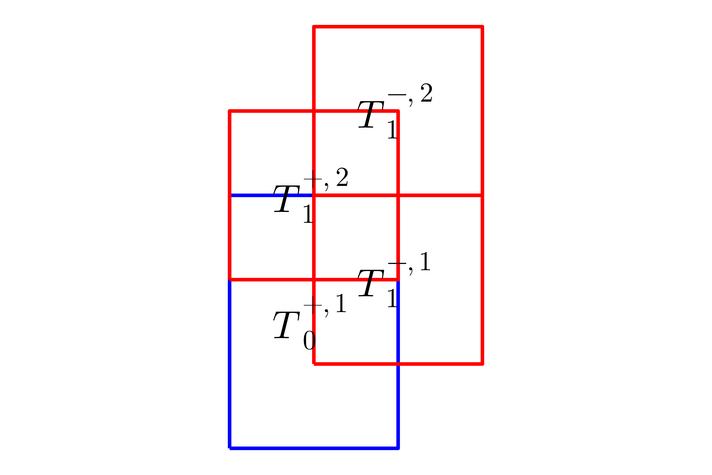}}	\fbox{\includegraphics[width=0.25\textwidth]{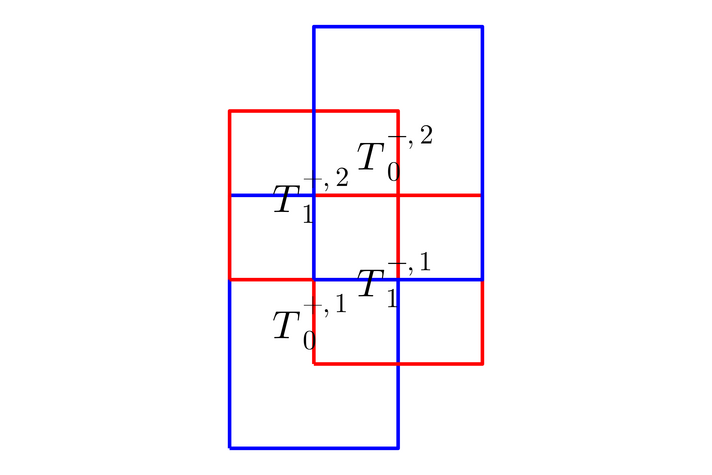}}	\fbox{\includegraphics[width=0.25\textwidth]{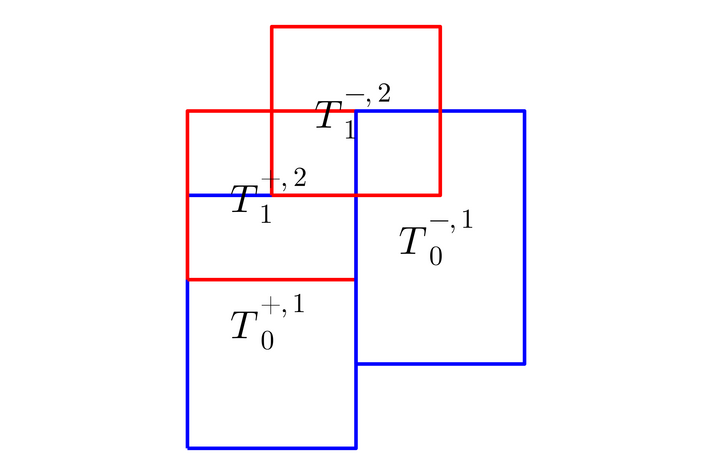}}	\fbox{\includegraphics[width=0.25\textwidth]{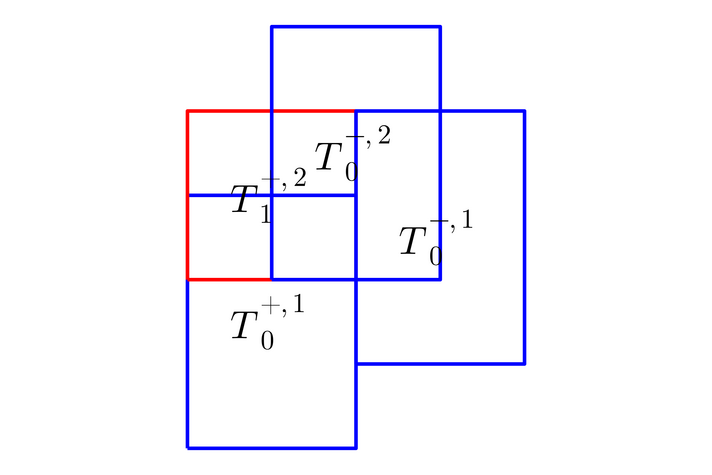}}	\fbox{\includegraphics[width=0.25\textwidth]{Figures/bdry_strings_base_compatibilities/v/base_cases/pq_qq.png}}
	\fbox{\includegraphics[width=0.25\textwidth]{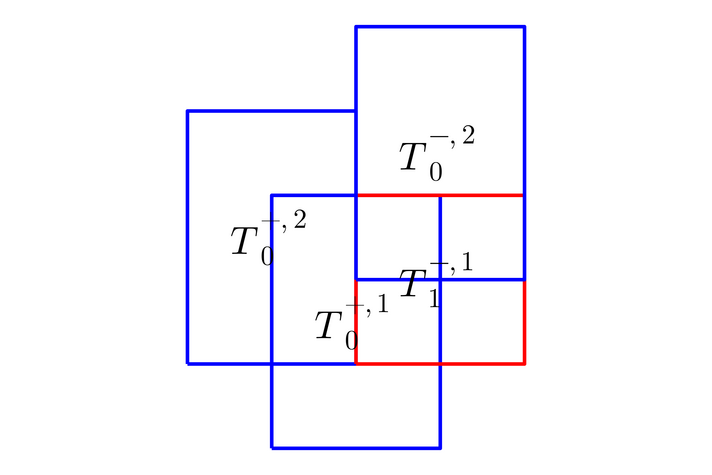}}
	\fbox{\includegraphics[width=0.25\textwidth]{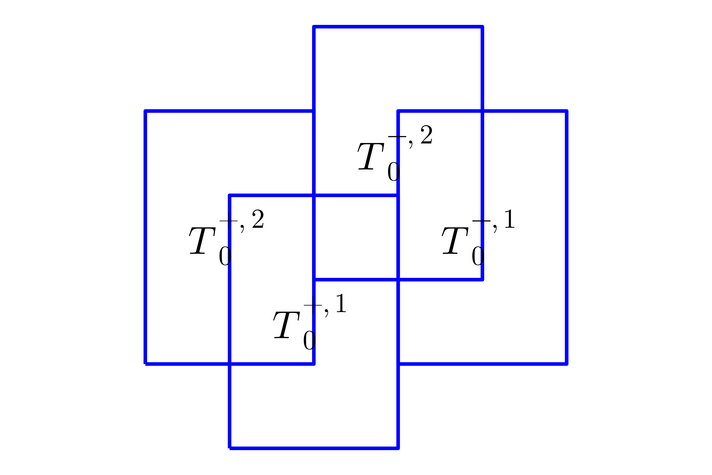}}
	\caption{Possible overlaps in a stacked almost palindrome zero-one vertical boundary string with labeling as Figure \ref{fig:zero_one_gap}.}
	\label{fig:vertical-zero-one-overlaps}
\end{figure}
%Lemma \ref{lemma:zero-one-fixed-offsets} will be a key tool  in proving pairwise compatibility.

We now associate function data to zero-one boundary strings.  Recall the affine offsets associated to the hyperbola bases,
\[
\begin{aligned}
a_{p,1} &= 0 \qquad a_{p,2} = -\I \\
a_{q,1} &= 0 \qquad a_{q,2} = 1-\I,
\end{aligned}
\]
for $p,q = 0/1, 1/1$.

The {\it zero-odometer} is any translation of $o_{0/1}:T_{0/1} \to \Z$, defined by $o_{0/1}(0)=o_{0/1}(1) = o_{0/1}(i) = o_{0/1}(1+ \I) = 0$ and $o_{0/1}(2 i) = o_{0/1}(1 + 2 \I) = -1$.  The {\it one-odometer} is any translation of $o_{1/1}:T_{1/1} \to \Z$, defined by $o_{1/1}(0)=o_{1/1}(1)=o_{1/1}(1+ \I)=0$ and $o_{1/1}(\I) = -1$. The {\it enlarged one-odometer} is any translation of $o_{1/1}^d : T_{1/1}^d \to \Z$ defined by $o_{1/1}^d = o_{1/1}$ on $T_{1/1}$ and $o_{1/1}^d(\I-1)=-2$, $o_{1/1}^d(1-\I) = 0$.

A sequence of zero/one-odometers $\{o_i\}$ {\it respects} a zero-one boundary string $\{T_i\}$ if each successive tile $T_i$ 
is the domain of $o_i$ and 
\begin{equation} \label{eq:zero-one-odometer_slopes-horizontal}
s(o_{i+1}) - s(o_{i}) = 
\begin{cases}
	0 	 &\quad\mbox{horizontal $p p$} \\
	0 	 &\quad\mbox{horizontal $q q$} \\
	-1/2 	 &\quad\mbox{horizontal $p q$} \\
	1/2 	 &\quad\mbox{horizontal $q p$} 
\end{cases} 
\end{equation}
and
\begin{equation} \label{eq:zero-one-odometer_slopes-vert}
	s(o_{i+1}) - s(o_{i}) = 
	\begin{cases}
		-\I 	 &\quad\mbox{vertical $p p$ } \\
		1-\I 	 &\quad\mbox{vertical $q q$ } \\
		1/2 - \I 	 &\quad\mbox{vertical $p q$ } \\
		1/2 - \I 	 &\quad\mbox{vertical $q p$ } 
	\end{cases} 
\end{equation}
where, for example, `vertical $pq$' indicates that the string is vertical, $i = p$, and $i+1$ is $q$. 

From Table \ref{table:zero-one-single-overlaps}, one can see some consecutive 
pairs of tiles do not overlap. This means odometers corresponding to such tiles may blow up across the boundary. We fix this by requiring a further compatibility relation between pairs of non-overlapping 
tiles. We assume that if $T_i, T_{i+1}$ are a consecutive sequence of horizontal tiles that do not overlap %pp qp non-reversed pp pq reversed
then, after a shared translation, $o_i$ and $o_{i+1}$ are constant across the shared boundary. That is, after the translation, $o_i(x) = o_{i+1}(y)$
for all $|y-x| =1$.  In the vertical case, if $T_i$ and $T_{i+1}$ do not overlap
we assume that after a shared translation, $o_{i+1}(y)-o_{i}(x) =-1$ for $|y-x| = 1$ (in this case they must be one-tiles).

We now check existence, using Lemma \ref{lemma:zero-one-fixed-offsets}. 
\begin{lemma} \label{lemma:stacked_zero_one_string}
	Given any word $w$, a sequence of zero-one odometers with a common extension which respects its boundary string or its reversal exists. 
	Moreover, if $w$ is an almost palindrome, then there exists a sequence of odometers $\{o_i^+\}$, and $\{o_i^-\}$
	respecting $\mathcal{B}_w = \{T_i^+\}$ and a sequence $\{o_i^-\}$ respecting the reversed string 
	$\mathcal{B}^r_{\rev(w)}$ which have a common extension to the stacked string  where $s(o_{1/1}^+) - s(o_{1/1}^-) = 0$ in the vertical case and $s(o_{0/1}^+) - s(o_{0/1}^-) = a_{p,2}$ in the horizontal case.
\end{lemma}

\begin{proof}
	We note the forms of the odometers after translation.  The zero-odometer translated by $-a_{p,2}$,  $\hat{o}_0: T_{0/1} \to \Z$ is given by $o_{0/1}(0)=o_{0/1}(1) = -1$ and $o_{0/1}(i) = o_{0/1}(1+ \I) = 0 = o_{0/1}(2 \I) = o_{0/1}(1 + 2\I) = 0$.  The one-odometer translated by $-a_{q,2}$,  $\hat{o}_q:T_{1/1} \to \Z$ is given by 
	$\hat{o}_q(0)=\hat{o}_q(\I) = \hat{o}_q(1+ \I)=0$ and $\hat{o}_q(1) = -1$. Similarly, the enlarged one-odometer translated
	by $-a_{q,2}$, $\hat{o}^d:T_{1/1}^d \to \Z$ is given by $\hat{o}^d = \hat{o}$ on $T_{1/1}$ and $\hat{o}(1-\I)=-2$, $\hat{o}(\I-1)=0$.

	From the definitions, one can see that no three consecutive pairs of tiles in a boundary string can overlap - 
	only two consecutive pairs can. Therefore, existence of a sequence of zero/one odometers respecting a boundary string reduces to checking 
	compatibility between partial odometers for pairwise consecutive tiles. Since compatibility is an affine invariant relationship, 
	we can translate so that the first odometer is exactly $o_{1/1}$, $o_{0/1}$, $\hat{o}_{1/1}$, or $\hat{o}_{0/1}$. This then reduces the compatibility check to a finite one, 
	see Table \ref{table:zero-one-single-overlaps} and Table \ref{table:zero-one-single-overlaps-odometer}. 
	
	The existence problem for a stacked string is similar - by Lemma \ref{lemma:zero-one-fixed-offsets}, there 
	are only ten possible cases for overlaps between tiles in a stacked string. We have enumerated these cases in 
	Figures \ref{fig:horizontal-zero-one-overlaps} and \ref{fig:vertical-zero-one-overlaps}.
	
	Since $s(o_{1/1}^+) = s(o_{1/1}^-)$ in the vertical case, the compatibility check follows by inspecting Table \ref{table:zero-one-single-overlaps-odometer} and Figure \ref{fig:vertical-zero-one-overlaps}.

	In the horizontal case, since $w$ is an almost palindrome, by the
	argument as in Lemma \ref{lemma:almost_palindrome}, the slopes between $o_i^+$ and $o_i^-$ are fixed.
	From this and inspecting Table \ref{table:zero-one-single-overlaps-odometer} and Figure \ref{fig:horizontal-zero-one-overlaps}, we see that the reversed and non-reversed odometers are exactly 0 on the shared overlaps. 
\end{proof}

\iffalse
We record now a vertical translation of the odometers by their respective affine factors.
The zero-odometer translated by $-a_{p,2}$,  $\hat{o}_0: T_{0/1} \to \Z$ is given by $o_{0/1}(0)=o_{0/1}(1) = -1$ and $o_{0/1}(i) = o_{0/1}(1+ \I) = 0 = o_{0/1}(2 i) = o_{0/1}(1 + 2i) = 0$  The one-odometer translated by $-a_{q,2}$,  $\hat{o}_q:T_{1/1} \to \Z$ is given by 
$\hat{o}_q(0)=\hat{o}_q(i) = \hat{o}_q(1+ \I)$ and $\hat{o}_q(1) = -1$. Similarly, the enlarged one-odometer translated
by $-a_{q,2}$, $\hat{o}^d:T_{1/1}^d \to \Z$ is given by $\hat{o}^d = \hat{o}$ on $T_{1/1}$ and $\hat{o}(1-\I)=-2$, $\hat{o}(i-1)=0$. 
\fi

\subsection{Pseudo-square tiles and boundary strings}

\begin{figure}
	\includegraphics[width=0.5\textwidth]{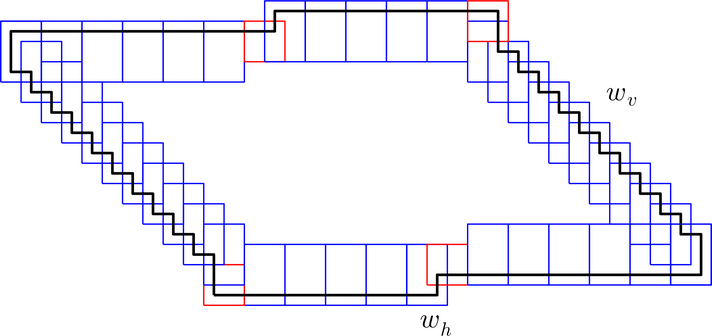}
	\caption{The boundary of a $(w_h, w_v)$-pseudo-square tile where $w_h$ is the horizontal word $q p^k q p^{k+1}$
		and $w_v$ is the vertical word $p^{2k+1} q$ for $k = 5$. The boundary word is outlined in the dual lattice in black.} 
	\label{fig:pseudo_square_boundary_word}
\end{figure}

We now associate horizontal and vertical boundary strings to tiles and partial odometers. Let $(w_h, w_v)$ denote almost palindromes which define zero-one horizontal and vertical boundary strings respectively. 

\begin{definition} \label{def:tile_boundary_strings}
	A $(w_h, w_v)$-pseudo-square is a tile, $T$, which can be decomposed along its boundary into a sequence of subtiles 
	\[
	 \mathcal{T}_{h,v} := \{T_{i,h}^+\} \cup \{ T_{i,h}^-\} \cup \{ T_{i,v}^+\} \cup \{ T_{i,v}^-\}
	\]
	each of which respectively form a $w_h$, $\rev(w_h)$, $w_v$, and $\rev(w_v)$ zero-one horizontal, reversed horizontal, vertical, and reversed vertical boundary string. 	That is, $\mathcal{T}_{h,v} \subset T$, $c(T_{1,h}^+) = c(T_{1,v}^-) =  c(T)$ and $\partial^{-} T \cap \mathcal{T}_{h,v} = \partial^{-} T$. 
	
	A partial odometer $o: T \to \Z$ respects $(w_h, w_v)$
	if its restrictions to $\mathcal{T}_{h,v}$ respect $w_h$-horizontal, $w_v$-vertical, and $\rev(w_h)$-reversed-horizontal
	and $\rev(w_v)$-reversed-vertical zero-one boundary strings respectively. 	
\end{definition}
We sometimes overload notation and also refer to the word describing the boundary string as a set of tiles. 

We now extend the rotation operator to pseudo-square tiles. For a binary word $w$, let $\mathcal{F}(w)$ denote the {\it flipping operator}
which flips every $p$ to a $q$ and vice versa. Then, 
\begin{equation} \label{eq:rotate_boundary_string}
\mathcal{R}(w_h, w_v) = (\mathcal{F}(w_v), \mathcal{F}(w_h))
\end{equation}
sends a pair of horizontal/vertical strings to a rotated pair. We now extend this to tiles. If $T$ is a $(w_h, w_v)$-pseudo-square
then
\begin{equation} \label{eq:rotate_tile}
\mathcal{R}(T) = T_r, 
\end{equation}
where $T_r$ is a $\mathcal{R}(w_h, w_v) $-pseudo-square with $c(T_r) = c(T)$. 

We now define a map  $G: (w_h, w_v) \mapsto (w_1, w_2)$. We start by defining it for pairs of horizontal zero-one tiles (strings),
\begin{equation} \label{eq:g_boundary_word_pairwise}
\begin{cases}
g(p*q) &\to 1*\I \\
g(q*p) &\to 1*1 \\
g(q_d*p)&\to \I*1 \\
g(p*p) &\to  1*1 \\
g(q*q_d) &\to  1*1 \\
g(q_d*q_d) &\to  \I*1,
\end{cases}
\end{equation}
where $q_d$ indicates a $T^d_{1/1}$ tile. Next extend the map to $w_h$ by, 
\begin{equation} \label{eq:g_boundary_word_w_h}
G(w_h) = g(w_h[1:2])*g(w_h[2:3])*\cdots*g(w_h[|w_h|-1, |w_h|)*1 
\end{equation}
and extend this to $w_v$ by 
\begin{equation}
G(w_v) = \I G(\mathcal{F}(w_v)),
\end{equation}
where $\I \cdot w_1$ denotes component multiplication in $\Z[\I]$ (for example, $\I \cdot (\I*1) = -1*\I$). And finally, extend the map pairwise $G(w_h,w_v) = (G(w_h), G(w_v))$.

Our next lemma uses this to express $(w_h,w_v)$ pseudo-squares using the boundary words from Section \ref{sec:almost_tilings}. 
See Figures \ref{fig:boundary_string_boundary_word} and \ref{fig:pseudo_square_boundary_word} for an illustration of this. 

\begin{figure}
	\includegraphics[width=0.7\textwidth]{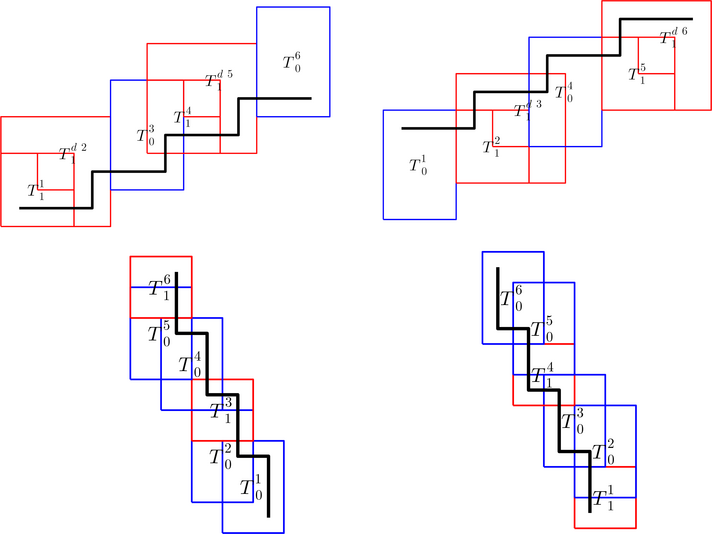}
	\caption{From top left to bottom right, $w_h =  q^2pq^2p$, reversed $\rev(w_h)$, $w_v = \mathcal{F}(w_h)$, reversed $\rev(w_v)$.
	The boundary words, $G(w_h, w_v)$ are drawn in black.}\label{fig:boundary_string_boundary_word}
\end{figure}

\begin{lemma}
The boundary word of $(w_h,w_v)$-pseudo-square $T$ can be written as $w_1*w_2*\widehat{\rev(w_1)}*\widehat{\rev(w_2)}$ where
$(w_1,w_2) = G(w_h, w_v)$. In particular, $T$ is 180-degree symmetric and $\mathcal{R}(T)$ is a 90-degree rotation of $T$.
\end{lemma} \label{lemma:pseudo-square-boundary-word}
\begin{proof}
	We first check that $G(w_h)$ traces out the lower boundary of $w_1$. By checking Table \ref{table:zero-one-single-overlaps}, 
	we see that the \eqref{eq:g_boundary_word_pairwise} does trace out the lower boundary for each pair of tiles. Indeed, 
	if neither tile in the pair $(T^1, T^2)$ is $T_{1/1}^d$, the path starts at $c(T^1)$ and ends at $c(T^2)$. Otherwise, the path 
	starts or ends at $c(T_{1/1}^d) + (1 - \I)$. The extra $*1$ in \eqref{eq:g_boundary_word_w_h} ensures 
	the path ends at the lower right corner of $T_{0/1}$. The lower right corner of $T_{0/1}$, after a 90-degree clockwise rotation 
	maps to the lower-left corner of $T_{1/1}$. In general, the lower boundary of every pair of horizontal tiles in Table \ref{table:zero-one-single-overlaps} maps to the right boundary of the flipped pair of vertical tiles. We can also use the table to check that the reversed string $\rev(w_h)$ is a 180-degree rotation of $w_h$, thus $\rev(w_1)$ traces the top boundary of $\rev(w_h)$.  
\end{proof}

%The rotation operator sends the boundary words $(w_1, w_2)$ to $(-\I w_2, -i w_1)$. 

This next lemma shows that if a partial odometer respects a pseudo-square, then it has a common 
extension to the plane. 

\begin{lemma}
Let $T$ be a $(w_h,w_v)$-pseudo-square,  $(w_1,w_2) = G(w_h, w_v)$, and suppose $o: T \to \Z$ respects $(w_h, w_v)$. If the conditions in Lemma \ref{lemma:almost_square_tiling} on $(w_1,w_2)$ are met, $T$ generates a $(v_1, v_2) := (\sum w_1 + \I, \sum w_2 - 1)$ regular almost-tiling.

Further suppose $(v_{1}, v_2) = (v_{n/d, 1}, v_{n/d,2})$ for some reduced fraction $0 < n/d < 1$. Then, the surrounding of $T$ with respect to $(v_{n/d, 1}, v_{n/d, 2})$ consists of two stacked zero-one horizontal and two vertical boundary strings and the translation condition,
\begin{equation} \label{eq:translation_condition}
o(x \pm v_{n/d,i}) = o(x) \pm a_{n/d,i}^T x + k_{n/d, \pm i}  \quad \mbox{for $x \in \Z^2$},
\end{equation}
where $k_{n/d, \pm i}$ are constants and $i \in \{1,2\}$ selects the lattice vector, uniquely extends $o:T \to \Z$ to the plane. 
\end{lemma} \label{lemma:common-extension-pseudo-square}

\begin{proof}
The first claim follows from Lemmas \ref{lemma:pseudo-square-boundary-word} and \ref{lemma:almost_square_tiling}.  We next check that the interfaces $(T, T \pm v_{n/d,2})$ and $(T, T \pm v_{n/d, 1})$ are stacked horizontal or vertical zero-one boundary strings respectively. 

Let $A$ be the horizontal string for $T$ and $B$ the reversed horizontal string for 
$T - v_{n/d, 2}$.  By Lemma \ref{lemma:pseudo-square-boundary-word}, the first tile in $B$ is located at $c(T) - v_{n/d,2} + v_{n/d,2} - (2 \I - 1)$. Thus, the offset between the first tile in $B$ and the first tile in $A$ is $2 \I - 1 = v_{0/1,2} + \I$, the correct
initial offset for a stacked string. 

Similarly if $C$ is the vertical string for $T$ and $D$ the reversed vertical string for $T + v_{n/d,1}$, then 
the offset between the respective first tiles in $C$ and $D$ is $(\I+1) = -v_{1/1,1}$. The other two interfaces are stacked strings by the above arguments for $T - v_{n/d,1}$ and $T + v_{n/d,2}$.  

Let $\mathcal{T} = \{ T + i v_{n/d, 1} + j v_{n/d,2}\} $ denote the almost tiling of $T$ and for $T_{i,j} \in T$, let 
$o_{i,j}:T_{i,j} \to \Z$ denote the translations of $o$ by $(i v_{n/d,1} + j v_{n/d,2}, i a_{n/d,1} + j a_{n/d, 2})$.
By definition, each $o_{i,j}$ respects $(w_h,w_v)$ on $T_{i,j}$. Restrict the $o_{i,j}$ to 
the stacked boundary strings and check, by repeating the above argument, that the slope differences between the first two perpendicular tiles in the stacked horizontal strings are $-a_{p,2}$. Similarly, the slope difference for the stacked vertical strings are $a_{q,1} = 0$.
Therefore, by Lemma \ref{lemma:stacked_zero_one_string}, each pair of odometers is compatible. This together with Lemma \ref{lemma:pairwise_compatibility} implies there is a common extension of $o_{i,j}$
to the plane. 
\end{proof}

We also require the notion of a tile odometer respecting only a horizontal boundary or vertical boundary string.
\begin{definition} \label{def:partial_tile_boundary_strings}
	A $w_{h}$ or $w_v$-pseudo-square is a tile $T$, whose boundary contains (but may not be equal to)
	\[
	\mathcal{T}_{h,*} = \{T_{i,h}^+\} \cup \{ T_{i,h}^-\} 
	\]
	or
	\[
	 \mathcal{T}_{*,v} = \{ T_{i,v}^+\} \cup \{ T_{i,v}^-\}, 
	\]
	where $T$ are as in Definition \ref{def:tile_boundary_strings} and either $c(T) = c(T_{1,h}^+)$ or $c(T) = c(T_{1,v}^-)$
	
	A partial odometer $o: T \to \Z$ respects $w_{h}$ or $w_v$ 
	if its restriction to $\mathcal{T}_{h,*}$ are $w_h$-horizontal and $\rev(w_h)$-reversed-horizontal
	strings or its restriction to $\mathcal{T}_{*,v}$ are $w_v$ and and $\rev(w_v)$-reversed-vertical zero-one boundary strings respectively. 	
\end{definition}

\subsection{Explicit formulae for zero-one boundary strings} \label{subsec:bs_explicit_formulae}
We collect in this section some explicit formulae for zero-one boundary strings which are straightforward consequences
of the definitions. In particular, these will correspond
to the degenerate base cases in \eqref{eq:words_degenerate_cases}.

The formulae are only used to verify the explicit odometers in Section \ref{sec:degenerate_cases} and may be skipped on a first read.

\subsubsection{Horizontal boundary strings} \label{subsubsec:horizontal_bs_explicit}

We first note the form of the odometers after a translation. 
The zero-odometer translated by vector $v = 0$, affine factor $a = -a_{p,2}$, and constant $b = -1$,  $\hat{o}_{0/1}: T_{0/1} \to \Z$ is given by $\hat{o}_{0/1}(0)=\hat{o}_{0/1}(1) = -1$ and $\hat{o}_{0/1}(\I) = \hat{o}_{0/1}(1+ \I) = 0 = \hat{o}_{0/1}(2 \I) = \hat{o}_{0/1}(1 + 2 \I) = 0$.  The one-odometer translated by
$v = 0$, $a =-a_{q,2}$, and $b = 0$,  $\hat{o}_{1/1}:T_{1/1} \to \Z$ is given by 
$\hat{o}_{1/1}(0)=\hat{o}_{1/1}(\I) = \hat{o}_{1/1}(1+ \I)=0$ and $\hat{o}_{1/1}(1) = -1$. Similarly, the enlarged one-odometer translated
by the same parameters, $\hat{o}_{1/1}^d:T_{1/1}^d \to \Z$ is given by $\hat{o}_{1/1}^d = \hat{o}_{1/1}$ on $T_{1/1}$ and $\hat{o}_{1/1}(1-\I)=-2$, $\hat{o}_{1/1}(\I-1)=0$.

Now, let  $\{o_i^+\}$ (resp. $\{o_i^-\}$) respect an arbitrary (resp. reversed) horizontal boundary string.
If $o_{1}^+ \in \{ o_{0/1}, o_{1/1}\}$,
then, $o_i^+ \in \{o_{0/1}, o_{1/1}\}$ depending on the respective letter. Similarly, if $o_{1}^- \in \{ \hat{o}_{0/1}, \hat{o}_{1/1}\}$
then $o_i^- \in \{\hat{o}_{0/1}, \hat{o}_{1/1}\}$.

\subsubsection{Vertical boundary strings} \label{subsubsec:vertical_bs_explicit}
The vertical boundary string case involves more computations since the translations involve non-zero affine factors. 
Fix $k \geq 1$.  The following functions, corresponding to the quadratic growth of the hyperbola bases, will be used:
\begin{equation}
t(j) = -j(j+1)/2 \qquad q(j) = -\lfloor \frac{j^2}{4} \rfloor
\end{equation}
In each case, let $\{o_i^+\}$  (resp. $\{o_i^-\}$) respect the indicated (resp. reversed) vertical boundary string.

\subsubsection*{Case 1: $p^{k} q$ and its reversal}
Suppose $o^+_1 = o_{0/1}$ on $T_{0/1}$, the first tile in $p^k q$. Then,  for $1 \leq j \leq k$
\begin{equation} \label{eq:explicit-forms-bs-1-a}
o^+_{j} = 
\begin{bmatrix}
t(j) &t(j) \\
t(j-1) &t(j-1) \\
t(j-2) &t(j-2) 
\end{bmatrix}
\end{equation}
and
\begin{equation} \label{eq:explicit-forms-bs-1-b}
o^+_{k+1} = 
\begin{bmatrix}
t(k+1) & t(k+1) + 1 \\
t(k) &t(k)
\end{bmatrix}.
\end{equation}
If $o^-_1 = o_{1/1}$ on $T_{1/1}$, the first tile in $\rev(p^k q)$, then
\begin{equation} \label{eq:explicit-forms-bs-1-c}
o^-_{j} = 
\begin{bmatrix}
t(j+1) + 2 & t(j+1) +3 \\
t(j) + 1 & t(j) + 2 \\
t(j-1) & t(j-1)+1
\end{bmatrix}
\end{equation}
for $2 \leq j \leq k+1$.

\subsubsection*{Case 2: $p q^{k}$ and its reversal}
Suppose $o_{1}^+ = o_{0/1}$ on $T_{0/1}$, the first tile in $p q^k$. Then for $2 \leq j \leq k+1$, $j' = 2(j-1)$
\begin{equation}\label{eq:explicit-forms-bs-2-a}
o_j^+ = 
\begin{bmatrix}
q(j'+2) +1 & q(j'+1) \\
q(j'+1) +1 & q(j')
\end{bmatrix}.
\end{equation}
If $o_1^{-} = o_{1/1}$ on $T_{1/1}$, the first tile in $\rev(p q^k)$ then for $1 \leq j \leq k$
\begin{equation}\label{eq:explicit-forms-bs-2-b}
o_j^- = 
\begin{bmatrix}
q(2 j) & q(2j -1) \\
q(2j-1) & q(2j-2) \\
\end{bmatrix}
\end{equation}
and
\begin{equation}\label{eq:explicit-forms-bs-2-c}
o_{k+1}^- = 
\begin{bmatrix}
q(2k+2) & q(2k+1) - 1\\
q(2k+1) & q(2k) \\
q(2k) & q(2k-1)
\end{bmatrix}.
\end{equation}

\subsubsection*{Case 3: $p q^{k} p q^{k+1}$ and its reversal}
Suppose $o_{1/1}^+ = o_{0/1}$ on $T_{0/1}$, the first tile in $p q^{k} p q^{k+1}$.  Then $o_{j}^+$ for $1 \leq j \leq (k+1)$ are as Case 2 above. 
Then 
\begin{equation}\label{eq:explicit-forms-bs-3-a}
o_{k+2} = 
\begin{bmatrix}
q(2(k+1) +3) +3 & q(2(k+1)+2) + 1 \\
q(2(k+1)+2)+3 & q(2(k+1)+1) + 1 \\
q(2(k+1)+1) & q(2(k+1))+1
\end{bmatrix}
\end{equation}
and for $k+3 \leq j \leq 2k +3$, $j'' = 2(j-2)$
\begin{equation}\label{eq:explicit-forms-bs-3-b}
o_{j} = 
\begin{bmatrix}
q(j''+4) + 3 & q(j''+3)+1 \\
q(j''+3) + 3 & q(j''+2) + 1 
\end{bmatrix}.
\end{equation}
If $o_1^{-} = o_{1/1}$ on $T_{1/1}$, the first tile in $\rev(p q^{k} p q^{k+1})$ then $o_{j}^-$ for $1 \leq j \leq (k+2)$ the first $(k+2)$ are as Case 2 above. 
Then, for $(k+3) \leq j \leq 2k+2$, 
\begin{equation}\label{eq:explicit-forms-bs-3-c}
o^-_{j} = 
\begin{bmatrix}
q(2(j-1) + 2) + 1 & q(2(j-1) + 1) \\
q(2(j-1) + 1) + 1 & q(2(j-1)) 
\end{bmatrix}
\end{equation}
and for $w = 2(k+1)$
\begin{equation}\label{eq:explicit-forms-bs-3-d}
o^-_{2k+3} = 
\begin{bmatrix}
q(2 w + 2) + 1 & q(2w + 1) \\
q(2w + 1) + 1 & q(2w) \\
q(2w) + 1 & q(2w-1) 
\end{bmatrix}.
\end{equation}

\section{Base cases} \label{sec:degenerate_cases}
Before we extend the hyperbola recursion to all odometers and tiles, 
we study a degenerate family and in fact prove Theorem \ref{theorem:odometers}
for this family. The reader is encouraged to skim or skip this section and come back to it only after 
reading Section \ref{sec:odometers}.

The reduced fractions which we analyze here are those in a Farey quadruple where at least one of the two parents is $(0/1)$ or $(1/1)$. 
Specifically, we prove the following.

\begin{prop} \label{prop:base_cases}
For each Farey quadruple of the form $\q_{(w)} = (p_1, q_1, p_2, q_2)$ where $w = 3^k$ or $2^k$ for $k \geq 0$
there is a quadruple of standard and alternate tile odometers
\[
(o_{p_1}, o_{q_1}, o_{p_2}, o_{q_2}) \quad \mbox{ and } \quad (\hat{o}_{p_1}, \hat{o}_{q_1}, \hat{o}_{p_2}, \hat{o}_{q_2}) 
\]
with finite domains, $T(o_{n/d}) = T_{n/d}$ and $T(\hat{o}_{n/d}) = \hat{T}_{n/d}$. 
For each such $w$, the tile odometers of the child Farey pair satisfy the following properties. 
\begin{enumerate}[label=(\alph*)]
	\item Under the lattice $L'(n/d)$, $T(n/d)$ generates an almost pseudo-square tiling.
	\item $\hat{T}(n/d)$ covers $\Z^2$ under $L'(n/d)$.
	\item There exist unique, distinct recurrent extensions $o_{n/d}:Z^2 \to \Z$ and $\hat{o}_{n/d}:\Z^2 \to \Z$ satisfying 
	the correct growth dictated by \eqref{eq:periodicity}.
	\item $T(n/d)$ is a $(w_h,w_v)$-pseudo-square which $o_{n/d}$ respects.
	\begin{center}
		\begin{tabular}{ c  c|c c }
			Standard case &  &$w_h$ &  $w_v$ \\ 
			\hline
			$3^k$ odd, $k \geq 1$  & $1/d$, $d \geq 4$ even & $qp^kqp^{k+1}$ & $p^{2k+1}q$  \\ 
			$3^k$ even, $k \geq 0$ & $1/d$, $d \geq 3$ even & $qp^{k+1}$ & $p^{2(k+1)}q$ \\ 
			$2^k$ odd,  $k \geq 0$& $\mathcal{R}(1/d)$ $d \geq 3$ even & $q^{2(k+1)}p$ & $pq^{k+1}$ \\ 
			$2^k$ even, $k \geq 1$ &  $\mathcal{R}(1/d)$, $d \geq 4$ even  & $q^{2k+1}p$ & $pq^kpq^{k+1}$
		\end{tabular}
	\end{center}
	The first column denotes a word which selects a degenerate Farey quadruple and the parity of the reduced fraction displayed in the second column.
	\item $\hat{T}(n/d)$ is a $w_{h/v}$-pseudo-square which $\hat{o}_{n/d}$ respects
	\begin{center}
		\begin{tabular}{ c  c|c c }
			Alternate case &  &$w_h$ &  $w_v$ \\ 
			\hline
			$3^k$ odd, $k \geq 1$  & $1/d$, $d \geq 4$ even & - &$p^{2k+1}q$  \\ 
			$3^k$ even, $k \geq 0$ & $1/d$, $d \geq 3$ even & $qp^{k+1}$ & - \\ 
			$2^k$ odd,  $k \geq 0$& $\mathcal{R}(1/d)$ $d \geq 3$ even & - & $pq^{k+1}$ \\ 
			$2^k$ even, $k \geq 1$ &  $\mathcal{R}(1/d)$, $d \geq 4$ even  & $q^{2k+1}p$ & - 
		\end{tabular}
	\end{center}
	In particular, the alternates coincide with the standards on one set of boundaries. 
	\item  Some later odometers contain exact translations of earlier odometers. To state this succinctly, 
	write $w(p)$ and $w(q)$, respectively, for the odd and even reduced fraction in the child Farey pair of $\q_{(w)}$
	and let $T(n/d)(v) = T(n/d) \cup (T(n/d) + v)$ for $v \in \Z[\I]$ and $n/d \in \{p,q\}$. 
	The following holds for all $k \geq 1$: 
	\begin{equation} \label{eq:3^k_decomp}
	\begin{aligned}
	 T(3^{k}(p)) &\supset T(q, v_{q,1} + v_{p,1} + v_{p,2}) \qquad \mbox{ offset $=0$} \\
	 \hat{T}(3^{k}(p))	 &\supset T(q, v_{q,1} + v_{p,1} + 2 v_{p,2}) \qquad \mbox{ offset $=0$ }  \\
	 \hat{T}(3^{k}(q)) &\supset T(q, -v_{q,1} - v_{p,1} ) \qquad \mbox{ offset $=v_{p,2} +v_{p,1}$ } 
	 \end{aligned}
	 \end{equation}
	 where $(p,q) = 3^{k-1}(p,q)$ and
	 \begin{equation} \label{eq:2^k_decomp}
	 \begin{aligned}
	 T(2^{k}(q)) &\supset T(p, v_{p,2} + v_{q,2} - v_{q,1}) \qquad \mbox{ offset $=v_{q,1}$ } \\
	 \hat{T}(2^{k}(q)) &\supset T(p, v_{p,2} + v_{q,2} - 2 v_{q,1}) \qquad \mbox{ offset $=v_{q,1}$}\\
	 	\hat{T}(2^{k}(p)) &\supset T(p, v_{p,2} + v_{q,2}) \qquad \mbox{ offset $=v_{q,1}$ }
	 	\end{aligned}
	 \end{equation}
	 where $(p,q) = 2^{k-1}(p,q)$. The third column records $c(T_1) - c(T_2)$ where $T_1$ is the tile in the second column
	 and $T_2$ is the tile in the first column. 
	
The tile odometers for $k \geq 1$ have an analogous decomposition with affine factors and translations dictated by \eqref{eq:3^k_decomp} and \eqref{eq:2^k_decomp}. For example, the restriction of $o_{3^{k}(p)}$ to $T(q, v_{q,1} + v_{p,1} + v_{p,2})$
is exactly equal to translated earlier tile odometers, $o_{q}^1 \cup o_{q}^2$ where $c(T(o_{q}^1)) - c(T_{o_{3^{k}(p)}}) = 0$, 
	 $T(o_{q}^1) \cup T(o_{q}^2) = T(q, v_{q,1} + v_{p,1} + v_{p,2})$,  $s(o_{q}^1) - s(o_{q}^2) = a_{p,2}$, and $s(o_{q}^2) - s(o_{3^{k}(p)}) = 0$. 
	 
	 \item Some later odometers contain partial translations of earlier odometers. The following holds for all $k \geq 1$ (using the same 
	 notation as the previous item):
	\begin{equation} \label{eq:even_staircase_decomp}
	 T(3^{k}(q)) \supset T(q, v_{p,1}+2v_{p,2}) \qquad \mbox{ offset $=0$ } \\
	\end{equation}
	where $(p,q) = 3^{k-1}(p,q)$ and 
	\begin{equation} \label{eq:odd_staircase_decomp}
	T(2^{k}(p)) \supset T(p,-2v_{q,1}+v_{q,2}) \qquad \mbox{ offset $=2v_{q,1}$ } \\
	\end{equation}
	where $(p,q) = 2^{k-1}(p,q)$. The tile odometers for $k \geq 1$ have an analogous decomposition (as in the previous item) but only after removing two corner cells from each of the subtiles on the right-hand-side: 
	\begin{equation} \label{eq:truncated_even_staircase}
	T^{sm}(q) = T(q) \backslash \{ c_1 \cup c_2 \} \qquad \mbox{ where $(p,q) = 3^{k-1}(p,q)$}
	\end{equation}
	where $c_1 =  c(T(q))$ and $c_2 = c_1 +  (v_{q,1} +v_{q,2}-v_{p,2})$ and
	\begin{equation} \label{eq:truncated_odd_staircase}
	T^{sm}(p) = T(p) \backslash \{ c_1' \cup c_2' \} \qquad \mbox{ where $(p,q) = 2^{k-1}(p,q)$}
	\end{equation}
	where $c_1' =  c(T(p)) + v_{p,1}-\I$ and $c_2' =  c(T(p))+v_{p,2}+1$.

\end{enumerate}

\end{prop}

This family will form the base cases for the general recursion in the subsequent section. 
As noted above, there is a recursive structure here but with some `errors' 
in the full decomposition. If the tile sizes are reduced to avoid these errors, then later tiles
will be too small to cover $\Z^2$. 

Since these errors are limited to the degenerate family and the odometers for this family 
are so simple, we take the cumbersome but elementary approach and provide the exact formulae. 
One could avoid this by adding additional cases to the general recursion. 

\subsection{Base points}
We first check that the base points of the hyperbola $0/1$ and $1/1$, are on $\partial \Gamma_F$ via an explicit construction. 
We recall a criteria for checking recurrence from the sandpile literature.  
Let $s: \Z^2 \to \Z$ and let $H$ be a finite induced subgraph of the $F$-lattice. $H$ is {\it allowed} for $s$ if there is 
a vertex $v$ of $H$ where $s(v)$ is at least the in-degree of $v$ in $H$ and otherwise is {\it forbidden}. 

\begin{prop} \label{prop:recurrence} \cite{holroyd2008chip}
	An integer superharmonic function $g$ is recurrent if and only if every nonempty induced subgraph 
	of the $F$-lattice is allowed for $s := \Delta g+1$. 
\end{prop}
In particular, Proposition \ref{prop:recurrence} reduces verifying recurrence of a function to checking a condition on its Laplacian
(which is no surprise given the function $s$ in the statement is usually referred to as a recurrent sandpile \cite{levine2010sandpile}). 
The equivalence between the two definitions is given in \cite[Proposition 3.3]{bou2021convergence}.

\begin{lemma} \label{lemma:checkerboard_recurrence}
The functions 
\[
g_{0/1}(x) = -\frac{x_2 (x_2 + 1)}{2} \qquad g_{1/1}(x) = -\lfloor \frac{(x_2-x_1)^2}{4} \rfloor
\]
are odometers for $0/1$ and $1/1$ respectively. 
\end{lemma}

\begin{proof}
The growth condition can be checked using the definition \eqref{eq:quadratic_growth}.
Moreover, $\Delta g_{0/1}(x) = \Delta g_{1/1}(x) = - 1\{ (x_1+x_2) \mbox{ is odd}\}$. 
By Proposition \ref{prop:recurrence} it remains to check that every nonempty induced subgraph $H$ of the $F$-lattice is allowed 
for $s = 1\{ (x_1+x_2) \mbox{ is even}\}$. Let $x$ denote the lower left vertex of $H$. That is $x$ has minimal $x_1$ coordinate 
and of all other $y \in H$ with $y_1 = x_1$, $x_2$ is minimal. This implies the only possible neighbors of $x$ in $H$ are $x+e_1$ or $x+e_2$.  If $(x_1 + x_2)$ is even, then $s(x) = 1$ so we may suppose otherwise.
If $(x+e_2) \in H$, then $s(x+e_2) = 1$ and by our choice of $x$,  $(x+e_2 - e_1) \not \in H$, thus $s(x+e_2)$ is larger
than its in-degree in $H$, completing the proof.  
\end{proof}

\subsection{Staircases}
\begin{figure}
\includegraphics[scale=0.35]{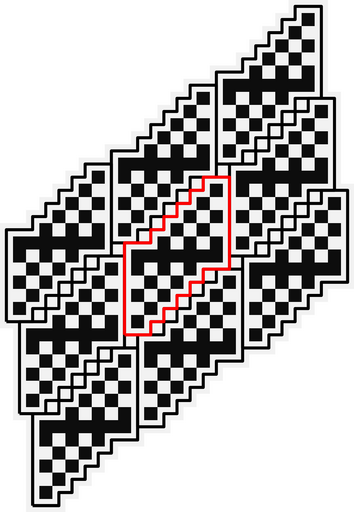} \qquad
\includegraphics[scale=0.35]{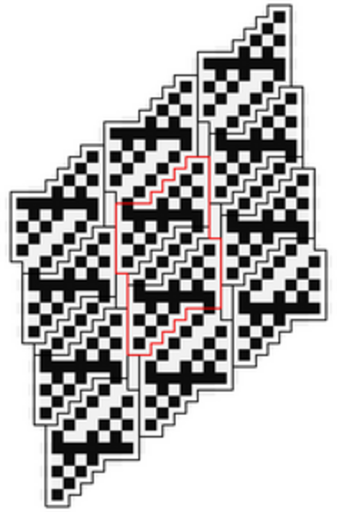} 
\caption{A period of the Laplacian of a staircase odometer on the left and its alternate. 
	The string is $22$ and the fraction is $3/4$. Each tile is outlined in the dual lattice.} \label{fig:staircase_odd}
\end{figure}
\begin{figure}
	\includegraphics[scale=0.35]{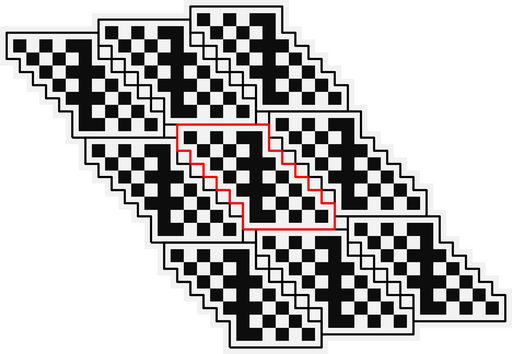}\quad
	\includegraphics[scale=0.35]{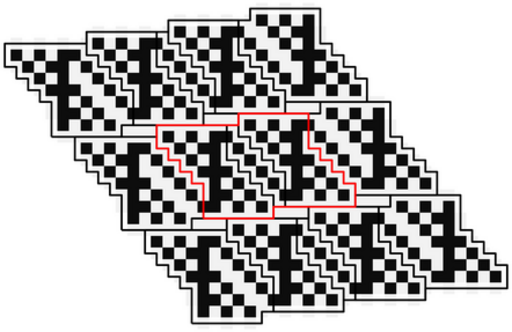}
	\caption{The rotated standard and alternate staircase odometers of Figure \ref{fig:staircase_odd}. The string is $33$
		and the fraction is $1/7$.} \label{fig:staircase_even}
\end{figure}

The {\it staircase} fractions are the reduced fractions of the form $\frac{1}{d}$ for $d$ odd and their rotations, 
$\mathcal{R}(1,d) = (\frac{d-1}{2}, \frac{d+1}{2})$. These fractions are the even (resp. odd child) in Farey quadruples 
$\q_{3}^k$  (resp. $\q_{2}^k$) for $k \geq 0$. The constructed tiles and Laplacians 
will respect the rotational invariance inherited from Lemma \ref{lemma:rot_invariance}. 

\iffalse
 We recall the translations
\begin{equation}
\begin{aligned}
v_{1/d,1} &= \quad v_{1/d,2}  \\
a_{1/d,1} &= 0 \quad a_{1/d,2} =  
\end{aligned}
\end{equation}
and for $p =  (\frac{d-1}{2}, \frac{d+1}{2})$, $d \geq 3$ odd, 
\begin{equation}
\begin{aligned}
v_{p,1} &= \quad v_{p,2}  \\
a_{p,1} &= 0 \quad a_{p,2} = 
\end{aligned}
\end{equation}
\fi

We start with the standard even child of $\q_{3^k}$ for $k \geq 0$.  %$k = (d-1)/2-1$. $
\begin{lemma} \label{lemma:basecase_3k_even_std}
	For each $d \geq 3$ odd, 
	\[
	T_{1/d} := \{ (x_1,x_2) \in [2-d, d] \times [0,d+1]:  1 \leq (x_1+x_2) \leq (d+2)\} \cup  \{ (0, 0), (2 , d+1)\}
	\]
	 is a $(q p^{k+1}, p^{2(k+1)}q)$-pseudo-square which $g_{1/d}:T_{1/d} \to \Z$, given by,
	 \[
	 g_{1/d}(x) = -\frac{1}{2} x_2 (x_2 + 1) + (x_2 + \min(x_1-1,0)) + 1\{x \in \{ (0, 0), (2 , d+1)\} \} \
	 \]
	 respects. Recall that $d = 2k + 3$. 
	  
\end{lemma}

\begin{proof}
We start by observing that the bottom boundary of $T_{1/d}$ is $q p^{k+1}$ zero-one horizontal boundary string. Indeed, 
\[
T_{1/d} \supset A_1 \cup A_2 := \{ [0,1] \times [0,1] \} \cup \{ [2,d] \times [0,2]\} 
\]
is a $q p^{k+1}$ horizontal boundary string: 
\[
A_1 = T_{1/1} \quad A_2 = v_{p,1} + \cup_{j=1}^{k+1} ( T_{0/1} + (j-1) v_{p,1}\}   . 
\]
Moreover an inspection of the formula shows that $g_{1/d} = o_{1/1}$ on $A_1$ 
and $g_{1/d} = o_{0/1}$ on the translations of $T_{0/1}$ which form $A_2$.

The top boundary of $T_{1/d}$ is a $\rev(p^{k+1} q)$ zero-one horizontal reversed boundary string. Indeed, 
\[
T_{1/d} \supset A_2^r \cup A_1^r := \{ [2-d, 0] \times [d-1,d+1]\} \cup \{[1,2] \times [d,d+1]\} 
\]
is a $p^{k+1} q$ zero-one reversed horizontal boundary string:
\[
A_2^r = \cup_{j=0}^{k} \{ T_{0/1} + j v_{p,1}\} \quad A_1^r = T_{1/1} + k v_{p,1} + v_{q,1}+1.
\]
To check that $g_{1/d}$ respects the string, it is convenient to consider the translation 
\[
g_{1/d}^r  = g_{1/d} - (1,-d)^T x - \frac{1}{2} d (d+1) + 1
\]
and recall the translated versions of the zero-one odometers, $\hat{o}_{1/1}$ and $\hat{o}_{0/1}$ defined in Section \ref{subsubsec:horizontal_bs_explicit}. 
Once we make this translation, we can use the formula to compute
\[
g_{1/d}^r( [1,2] \times [d,d+1]) = \hat{o}_{1/1} = \begin{bmatrix} 0 & 0 \\ 0 & -1 \end{bmatrix}
\]
and $g_{1/d}^r([2-d,0] \times [d-1,d+1]) = 0$ and $g_{1/d}^r([2-d,0], d-1) = -1$ which coincides with copies of $\hat{o}_{0/1}$. 
	
	The check for the right and left boundaries proceeds by comparing to the explicit formulae for the degenerate zero-one strings given in Section \ref{subsubsec:vertical_bs_explicit}. We use the notation defined there. We first check that the right-boundary is a $p^{2(k+1)}q$ string by observing
	\[
	T_{1/d} - (d-1,0) \supset A_1 \cup A_2 := \cup_{j=0}^{2(k+1)-1} (T_{1/1} + j v_{p,2}) \cup (T_{0/1} + (2(k+1)-1) v_{p,2} + v_{q,2}).
	\] 
	Indeed, the upper right corner of each $v_{p,2}$ translation of $T_{1/1}$ satisfies the equality $x_1 + x_2 = 3$ and the upper right corner
	of $T_{0/1} + (2(k+1)-1) v_{p,2} + v_{q,2}$ is $(3-d,d+1)$. To check the formula matches \eqref{eq:explicit-forms-bs-1-a}
	observe that on $A_1 \cup A_2 + (d-1,0)$, $g_{1/d} = -\frac{1}{2} x_2(x_2+1) + x_2$. And so, on each $T_j := T_{1/1} + (j-1) v_{p,2} + (d-1,0)$, 
	\[
	g_{1/d}\restriction_{T_j} =  
	\begin{bmatrix}
	t(j+1) & t(j+1) \\
	t(j) & t(j) \\
	t(j-1) & t(j-1) 
	\end{bmatrix}
	+
	\begin{bmatrix}
	j+1 & j+1 \\
	j  &  j  \\
	j-1 & j-1 	
	\end{bmatrix} 
	= 
	\begin{bmatrix}
	t(j) & t(j) \\
	t(j-1) & t(j-1) \\
	t(j-2) & t(j-2)
	\end{bmatrix}.
	\]	
	The formula also implies it coincides with \eqref{eq:explicit-forms-bs-1-b} on $A_2$. 
	
	The argument for the left boundary is symmetric. Start by observing 
	\[
	T_{1/d} \supset A_1 \cup A_2 := T_{0/1} \cup \cup_{j=0}^{2(k+1)-1} (T_{1/1} + j v_{p,2} + \I).
	\]
	Since, on the left boundary $g_{1/d} = -\frac{1}{2} x_2(x_2+1) + (x_1+x_2-1)$, $g_{1/d}\restriction_{A_1} = o_{0/1}$. 
	The lower left corner of each $T_j := T_{1/1} + \I + (j-2) v_{p,2}$ for $2 \leq j \leq 2(k+1)+1$ lies on the left boundary, $(x_1 + x_2) = 1$. 
	Hence, 
	\[
	g_{1/d}\restriction_{T_j} =  
	\begin{bmatrix}
	t(j+1) & t(j+1) \\
	t(j) & t(j) \\
	t(j-1) & t(j-1) 
	\end{bmatrix}
	+
	\begin{bmatrix}
	2 & 3 \\
	1  &  2  \\
	0 &  1 	
	\end{bmatrix} 
	\]
	which is exactly \eqref{eq:explicit-forms-bs-1-c}.
\end{proof}

Next is the standard odd child of $\q_{2^k}$. 
\begin{lemma} \label{lemma:basecase_2k_odd_std}
	For each $p_1 = \frac{\frac{d-1}{2}}{\frac{d+1}{2}}$, $d \geq 3$ odd,
	\[
	T_{p_1} := \{ (x_1,x_2) \in [0,d+1] \times [1,2d-1]:  -1 \leq (x_2-x_1) \leq d\} \cup  \{ (0, d+1), (d+1, d-1)\}
	\]
	is a $(q^{2(k+1)}p,p q^{(k+1)})$-pseudo-square which $g_{p_1}:T_{p_1} \to \Z$, given by, 
	\[
	g_{p_1}(x) = -\lfloor \frac{(x_2-x_1)^2}{4} \rfloor + \min(d-x_2,0) + 1\{x \in \{ (0, d+1), (d+1, d-1)\} \} 
	\]
	respects. 
\end{lemma}

\begin{proof}
The computations are identical to that of previous Lemma. In this case, the vertical strings match the formula
given by \eqref{eq:explicit-forms-bs-2-a}, \eqref{eq:explicit-forms-bs-2-b}, \eqref{eq:explicit-forms-bs-2-c}.
%The argument is a rotation of the previous Lemma. We start by checking that the bottom boundary of $T_{1/d}$ is a $q^{2(k+1)}p$ zero-one horizontal boundary string. First, 
%\[
%T_{p_1} \supset A_1 \cup A_2 := T_{1/1} \cup \cup_{j=2}^{2(k+1)} (T_{1/1}^d + (j-1) v_{q,1}) \cup (T_{0/1} + (2(k+1)-1)v_{q,1} + v_{p,1})
%\]
%is the bottom boundary as the lowermost square of each $T_j := T_{1/1}^d + (j-1) v_{q,1}$ satisfies $(x_2 - x_1) = -1$. 
%Also, on the bottom boundary, $g_{p_1} = -\lfloor \frac{(x_2-x_1)^2}{4}  + 1_{(0,d+1)}$.
%Hence, $c(T_{p_1}) = (0,1) = c(T_{1/1})$, $g_{p_1} = o_{0/1}$ on $T_{1/1}$ and $g_{p_1} = o^d_{0/1} on each $T_j$
%and $g_{p_1} = o_{1/1}$ on $A_2$. 
%
%
%The top boundary of $T_{p_1}$ is a $\rev(q^{2(k+1)}p)$ zero-one horizontal reversed boundary string.
%To check that $g_{p_1}$ respects the string, it is convenient to consider the translation 
%\[
%g_{p_1}^r  = g_{p_1} - (1,-d)^T x - \frac{1}{2} d (d+1) + 1
%\]
%and recall the translated versions of the zero-one odometers, $\hat{o}_{1/1}$ and $\hat{o}_{0/1}$ defined in Section \ref{subsubsec:horizontal_bs_explicit}. 
%Once we make this translation, we can use the formula to compute
%\[
%g_{1/d}^r( [1,2] \times [d,d+1]) = \hat{o}_{1/1} = \begin{bmatrix} 0 & 0 \\ 0 & -1 \end{bmatrix}
%\]
%and $g_{1/d}^r([2-d,0] \times [d-1,d+1]) = 0$ and $g_{1/d}^r([2-d,0], d-1) = -1$ which coincides with copies of $\hat{o}_{0/1}$. 
\end{proof}

We next construct the alternate staircase odometers. By Lemma \ref{lemma:common-extension-pseudo-square}, each of the tile odometers defined in the previous two lemmas extend to $\Z^2$ under $L'(n/d)$ with the correct growth \eqref{eq:periodicity_growth}. However, in the next two cases, we require a different argument as the tiling of alternate tiles will result in overlaps.  We start with the alternate even child of $\q_{3^k}$

\begin{lemma} \label{lemma:basecase_3k_even_alt}
	For each $d \geq 3$ odd, there is a unique function $\hat{g}_{1/d}: \Z^2 \to \Z$ with 
	\begin{align*}
	\hat{g}_{1/d}(x) &= -\frac{1}{2} x_2 (x_2 + 1)  + \min(0, 2 - x_1) + \min(0, d-2 + x_1)  \quad \mbox{for  $x \in \hat{T}_d$}\\
	&+ \max(x_1+x_2-2,0) + 1_{A \cup B} 
	\end{align*}
	where 
	\begin{align*}
	A &= \{ -(d-1) \times [-1,0]\} \cup \{ 3 \times [d,d+1]\} \\
	B &=  \{ (x_1,x_2) :  (x_1+x_2) = 2 \mbox{ and } 0 < x_2 < d \}    \\ 
	C &= \{  [1-d,0] \times -1  \} \cup \{ [4-d,3] \times (d+1) \} \\ 
	D&=  \{ (x_1, x_2) \in [ 4-2d, d] \times [0, d]:  -d + 2 \leq x_1 + x_2 \leq d+2 \}  \\
	\hat{T}_{1/d} &=  A \cup B \cup C \cup D
	\end{align*}
	and
	\[
	\begin{aligned}
	\hat{g}_{1/d}(x \pm v_{1/d,i}) &= \hat{g}_{1/d}(x) \pm a_{1/d,i}^T x + k_{1/d,i}  \quad \mbox{for $x \in \Z^2$},
	\end{aligned}
	\]
	where $k_{p,\pm i} \in \Z$ is a constant and $i \in \{1,2\}$ selects the lattice vector.
	Moreover, $\hat{T}_{1/d}$ is a $w_h = q p^{k+1}$ pseudo-square which $\hat{g}_{1/d}$ respects. 
\end{lemma}
\begin{proof}
	Consistency of the first condition and translation by $(v_{1/d,1}, a_{1/d,1})$ comes after checking that
	$\hat{g}_{1/d}(x) = \hat{g}_{1/d}(x + v_{1/d,1})$ for $x, (x+ v_{1/d,1}) \in \hat{T}_{1/d}$. In particular, this shows the constant 
	$k_{1/d,i} = 0$ for $i = 1$. Write $\hat{g}_{1/d}: \mathbf{T}_h \to \Z$ for the common extension of $\hat{g}_{1/d}:T\ \to \Z$
	to $\mathbf{T}_h := \cup _{i \in \Z} (T + \I v_{1/d,1})$. Note that since there are no gaps between $T$ and $T \pm v_{1/d,1}$, 
	$\mathbf{T}_h$ is simply connected.

	From the formula and Section \ref{subsubsec:horizontal_bs_explicit} we may check there is a $w_h$ horizontal boundary string starting at $(0,0)$ and a $\rev(w_h)$ reversed horizontal boundary string starting at $(4-2d,d-2)$ which $\hat{g}_{1/d}$ respects on $T_{1/d}$. In fact, the top boundary of $\mathbf{T}_h$ is a repeating sequence of $\rev(w_h)$ and the bottom boundary 
	is a repeating sequence of $w_h$ and $\hat{g}_{1/d}$ respects both infinite strings. This implies $( \pm v_{1/d,2}, \pm a_{1/d,2})$ translations of $\hat{g}_{1/d}$ form stacked boundary strings and thus have a common extension to $\cup_{j \in \Z} (\mathbf{T}_{h} + j v_{1/d,2})$. 
	
	Vertical translations of $\mathbf{T}_{h}$ cover the plane since $\mathbf{T}_{h}$ is simply connected and each $\pm v_{1/d,2}$ interface is a stacked zero-one boundary string which is simply connected by Lemma \ref{lemma:zero-one-fixed-offsets}. 
\end{proof}

Next is the alternate odd child of $\q_{2^k}$. 
\begin{lemma} \label{lemma:basecase_2k_odd_alt}
	For each $p_1 = \frac{\frac{d-1}{2}}{\frac{d+1}{2}}$, $d \geq 3$ odd, there is a unique function $\hat{g}_{p_1}: \Z^2 \to \Z$ with 
	\begin{align*}
	\hat{g}_{p_1}(x) &= -\lfloor \frac{(x_2-x_1)^2}{4} \rfloor + \min(d-x_2-1,0) + \min((2d-1)-x_2,0)  \quad \mbox{ for $x \in \hat{T}_{p_1}$} \\
	&+ \max((x_2-x_1)-d+1, 0) + 1_{A \cup B} 
	\end{align*}
	where
	
	\begin{align*}
	A &= \{[-1,0] \times 2 d   \} \cup \{ [d,d+1] \times d-2 \} \\
	B &=  \{ (x_1,x_2) :  (x_2-x_1) = d-1 \mbox{ and } 0 < x_1 < d \}    \\ 
	C &= \{  -1 \times [d+1,2d] \} \cup \{ (d+1) \times [d-2,2d-3]\} \\ 
	D &=  \{ (x_1, x_2) \in [0, d] \times [1, 3d-3]:  -1 \leq x_2-x_1 \leq 2d-1 \}  \\ %good
	\hat{T}_{p_1} &=  A \cup B \cup C \cup D
	\end{align*}
	and
	\[
	\begin{aligned}
	\hat{g}_{p_1}(x \pm v_{p,i}) &= \hat{g}_{p_1}(x) \pm a_{p_1,i}^T x + k_{p_1,i}  \quad \mbox{for $x \in \Z^2$},
	\end{aligned}
	\]
	where $k_{p_1,\pm i} \in \Z$ is a constant and $i \in \{1,2\}$ selects the lattice vector.
	Moreover, $\hat{T}_{p_1}$ is a $w_v = p q^{k+1}$ pseudo-square which $\hat{g}_{p_1}$ respects. 
\end{lemma}

\begin{proof}
	Consistency of the first condition and translation by $(v_{p_1,2}, a_{p_1,2})$ comes after checking that
	\[
	\hat{g}_{p_1}(x) = \hat{g}_{p_1}(x + v_{p_1,2}) +  a_{p_1,2}^T x - (\frac{d-3}{2})^2
	\]
	for $x, (x+ v_{p_1,2}) \in \hat{T}_{p_1}$ and  
	\[
	\hat{g}_{p_1}(x) = \hat{g}_{p_1}(x - v_{p_1,2}) -  a_{p_1,2}^T x - (\frac{d-3}{2})^2 + 1
	\]
	for $x, (x- v_{p_,1}) \in \hat{T}_{p_1}$. Write $\hat{g}_{p_1}: \mathbf{T}_v$ for the common extension of $\hat{g}_{p_1}:T \to \Z$ to $\mathbf{T}_v := \cup_{i \in \Z} (T + \I v_{p_1,2})$. As $\mathbf{T}_v$ is a rotation of $\mathbf{T}_h$ from the previous lemma, 
	it is also simply connected. 

	By checking the formula for $\hat{g}_{p_1}$ against \eqref{eq:explicit-forms-bs-2-a}, \eqref{eq:explicit-forms-bs-2-b}, \eqref{eq:explicit-forms-bs-2-c},
	we see that there is a $\rev(w_v)$ boundary string starting at $(0,1)$ in $\hat{T}_{p_1}$ which $\hat{g}_{p_1}$ respects. 
	There is also a $(a_{p_1,2} - a_{1/1,2})$ translated $w_v$ boundary string starting at $(d-1, 2d-3)$. 
	%Indeed, if we let $b$ denote the odometer for a  a $p q^{k+1}$ string, the translation, $b^T$, which appears in $\hat{g}_{p_1}$ starting at $(d-1, 2 d -3)$ is given by  $s(b^T) - s(b) = (a_{p_1,2} - a_{1/1,2})$.	
	%\[
	%b^T(x) := b(x) + (a_{p_1,2} - a_{1/1,2})^T x -  \frac{1}{4}(d-3)(d+3). %this is coordinate dependent
	%\]
	Actually, a stronger statement holds: $\hat{g}_{p_1}$ on $\mathbf{T}_v$ respects infinite repeating $p q^{k+1}$ strings on both sides. Thus, we may combine the strips to extend $\hat{g}_{p_1}$ to the plane. The tiling by $\mathbf{T}_v$ leaves no gaps by the same argument as the previous lemma. 	
\end{proof}

We next record the formula for the Laplacians.  Write  $\partial^{-} T = \{ x \in T : \exists y \not \in T \mbox{ such that } |y-x|=1\}$.
For $x \in T$, write $\mathcal{R}(x)$ for the image of $x$ under $\mathcal{R}(T)$, defined in \eqref{eq:rotate_tile}. 
If $T$ generates a tiling, this definition extends $\mathcal{R}(x)$ to $x \in \Z^2$. 

\begin{lemma} \label{lemma:staircase_laplacian}
	For $d \geq 3$ odd, let $T_{1/d}$, $\hat{T}_{1/d}$ be the tiles and $g_{1/d}$, $\hat{g}_{1/d}$ the plane extensions of the objects given in Lemmas \ref{lemma:basecase_3k_even_std}
	and \ref{lemma:basecase_3k_even_alt}. The Laplacian, $\Delta g_{1/d}$, satisfies
	\begin{equation}
	\begin{aligned}
		\Delta g_{1/d}(x)  &=  \Delta g_{1/d}(x \pm v_{1/d, *})  \mbox{ for all $x \in \Z^2$} \\
		\Delta g_{1/d}(x) &= -1\{ (x_1+x_2) \mbox{ is odd } \} \\
		&-1\{ (x_1+x_2) \mbox{ is even and $x_1 = 1$} \} \qquad \mbox{ on $T_d \backslash \partial^{-}T_{1/d}$} \\
	%	& + 1_{\{ (0,1) \cup (2,d)\}}  \qquad \mbox{ on $T_d \backslash \partial^{-}T_{1/d}$} \\
		\Delta g_{1/d}(x) &= 0 \qquad \mbox{ on $\partial^{-}T_{1/d}$}
	\end{aligned}
	\end{equation}

	The Laplacian of the alternate, $\Delta \hat{g}_{1/d}$ satisfies 
	\begin{equation}
	\begin{aligned}
	\Delta \hat{g}_{1/d}(x)  &=  \Delta \hat{g}_{1/d}(x \pm v_{1/d, i})  \mbox{for $i = 1,2$, for all $x \in \Z^2$} \\
	\Delta \hat{g}_{1/d}(x) &= -1\{ (x_1+x_2) \mbox{ is odd } \}\\
	&  -1\{ (x_1+x_2) \mbox{ is even and $x \in \{-(d-2) \times [0,d-1]\} \cup \{2 \times [1, d] \}$} \} \\
	&  -1\{ (x_1 +x_2) = 2  \mbox{ and $1 \leq x_2 \leq d-1$} \}\\
	&+ 1 \{ (x_1 + x_2) = 1 \mbox{ and $1 \leq x_2 \leq d-2$} \} \\
	&+ 1 \{ (x_1 + x_2) = 3 \mbox{ and $2 \leq x_2 \leq d-1$} \} \qquad \mbox{ on $T_{1/d} \backslash \partial^{-}T_{1/d}$} \\
%	& + 1\{d > 3\} 1_{(1-d,1) \cup (3,d-1)}  \qquad \mbox{ on $T_d \backslash \partial^{-}T_{1/d}$} \\
	\Delta \hat{g}_{1/d}(x) &= -1_{ (1-d,0) \cup (3,d)} \qquad \mbox{ on $\partial^{-}T_{1/d}$}.
	\end{aligned}
	\end{equation}
	
	In both cases, the Laplacians are 180-degree symmetric and  
	\begin{align*}
	\Delta g_{\mathcal{R}(1/d)}(x) &= \Delta g_{1/d}(\mathcal{R}(x)) \\
	\Delta \hat{g}_{\mathcal{R}(1/d)}(x) &= \Delta \hat{g}_{1/d}(\mathcal{R}(x)).
	\end{align*}

\end{lemma}

\subsection{Doubled staircases}
\begin{figure}
	\includegraphics[scale=0.35]{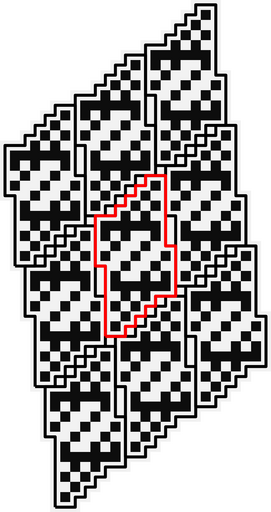} \qquad
	\includegraphics[scale=0.35]{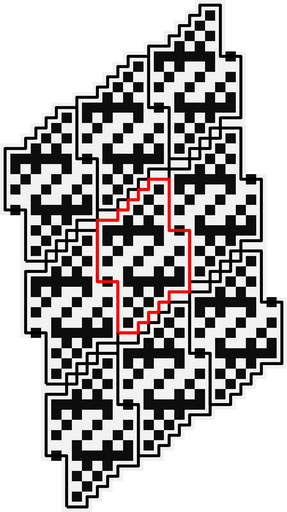} 
	\caption{A period of the Laplacian of a doubled staircase odometer on the left and its alternate. 
		The string is $22$ and the reduced fraction is $5/7$. Each tile is outlined in the dual lattice.} \label{fig:doubled_staircase_odd}
\end{figure}
\begin{figure}
	\includegraphics[scale=0.35]{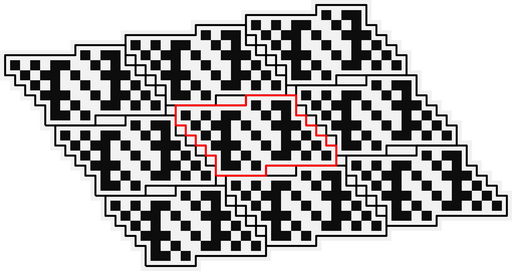}\quad
	\includegraphics[scale=0.35]{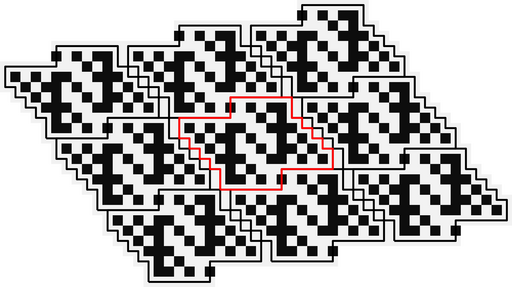}
	\caption{The rotated standard and alternate doubled staircase odometers of Figure \ref{fig:doubled_staircase_odd}. The string is $33$
		and the reduced fraction is $1/6$.} \label{fig:doubled_staircase_even}
\end{figure}

The {\it doubled staircases} are the reduced fractions of the form $\frac{1}{d}$ for $d \geq 4$ 
even and their rotations $\mathcal{R}(1,d) = (d-1,d+1)$. These are respectively the odd and even child
in Farey quadruples  $\q_{3^k}$ and $\q_{2^k}$ for $k \geq 1$. In particular, doubled staircases are siblings of the staircases.

We start with the odd standard child of $\q_{3^k}$. 
\begin{lemma} \label{lemma:basecase_3k_odd_std}
For each $d \geq 4$ even,
\begin{align*}
A &=  (0,0) \cup (d+2,d+2) \\
B &= \{ (x_1, x_2) : (x_1 + x_2) = (d+2) \mbox{ and } 1 < x_2 < d+1\} \\
C &=  \{ (x_1, x_2) \in [-d+2, 2 d] \times [0, d+2] : 1 \leq (x_1 + x_2) \leq 2d + 3\} \\ 
D &= \{ (x_1, x_2) : x_1 \geq d \mbox{ and } x_2 = 0 \mbox{ or } x_1 \leq 2 \mbox{ and } x_2 = d+2 \}  \\
T_{1/d} &:=  \{A \cup B \cup C \} \backslash D
\end{align*}
is a $(q p^k q p^{k+1}, p^{(2k+1)} q)$ pseudo-square which $g_{1/d}: T_{1/d} \to \Z$, given by, 
\[
\begin{aligned}
g_{1/d}(x) &= -\frac{1}{2} x_2 (x_2 + 1) + (x_2 + \min(0, x_1-1)) + \min(0, d+1-x_1) \quad \mbox{ for $x \in T_{1/d}$} \\
&+ \max((x_2+x_1)-d-2,0) + 1_{A \cup B}
\end{aligned}
\]
respects. 
\end{lemma}
\begin{proof}
	The computation in Lemma \ref{lemma:basecase_3k_even_std} also shows that the bottom and top boundaries of $T_{1/d}$ 
	are given by  $q p^k q p^{k+1}$ and $\rev(q p^k q p^{k+1})$ zero-one horizontal boundary strings which $g_{1/d}$ respects. Also, the explicit formulae for $p^{(2k+1) q}$ given in \eqref{eq:explicit-forms-bs-2-a}, \eqref{eq:explicit-forms-bs-2-b}, \eqref{eq:explicit-forms-bs-2-c}
	also shows that the left and right boundaries of $T_{1/d}$ are reversed and non-reversed vertical $p^{(2k+1) q}$
	zero-one boundary strings which $g_{1/d}$ respects. 
\end{proof}

Next is the even standard child of $\q_{2^k}$. 
\begin{lemma}\label{lemma:basecase_2k_even_std}
For each $q_1 = \frac{d-1}{d+1}$ $d \geq 4$ even, 
\begin{align*}
A &=  (-1,2d+1) \cup (d+1,d-1) \\
B &= \{ (x_1, x_2) : (x_2-x_1) = d \mbox{ and } 0 < x_1 < d\} \\
C &=  \{ (x_1, x_2) \in [-1, d+1] \times [1, 3d-1] : -1 \leq (x_2 -x_1) \leq 2d + 1\} \\ 
D &= \{ (x_1, x_2) : x_1 = -1 \mbox{ and } x_2 \leq d+1 \mbox{ or } x_1 = d+1 \mbox{ and } x_2 \geq 2d-1 \}  \\
T_{q_1} &:=  \{A \cup B \cup C \} \backslash D
\end{align*}
is a $(q^{(2k+1)} p, p q^k p q^{(k+1)})$-pseudo-square which $g_{q_1}:T_{q_1} \to \Z$, given by,
\[
\begin{aligned}
g_{q_1}(x) &= -\lfloor \frac{(x_2-x_1)^2}{4} \rfloor + \min(0, d-x_2) + \min(0, 2d-x_2) \quad \mbox{ for $x \in T_{q_1}$} \\
&+ \max((x_2-x_1)-d,0) + 1_{A \cup B}
\end{aligned}
\]
respects. 
\end{lemma}
\begin{proof}
	As in Lemma \ref{lemma:basecase_3k_odd_std}, the computations are identical to Lemma \ref{lemma:basecase_3k_even_std}. 
	For the vertical strings in this case, we check against the formulae \eqref{eq:explicit-forms-bs-3-a}, \eqref{eq:explicit-forms-bs-3-b}, \eqref{eq:explicit-forms-bs-3-c}.
\end{proof}

As for the standard staircases, by Lemma \ref{lemma:common-extension-pseudo-square}, each of the doubled staircase tile odometers defined in the previous two lemmas extend to $\Z^2$ under $L'(n/d)$ with the correct growth \eqref{eq:periodicity_growth}. Again, however, for the alternates the argument is different.

We start with the alternate odd child of $\q_{3^k}$.

\begin{lemma} \label{lemma:basecase_3k_odd_alt}
	For each $d \geq 4$ even, there is a unique function $\hat{g}_{1/d}: \Z^2 \to \Z$ with 
	\[
	\begin{aligned}
	\hat{g}_{1/d}(x) &= -\frac{1}{2} x_2 (x_2 + 1) + \min(0, x_1) + \min(0, d-x_1-1) \quad \mbox{ for $x \in \hat{T}_{1/d}$} \\
	&+ \max((x_2+x_1)-d,0) + 1_{A \cup B}
	\end{aligned}
	\]
	where
	\begin{align*}
	A &=  (-1,-1) \cup (d,d+2) \\
	B &= \{ (x_1, x_2) : (x_1 + x_2) = d \mbox{ and } 0 < x_1 < d-1\} \\
	C &=  \{ (x_1, x_2) \in [-d+1, 2 d-2] \times [-1, d+2] : -1 \leq (x_1 + x_2) \leq 2d + 1\} \\ 
	D &= \{ (x_1, x_2) : x_1 \geq d \mbox{ and } x_2 \in [-1,0] \mbox{ or } x_1 \leq -1 \mbox{ and } x_2 \in [d+1,d+2] \}  \\
	\hat{T}_{1/d} &:=  \{A \cup B \cup C \} \backslash D
	\end{align*}
	and
	\[
	\hat{g}_{1/d}(x \pm v_{1/d,i}) = \hat{g}_{1/d}(x) \pm a_{1/d,i}^T x + k_{1/d, \pm i}  \quad \mbox{for $x \in \Z^2$},
	\]
	where $k_{1/d, \pm i}$ is a constant and $i \in \{1,2\}$ selects the lattice vectors. 	Moreover, $\hat{T}_{1/d}$ is a $w_v = p^{(2k+1)} q$ pseudo-square which $\hat{g}_{1/d}$ respects.
\end{lemma}
\begin{proof}
	By comparing against \eqref{eq:explicit-forms-bs-1-c}, one sees there is a reversed $\rev(p^{(2k+1)} q)$-vertical boundary
	string starting at $(-1,-1)$ in $\hat{T}_{1/d}$ which $\hat{g}_{1/d}$ respects. Also, after a slope $a_{0/1,2}$ translation, the $p^{(2k+1)} q$ string given by \eqref{eq:explicit-forms-bs-1-a} and \eqref{eq:explicit-forms-bs-1-b} coincides with $\hat{g}_{1/d}$ starting at $(2d-3, 1)$. Thus, $\hat{g}_{1/d}$ is compatible with its horizontal translates and there are no gaps in $\cup_{i \in \Z} (\hat{T}_{1/d} + \I v_{1/d,1})$.  
	
	Also, after translation by $v_{1/d,2} = (-(d-1), d+1)$, the bottom boundary of $\hat{T}_{1/d}$, given by,  
	\[
	L := \{ x_2 = -1, -1 \leq x_1 \leq d-1\} \cup (d-1,0) \cup \{ x_2 = 1, d-1 \leq x_1 \leq 2(d-1)\}
	\]
	maps to 
	\[
	U_p := \{ x_2 = d, -d \leq x_1 \leq 0\} \cup (0,d+1) \cup \{ x_2 = d+2, 0 \leq x_1 \leq d-1\}
	\]
	and $U := U_p \cup (d,d+2) \backslash (-d,d)$ is the top boundary. In particular, 
	\[
	\hat{g}_{1/d}(x+v_{1/d,2}) = \hat{g}_{1/d}(x) + a_{1/d,2}^T x - (\frac{d(d+5)}{2} -1) 
	\]
	for $x \in L$. Similarly, 
	\[
	\hat{g}_{1/d}(x-v_{1/d,2}) = \hat{g}_{1/d}(x) - a_{1/d,2}^T x - \frac{d(d-1)}{2}
	\]
	for $x \in U$. This shows compatibility of $\hat{g}_{1/d}$ with its vertical translates and that $\cup_{j \in \Z} (\hat{T}_{1/d} + j v_{1/d,2})$ has no gaps. 
\end{proof}

% use trick to eliminate double ii 

Finally, we give the alternate even child of $\q_{2^k}$.
\begin{lemma} \label{lemma:basecase_2k_even_alt}
	For each $q_1 = \frac{d-1}{d+1}$, $d \geq 4$ even, there is a unique function $\hat{g}_{q_1}: \Z^2 \to \Z$ with
	\[
	\begin{aligned}
	\hat{g}_{q_1}(x) &= -\lfloor \frac{(x_2-x_1)^2}{4} \rfloor + \min(0, d-x_2) + \min(0, 2d-1-x_2) \quad \mbox{ for $x \in T_{q_1}$} \\
	&+ \max(0, x_2-x_1-d) + 1_{A \cup B}
	\end{aligned}
	\]
	where
	\begin{align*}
	A &=  (-2,2d) \cup (d+1,d-1) \\
	B &= \{ (x_1, x_2) : (x_2-x_1) = d \mbox{ and } 0 < x_1 < d-1\} \\
	C &=  \{ (x_1, x_2) \in [-2, d+1] \times [1, 3d-2] : -1 \leq (x_2 -x_1) \leq 2d + 1\} \\ 
	D &= \{ (x_1, x_2) : x_1 \leq -1 \mbox{ and } x_2 \leq d+1 \mbox{ or } x_1 \geq d \mbox{ and } x_2 \geq 2d \}  \\
	T_{q_1} &:=  \{A \cup B \cup C \} \backslash D
	\end{align*}
	and
	\[
	\hat{g}_{q_1}(x \pm v_{q_1,i}) = g_{q_1}(x) \pm a_{q_1,i}^T x + k_{q_1, \pm i}  \quad \mbox{for $x \in \Z^2$},
	\]
	where $\hat{k}_{d, \pm i}$ is a constant and $i \in \{1,2\}$ selects the lattice vectors. Moreover, $\hat{T}_{1/d}$ is a $w_h = q^{(2k+1)} p$ pseudo-square which $\hat{g}_{q_1}$ respects.
\end{lemma}

\begin{proof}
	The proof is similar to Lemma \ref{lemma:basecase_2k_even_alt}.  First, we check that there is a $q^{(2k+1)} p$ horizontal boundary string
	starting $(0,1)$ which $\hat{g}_{q_1}$ respects. As in Lemma \ref{lemma:basecase_3k_even_std}, there is a 
	$\rev(q^{(2k+1)} p)$ reversed horizontal boundary string starting at $(-2, 2(d-1))$ which $\hat{g}_{q_1}$ respects. This implies 
	compatibility and no gaps in the vertical direction.

	For the other direction, after translation by $v_{q_1,1} = (d+1, d-1)$, the left boundary of $\hat{T}_{q_1}$, given by,  
	\[
	L := \{ x_1 = 0, 1 \leq x_2 \leq d\} \cup (-1, d) \cup \{ x_1 = -2,  d \leq x_2 \leq 2 d\}
	\]
	maps to 
	\[
	R_p := \{ x_1 = d+1, d \leq x_2 \leq 2d-1\} \cup (d, 2d-1) \cup \{ x_1 = d-1,  2d-1 \leq x_2 \leq 3 d-1\}
	\]
	and $R := R_p \cup (d+1,d-1) \backslash (d-1, 3d-1)$ is the top boundary. In particular, 
	\[
	\hat{g}_{q_1}(x+v_{q_1,1}) = \hat{g}_{q_1}(x)
	\]
	for $x \in L$. Similarly, 
	\[
	\hat{g}_{q_1}(x-v_{q_1,2}) = \hat{g}_{q_1}(x) 
	\]
	for $x \in R$. This shows compatibility of $\hat{g}_{q_1}$ with its horizontal translates and that $\cup_{j \in \Z} (\hat{T}_{q_1} + j v_{q_1,1})$ has no gaps. 
\end{proof}

\begin{lemma} \label{lemma:doubled_staircase_laplacian}
	For $d \geq 4$ even, let $T_{1/d}$, $\hat{T}_{1/d}$ be the tiles and $g_{1/d}$, $\hat{g}_{1/d}$ the plane extensions of the objects given in Lemmas \ref{lemma:basecase_3k_odd_std} and \ref{lemma:basecase_3k_odd_alt}. The Laplacian, $\Delta g_{1/d}$, satisfies
	\begin{equation}
	\begin{aligned}
	\Delta g_{1/d}(x)  &=  \Delta g_{1/d}(x \pm v_{1/d, *})  \mbox{ for all $x \in \Z^2$} \\
	\Delta g_{1/d}(x) &= -1\{ (x_1+x_2) \mbox{ is odd } \} \\
	&  -1\{ (x_1+x_2) \mbox{ is even and $x \in \{1 \times [1,d]\} \cup \{d+1 \times [2, d+1] \}$} \} \\
	&  -1\{ (x_1 +x_2) = d+2  \mbox{ and $2 \leq x_2 \leq d$} \}\\
	&+ 1 \{ (x_1 + x_2) =d+3 \mbox{ and $3 \leq x_2 \leq d$} \} \\
	&+ 1 \{ (x_1 + x_2) =d+1 \mbox{ and $2 \leq x_2 \leq d-1$} \} \qquad \mbox{ on $T_d \backslash \partial^{-}T_{1/d}$} \\
	\Delta g_{1/d}(x) &= 0 \qquad \mbox{ on $\partial^{-}T_{1/d}$}
	\end{aligned}
	\end{equation}

	The Laplacian of the alternate, $\Delta \hat{g}_{1/d}$ satisfies 
	\begin{equation}
	\begin{aligned}
	\Delta \hat{g}_{1/d}(x)  &=  \Delta \hat{g}_{1/d}(x \pm v_{1/d, *})  \mbox{ for all $x \in \Z^2$} \\
	\Delta \hat{g}_{1/d}(x) &= -1\{ (x_1+x_2) \mbox{ is odd } \}\\
	&  -1\{ (x_1+x_2) \mbox{ is even and $x \in \{0 \times [-1,d-1]\} \cup \{d-1 \times [2, d+2] \}$} \} \\
	&  -1\{ (x_1 +x_2) = d  \mbox{ and $2 \leq x_2 \leq d-1$} \}\\
	&+ 1 \{ (x_1 + x_2) =d+1 \mbox{ and $3 \leq x_2 \leq d$} \} \\
	&+ 1 \{ (x_1 + x_2) =d-1 \mbox{ and $1 \leq x_2 \leq d-2$} \} \qquad \mbox{ on $T_d \backslash \partial^{-}T_{1/d}$} \\
%	& + 1_{ (-1,0) \cup (d,d+1)  }  \qquad \mbox{ on $T_d \backslash \partial^{-}T_{1/d}$} \\
	\Delta \hat{g}_{1/d}(x) &= -1_{ (d-1,0) \cup (0,d+1)} \qquad \mbox{ on $\partial^{-}T_{1/d}$}
	\end{aligned}
	\end{equation}
	
	In both cases, the Laplacians are 180-degree symmetric and  
	\begin{align*}
	\Delta g_{\mathcal{R}(1/d)}(x) &= \Delta g_{1/d}(\mathcal{R}(x)) \\
	\Delta \hat{g}_{\mathcal{R}(1/d)}(x) &= \Delta \hat{g}_{1/d}(\mathcal{R}(x)).
	\end{align*}

\end{lemma}

\subsection{One-sided recurrence}
In this section we prove that the constructed functions are recurrent. 
In fact, we prove a sufficient property which we later use to prove recurrence in the general construction. 
First, we observe that recurrence is preserved under rotations which flip parity. 
\begin{lemma}
	Suppose $v: \Z^2 \to \Z$ is integer superharmonic and recurrent and 
	\[
	\Delta v = \Delta (v \circ \mathcal{R})
	\]
	for a 90-degree rotation and translation $\mathcal{R}(\Z^2) = \Z^2$ which flips parity: if $(x_1+x_2)$ is even and $y = \mathcal{R}(x)$
	then $(y_1 + y_2)$ is odd. Then $v \circ \mathcal{R}$ is integer superharmonic and recurrent. 
\end{lemma} \label{lemma:rot_recurrence_invariance}
\begin{proof}
	This is an immediate consequence of the definition. Indeed, let $s_1 = 	\Delta v + 1$ and $s_2 = \Delta (v\circ \mathcal{R}) + 1$.
	Since $\mathcal{R}$ is a bijection, every finite induced subgraph of the rotated $F$-lattice can be written 
	as $\mathcal{R} \circ H$, for $H$ a finite induced subgraph of the $F$-lattice. Since $v$ is recurrent, 
	there is a vertex $x \in H$ with $s_1(x)$ larger than its in-degree in $H$. Let $y = \mathcal{R}(x)$. By assumption, $s_2(y) = s_1(x)$ 
	and as $\mathcal{R}$ is a rotation and flips the parity of $x$, the horizontal/vertical neighbors of $x$ become the vertical/horizontal neighbors of $y$ and the edges between either pair of neighbors are preserved.
\end{proof}

In light of Lemma \ref{lemma:rot_recurrence_invariance} and the observed rotational invariance of the Laplacians, we need only prove recurrence for $\q_{3^k}$, $k \geq 0$.

We start with the standard even child.  Figure \ref{fig:staircase_even} will be a useful 
reference in the next two proofs. 

\begin{lemma} \label{lemma:staircase_recurrent}
For each $d \geq 3$ odd, the extension of $g_{1/d}:Z^2 \to \Z$ defined in Lemma \ref{lemma:basecase_3k_even_std}
is integer superharmonic and recurrent.  
\end{lemma}
\begin{proof}
	By the explicit formula of the Laplacian, Lemma \ref{lemma:staircase_laplacian}, it suffices to check recurrence. 	
		
	Let $s = \Delta g_{1/d} + 1$ and suppose, for sake of contradiction, that there is an induced subgraph of the $F$-lattice, $H$,
	which is forbidden for $s$. Let $c^{0} = -\infty$ and for $j \geq 1$, let 
	\begin{equation}
	\begin{aligned}
	c^{j} &= \min\{ x_1 > c^{j-1} : x \in H\}  \\ 
	V^{j} &= \{ x \in H : x_1 = c^j\}.
	\end{aligned}
	\end{equation}
	In words, sets of possibly disjoint vertical lines enumerated from left to right. Since $H$ is forbidden, it is nonempty, hence
	$V^{1}$ exists.
	
	Write $\partial^h T(1/d) := \partial^{-}T(1/d) \cap \{ x_2 \in \{0, d+1\}\}$. We prove the following by induction on $j \geq 0$ for all $L'(1/d)$ translates of $T(1/d)$.  We use periodicity of $s$ and state (and prove) the claims for $T(1/d)$:  
	\begin{enumerate}
		\item If $\bigcup_{j'<j} V^{j'} \cap T(1/d) = \emptyset $ and $V^j \cap T(1/d) \neq \emptyset$
		then $c_j =1 $. 
		\item If $V^{j} \cap T(1/d) \not=\emptyset$, then $V^{j} \cap \partial^h T(1/d) = \emptyset$.
		\item If $V^{j} \cap T(1/d) \not=\emptyset$ then $V^{j+1} \cap T(1/d) \not=\emptyset$.
	\end{enumerate}
	The third condition will result in a contradiction as each tile $T(1/d)$ is finite. The idea is to continually use the fact that $H$ is forbidden.
	
	The base case $j = 0$ corresponds to $V^0 = \emptyset$ and thus the claims are vacuously true for the base case. Let 
	$j \geq 1$ be given and suppose that the claims are true for all $j' < j$. 
	
	{\it Proof of (1)}\\
	We argue by contradiction. Suppose $\bigcup_{j'<j} V^{j'} \cap T(1/d) = \emptyset$, $V^j \cap T(1/d) \neq \emptyset$, and $c_j \not = 1$. 
	If $y \in V^j$ and is even, then, since $c_j \not = 1$, $s(y) = 1$ and therefore
	$y-e_1 \in H$, since $H$ is forbidden. This either contradicts the assumption $\bigcup_{j' < j} V^{j'} \cap T(1/d) = \emptyset$
	or $y-e_1 \not \in T(1/d)$. In the latter case, there are two subcases, (i) $y = (2-d,d)$ 
	or (ii) $y = (0,0)$. In case (i), $s(y-e_1) = 1$ and so $y-e_1-e_2 \in H$, contradicting inductive (2). 
	In case (ii), $s(y-e_1) = 0$ and either $y-e_1-e_2 \in H$, contradicting inductive (2) or $y-e_1 + e_2 \in H$. 
	In the latter case, $s(y-e_1+e_2) = 1$ and there is an even-odd chain of points all with $s(x_i) = 1$ 
	ending at $x_k \in \partial^h T(1/d) - v_{1/d,1}$, contradicting inductive (2). (In other words, the vertical boundaries
	of $T(1/d)$ are $F$-lattice connected).  
	
	Otherwise if $y$ is odd and $s(y) = 1$, then $y \pm e_2 \in H$ and at least one such neighbor is in $T(1/d)$, a contradiction by the above. If $y$ is odd and $s(y) = 0$, then both $y \pm e_2$ are in $T(1/d)$ and at least one neighbor is in $V^{j}$ as $H$ is forbidden, again resulting in a contradiction. 	
	
	{\it Proof of (2)}\\
	Suppose there is $y \in \partial^h T(1/d) \cap V^j$. By the explicit formula, $s(y) = 1$. If $c_j = 1$, then $y \in \{(1,0), (1,d+1)\}$. 
	If $y = (1,0)$, then $y' \in \{y-e_2, y-e_2-e_1\}$ satisfies $s(y') = 1$ and $y' \in H$, contradicting the inductive hypothesis as $y-e_2-e_1$ is on $\partial^h (T(1/d)-v_{1/d,1})$.
	If $y = (1,d+1)$, then $y+e_2$ and $y+e_2-e_1 \in H$ and $s(y+e_2) = 1$ and $s(y+e_2-e_1)  = 0$. Since $H$ is forbidden at least one of $(y+e_2-e_1 \pm e_2)$ 
	must be in $H$, contradicting inductive (2) as in the Proof of (1). 
	
	If $y = (2,d+1)$, then $s(y) = 1$ and $y-e_1 \in H \cap \partial^h T(1/d)$, contradicting inductive (2). Otherwise, by inductive (1), $c_{j} > 1$ and so $x_2 = 0$. If $x$ is even, then $x-e_1 \in H$, contradicting inductive (2). Otherwise if $x$ is odd, then $x \pm e_2 \in H$ 
	and $x-e_2 \in T'^{1/d} := T(1/d) - v_{1/d,2}$.  Since $x-e_2 + v_{1/d,2} \leq 1$, by inductive (1) applied to $T'^{1/d}$, $x-e_2 + v_{1/d,2} = 1$  and $x = (d,0)$. In this case, $s(x+e_2) = 1$ and so $x+e_2-e_1 \in H$. However, $s(x+e_2-e_1) = 0$ and since $x$ is odd, this means $x+e_2-e_1 \pm e_2 \in H$, however, by the same argument as the Proof of (1) this contradicts inductive (2).

	{\it Proof of (3)} \\	
	We may suppose by the above arguments that $x \in V^{j} \cap \{ T(1/d) \backslash \partial^{-}T(1/d)\}$ and $x_1 \geq 1$. 
	
	If $c^{j} = 1$, then there are no $y \in H \cap T(1/d)$ with $y_1 < 1$.
	Therefore if $x$ is even, $x + e_1 \in H$. If $x$ is odd, at least one neighbor $x \pm e_2 \in H$ and that neighbor is even.
	
	Otherwise, suppose $c^j > 1$. By the same argument, there must be an even $x \in H \cap T(1/d)$. However,
	since $c^j > 1$, $s(x) = 1$, meaning $x \pm e_1 \in H$. 
\end{proof}

Next is the alternate even child.
\begin{lemma} \label{lemma:alt-staircase-recurrent}
	For each $d \geq 3$ odd, the extension of $\hat{g}_{1/d}:Z^2 \to \Z$ defined in Lemma \ref{lemma:basecase_3k_even_alt}
	is integer superharmonic and recurrent. 
\end{lemma}
\begin{proof}
	Let $\hat{s} = \Delta \hat{g}_{1/d} + 1$ and begin the proof as in Lemma \ref{lemma:staircase_recurrent}
	except modify the induction hypotheses as follows. 
	
	Write $\partial^h \hat{T}(1/d) := \partial^{-}\hat{T}(1/d) \cap \{ x_2 \in \{-1, 0, d, d+1\}\}$. 
	We claim that the following holds for all $L'(1/d)$-translations of $\hat{T}(1/d)$ and all $j \geq 1$. We use periodicity of $s$ and state (and prove) the claims for $\hat{T}(1/d)$:  
	
	\begin{enumerate}
		\item If $\bigcup_{j' < j} V^{j'} \cap \hat{T}(1/d) = \emptyset$ and $V^{j} \cap \hat{T}(1/d) \not=\emptyset$ then $c^{j} \in \{ 2-d, 2 \}$
		or $V^{j}$ is a singleton, $\{  (1-d,0), (1,1) \}$.   
		Moreover, if $c_j \in \{2-d, 2\}$, then $V^{j'} \cap \hat{T}(1/d) = \emptyset$ for $j' < j$. 
		
		\item If $V^{j} \cap \hat{T}(1/d) \not=\emptyset$, then $V^{j} \cap \partial^h \hat{T}(1/d) = \emptyset$. 
		\item If $V^{j} \cap \hat{T}(1/d) \not=\emptyset$ then $V^{j+1} \cap \hat{T}(1/d) \not=\emptyset$.
	\end{enumerate}
	The base case $j = 0$ corresponds to $V^0 = \emptyset$ and thus the claims are vacuously true for the base case. Let 
	$j \geq 1$ be given and suppose that the claims are true for all $j' < j$.

	{\it Proof of (1)}	\\
	Suppose $\bigcup_{j' < j} V^{j'} \cap \hat{T}(1/d) = \emptyset$ and $V^{j} \cap \hat{T}(1/d) \not=\emptyset$ but $c^{j} \not \in \{ 2-d, 2 \}$ and take $x \in V^{j}$
	for $x \not \in \{ (1-d,0), (1,1)\}$. 
	
	If $x = (3,d)$, then since $x$ is even and $\bigcup_{j' < j} V^{j'} \cap \hat{T}(1/d) = \emptyset$,
	$x + e_1 \in H$. However, $x+e_1$ is odd and $\hat{s}(x+e_1)=1$, so we can build an odd-even chain all with $\hat{s}(x') = 1$ and in $H$: $\{x+e_1, x+e_1+e_2, x+e_2, x+2e_2, x+2e_2-e_1\}$, 
	contradicting inductive (2) as $(x+2e_2 -e_1) \in (\partial^h \hat{T}(1/d) + v_{1/d,2}) \cap V^{j-1}$. 
		
	Next, if $(x_1 + x_2) = 2$ and $1 \leq x_2 \leq d-1$, then $\hat{s}(x) =0$ 
	and $x + e_1 \in H$. By our assumption, in this case, $x_2 \geq 2$, so $\hat{s}(x+e_1) = 1$ and hence $x+e_1, x+e_1 + e_2 , x+e_2, x+ 2 e_2 \in H$.
	However, $\hat{s}(x+2 e_2) = 1$ and even, a contradiction if $(x + 2 e_2 - e_1) \in \hat{T}(1/d)$ as its in-degree is at most one by the case we are in.
	Otherwise $x + 2 e_2 - e_1 \in \partial^h \hat{T}(1/d) + v_{n/d,2}$, contradicting inductive (2).

	Otherwise, if $x$ is odd, either $\hat{s}(x) = 1$ and one of $s(x \pm e_2) = 1$
	or $\hat{s}(x) = 0$ and both $\hat{s}(x \pm e_2) = 1$. If $x$ is even, then $\hat{s}(x) = 1$. 
	Both cases lead to a contradiction as we cannot have even $x \in V^{j}$ with $\hat{s}(x) = 1$. 
	
	Now, suppose $c_j = 2$ but $\cup_{j'<j} V^{j'} \cap \hat{T}(1/d) \not = \emptyset$. 
	By inductive (3),  there is some $y \in V^{j'}$, $j' < j$ with $(y_1 + y_2) = 1$. 
	If $y = (2-d,d-1)$, then by inductive (1),  $y-e_2, y-e_2+e_1 \in H$ so we may assume $y_2 \leq d-2$.
	In fact, iterating this shows that $(1,0) \in H$, contradicting inductive (2) 
	as $(1,0) \in V^{j-1} \cap \partial^h \hat{T}(1/d)$.

	{\it Proof of (2)}\\
	Suppose not and take $y \in V^{j} \cap \{ \partial^h \hat{T}(1/d) \}$ so that $\hat{s}(y) = 1$.
	We divide into subcases (i) $y_2 \in \{-1, 0\}$ and (ii) $y_2 \in  \{d, d+1\}$.
	
	In case (i), if $y$ is even, then $y-e_1 \in H \cap \partial^h \hat{T}(1/d) \cap V^{j-1}$, a contradiction. 
	Otherwise, $y$ is odd and there are three subcases. In the first subcase $y-e_2 \in \hat{T}(1/d) - v_{1/d,2}$ and 
	$y-e_2 + v_{1/d,2} < 2$, contradicting (1). In the second subcase $y-e_2 \in \hat{T}(1/d) - v_{1/d,2} + v_{1/d,1}$
	and $y-e_2 + v_{1/d,2} - v_{1/d,1} < 2$ contradicting (1).  In the third subcase,  $\hat{s}(y) = \hat{s}(y-e_2) = \hat{s}(y-e_2-e_1) = 1$, 
	so $y-e_1 \in H$, contradicting inductive (2) for $V^{j-1}$. 
	
	In case (ii), we may similarly assume $y$ is odd, in which case by (1) applied to all of $\hat{T}(1/d) + i_1 v_{1/d,1} + i_2 v_{1/d,2}$ for $|i_j| \leq 1$, $y \in  \{ (3-d,d), (3,d+1)\}$
	and $\hat{s}(y) = \hat{s}(y+e_2) = \hat{s}(y+e_2-e_1)=1$ and all points are in $H$, contradicting inductive (2) for $V^{j-1}$. 
	
	{\it Proof of (3)} \\
	Take $y \in V^{j} \cap \hat{T}(1/d) \backslash \partial^{h} \hat{T}(1/d)$ and argue, as in the proof of (3) in Lemma \ref{lemma:staircase_recurrent}, that either $y + e_1 \in H$ or one of $y \pm e_2 + e_1 \in H$. If $c_j \in \{2-d, 2\}$, since $V^{j-1} \cap \hat{T}(1/d) = \emptyset$, 
	then there must be even $y' \in V^{j-1} \cap \hat{T}(1/d)$ with $y'+e_1 \in V^{j} \cap \hat{T}(1/d)$. 
	Similarly, if $y \in \{(1-d,0), (1,1)\}$ then $y$ is even and so $y + e_1 \in V^{j}$. 
	
	Otherwise, if $y$ is odd and $\hat{s}(y) = 0$, then $s(y \pm e_2) = 1$ and so at least one of $(y \pm e_2 + e_1) \in H$. 
	If $y$ is odd and $\hat{s}(y) = 1$ then both $y' \in \{y \pm e_2\} \in H$ and at least one $\hat{s}(y') = 1$ with $y' \in \hat{T}(1/d)$
	or $y' \in \partial^{H} \hat{T}(1/d) \pm v_{1/d,2}$, the latter case contradicting (1). 
	
	If $y$ is even and $\hat{s}(y) = 0$ but $y+e_1 \not \in H$, then $y-e_1 \in H$. 
	If $y = (3-d,d)$, then by inductive (1) and the formula of $\hat{s}$, $\{y-e_1, y-e_1-e_2, y-e_2, y-2e_2, y-2e_2+e_1, y - e_2 + e_1\} \subset H$.
	Otherwise, $\hat{s}(y-e_1) = 1$ and similarly $\{y-e_1 -e_2, y-e_2, y-2 e_2, y- 2 e_2 + e_1, y - e_2 + e_1\} \subset H$.
\end{proof}

Next is the standard odd child. Figure \ref{fig:doubled_staircase_even} will be a reference in the next two proofs. 
\begin{lemma} \label{lemma:doubled-staircase-recurrent}
	For each $d \geq 4$ even, the extension of $g_{1/d}$ defined in Lemma \ref{lemma:basecase_3k_odd_std} is
	integer superharmonic and recurrent. 
\end{lemma}
\begin{proof}
	Let $s = \Delta g_{1/d} + 1$ and begin the proof as Lemma \ref{lemma:alt-staircase-recurrent}
	modifying the inductive hypotheses as follows. 
	
	Write $\partial^h T(1/d) := \partial^{-}T(1/d) \cap \{ x_2 \in \{0, 1, d+1, d+2\}\}$
	and $T^1, T^2$ for the two copies of $T(1/(d-1))$ contained within $T(1/d)$: $T^1 = T(1/(d-1))$ and $T^2 = T(1/(d-1)) + v_{1/(d-1),1} + 1 + \I$

	We claim following holds for all $L'(1/d)$ translations of $T(1/d)$ and all $j \geq 1$. We use periodicity of $s$ and state (and prove) the claims for $T(1/d)$.
	\begin{enumerate}
		\item If $\bigcup_{j' < j} V^{j'} \cap T(1/d) = \emptyset$ and $V^{j} \cap T(1/d) \not=\emptyset$ then $c^{j} \in \{ 1, d+1 \}$
		or $V^{j} = (d,2)$.  Moreover, $c_j \in \{1, d+1\}$, then $V^{j'} \cap T(1/d) = \emptyset$ for $j' < j$. 
		If $V^{j} \cap T^1 \neq \emptyset$, then $V^{j'} \cap T(1/d) = \emptyset$ for $c_{j'} < 1$.
		If $V^{j} \cap T^2 \neq \emptyset$, then $V^{j'} \cap \{ T(1/d) \backslash (d,2)\} = \emptyset$ for $c_{j'}  < d+1$. 
		\item If $V^{j} \cap T(1/d) \not=\emptyset$, then $V^{j} \cap \partial^h T(1/d) = \emptyset$. 
		\item If $V^{j} \cap T(1/d) \not=\emptyset$ then $V^{j+1} \cap T(1/d) \not=\emptyset$.
	\end{enumerate}
		The base case $j = 0$ corresponds to $V^0 = \emptyset$ and thus the claims are vacuously true for the base case. Let 
	$j \geq 1$ be given and suppose that the claims are true for all $j' < j$.

	{\it Proof of (1)}	\\
	Suppose $\bigcup_{j' < j} V^{j'} \cap T(1/d) = \emptyset$ and $V^{j} \cap T(1/d) \not=\emptyset$ but $c^{j} \not \in \{ 1, d+1 \}$ and take $x \in V^{j}$
	for $x \not \in \{ (1-d,0), (1,1)\}$.  
	
	If $(x_1 + x_2) = (d+2)$ and $2 \leq x_2 \leq d$, then $s(x) = 0$ and $x+e_1 \in H$.
	By assumption, in this case, $x_2 \geq 3$, so $s(x+e_1) =1$ and  $x+e_1, x+e_1 + e_2 , x+e_2, x + 2 e_2 \in H$.
	However, $s(x+ e_2) = 1$ and even, a contradiction. 

	The rest of the proof proceeds along the same lines as the corresponding proof of (1) in Lemma \ref{lemma:alt-staircase-recurrent}. Indeed, if $V^{j} \cap T^2 \not = \emptyset$, but $V^{j'} \cap T(1/d) \neq \emptyset$ for $c_{j'}  < d+1$, then $V^{j-1} \cap T(1/d) \not = \emptyset$, in which case by inductive (3) there is some $y \in V^{j'}$, $j' < j$ with $(y_1 + y_2) = d+1$.

	{\it Proof of (2)} \\
	The proof is similar to the corresponding proof of (2) in Lemma \ref{lemma:alt-staircase-recurrent}
	except here instead of comparing to just the translations of $T(1/d)$, compare to the embedded subtiles $T^1$ 
	and $T^2$ also.

	Suppose not and take $y \in V^{j} \cap \{ \partial^h T(1/d) \}$ so that $s(y) = 1$.
	We divide into subcases (i) $y_2 \in \{0, 1\}$ and (ii) $y_2 \in  \{d+1, d+2\}$.
	
	In case (i), if $y$ is even, then $y-e_1 \in H \cap \partial^h T(1/d) \cap V^{j-1}$, a contradiction. 
	Otherwise, $y$ is odd and there are three subcases. In the first subcase $y-e_2 \in T(1/d) - v_{1/d,2}$ and 
	$y-e_2 + v_{1/d,2} < 1$, contradicting (1). In the second subcase $y-e_2 \in T^2$
	and $y-e_2 + v_{1/d,2} - v_{1/d,1} < d+1$ contradicting (1).  In the third subcase,  $s(y) = s(y-e_2) = s(y-e_2-e_1) = 1$, 
	so $y-e_1 \in H$, contradicting inductive (2) for $V^{j-1}$. 
	
	In case (ii), we may similarly assume $y$ is odd, in which case we apply (1) to all of $T' + v_{1/d,2}$ for $T' \in \{T(1/d), T^1, T^2\}$, $y \in  \{ (2,d+1), (d+1,d+2)\}$.
	In the first subcase,  $s(y) = s(y+e_2) = s(y+e_2-e_1)=1$ and all are in $H$, contradicting inductive (2) for $V^{j-1}$. 
	In the second subcase, $s(y) = s(y+e_2) = 1$ and $s(y+e_2-e_1) = 0$. By inductive (2), $y+2 e_2-e_1 \in H$.
	However, by iterating, this means we can build a chain of points $y_i$ with $s(y_i) = 1$ and $y_i \in H$
	that eventually intersects $\partial^h T(1/d) + v_{1/d,2}$, contradicting inductive (2).  

	{\it Proof of (3)} \\
	The argument repeats the proof of (3) in Lemma \ref{lemma:alt-staircase-recurrent}. 
\end{proof}

We conclude with the alternate odd child. The key difference/simplification here is that the vertical boundaries 
in the $L'(n/s)$ tiling of $\hat{T}(1/d)$ are connected with respect to the $F$-lattice. 
\begin{lemma} \label{lemma:doubled-alt-staircase-recurrent}
	For each $d \geq 4$ even, the extension of $\hat{g}_{1/d}$ defined in Lemma \ref{lemma:basecase_3k_odd_alt}
	is integer superharmonic and recurrent. 
\end{lemma}
\begin{proof}
	Let $\hat{s} = \Delta \hat{g}_{1/d} + 1$ and begin the proof as Lemma \ref{lemma:alt-staircase-recurrent}
	modifying the inductive hypotheses as follows.
	
	Write $\partial^h \hat{T}(1/d) := \partial^{-}\hat{T}(1/d) \cap (T^1 \cup T^2) \cap \{x_2 \in \{-1,1,d,d+2\}\}$
	where $T^1, T^2$ are the two copies of $T(1/(d-1))$ contained within $\hat{T}(1/d)$: $T^1 = T(1/(d-1))$ and $T^2 = T(1/(d-1)) + v_{1/(d-1),1} + 2 \I$
	
	We claim the following holds for all $L'(1/d)$ translations of $\hat{T}(1/d)$ and all $j \geq 0$. We use periodicity of $s$ and state (and prove) the claims for $\hat{T}(1/d)$:  
	\begin{enumerate}
		\item If $\bigcup_{j' < j} V^{j'} \cap \hat{T}(1/d) = \emptyset$ and $V^{j} \cap \hat{T}(1/d) \not=\emptyset$ then $c^{j} \in \{ 0, d-1 \}$
		or $V^{j} = (d-3,1)$.  Moreover, if $c_j = (d-1)$, then $V^{j'} \cap \hat{T}(1/d) \cap T^2 \cap \{x_2 > 3\} = \emptyset$ for $j' < j$. 
		\item If $V^{j} \cap \hat{T}(1/d) \not=\emptyset$, then $V^{j} \cap \partial^h \hat{T}(1/d) = \emptyset$
		\item If $V^{j} \cap \hat{T}(1/d) \not=\emptyset$ then $V^{j+1} \cap \hat{T}(1/d) \not=\emptyset$.
	\end{enumerate}
			The base case $j = 0$ corresponds to $V^0 = \emptyset$ and thus the claims are vacuously true for the base case. Let 
	$j \geq 1$ be given and suppose that the claims are true for all $j' < j$. 
	
	{\it Proof of (1)}	\\
	Suppose $\bigcup_{j' < j} V^{j'} \cap \hat{T}(1/d) = \emptyset$ and $V^{j} \cap \hat{T}(1/d) \not=\emptyset$ but $c^{j} \not \in \{0, d-1 \}$ and take $x \in V^{j}$
	for $x \neq (d-2,1)$. By a similar argument as in Lemma \ref{lemma:doubled-staircase-recurrent} this contradicts either the assumption or inductive (2).
	In this case since $\partial^h$ is smaller, there is an extra step which uses the connected path of 1s on the left boundary. 

	Now, suppose $c_j = (d-1)$ but $y \in V^{j'} \cap \hat{T}(1/d) \cap T_2 \cap \{x_2 > 3\} \neq \emptyset$.
	If $\hat{s}(y) = 1$, then $y-e_1 \in H$. 
	If $\hat{s}(y) = 0$ then both $\hat{s}(y \pm e_2) = 1$ and at least one such neighbor $y \pm e_2$ is in $H$. Iterating,
	this means there is $y \in H$ on $\partial^{-}T^2$. This either immediately contradicts inductive (2) or there is an $F$-lattice path of $\hat{s}(y_i)=1$
	from $y$ to $(0,d+1)$. This implies either $(0, d) \in H$ and $(-1,d) \in H$, contradicting inductive (1) or $(0,d+2) \in H$ and $(1,d+2) \in H$ which contradicts inductive (2).

	{\it Proof of (2)} \\
	Suppose not and take $y \in V^{j} \cap \partial^h T(1/d)$. We divide into subcases (i) $y_2 = -1$ and (ii) $y_2 = d+2$. 
	
	Case (i) is identical to the previous proof. In case (ii), by (1),  $y \in \{ (d-1,d+2), (d,d+2\}$. 
	Since $(d,d+2) \in H$ implies $(d-1,d+2)$ in $H$, we may take $y = (d-1,d+2)$. In this case, $y \in H$ implies the existence of a chain of points $\{y_i\} \subset H$ with $\hat{s}(y_i) = 1$ ending at $y + v_{1/d,2}$, contradicting inductive (2) for $T(1/d) + v_{1/d,2}$.

	{\it Proof of (3)} \\
	The argument is almost identical to the proof of (3) in Lemma \ref{lemma:alt-staircase-recurrent}
	in which we argue case by case. The only new cases are case (i) $y \in \{ (1-d,0), (3,d)\}$ and case (ii) 
	$y = (d-3,1)$. 
	
	In case (i), both $\hat{s}(y\pm e_2) = 1$ hence one of $y \pm e_2 + e_1$ is in $H$. 
	
	In case (ii), if $y + e_1 \not \in H$, then $\{y -e_1, y-e_1-e_2, y-e_2, y-2e_2, y-2e_2 + e_1\} \subset H$. 
\end{proof}

\section{Odometers and tiles} \label{sec:odometers}
This section extends the hyperbola recursion of Section \ref{sec:hyperbola} to odometers and tiles. That is, we associate to each rational in a Farey quadruple a pair of tiles and odometers. For continuity of the literature, the formality which we use to define the recursion 
is similar to \cite{levine2017apollonian}. 

\subsection{Standard tiles}
\begin{figure}
	\includegraphics[scale=0.3]{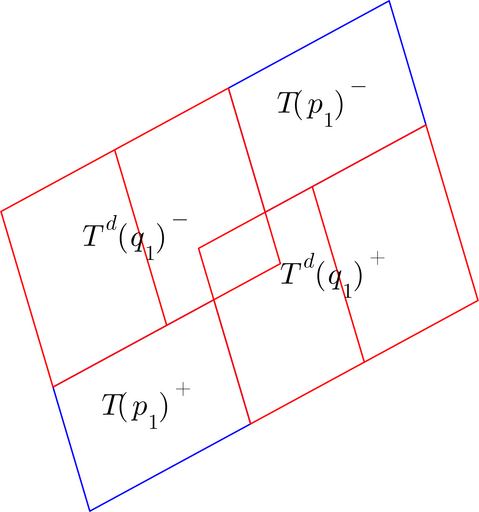}
	\includegraphics[scale=0.3]{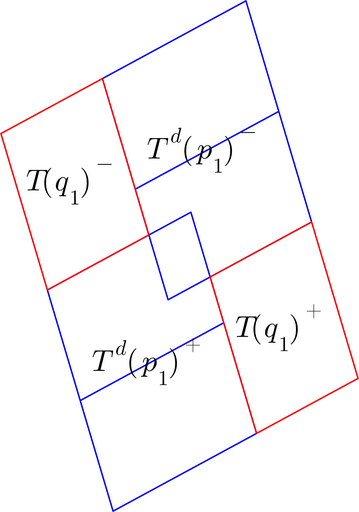}\\
	\includegraphics[scale=0.3]{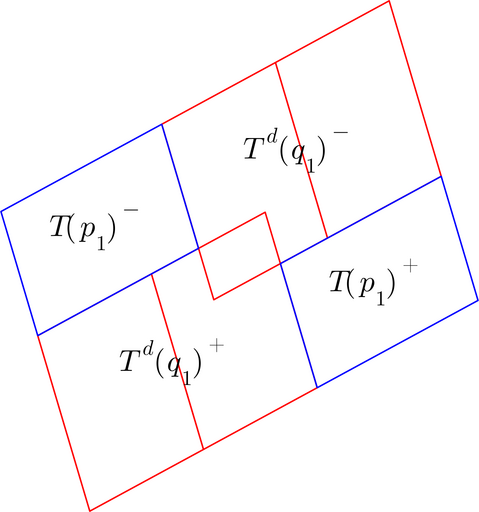}
	\includegraphics[scale=0.3]{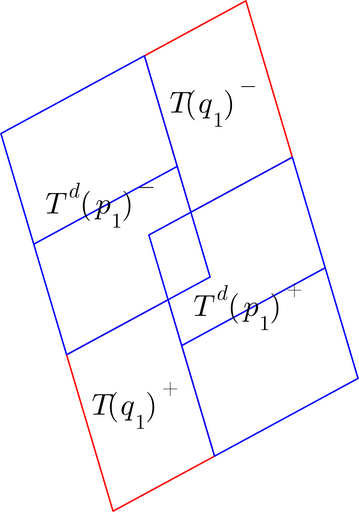}\\
	\caption{The two possible orientations for a standard tile pair from Definition \ref{def:standard_tiles}. The left column is the odd child and the right column is the even child. The first row is the odd-first orientation and the second is the even-first orientation.
		Only one type of overlap between parents (corresponding to Type 1 children) is displayed - see Figure \ref{fig:standard_tile_overlaps} for the other types of overlaps.} \label{fig:standard_tile_orientations}
\end{figure}

Let $(p_0, q_0) =  \mathcal{C}(p_1,q_1)$ form a Farey quadruple labeled by a recursion word $w \in F_3^*$. Let $\mathcal{W}_1$ count the number of 1s in $w$.

\begin{definition}  \label{def:standard_tiles}

	A pair of tiles $(T(p_0), T(q_0)) \subset \Z^2$ are {\it standard tiles} for $(p_0, q_0)$ if
	they appear in Proposition \ref{prop:base_cases}, are $(T_{0/1}, T_{1/1})$ from Section \ref{subsec:zero-one-tiles}, or 
	have the {\it standard tile decomposition}:
	\begin{equation} \label{eq:standard_decompositions}
	\begin{aligned}
	T(p_0) &= T(p_1)^+ \cup T(p_1)^- \cup T^d(q_1)^+ \cup T^d(q_1)^- \\
	T(q_0) &=  T^d(p_1)^+ \cup T^d(p_1)^- \cup T(q_1)^+ \cup T(q_1)^- 
	\end{aligned}
	\end{equation}
	where $T^d(n/d)^{\pm}$ denote doubled tiles, 
	\begin{equation} \label{eq:doubled_tiles}
	\begin{aligned}
	T^d(p)^{\pm} &= T(p)^{\pm,1} \cup T(p)^{\pm, 2} := T(p) \cup (T(p) +  v_{p,2}) \\
	T^d(q)^{\pm} &= T(q)^{\pm,1} \cup T(q)^{\pm, 2} := T(q) \cup (T(q)+ v_{q,1}) 
	\end{aligned}
	\end{equation} 
	and $(T(p_1), T(q_1))$, with or without superscripts, are standard tiles for $(p_1,q_1)$.
	The tile positions
	in \eqref{eq:standard_decompositions} depend on the parity of $\mathcal{W}_1$: 
	if $\mathcal{W}_1$ is odd
	\begin{equation} \label{eq:offsets_w1_odd}
	\begin{aligned}
	c(T) - c(T(p_0)) &= 
	\begin{cases}
	0 &\mbox{if $T = T(p_1)^+$} \\
	2 v_{q_1,1} + v_{q_1,2} &\mbox{if $T = T(p_1)^-$} \\
	v_{p_1,1} &\mbox{if $T = T^d(q_1)^+$} \\
	v_{p_1,2} &\mbox{if $T = T^d(q_1)^-$} 
	\end{cases} \\
	c(T) - c(T(q_0))  &= 
	\begin{cases}
	0 &\mbox{if $T = T^d(p_1)^+$} \\
	v_{q_1,1} + v_{q_1,2} &\mbox{if $T = T^d(p_1)^-$} \\
	v_{p_1,1} &\mbox{if $T = T(q_1)^+$} \\
	2 v_{p_1,2} &\mbox{if $T = T(q_1)^-$} 
	\end{cases}
	\end{aligned}
	\end{equation}
	otherwise
	\begin{equation} \label{eq:offsets_w1_even}
	\begin{aligned}
	c(T) - c(T(p_0))  &= 
	\begin{cases}
	0 &\mbox{if $T = T^d(q_1)^+$} \\
	v_{p_1,1} + v_{p_1,2} &\mbox{if $T = T^d(q_1)^-$} \\
	2 v_{q_1,1} &\mbox{if $T = T(p_1)^+$} \\
	v_{q_1,2} &\mbox{if $T = T(p_1)^-$} 
	\end{cases} \\
	c(T) - c(T(q_0)) &= 
	\begin{cases}
	0 &\mbox{if $T = T(q_1)^+$} \\
	v_{p_1, 1} + 2 v_{p_1,2} &\mbox{if $T = T(q_1)^-$} \\
	v_{q_1,1} &\mbox{if $T = T^d(p_1)^+$} \\
	v_{q_1,2} &\mbox{if $T = T^d(p_1)^-$},
	\end{cases}
	\end{aligned}
	\end{equation}
	see Figure \ref{fig:standard_tile_orientations}. 
	
	The tiles in the standard tile decomposition of $T(n/d)$ will be called subtiles of $T(n/d)$. 
	The tile orientations in \eqref{eq:offsets_w1_odd} and \eqref{eq:offsets_w1_even} will be referred
	to as the odd-first and even-first orientations respectively. A single tile in a pair of standard tiles is a {\it standard tile}.
\end{definition}

\begin{figure}
	\includegraphics[scale=0.35]{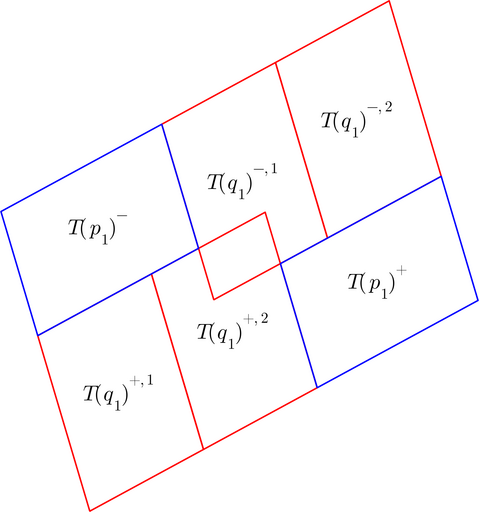}
	\includegraphics[scale=0.35]{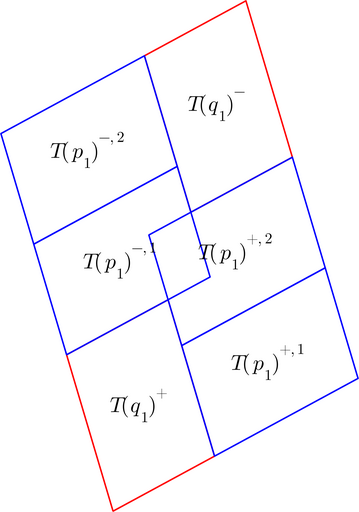} \\
	\includegraphics[scale=0.35]{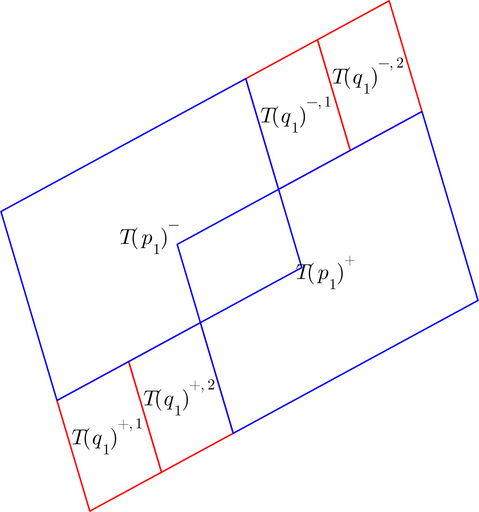}
	\includegraphics[scale=0.35]{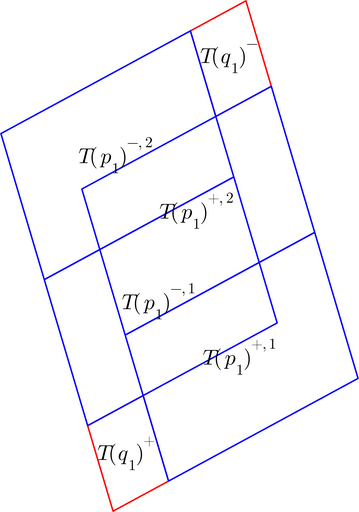} \\
	\includegraphics[scale=0.35]{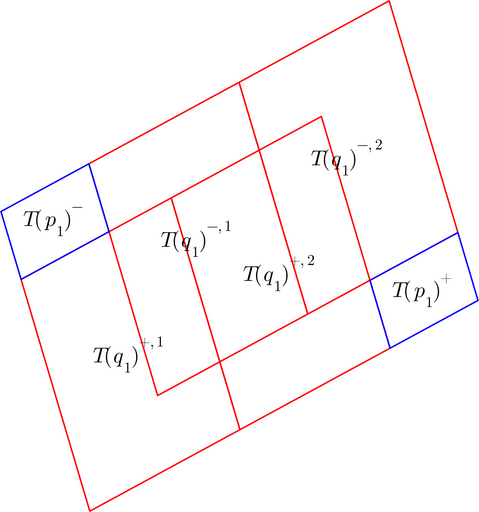}
	\includegraphics[scale=0.35]{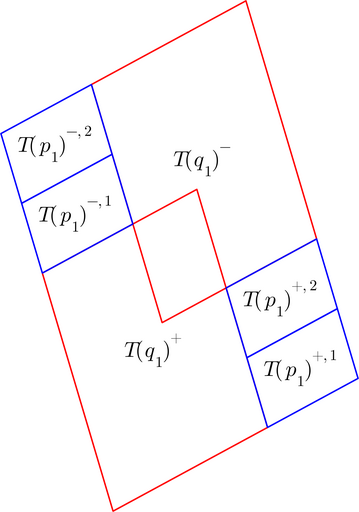} 
	\caption{All possible overlaps between parents in a standard tile.  The left column denotes the odd child and the right column is the even child. The rows from top to bottom correspond to Type 1, 2, then 3 children. 
		The parent tiles are labeled in their centers with notation from Definition \ref{def:standard_tiles}.
		The orientation is the even-first orientation.} \label{fig:standard_tile_overlaps}
\end{figure}

Before proving general existence of standard tiles, we derive an extension of rotation invariance Lemma \ref{lemma:rot_invariance}, to tiles
and a tiling property, assuming existence up to a certain depth. 
\begin{lemma}\label{lemma:standard_boundary_word_recursion}
	Suppose standard tiles exist for all $n/d \in \mathcal{T}_m$, for $m \leq m_0$, some $m_0 \geq 1$. 
	Then, the following properties are satisfied for each such $0 < n/d < 1$.
	\begin{enumerate}
		\item Rotation invariance:  $\mathcal{R}(T(n/d)) = T(\mathcal{R}(n/d))$
		\item Boundary tiling: starting at $c(T(n/d))$ the boundary word of $T(n/d)$ can be written as $w = w_1 * w_2 * \widehat{\rev(w_1)} * \widehat{\rev(w_2)}$ where $w_1,w_2$ satisfy the hypotheses in Lemma \ref{lemma:almost_square_tiling} 
		and $\sum w_1 + \I =  v_{n/d, 1}$ and $\sum w_2 -1= v_{n/d,2}$.
	\end{enumerate}
\end{lemma} 

%
%Observe we that the standard tile recursion is rotation invariant. Specifically, 
%a horizontal edge of every tile is the vertical edge of some other tile in the same depth of the recursion tree. 
%We make this precise in the next lemma by first proving that 
%every tile indeed has ensures that odd standard tiles are sent to even standard tiles this will simplify the proof. 
%We use the previous lemmas to show that standard tiles periodically cover space and have certain compatibility properties. First, we prove 
%a property of the boundary word of standard tiles. 
%
%\begin{lemma} \label{lemma:standard_boundary_word_recursion}
%	If $0 < n/d < 1$, then the boundary word of $T(n/d)$ can be written as 
%	$w = w_1 * w_2 * \widehat{\rev(w_1)} * \widehat{\rev(w_2)}$ and $w_1,w_2$ satisfy the hypotheses in Lemma \ref{lemma:almost_square_tiling}
%	and $\sum w_1 + \I= v_{n/d,1}$ and $\sum w_2 -1= v_{n/d,2}$. 
%\end{lemma}

\begin{proof}
	Both properties are true by Proposition \ref{prop:base_cases} if $T(n/d)$ does not have a standard tile decomposition. 
	Thus, we may assume $T(n/d)$ has a standard tile decomposition and that both properties are satisfied for the parent tiles 
	of $T(n/d)$. 
	
	{\it Step 1: Rotation invariance}  \\
	If $(n/d)$ is even, then its parent tiles are $T^d(p_1)^{\pm}$ and $T(q_1)$. By Lemma \ref{lemma:lattice_rotation}
	and the inductive hypothesis, 
	$\mathcal{R}(T^d(p_1)^{\pm}) = T^d(\mathcal{R}(p_1))$ and $\mathcal{R}(T(q_1)) = T(\mathcal{R}(q_1))$. 
	By Lemma \ref{lemma:rot_invariance}, these are the parent tiles of $T(\mathcal{R}(n/d))$.  This together 
	with the standard decomposition and Lemma \ref{lemma:lattice_rotation} again implies $T(\mathcal{R}(n/d)) = \mathcal{R}(T(n/d))$.
	The proof applies to $(n/d)$ odd after observing $\mathcal{R}$ on the parent tiles is an involution. Indeed, the second hypothesis
	implies 180-degree symmetry.

	{\it Step 2: Boundary tiling} \\
Every base case tile, $T(n/d)$ has a {\it $(w_1,w_2)$-boundary word} of the form $w_1 * w_2 * \widehat{\rev(w_1)} \widehat{\rev(w_2)}$ and $T(\mathcal{R}(n/d))$ has a $(-\I w_2, -\I w_1)$ boundary word both of which start at $c(T)$. Moreover, each such $w_1$ and $w_2$ satisfy the desired properties. By induction and the standard decomposition every subsequent standard tile has a boundary word decomposition. 
Thus, by Step 1, it suffices to show that the {\it bottom edges}, the $w_1$ in the boundary word decompositions satisfy the conditions in Lemma \ref{lemma:almost_square_tiling}.

	We start by rewriting the standard tile recursion to produce just the two bottom edges for each odd-even tile pair 
	for each child Farey pair in a Farey quadruple: $(w_t, v_t) \in F_2 \times F_2$. The base cases are quadruples $\q_{2^k}$ for $k \geq 0$ and $\q_{3^k}$ for $k > 0$ for which we have explicit formulae from Section \ref{sec:degenerate_cases} for the odd-even pair of edges $(w_0, v_0)$: 
	\[
	\begin{pmatrix}
	w_0 \\
	v_0 
	\end{pmatrix} 
	= \begin{cases}
	\begin{pmatrix} 1*( 1*\I)^{2 k +1}*(1*1)  \\
	1*(1*\I)^{2k} (1*1)
	\end{pmatrix} &\mbox{ $\q_{2^k}$ for $k \geq 0$ } \\ \\
	\begin{pmatrix} 
	(1*1)^k *(1*\I) *(1*1)^{k+1}*1  \\
	1*(1*1)^{k+1}
	\end{pmatrix} &\mbox{ $\q_{3^k}$ for $k > 0$.} 
	\end{cases}
	\]
	Now, given a recursion word $str \in F_3^*$ describing a quadruple $\q_{str}$ the standard 
	tile decomposition implies 
	\begin{equation}
	\begin{pmatrix}
	w_{t+1} \\
	v_{t+1} 
	\end{pmatrix} 
	= \begin{cases}
	\begin{pmatrix} 
	w_t*\I*v_t*\I*v_t  \\
	w_t*\I*v_t
	\end{pmatrix} &\mbox{ if $\sum (str[i]=1)$ is odd} \\ \\
	
	\begin{pmatrix} 
	v_t*\I*v_t*\I*w_t  \\
	v_t*\I*w_t
	\end{pmatrix} 
	&\mbox{ if $\sum (str[i]=1)$ is even,} 
	\end{cases}
	\end{equation}
	where $w_t,v_t$ are the edges of the odd-even parent Farey pair in $\q_{str}$.

	It will be convenient to augment the recursion so that it generates the bottom edge concatenated with an extra i. The augmented
	recursion has as base cases 
	\[
	\begin{pmatrix}
	\tilde{w}_0 \\
	\tilde{v}_0 
	\end{pmatrix} 
	= \begin{cases}
	\begin{pmatrix} 1*( 1*\I)^{2 k +1}*(1*1*\I)  \\
	1*(1*\I)^{2k} (1*1*\I)
	\end{pmatrix} &\mbox{ if $str = 2^k$ for $k \geq 0$ } \\ \\
	\begin{pmatrix} 
	(1*1)^k (1*\I) *(1*1)^{k+1}*(1*\I)  \\
	(1*1)^{k+1}*(1*\I)
	\end{pmatrix} &\mbox{ if $str = 3^k$ for $k > 0$ } 
	\end{cases}
	\]
	and the recursive step is
	\begin{equation} \label{eq:tiling_recursion}
	\begin{pmatrix}
	\tilde{w}_{t+1} \\
	\tilde{v}_{t+1} 
	\end{pmatrix} 
	= \begin{cases}
	\begin{pmatrix} 
	\tilde{w}_t*\tilde{v}_t*\tilde{v}_t  \\
	\tilde{w}_t*\tilde{v}_t
	\end{pmatrix} &\mbox{ if $\sum (str[i]=1)$ is odd } \\ \\
	
	\begin{pmatrix} 
	\tilde{v}_t*\tilde{v}_t*\tilde{w}_t  \\
	\tilde{v}_t*\tilde{w}_t
	\end{pmatrix} 
	&\mbox{ if $\sum (str[i]=1)$ is even, }
	\end{cases}
	\end{equation}
	where similarly $\tilde{w}_t,\tilde{v}_t$ are the augmented edges of the odd-even parents of str. 
	Use induction and compute using \eqref{eq:tiling_recursion} to show that the augmented words sum to the desired lattice vectors. 
	It remains to verify the rest of the hypotheses for which we use \eqref{eq:tiling_recursion} and the forms of the base cases.
	We split into cases based on the structure of the recursion word $\q_{str}$.

	{\it Case 1: $str = 2^k*s*str'$ for $k \geq 0$, $s \in \{1,3\}$ and $|str'| \geq 0$} \\
	If $k = 0$ and $s = 3$ then proceed to Case 2. Otherwise, write 
	\[
	\begin{pmatrix}
	p_k \\
	q_k 
	\end{pmatrix} 
	= 
	\begin{pmatrix} 1*( 1*\I)^{2 k +1}*(1*1*\I)  \\
	1*(1*\I)^{2k} (1*1*\I)
	\end{pmatrix}.
	\]
	By Lemma \ref{lemma:almost_palindrome} applied with initial string corresponding to either $(p_k,q_k)$ or $(q_k, p_{k-1})$
	we have that for all $t \geq 1$, both $w_t$ and $v_t$ are of the form $p*\tilde{w}*q$ where $\tilde{w}$ is a palindrome in the letters $p,q$
	where 
	\[
	(p,q) = \begin{cases} (p_k, q_k) &\mbox{ if $s = 1$} \\
	(q_k, p_{k-1}) &\mbox{ if $s = 3$}.
	\end{cases}
	\]
	An inspection of the formula shows that 
	\begin{equation}\label{eq:letter_reversal}
	\rev(b)*\I = \I*b \quad \mbox{ for $b \in \{p,q\}$}
	\end{equation}
	as words in the letters $\{1,i\}$. 
	We claim this implies that case (b) of Lemma \ref{lemma:almost_square_tiling} holds. First take $s = 1$ and write 
	in the letters $\{1,i\}$, 
	\begin{align*}
	p_k \tilde{w} q_k &=  (1*1*\I)*(1*\I)^{2k}*(1*1*\I)*\tilde{w}*1*(1*\I)^{2k}*(1*1*\I) \\
	&= (1*1*\I)*\left( (1*\I)^{2k}*(1*1*\I)*\tilde{w}*1*(1*\I)^{2k}*1\right)*1*\I \\
	&=: (1*1*\I)*\tilde{v}*1*\I.
	\end{align*}
	As we have augmented a trailing $i$, it suffices to show $\tilde{v}$ is a palindrome in the letters $\{1,i\}$: 
	\begin{align*}
	\rev(\tilde{v}) &= 1*(\I*1)^{2k}*1*\rev{\tilde{w}}*(\I*1*1)*(\I*1)^{2k}  \\
	&= (1*\I)^{2k}*1*1*(\rev{\tilde{w}}*\I)*1*(1*\I)^{2k}*1 \\
	&= (1*\I)^{2k}*(1*1*\I)*\tilde{w}*1*(1*\I)^{2k}*1 \\
	&= \tilde{v},
	\end{align*}
	where in the second to last step we used \eqref{eq:letter_reversal}. The argument when $s = 3$ and $k \geq 1$ proceeds in the same fashion 
	using \eqref{eq:letter_reversal}. % uses \eqref the same similar. 
%	
%	If $s = 3$, then 
%	\begin{align*}
%	q_k \tilde{w} p_{k-1} &=  1*(1*\I)^{2 k }*(1*1*\I)*\tilde{w}*1*(1*\I)^{2k -1}*(1*1*\I) \\
%	&= (1*1*\I)*\left( (1*\I)^{2k-1}*(1*1*\I*\tilde{w}*1*(1*\I)^{2 k-1}*1)*1*\I \\
%	&=: (1*1*\I)*\tilde{v}*1*\I
%	\end{align*}
%	and
%	\begin{align*}
%	\rev(\tilde{v}) &= 1*(\I*1)^{2k-1}*1*\rev{\tilde{w}}*(\I*1*1)*(\I*1)^{2k-1}  \\
%	&= (1*\I)^{2k-1}*1*1*(\rev{\tilde{w}}*\I)*1*(1*\I)^{2k-1}*1 \\
%	&= (1*\I)^{2k-1}*(1*1*\I)*\tilde{w}*1*(1*\I)^{2k-1}*1,
%	\end{align*}

	{\it Case 2:  $str = 3^k*s*str'$ for $k > 0$, $s \in \{1,2\}$ and $|str'|\geq 0$} \\	
	The argument is similar to Case 1, however, the letters in this case are: 
	\[
	\begin{pmatrix}
	p_k \\
	q_k
	\end{pmatrix} 
	= 
	\begin{pmatrix}
	h_k*h_{k+1} \\
	h_{k+1}
	\end{pmatrix}
	:= 
	\begin{pmatrix}(1*1)^k*(1*\I)*(1*1)^{k+1}*(1*\I)  \\
	(1*1)^{k+1}(1*\I)
	\end{pmatrix}.
	\]
	By Lemma \ref{lemma:almost_palindrome} both $w_t$ and $v_t$ are of the form $p*\tilde{w}*q$ where $\tilde{w}$ is a palindrome in the letters $p,q$ where 
	\[
	(p,q) = \begin{cases} (p_k, q_k) &\mbox{ if $s = 1$} \\
	(q_{k-1}, p_k) &\mbox{ if $s = 2$}.
	\end{cases}
	\]
	Compute to see that
	\begin{equation}\label{eq:letter_reversal_case2}
	\begin{aligned}
	\rev(b)*\rev(h_{k+1})*\I &= (\rev(h_{k+1})*\I)*b  \\
	&= \I*h_{k+1}*b \quad \mbox{ for $b \in \{p_k,q_k\}$}.
	\end{aligned}
	\end{equation}
	and
	\begin{equation}\label{eq:letter_reversal_case3}
	1*(1*1)^k*\rev(b) = b*	1*(1*1)^k  \quad \mbox{ for $b \in \{p_k,q_{k-1}\}$}.
	\end{equation}
	The rest of the argument is similar to Case 1: when $s = 1$ use \eqref{eq:letter_reversal_case2}
	and when $s = 2$ use \eqref{eq:letter_reversal_case3}. 
\end{proof}

The next lemma uses the abstract recursion on binary words in \eqref{eq:child_farey_quad_continuedfrac} as well as the pseudo-square boundary decomposition. For notational convenience, 
write $w_h(n/d)$ for the binary word associated to the reduced fraction 
in \eqref{eq:child_farey_quad_continuedfrac} with initial seed $w_0 = \{\}$ and $\mathbf{Q}_{\{\}} =(qqp,qp,p,q)$. Also write $w_v(n/d) = \mathcal{F}(w_h( \mathcal{R}(n/d)))$.

\begin{lemma} \label{lemma:standard_tile_properties}
	A standard tile, $T(n/d)$ exists for every reduced rational $0 \leq n/d \leq 1$. Moreover, 
	when $0 < n/d < 1$, the tile has the following properties. 
	\begin{enumerate}[label=(\alph*)]
		\item $T(n/d)$ generates a $(v_{n/d,1}, v_{n/d,2})$-regular almost pseudo-square tiling. 
		\item Each $T(n/d)$ is a $(w_{h}(n/d), w_v(n/d))$-pseudo-square with offsets respecting the tiling: 
			\begin{enumerate}[label=(\roman*)]
				\item $c(T) - c(T(n/d)) = v_{n/d,1} - (\I+1)$ where $T$ is the last tile of $w_{h}(n/d)$
				\item $c(T') - c(T(n/d)) = v_{n/d,2} - (2 \I-1)$ where $T'$ is the last tile of $\rev(w_{v}(n/d))$.
			\end{enumerate}
		\item The surrounding of $T(n/d)$ with respect to $(v_{n/d,1}, v_{n/d,2})$ consists of two stacked zero-one horizontal and two vertical boundary strings.
		\item When $T(n/d)$ has a standard decomposition, the shared boundary between neighboring subtiles is part of or is a stacked horizontal or vertical zero-one boundary string. 
	\end{enumerate}
\end{lemma}

\begin{proof}
	In light of Proposition \ref{prop:base_cases}, we may assume $T(n/d)$ has a standard decomposition 
	and the statements are true for its parents, $T(p_1)$ and $T(q_1)$.  Also, by Lemma \ref{lemma:standard_boundary_word_recursion}
	to prove part (a) it suffices to show that $T(n/d)$ is a topological disk, \ie, does not have any internal holes. This however follows from parts (c) and (d) and Lemma \ref{lemma:zero-one-fixed-offsets}, 
	hence it remains to prove (b), (c), and (d). We assume that the decomposition given is in the even-first orientation, 
	otherwise flip the subsequent statements.

	{\it Step 1: (b)}
	
	By the inductive hypothesis, $T(p_1)$ and $T(q_1)$ are $(w_h(p_1), w_v(p_1))$ and $(w_h(q_1), w_v(q_1))$
	pseudo-squares. By rotation, we may assume $T(n/d)$ is odd. Let $T$ be the last tile in $T(q_1)^{+,1}$ and $T'$ the first tile in $T(q_1)^{+,2}$. 
	By definition $T$ is a translation of $T_{0/1}$ and $T'$ a translation of $T_{1/1}$. Moreover, 
	by the definition of the standard decomposition and the inductive hypothesis on the offsets, $c(T')-c(T) = -(1+ \I)$, 
	in particular, we can glue the two boundary strings together. A similar argument applies to the interface between 
	$T(q_1)^{+,2}$ and $T(p_1)$. This shows if $(n/d)$ is odd, then $T(n/d)$ respects $w_h(q_1)*w_h(q_1)*w_h(p_1)$ otherwise
	it respects $w_h(q_1)*w_h(p_1)$. A symmetric argument applies to the vertical boundary strings and a computation
	shows that the offsets respect the tiling. 
	
	{\it Step 2: (c)} \\
	Let $A$ be the horizontal string for $T(n/d)$ and $B$ the reversed horizontal string for 
	$T(n/d) - v_{n/d, 2}$. Set $c(T(n/d)) = 0$. By part (b), the first tile in $B$ is located at $c(T(n/d)) - v_{n/d,2} + v_{n/d,2} - (2 \I - 1)$. Thus, the offset between the first tile in $B$ and the first tile in $A$ is $2 \I - 1 = v_{0/1,2} + \I$, the correct
	initial offset for a stacked string. 
	
	Similarly if $C$ is the vertical string for $T(n/d)$ and $D$ the reversed vertical string for $T(n/d) + v_{n/d,1}$, then 
	the offsets between the first tiles is $(\I+1) = -v_{1/1,1}$. The other two sides are stacked strings by the above arguments 
	for $T(n/d) - v_{n/d,1}$ and $T(n/d) + v_{n/d,2}$.

	{\it Step 3: (d)} \\
	We state the arguments with the aid of Figure \ref{fig:standard_tile_overlaps}.
	First consider the three possible boundaries between $T(q_1)^{+,1/2}$ and $T(p_1)^{-}$ when $T(n/d)$ is odd.
	By the inductive hypothesis, the offset between the first tile in the reversed zero-one horizontal 
	boundary string for $T(q_1)^{+,1}$ and the first tile for the zero-one horizontal string in $T(p_1)^-$ is $v_{0/1,2} + \I$. 
	Since the initial offset is correct, the rest of the interface forms part of a stacked horizontal zero-one boundary string
	by part (3) of Lemma \ref{lemma:almost_palindrome}. Indeed, every letter other than the first matches across the interface. 
	Reversing $p_1$ and $q_1$ above shows the three possible interfaces between $T(p_1)^+$ and $T(q_1)^{-,1/2}$ also form part of a stacked horizontal boundary string.
	When $T(n/d)$ is even, the interface between $T(p_1)^{+,1}$ and $T(p_1)^{+,2}$ is exactly a stacked horizontal 
	boundary string by the inductive hypothesis part (c). By rotation, the above arguments apply to the vertical interfaces. 
\end{proof}

\subsection{Weak standard odometers}

Our current goal is to extend the standard tile decomposition to odometers. In order to do so, we must define the operation of doubling a partial odometer. However, in the course of the recursion, doubled odometers may need to be corrected in the interior of the tile
so we need a notion of tile that only depends on the boundary in the pseudo-square decomposition. 

\begin{definition} \label{def:partial_tile}
	Let $T(n/d)$ be a $(w_{h}(n/d), w_v(n/d))$ pseudo-square. A partial tile
	$T^{h^{\pm}/v^{\pm}}(n/d)$ is a tile which coincides with $T(n/d)$ 
	on one of the four sides: 
	\begin{equation}
	T^{h^{\pm}/v^{\pm}}(n/d) \cap T(n/d)
	\supset
	\begin{cases}
	w_h(n/d) &\mbox{ case $h^+$} \\
	\rev(w_h(n/d)) &\mbox{ case $h^-$} \\
	w_v(n/d) &\mbox{ case $v^+$ }  \\
	\rev(w_v(n/d)) &\mbox{ case $v^-$}.
	\end{cases}
	\end{equation}	
	That is, for example, $T^{h^+}$ contains the zero-one subtiles of $T(n/d)$'s horizontal boundary string, $w_h(n/d)$, with $w_h$ as in Definition \ref{def:tile_boundary_strings}. 
	A boundary tile $T^b(n/d)$ coincides with $T(n/d)$ on every side: 
	\begin{equation}
	T^b(n/d) \cap T(n/d)
	\supset  w_h(n/d) \cup \rev(w_h(n/d)) \cup w_v(n/d) \cup \rev(w_v(n/d)).
	\end{equation}	
	
	\iffalse
	Let $(T(p_0),T(q_0))$ be standard tiles for $(p_0, q_0)$ with decompositions given by \eqref{eq:standard_decompositions}.
	A halved tile is a subset of a standard tile which contains the boundary of at least
	two orientation and parity specified outer tiles:	
	A halved tile $T^{h}(p_0) \subset T(p_0)$ in the odd-first orientation is 
	\begin{equation}
	T^{h}(p_0) \cap T(p_0) \supset
	\begin{cases}
	T(p_1)^+ \cup T(q_1)^{-,1} &\mbox{left halved} \\
	T(p_1)^{-} \cap T(q_1)^{+,2} &\mbox{right halved}
	\end{cases}
	\end{equation}
	and in the even-first orientation 
	\begin{equation}
	T^{h}(p_0) \cap T(p_0) \supset
	\begin{cases}
	T(p_1)^- \cup T(q_1)^{+,1} &\mbox{left halved} \\
	T(p_1)^{+} \cup T(q_1)^{-,2} &\mbox{right halved}.
	\end{cases}
	\end{equation}
	Similarly, a halved tile $T^{h}(q_0) \subset T(q_0)$ in the odd-first orientation and even-first orientation is 
	\begin{equation}
	T^{h}(q_0) \cap T(q_0) \supset
	\begin{cases}
	T(p_1)^{+,1} \cup T(q_1)^{+} &\mbox{lower halved} \\
	T(p_1)^{-,2} \cup T(q_1)^{-} &\mbox{upper halved}.
	\end{cases}
	\end{equation}
	(Both cases are the same). 
	
	The boundaries on which the tiles match will be called halved boundaries.
	\fi
\end{definition}

We now define the notion of a doubled odometer using partial tiles. But before we do so, we 
define an operation, for convenience, which will allow us to quickly pass from a tile decomposition 
to an odometer decomposition. Say $T(n/d)$ and $T(n'/d')$ are two tiles 
with  $c(T(n'/d')) - c(T(n/d)) = v_{n/d,i}$. Then, two partial odometers $o(n/d)$ and $o(n'/d')$ {\it respect the tile translations} if $s(o(n'/d')) - s(o(n/d)) = a_{n/d,i}$ and the domains of $o(n/d)$ and $o(n'/d')$ are $T(n/d)$ and $T(n'/d')$ respectively.

\begin{definition} \label{def:standard_odometer_doubling}
	For $(n/d) \in \{p_0,q_0\}$ let $T(n/d)$ be a standard tile and $T^{h^{\pm}/v^{\pm}}(n/d)$ a partial tile. Denote by
	 \begin{equation} \label{eq:doubled_tile_subset}
	 T^{d,p, \pm}(n/d) = T^b(n/d)^{1} \cup T^{h^{\pm}/v^{\pm}}(n/d)^2 :=  T^b(n/d) \cup (T^{h^{\pm}/v^{\pm}}(n/d) + v_{n/d})
	 \end{equation}
	 a partial doubled tile where $v_{n/d} = v_{p_0,2}$ or $v_{n/d} = v_{q_0,1}$ for $n/d = p_0$ odd or $n/d=q_0$ even respectively.
	 
	 The weak doubling of a partial odometer $o_{n/d}$ with domain $T^b(n/d)$  
	 is a partial odometer $d(o)^{\pm}_{n/d}:T^{d,p,\pm}(n/d) \to \Z$ with the decomposition
	 \begin{equation} \label{eq:doubled_odometer}
	 d(o)^{\pm}_{n/d} = o_{n/d} \cup o_{n/d}^*
	 \end{equation}
	 where $T(o_{n/d}) = T^b(n/d)^{1}$ and $T(o_{n/d}^*) = T^{h^{\pm}/v^{\pm}}(n/d)^2$
	 and, after being restricted to the relevant zero-one boundary strings from Definition \ref{def:partial_tile},
	 $o_{n/d}^*$ and $o_{n/d}$ respect the tile translations.  
\end{definition}

We now use these to partially define the standard recursion. The full recursion requires alternate tiles odometers
which are defined in the next two subsections. 

\begin{definition} \label{def:weak_standard_odometers}
	A pair of partial odometers $o_{p_0}:T^b(p_0) \to \Z$ and $o_{q_0}:T^b(q_0) \to \Z$ 
	are {\it weak standard tile odometers} for $(p_0, q_0)$ if they appear in Proposition \ref{prop:base_cases}, are $(o_{0/1}, o_{1/1})$ from Section \ref{subsec:zero-one-tiles} or if $(T(p_0), T(q_0))$ are standard tiles with standard decompositions
	for $(p_0,q_0)$ and the partial odometers have the standard decompositions:
	\begin{equation} \label{eq:std_odometer_decomposition}
	\begin{aligned}
	o_{p_0} = o_{p_1}^+ \cup o_{p_1}^- \cup d(o)_{q_1}^+ \cup d(o)_{q_1}^- \\
	o_{q_0} =  d(o)_{p_1}^+ \cup d(o)_{p_1}^- \cup o_{q_1}^+ \cup o_{q_1}^- 
	\end{aligned}
	\end{equation}
	where each $o_{n/d}$ is a weak standard tile odometer for $(n/d)$ and the offsets are specified by requiring the odometers respect the tile translations in Definition \ref{def:standard_tiles}. The weak standard tile odometers on the right-hand-side of \eqref{eq:std_odometer_decomposition} will be called weak subodometers of $o_{p_0}$ or $o_{q_0}$ respectively. 
\end{definition}

We say weak standard tile odometers $o_{n/d}$ and $o_{n/d}'$ are {\it lattice adjacent} if 
\begin{align*}
c(T(o_{n/d})) - c(T(o_{n/d})') &=  v_{n/d,i} \\
s(o_{n/d}) - s(o_{n/d}') &= a_{n/d,i},
\end{align*}
for $i \in \{1,2\}$. We now prove an analogue of Lemma \ref{lemma:standard_tile_properties}
for weak standard odometers.  

\begin{lemma} \label{lemma:weak_standard_odometer_properties}
	A weak standard odometer, $o(n/d)$ exists for every reduced rational $0 \leq n/d \leq 1$. Moreover, 
	when $0 < n/d < 1$, the odometer has the following properties. 
	\begin{enumerate}[label=(\alph*)]
		\item $o(n/d)$ respects $(w_{h}(n/d), w_v(n/d))$
		\item Let $A,B,C,D$ denote the the first and last tiles of $w_h(n/d)$ and the last and first tiles of $\rev(w_h)(n/d)$
		respectively (geometrically a counter-clockwise walk around the tile). Let $o_{Z}$ be the restriction of $o(n/d)$ to $Z \in \{A,B,C,D\}$. Then, $s(o_B) - s(o_A) = 0$, $s(o_C) - s(o_D) = a_{0/1,2}-a_{1/1,2}$ and $s(o_C) - s(o_B)= a_{n/d,2} - a_{1/1,2}$.
		\item Lattice adjacent $o_{n/d}'$ and $o_{n/d}$ are compatible. 
	\end{enumerate}
\end{lemma}
	
\begin{proof}	
	By Proposition \ref{prop:base_cases}, we may assume the tile $T(n/d)$ has a standard decomposition and hence 
	the subodometers of $o(n/d)$ exist and satisfy the inductive hypotheses. 
	
	{\it Step 1: Existence} \\
	Since subodometers exist, it suffices to show that the odometer decomposition is well-defined \ie, the subodometers
	have a common extension on their overlaps. By possibly deleting parts of the subodometers, 
	we may assume that the only overlaps are on the internal zero-one stacked boundary strings.
	By an inductive application of (a) and (b) the subodometers respect the corresponding stacked boundary strings and therefore
	by Lemma \ref{lemma:stacked_zero_one_string} have a common extension to $T^b(n/d)$.

%	Suppose $T(n/d)$ is odd and let  $o_{p_1}^+$ , $o_{p_1}^-$,  $d(o)_{q_1}^+$  $d(o)_{q_1}^-$ be weak standard odometers 
%	for the corresponding tiles in the decomposition. By possibly deleting part of the odometers in $d(o)_{q_1}^{\pm}$, 
%	we may assume that subodometers only overlap on the internal zero-one boundary strings as specified in part (d) of Lemma \ref{lemma:standard_tile_properties}. Therefore, compatibility follows when we prove that the translation and affine offsets between the boundary strings is correct. This will follow from parts (a) and (b) of the inductive hypothesis and the definition of zero-one boundary string. 
%	

	{\it Step 2. Inductive hypotheses} \\
	We copy the proof of parts (b) and (c) of Lemma \ref{lemma:standard_tile_properties}. To check that the affine offsets
	are the ones required in Lemma \ref{lemma:stacked_zero_one_string}, we use the fact that the subodometers respect the tiling together with an inductive application of (b). 
\end{proof}

\subsection{Alternate tiles}
We now construct alternate tiles. In this case, the decomposition depends on the last letter of the recursion word, in particular, the parents of the parents. If the last letter of the recursion word is $s$, we say that we are in the Type $s$ case. (Note that the initial quadruple is one of the base cases.)
\begin{definition} \label{def:alternate_tiles}
	
A pair of tiles $(\hat{T}(p_0), \hat{T}(q_0)) \subset \Z^2$ are {\it alternate tiles} for $(p_0, q_0)$ if
if they appear in Proposition \ref{prop:base_cases}, are $(T_{0/1},T_{1/1})$ from Section \ref{subsec:zero-one-tiles}
or if they have the {\it alternate tile decomposition}
\begin{equation} \label{eq:alt_decomp}
\begin{aligned}
\hat{T}(p_0) &= T^{ds}(p_1)^+ \cup  T^{ds}(p_1)^- \cup  T(q_1)^+ \cup T(q_1)^- \cup \hat{T}(q_1) \\
\hat{T}(q_0) &= T(p_1)^+ \cup T(p_1)^- \cup T^{ds}(q_1)^+ \cup T^{ds}(q_1)^- \cup \hat{T}(p_1)
\end{aligned}
\end{equation}
where $T^{ds}(n/d)^{\pm}$ denote doubled tiles where the doubling is different depending on the orientation:
\begin{equation} \label{eq:alt_doubled}
\begin{aligned}
T^{ds}(p)^{\pm} &= T(p)^{\pm,1} \cup T(p)^{\pm,2} :=  T(p) \cup (T(p) + S_{p}) \\
T^{ds}(q)^{\pm} &= T(q)^{\pm,1} \cup T(q)^{\pm,2} := T(q) \cup (T(q) + S_{q}),
\end{aligned}
\end{equation}
where
\begin{equation}
(S_{p}, S_{q}) = \begin{cases}
(v_{q,1} + v_{p,2}, -v_{q,1} + v_{p,2}) &\mbox{ if $\mathcal{W}_1$ is even} \\
(-v_{q,1} + v_{p,2}, v_{q,1} + v_{p,2}) &\mbox{ otherwise}  
\end{cases}
\end{equation}
and $(T(p_1), T(q_1))$ with or without superscripts are standard tiles for $(p_1,q_1)$ and $\hat{T}(n/d)$ is an alternate tile for $n/d$.

The standard tile positions in \eqref{eq:alt_decomp} depend on the parity of $\mathcal{W}_1$:
	if $\mathcal{W}_1$ is odd
\begin{equation} \label{eq:alt_offsets_w1_odd}
\begin{aligned}
c(T) - c(\hat{T}(p_0)) &= 
\begin{cases}
v_{p_1,1} &\mbox{ if $T = T(q_1)^+$} \\
2 v_{p_1,2} - v_{q_1,1} &\mbox{ if $T = T(q_1)^-$} \\
0 &\mbox{ if $T = T^{ds}(p_1)^+$} \\
v_{q_1,1} + v_{q_1,2} &\mbox{ if $T = T^{ds}(p_1)^-$} 
\end{cases} \\
 c(T) - c(\hat{T}(q_0)) &= 
\begin{cases}
v_{p_1,1} &\mbox{ if $T = T^{ds}(q_1)^+$} \\
v_{p_1,2}  &\mbox{ if $T = T^{ds}(q_1)^{-}$ } \\
0  &\mbox{ if $T = T(p_1)^+$ } \\
v_{p_1,2} + v_{q_1,2} + 2 v_{q_1,1} &\mbox{ if $T = T(p_1)^-$ } 
\end{cases}
\end{aligned}
\end{equation}
otherwise
\begin{equation} \label{eq:alt_offsets_w1_even}
\begin{aligned}
c(T) - c(\hat{T}(p_0)) &= 
\begin{cases}
0  &\mbox{ if $T = T(q_1)^+$} \\
v_{p_1,1} + 2 v_{p_1,2} + v_{q_1,1} &\mbox{ if $T = T(q_1)^-$} \\
v_{q_1,1} &\mbox{ if $T = T^{ds}(p_1)^+$} \\
v_{q_1,2} &\mbox{ if $T = T^{ds}(p_1)^-$} 
\end{cases} \\
c(T) - c(\hat{T}(q_0))  &= 
\begin{cases}
0 &\mbox{ if $T = T^{ds}(q_1)^+$} \\
v_{p_1,1} + v_{p_1,2}  &\mbox{ if $T = T^{ds}(q_1)^{-}$} \\
v_{q_1,1}  &\mbox{ if $T = T(p_1)^+$} \\
v_{p_1,2} + v_{q_1,2} - v_{q_1,1} &\mbox{ if $T = T(p_1)^-$}.
\end{cases}
\end{aligned}
\end{equation}

The alternate tile positions in \eqref{eq:alt_decomp} may depend on both the parity of $\mathcal{W}_1$, (even-first/odd-first) and the type of the child:
\begin{equation}
 c(\hat{T}(q_1)) - c(\hat{T}(p_0)) = 
\begin{cases}
v_{p_1,2} + v_{p_1,1} -v_{q_1,1} &\mbox{ odd-first and Type 1 or 3 } \\
- &\mbox{ odd-first and Type 2 } \\
v_{q_1,1} + v_{p_1,2} &\mbox{ even-first and Type 1 or 3 } \\
- &\mbox{ even-first and Type 2 } \\
\end{cases}
\end{equation}

\begin{equation}
 c(\hat{T}(p_1)) - c(\hat{T}(q_0)) = 
\begin{cases}
v_{p_1,2} + v_{q_1,1} &\mbox{ odd-first and Type 1 } \\
2 v_{q_1,1} + v_{q_1,2} &\mbox{ odd-first and Type 2 } \\
- &\mbox{ odd-first and Type 3 } \\
v_{p_1,2} + v_{p_1,1} - v_{q_1,1} &\mbox{ even-first and Type 1 } \\
v_{q_1,2} &\mbox{ even-first and Type 2 } \\
- &\mbox{ even-first and Type 3 } 
\end{cases}
\end{equation}
where in the cases indicated by $-$ the alternate tile is omitted. 

\end{definition}

We now make an important exception in the definition of $c(\cdot)$ for alternate tiles. 
If $\hat{T}(n/d)$ has an alternate decomposition, then
\begin{equation}
c(\hat{T}(n/d)) = 
\begin{cases}
 c(T(q_1)^{+,1})  &\mbox{even-first and $n/d$ even} \\
  c(T(q_1)^{+}) &\mbox{even-first and $n/d$ odd} \\
  c(T(p_1)^{+}) &\mbox{odd-first and $n/d$ even} \\
  	c(T(p_1)^{+,1}) &\mbox{odd-first and $n/d$ odd}.
\end{cases}
\end{equation}

\begin{figure}
	\centering
	\includegraphics[scale=0.37]{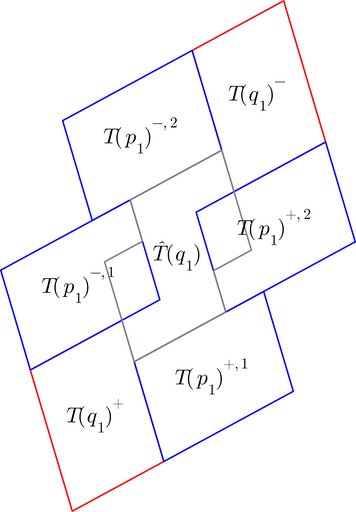}
	\includegraphics[scale=0.37]{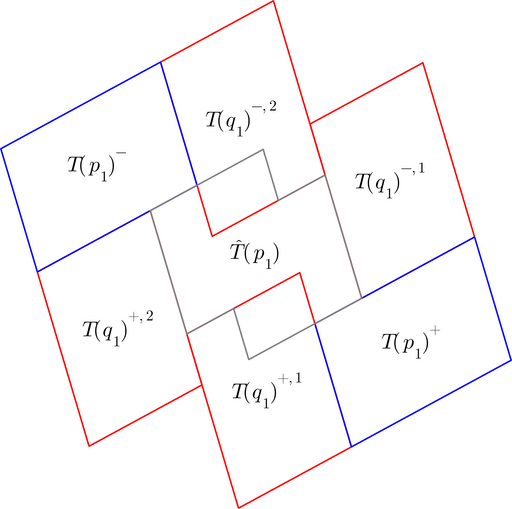} \\
	\includegraphics[scale=0.37]{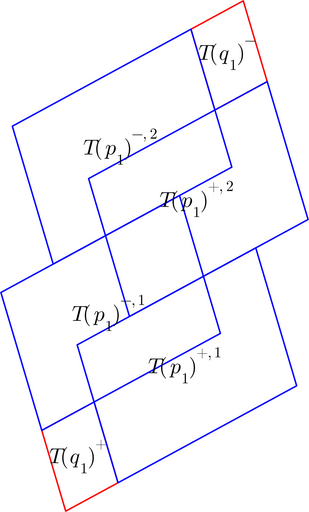}
	\includegraphics[scale=0.37]{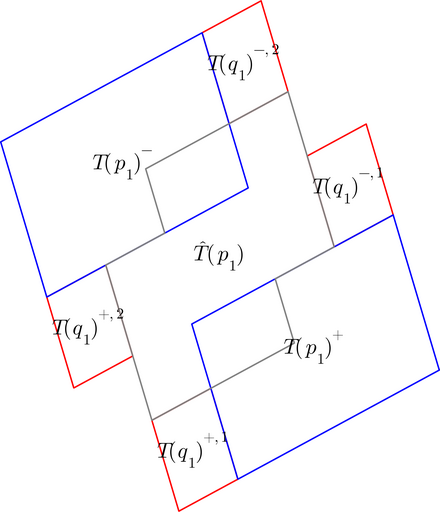} \\
	\includegraphics[scale=0.37]{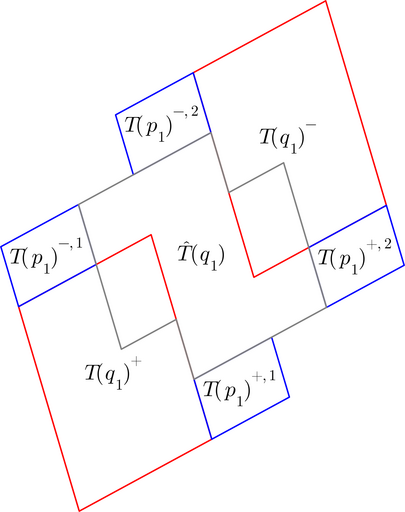}
	\includegraphics[scale=0.37]{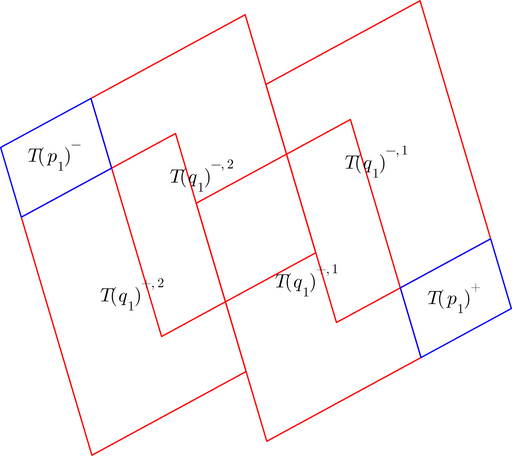} 
	\caption{As Figure \ref{fig:standard_tile_overlaps} (even-first) but for alternate tiles with labels from Definition \ref{def:alternate_tiles}.} \label{fig:alt_tile_overlaps}
\end{figure}

\begin{figure}
	\includegraphics[scale=0.38]{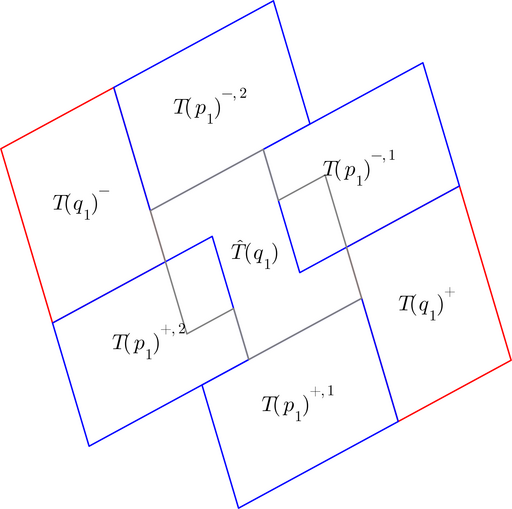}
	\includegraphics[scale=0.38]{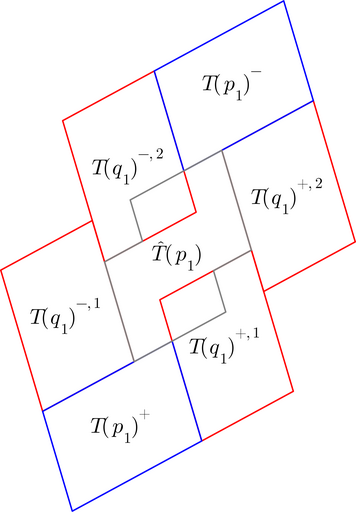} \\
	\includegraphics[scale=0.38]{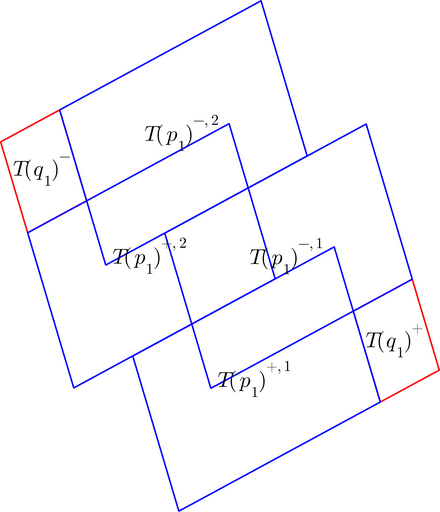}
	\includegraphics[scale=0.38]{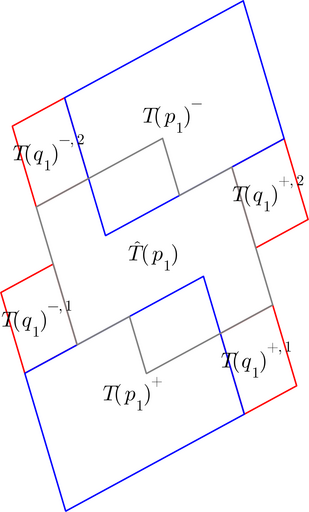} \\
	\includegraphics[scale=0.38]{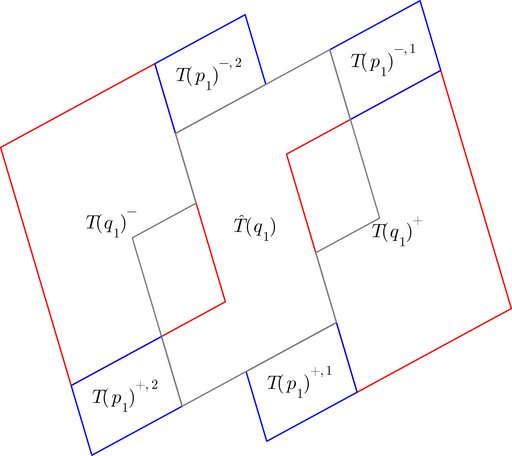}
	\includegraphics[scale=0.38]{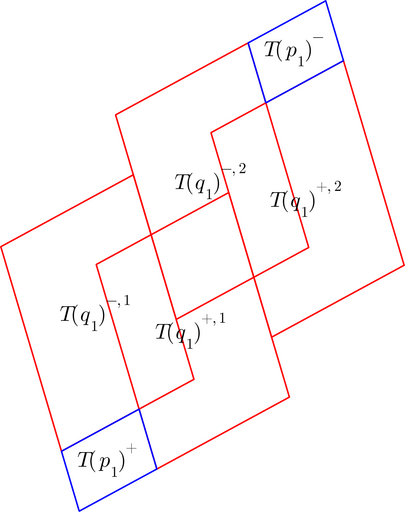} 
	\caption{As Figure \ref{fig:alt_tile_overlaps} but in the odd-first orientation} \label{fig:odd-first_alt_tile_overlaps}
\end{figure}

Next is the analogue of Lemma \ref{lemma:standard_tile_properties} for alternate tiles. 
\begin{lemma} \label{lemma:alternate_tile_properties}
	An alternate tile, $\hat{T}(n/d)$ exists for every reduced rational $0 \leq \frac{n}{d} \leq 1$. Moreover, 
	when $0 < n/d < 1$, the alternate tile has the following properties. 
	\begin{enumerate}
		\item Symmetry and rotation invariance: $\hat{T}(n/d)$ is 180-degree symmetric and 
		\[
		\mathcal{R}(\hat{T}(n/d)) = \hat{T}(\mathcal{R}(n/d)).
		\]
		\item $\hat{T}(n/d)$ covers space under the lattice $L'(n/d)$
		\item If $n/d$ is even $\hat{T}(n/d)$ is a $w_{h}(n/d)$-pseudo-square, otherwise a $w_{v}(n/d)$ pseudo-square with offsets respecting the tiling: 
		\begin{enumerate}
			\item Even case: $c(T_1) = c(\hat{T}(n/d))$ where $T_1$ is first tile of $w_{h}(n/d)$ %and 
		%	$c(T_2)-c(T(n/d)) = v_{p,2} + v_{q,2}-v_{q,1}$ where $T_2$ is the first tile of $\rev(w_{h}(n/d))$
			\item Odd case: $c(T_1') = c(\hat{T}(n/d)) + s$ where $T_1'$ is the first tile of $\rev(w_{v}(n/d))$
			where $s = 0$ in the even-first orientation, otherwise $s = -v_{q_1,1}+v_{p_1,2}$. 
			%and $c(T_2') - c(T(n/d)) = v_{p,2} + v_{q,1}+
		\end{enumerate}
		\item The surrounding of $\hat{T}(n/d)$ with respect to $(v_{n/d,1}, v_{n/d,2})$ consists of 
		either part of a stacked zero-one boundary string or a complete overlap on a subtile.
		\item When $\hat{T}(n/d)$ has an alternate decomposition, the shared boundary between neighboring subtiles is part of or is a stacked horizontal or vertical zero-one boundary string. 
	\end{enumerate}
\end{lemma}

\begin{proof}
	We may assume by Proposition \ref{prop:base_cases} that $\hat{T}(n/d)$ has an alternate decomposition. The proof 
	of (1) is identical to that of Lemma \ref{lemma:standard_boundary_word_recursion}. 
	
	{\it Step 1: (2)} \\
	By (1), we may assume $\hat{T}(n/d)$ is even. We also suppose $\hat{T}(n/d)$ is in the even-first orientation, as the odd-first argument is repetitive. Let $w_{n'/d', i}$ denote the boundary words 
	of the standard subtiles as specified by Lemma \ref{lemma:standard_boundary_word_recursion}. 
	Consider the enlarged tile, $T^e(n/d)$ with boundary word $w_1 = w_{q_1,1}*w_{q_1,1}*w_{p_1,1}$
	and $w_2 = w_{p_1,2}*w_{q_1,2}*w_{p_1,2}$. Translate the enlarged tile so that $c(T^e(n/d)) - c(\hat{T}(n/d)) = -v_{q_1,1}$. 
	
	By Lemma \ref{lemma:almost_square_tiling}, $T^e(n/d)$ generates a regular $(v_{n/d,1}+v_{q,1}, \sum v_{n/d,2})$ almost pseudo-square tiling.  (To see that $w_2$ satisfies
	the conditions needed in the Lemma, observe that $w_{p_1,2}*w_{q_1,2}*w_{p_1,2}*w_{q_1,2}$ consists of two concatenated vertical boundary strings of a standard tile.)
	
	We use this to show there are no gaps in the $(v_{n/d,1}, v_{n/d,2})$-regular tiling of $\hat{T}(n/d)$
	and the only subtiles which overlap are $T(q_1)^{-,1}$ and $T(q_1)^{+,2}$.	The proof proceeds along the lines of Figure \ref{fig:shifted_overlap}.
	
 	By 180-degree symmetry, it suffices to analyze the lower-right corner of a surrounding. Take the shifted tiling of $\hat{T}(n/d)$ with respect to  $(\sum v_{n/d,1}+v_{q,1}, \sum v_{n/d,2})$ and compare the lower-right surrounding: 
	\begin{align*}
	S^{s} := \hat{T}(n/d) \cup (\hat{T}(n/d) + v_{n/d,1} + v_{q_1,1}) \cup (\hat{T}(n/d) - v_{n/d,2}) \cup (\hat{T}(n/d) + v_{n/d,1} - v_{n/d,2} + v_{q_1,1})
%	=: \hat{T}(n/d) \cup T^{r} \cup T^{d} \cup T^{dr}
	\end{align*}
 	to the corresponding corner of the non-shifted tiling
	\begin{align*}
	S &:= \hat{T}(n/d) \cup (\hat{T}(n/d) + v_{n/d,1}) \cup (\hat{T}(n/d) - v_{n/d,2}) \cup (\hat{T}(n/d) -v_{n/d,2} + v_{n/d,1}).
%	&\hat{T}(n/d) \cup T^{r} \cup T^{d} \cup T^{dr}
	\end{align*}
	Since $T^e(n/d)$ generates a $(v_{n/d,1}+v_{q,1}, v_{n/d,2})$ almost pseudo-square tiling, each 
	pair of tiles in $S^{s}$ can only overlap on their boundaries and by definition of $T^e(n/d)$, there 
	are only two gaps in $S^{s}$ both of which are pseudosquares with a $(w_{q_1,1}, w_{p_1,2})$ boundary word. 
	Using the alternate decomposition, these two gaps are filled in the non-shifted tiling and $T(q_1)^{-,1}$ and $T(q_1)^{+,2}$ overlap
	completely, concluding the proof of this step. 
	
	{\it Step 2: (3) (4) and (5)} \\
	Given Step 1, the proof is similar to that of Lemma \ref{lemma:standard_tile_properties}. 
	Indeed, we can use the alternate decomposition to concatenate boundary strings of the standard subtiles 
	which make up the boundary of $\hat{T}(n/d)$. See Figures \ref{fig:alt_tile_overlaps}, \ref{fig:odd-first_alt_tile_overlaps},
	and the bottom of Figure \ref{fig:shifted_overlap}.
\end{proof}

\begin{figure}
	\includegraphics[width=0.6\textwidth]{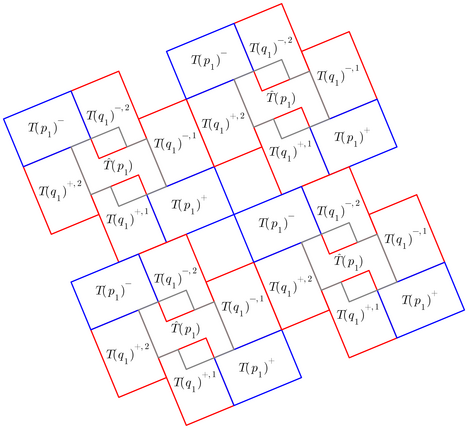}
	\includegraphics[width=0.6\textwidth]{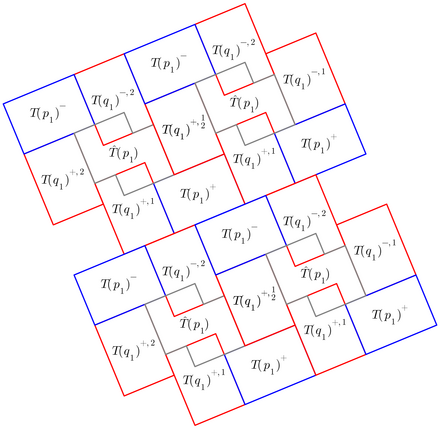}
	\caption{A visual explanation of the proof of Lemma \ref{lemma:alternate_tile_properties}. 
On the top is a lower-right surrounding of the shifted tiling and on the bottom is a lower-right surrounding of the non-shifted tiling. 
} \label{fig:shifted_overlap}
\end{figure}

\subsection{Weak alternate odometers}

The recursion for alternate odometers given the subtile placement is similar to the standard ones. 
To that end, we first extend the notion of boundary tile and doubled odometers to this shifted case.

\begin{definition} \label{def:alternate_doubled_tile}
	When $\hat{T}(n/d)$ has an alternate decomposition, an alternate boundary tile, $\hat{T}^b(n/d)$ coincides 
	with the standard boundary subtiles in its decomposition
	\begin{equation}
	\hat{T}^b(n/d) \cap \hat{T}(n/d)
	= \bigcup T^b(n'/d') \qquad \mbox{where $T(n'/d')$ ranges over the standard subtiles} 
	\end{equation}	
	
	Let $T^{ds,h/v, \pm}(n/d) = T^b(n/d) \cup (T^{h/v,\pm}(n/d) + S_{n/d})$ denote the shifted doubling of $T(n/d)$
	where $S_{n/d}$ is from \eqref{eq:alt_doubled}. As in Definition \ref{def:standard_odometer_doubling} extend 
	the weak doubled odometer $ds(o)_{n/d}: T^{ds,h/v,\pm}(n/d) \to \Z$ to this case. 
\end{definition}

We are now ready to state a weak version of the alternate odometer recursion, analogous to the standard one from before. 
\begin{definition} \label{def:weak_alternate_odometers}
	A pair of partial odometers $\hat{o}_{p_0}:T(p_0) \to \Z$ and $\hat{o}_{q_0}: T(q_0) \to \Z$ 
are {\it weak alternate tile odometers} for $(p_0, q_0)$ if they appear in Proposition \ref{prop:base_cases}, are $(o_{0/1}, o_{1/1})$ from Section \ref{subsec:zero-one-tiles} or $(\hat{T}(p_0), \hat{T}(q_0))$ are alternate tiles 
for $(p_0,q_0)$ and the partial odometers have the alternate decompositions:
\begin{equation}\label{eq:alt_odometer_decomposition}
\begin{aligned}
\hat{o}_{p_0} = ds(o)_{p_1}^+ \cup ds(o)_{p_1}^- \cup o_{q_1}^+ \cup o_{q_1}^- \cup \hat{o}(q_1)  \\
\hat{o}_{q_0} =  o_{p_1}^+ \cup o_{p_1}^- \cup ds(o)_{q_1}^+ \cup ds(o)_{q_1}^- \cup \hat{o}(p_1) 
\end{aligned}
\end{equation}
where $ds(o)_{n/d}^{\pm}$ and $\hat{o}(n/d)$ are the respective weak shifted doublings and weak alternate parent odometers
respecting the tile translations in Definition \ref{def:alternate_tiles}. 
In the indicated cases in Definition \ref{def:alternate_tiles}, we omit the alternate parent odometers.

\end{definition}

\begin{lemma} \label{lemma:weak_alternate_odometer_properties}
	A weak alternate odometer, $\hat{o}(n/d)$ exists for every reduced rational $0 \leq n/d \leq 1$. Moreover, 
	when $0 < n/d < 1$, the odometer has the following properties. 
	\begin{enumerate}[label=(\alph*)]
		\item $\hat{o}(n/d)$ respects $w_{h}(n/d)$ or $w_v(n/d)$ if $(n/d)$ is even or odd respectively.
		\item Lattice adjacent $\hat{o}_{n/d}'$ and $\hat{o}_{n/d}$ are compatible. 
	\end{enumerate}
\end{lemma}

\begin{proof}	
	Given Lemma \ref{lemma:alternate_tile_properties} the proof is identical to the proof of Lemma \ref{lemma:weak_standard_odometer_properties}.
\end{proof}

\section{Correcting the recursion} \label{sec:l_correction}
In this section we complete the weak recursion defined in the previous section. This step
requires us to possibly `correct' doubled odometers which overlap. This is done by either a chain of ancestors 
or by checking that immediate parents overlap on immediate grandparents. We start by developing the machinery 
to chain together ancestors and then use that to fully define the recursion.

\subsection{Corrected partial tiles}
\begin{figure}
	\includegraphics[width=0.5\textwidth]{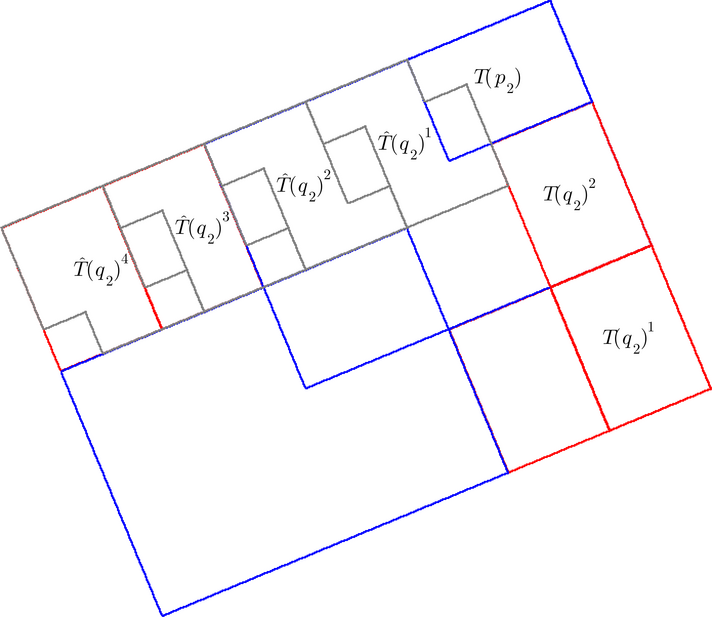}
	\caption{An odd $L$-correction with labels as in Definition  \ref{def:odd_l_correction} overlaid on its standard tile.}
\end{figure}

In this section we define a partial tile, Definition \ref{def:partial_tile}, for every standard tile which will aid us in fully defining doubled weak odometers which don't overlap on a common ancestor. The construction of this partial tile will involve a chain of ancestor tiles. Let $(p_0, q_0) = \mathcal{C}(p_1, q_1)$ be a Farey child pair in quadruple $\q_{w}$ with recursion word $w \in F_3^*$.
Recall the standard tile decomposition from Definition \ref{def:standard_tiles}. 	See Figure \ref{fig:l_corrections_orientations_odd}.

\begin{definition}[Odd $L$-correction] \label{def:odd_l_correction}

	Suppose $w = w_1 * s * 2^k$ where $|w_1| \geq 0$, $s \in \{1,3\}$, and $k \geq 0$.
	Let $(p_2, q_2)$ be the parent Farey pair corresponding to the string $w = w_1*s$ and let $(T(p_2), T(q_2))$ be standard tiles for $(p_2,q_2)$ and $\hat{T}(q_2)$ an alternate tile for $q_2$.

	An $L$-correction for $T(p_0)$, $T^{L,\pm}(p_0)$ is a partial tile, $T^{v, \pm}(p_0)$ with the following decomposition
	\begin{equation}\label{eq:l_correction_odd}
	T^{L,\pm}(p_0) = \bigcup_{j=0}^{k} (T(q_2) + K_1 j) \cup  (T(p_2) + K_2) \cup \left(K_3 + \bigcup_{i=1}^{2(k+1)-z} (\hat{T}(q_2) +  K_4 i)\right)
	\end{equation}
	where $c(T(p_2)) = c(T(q_2)) + K_1 k$ and $c(\hat{T}(q_2)) = c(T(p_2)) + K_2$ and the initial offset of $T(q_2)$ is specified by requiring $c(T(q_2)) = c(T(q_1))$ for an outer tile $T(q_1)$ in the standard decomposition of $T(p_0)$:
	\[
	\begin{cases}
	T(q_1)^{+,2}  &\mbox{ $10$} \\
	T(q_1)^{-,1} &\mbox{ $11$} \\
	T(q_1)^{+,1} &\mbox{ $00$} \\
	T(q_1)^{-,2} &\mbox{ $01$} 
	\end{cases}
	\]
	and the subsequent offsets are
	\[
	[K_1,K_2] = 
	\begin{cases}
	[v_{q_2,2}, v_{q_2,2} + (v_{q_2,1} - v_{p_2,1} )]  &\mbox{ $ 1 0 $} \\
	[v_{q_2,2}, -v_{p_2,2}] &\mbox{ $1 1$} \\
	[v_{q_2,2},v_{q_2,2}] &\mbox{ $0 0$} \\
	[-v_{q_2,2},-v_{p_2,2}-(v_{p_2,1} - v_{q_2,1})] &\mbox{ $0 1$} 
	\end{cases}
	\]
	and
	\[
	[K_3,K_4] =
	\begin{cases}
	[(-v_{q_2,2}+v_{p_2,2}) + (v_{p_2,1} - v_{q_2,1}), -v_{q_2,1}]  &\mbox{ $ 1 0 $} \\
	[v_{p_2,1}-v_{q_2,1}, v_{q_2,1}] &\mbox{ $1 1$} \\
	[-(v_{q_2,2}+v_{p_2,2}), v_{q_2,1}] &\mbox{ $0 0$} \\
	[0, -v_{q_2,1}] &\mbox{ $0 1$},
	\end{cases}
	\]
	where the right-hand side columns denote the case:
	\begin{equation} \label{eq:case_labels}
	\begin{aligned}
		1 0 =&\mbox{ $\mathcal{W}_1$ is odd and $+$}  \\
		1 1 =&\mbox{ $\mathcal{W}_1$ is odd and $-$}  \\
		0 0 =&\mbox{ $\mathcal{W}_1$ is even and $+$}  \\
		0 1 =&\mbox{ $\mathcal{W}_1$ is even and $-$}.
	\end{aligned}
	\end{equation}
The term $z$ in \eqref{eq:l_correction_odd} is either 0 or 1 and if $z = 0$, we say the $L$-correction is elongated. 
\end{definition}

The correction in the even-case is similar but the offsets are slightly different 
due to the lack of rotational symmetry in the parameterization. 	See Figure \ref{fig:l_corrections_orientations} and Figure \ref{fig:l_corrections_overlaps}. 

\begin{definition}[Even $L$-correction] \label{def:even_l_correction}

	Suppose $w = w_1 * s * 3^k$ where $|w_1| \geq 0$, $s \in \{1,2\}$, and $k \geq 0$.
	Let $(p_2, q_2)$ be the parent Farey pair corresponding to the string $w = w_1*s$ and let $(T(p_2), T(q_2))$ be standard tiles for $(p_2,q_2)$ and $\hat{T}(p_2)$ an alternate tile for $p_2$.

	 An $L$-correction for $T(q_0)$, $T^{L,\pm}(q_0)$ is a partial tile $T^{h, \pm}(q_0)$ with the following decomposition. 
	\begin{equation}\label{eq:l_correction_even}
	T^{L,\pm}(q_0) = \bigcup_{j=0}^{k} (T(p_2) + K_1 j) \cup  (T(q_2) + K_2) \cup \left(K_3 + \bigcup_{i=1}^{2(k+1)-z} (\hat{T}(p_2) +  K_4 i)\right)
	\end{equation}
	where $c(T(q_2)) = c(T(p_2)) + K_1 k$, $\hat{T}(p_2) = c(T(q_2)) + K_2$ and the initial offset is specified by requiring $c(T(p_2)) = c(T(p_1))$ for an outer $T(p_1)$ tile in the standard decomposition of $T(p_0)$:
	\[
	\begin{cases}
	T(p_1)^{+,1}  &\mbox{ $\{0 \mbox{ or } 1\}0$} \\
	T(p_1)^{-,2} &\mbox{ $\{0 \mbox{ or } 1\}1$} 
	\end{cases}
	\]
	and
	 \[
	 [K_1,K_2,K_4] = 
	 \begin{cases}
	[v_{p_2,1}, v_{p_2,1}, v_{p_2,2}] &\mbox{ $1 0$  } \\
	[-v_{p_2,1}, -v_{q_2,1}, v_{p_2,2}]  &\mbox{ $0 0$ } \\
	[-v_{p_2,1}, -v_{q_2,1}+v_{p_2,1}-v_{q_2,1}, -v_{p_2,2}]  &\mbox{ $1 1$ } \\
	[v_{p_2,1}, v_{p_2,1} + v_{p_2,2} -v_{q_2,2}, -v_{p_2,2}] &\mbox{ $0 1$ } 
	 \end{cases}
	 \]
	 and
	 \[
	 K_3 = 
	 \begin{cases}
	 v_{q_2,1}-v_{p_2,1}  + (v_{q_2,2} - v_{p_2,2} + v_{q_2,1} )  &\mbox{ $1 0$ and $s=2$} \\
	 v_{q_2,1}-v_{p_2,1}  	&\mbox{ $1 0$ and $s=1$} \\
	 v_{q_2,2} - v_{p_2,2}  &\mbox{ $0 0$ and $s=2$} \\
	 v_{p_2,1} - v_{q_2,1} &\mbox{ $0 0$ and $s=1$} \\
	 v_{q_2,2} - v_{p_2,2} + v_{q_2,1} &\mbox{ $1 1$ and $s=2$} \\
	 0 &\mbox{ $1 1$ and $s=1$ }\\
	 v_{q_2,2} - v_{p_2,2} - v_{p_2,1} + v_{q_2,1} &\mbox{ $0 1$ and $s=2$} \\
	 	0 &\mbox{ $0 1$ and $s=2$},
	 \end{cases}
	 \]
	 where the cases on the right are described by \eqref{eq:case_labels}. The term $z$ in \eqref{eq:l_correction_even} is either 0 or 1 and if $z = 0$, we say the $L$-correction is elongated.

\end{definition}
Note that the initial offset requirement assumes that $T(p_2) = T(p_1)$ in the odd case and $T(q_2) = T(q_1)$ in the even case
but this follows the standard tile decomposition or Proposition \ref{prop:base_cases}. We assert a final exception to the definition of $c(\cdot)$
\begin{equation} \label{eq:corner_l_correction}
c(T^{L,\pm}(n/d)) = c(T(n/d)).
\end{equation}
We next verify existence of $L$-corrections.

\begin{lemma} \label{lemma:l_correc_properties}
	For every $n/d \in \{p_0, q_0\}$ from Definition \ref{def:even_l_correction} or \ref{def:odd_l_correction}
	an $L$-correction exists and has the following properties. 
	\begin{enumerate} %counter-clockwise rotation
		\item Rotation invariance: 
		\begin{align*}
		\mathcal{R}^1(T^{L, +}(q_0)) &= T^{L, -}(\mathcal{R}(q_0)) \\
		\mathcal{R}^2(T^{L, +}(q_0)) &= T^{L,-}(q_0) \\
		\mathcal{R}^3(T^{L, +}(q_0)) &=  T^{L, +}(\mathcal{R}(q_0)) \\
		\mathcal{R}^4(T^{L, +}(q_0)) &= T^{L, +}(q_0) \, ,
		\end{align*}
		where $\mathcal{R}^i$ represents applying the rotation $\mathcal{R}$ repeatedly $i$ times.
		\item  The elongated $T^{L,\pm}(q_0)$ coincides with $T(q_0)$ on two sides
		\[
		T^{L, \pm}(q_0) \cap T(q_0) \supset
		\begin{cases}
			w_h(q_0) \cup w_v(q_0) &\mbox{ $1 0$  } \\
			w_h(q_0) \cup \rev(w_v(q_0))  &\mbox{ $0 0$ } \\
			\rev(w_h(q_0)) \cup \rev(w_v(q_0))  &\mbox{ $1 1$ } \\
			\rev(w_h(q_0)) \cup w_v(q_0)  &\mbox{ $0 1$ } 
		\end{cases}
		\] 
	where the cases on the right are described by \eqref{eq:case_labels}. 
	\end{enumerate}
\end{lemma}
\begin{proof}
	The stated rotation invariance follows from the decomposition and rotation invariance of standard and alternate tiles.
	To prove the second claim, we rotate and flip to assume we are in case `10' and $w = w_1 *s*3^k$ from Definition \eqref{def:even_l_correction}. We then proceed by induction on $k \geq 0$
	
	If $k = 0$, the claim follows by the standard decomposition of $T(q_0)$ and the definition of the $L$-correction. Indeed, the elongated $T^{L,+}(q_0) \cap T(q_0)\supset T(p_1)^{+,1} \cup T(q_1)^+$ which shows $T^{L, \pm}(q_0) \cap T(q_0) \supset w_h(q_0)$. 
	For the vertical direction, since $\hat{T}(p_1)$ is a $w_v(p_1)$-pseudo-square, it agrees with $T(p_1)$ on the vertical boundaries. 
	
	Now, suppose $k \geq 1$ is given and let $q_0'$ be the even child in $\q_{w[1:|w|-1]}$. By the inductive hypothesis, the elongated $T^{L, +}(q_0')$ coincides with $T(q_0')$ on the bottom and left boundaries. Also, by definition, we may write the elongated correction for $q_0$ as  
	\begin{align*}
	T^{+,L}(q_0) &= T(p_1) \cup (T^{+,L}(q_0') + v_{p_1, 1})  \\
	&\cup ( \hat{T}(p_1)  \cup \hat{T}(p_1) + v_{p_1,2}) + (v_{q_1,1} + v_{q_0,1} + v_{q_0,2}).
	\end{align*}
	We can then conclude using the standard decomposition for $T(q_0)$.
\end{proof}

\begin{figure}[!ht]
	\includegraphics[scale=0.5]{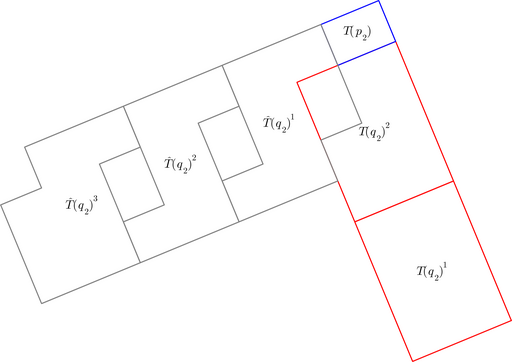}
	\includegraphics[scale=0.4]{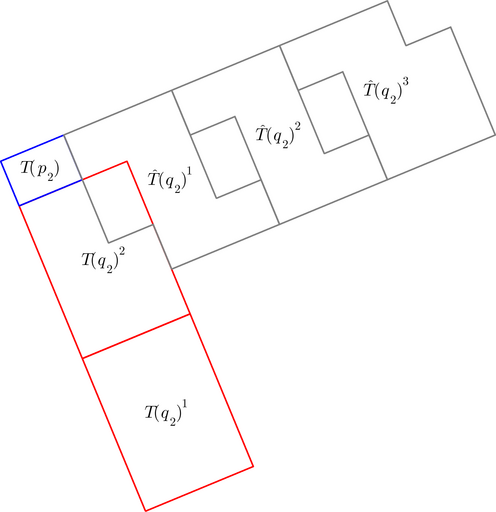} 
	\includegraphics[scale=0.5]{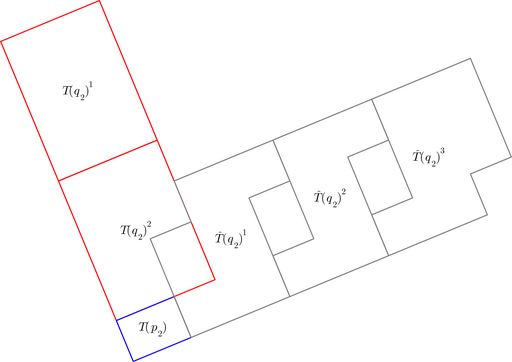}
	\includegraphics[scale=0.4]{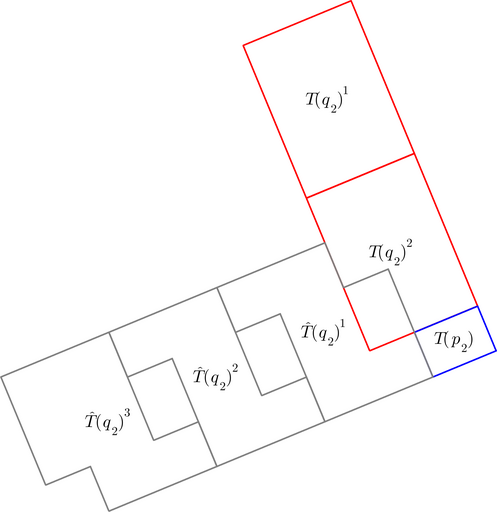}
	\caption{The four possible orientations for an odd $L$-correction as described in Definition \ref{def:odd_l_correction}. 
		From top left to bottom right, $10$, $00$, $11$ then $01$.} \label{fig:l_corrections_orientations_odd}
\end{figure}

\begin{figure}[!ht]
	\includegraphics[scale=0.4]{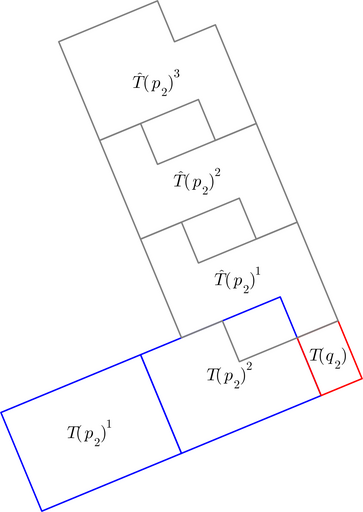}
	\includegraphics[scale=0.4]{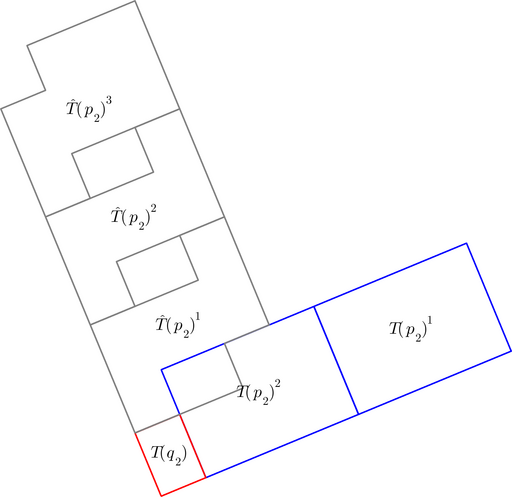} 
	\includegraphics[scale=0.4]{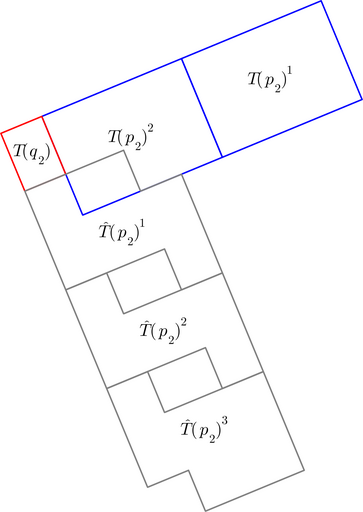}
	\includegraphics[scale=0.4]{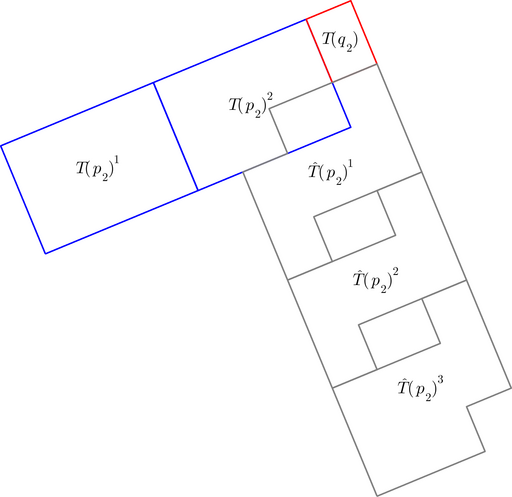}
	\caption{The four possible orientations for an even $L$-correction as described in Definition \ref{def:even_l_correction}. 
	From top left to bottom right, $10$, $00$, $11$ then $01$.} \label{fig:l_corrections_orientations}
\end{figure}

\begin{figure}[!ht]
	\includegraphics[width=0.49\textwidth]{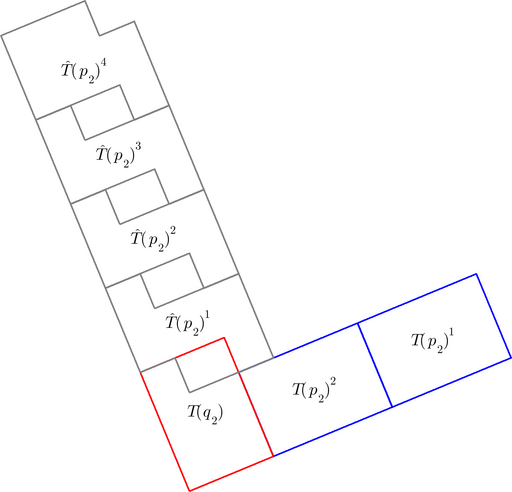}
	\includegraphics[width=0.49\textwidth]{Figures/tiles/corrected_tiles/even_tile/00_2.png}
	\caption{The two possible odd-even overlaps for an even $L$-correction as described in Definition \ref{def:even_l_correction}: 
		$s = 1$ on the left (elongated) and $s = 2$ on the right.} \label{fig:l_corrections_overlaps}
\end{figure}

\begin{figure}
		\includegraphics[width=0.49\textwidth]{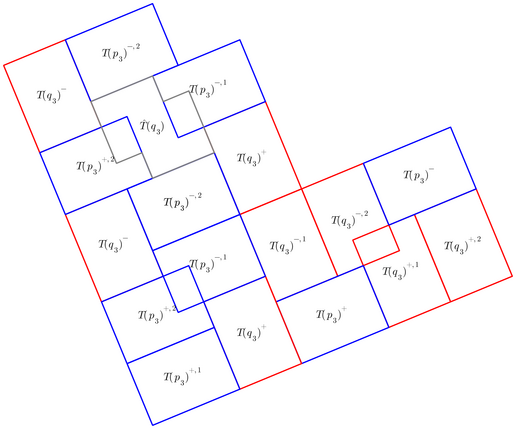}
			\includegraphics[width=0.49\textwidth]{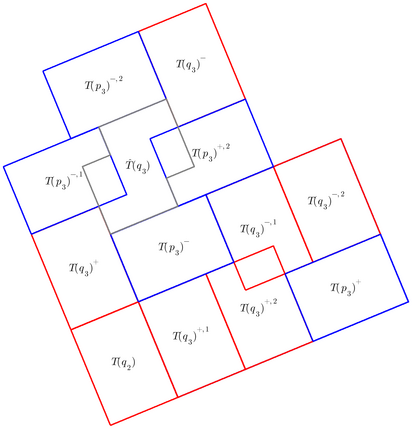}
	\caption{Decompositions of the triple $(\hat{T}(p_2)^1, T(q_2), T(p_2)^2)$ in Figure \ref{fig:l_corrections_overlaps}: $s=1$ on the left and $s=2$ on the right.	
}\label{fig:l_corrections_dd}
	
\end{figure}

\subsection{Tile odometers} \label{subsec:full_recursion}

\begin{figure}
	\includegraphics[scale=0.29]{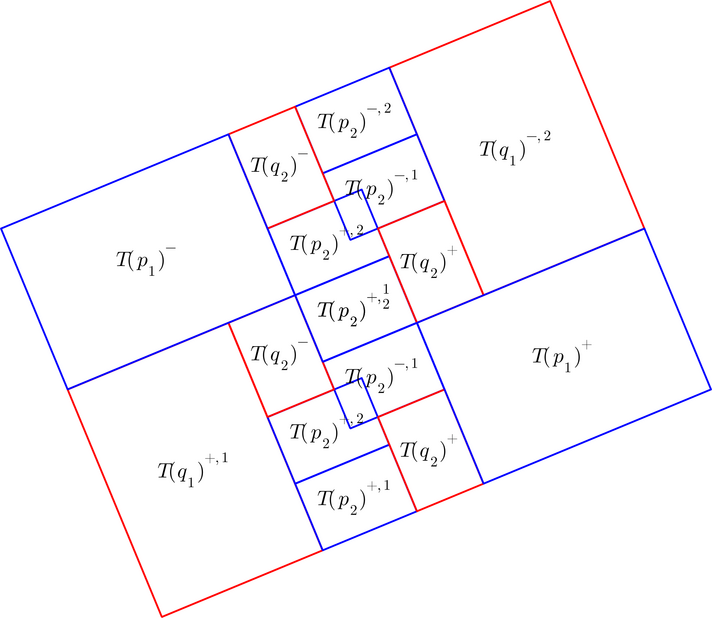} \quad
		\includegraphics[scale=0.29]{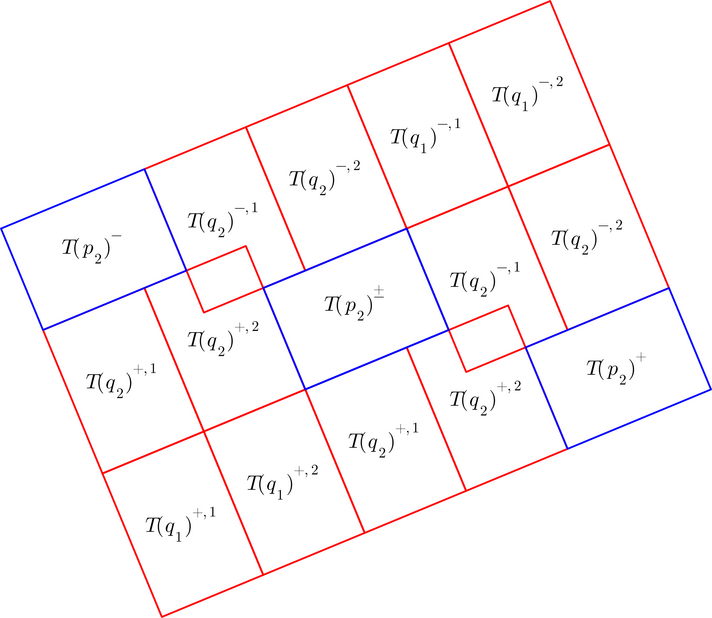}
	\caption{Double decomposition of $T(p_0)$ in the even-first orientation for recursion words $w = *1$ and $w = *2$ respectively.}
	\label{fig:std_double_decomp_no_correction}
\end{figure}
\begin{figure}
	\includegraphics[scale=0.29]{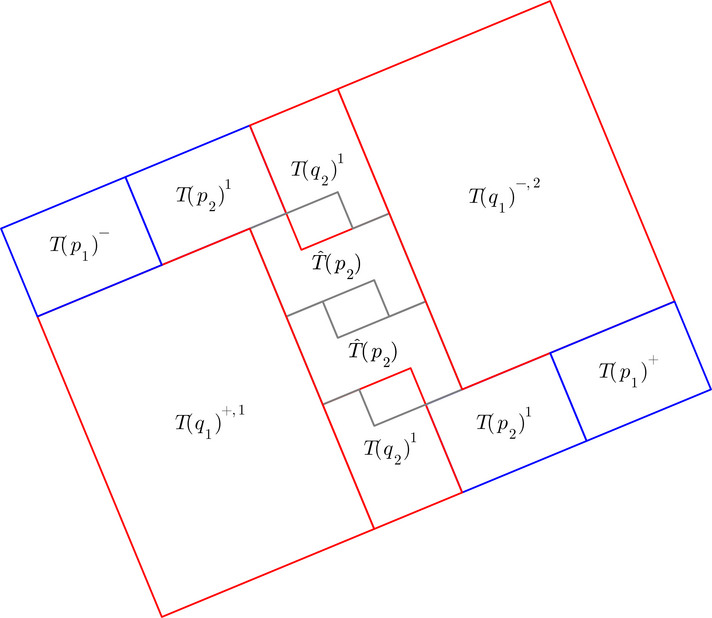} \quad
	\includegraphics[scale=0.29]{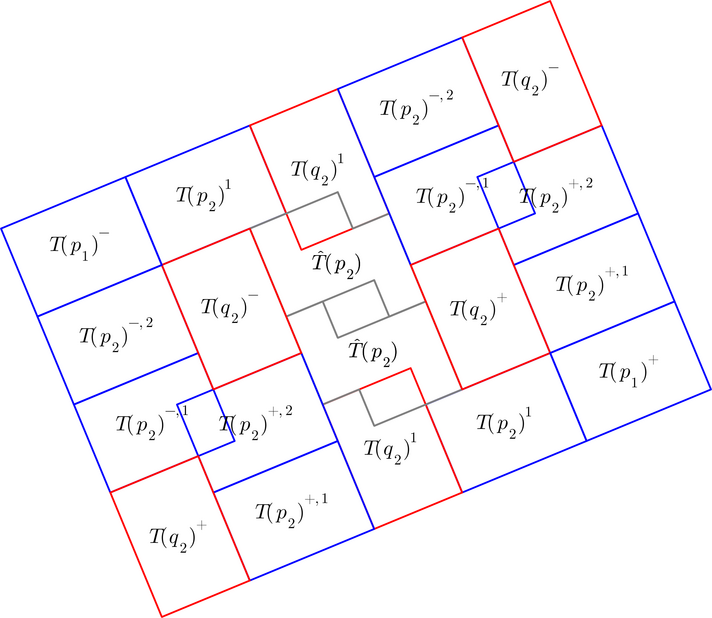} \\
	\includegraphics[scale=0.29]{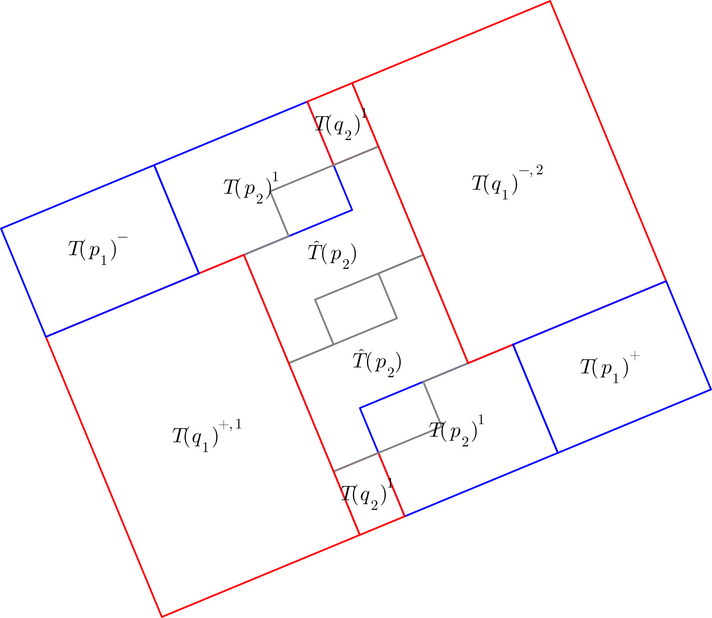} \quad
	\includegraphics[scale=0.29]{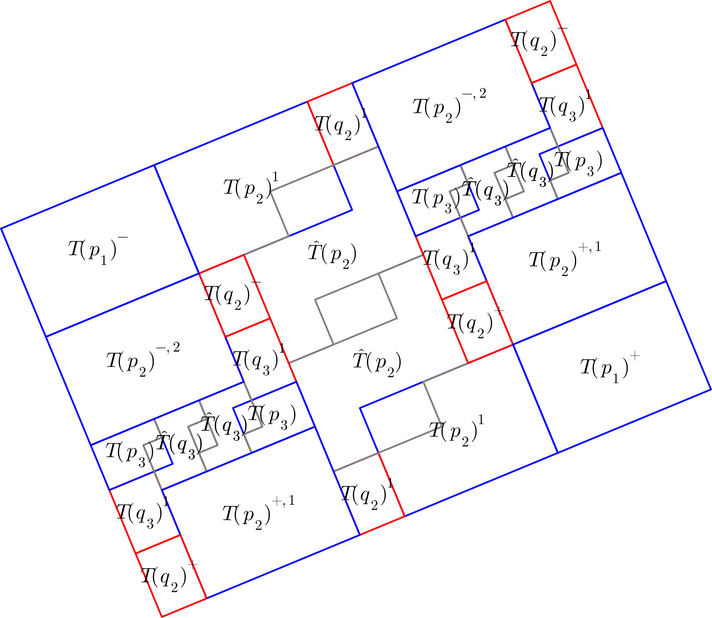}
	\caption{Decomposition and then double decomposition of corrected $T(p_0)$ in the even-first orientation. The first row is $w = *13$ and the second is $w = *23$}	\label{fig:std_double_decomp_correction}
\end{figure}

\begin{figure}
	\includegraphics[scale=0.3]{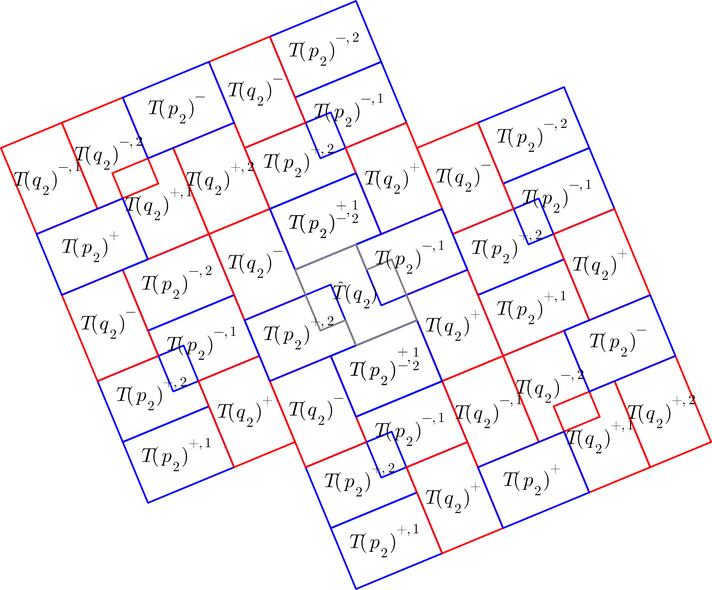} \quad
	\includegraphics[scale=0.29]{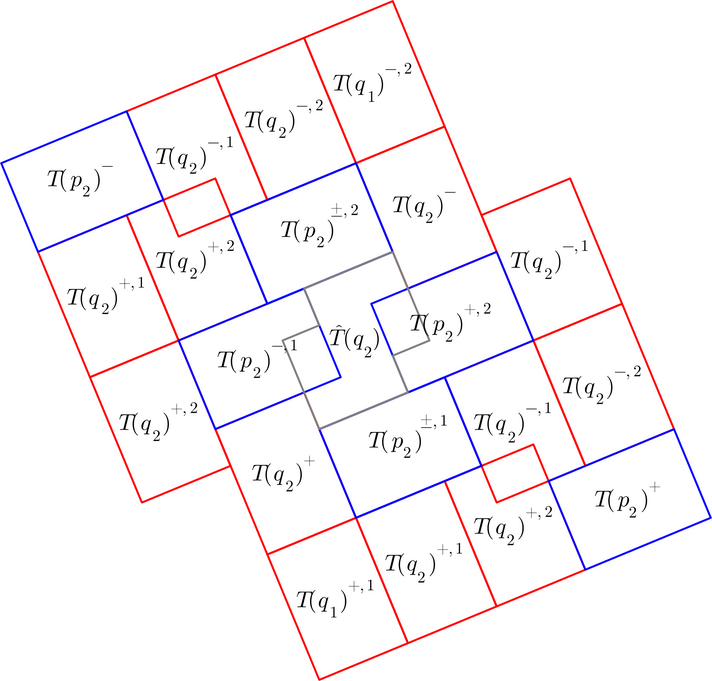}
	\caption{Double decomposition of $\hat{T}(q_0)$ in the even-first orientation for recursion words $w = *1$ and $w = *2$ respectively.}\label{fig:alt_double_decomp_no_correction}
\end{figure}
\begin{figure}
	\includegraphics[scale=0.29]{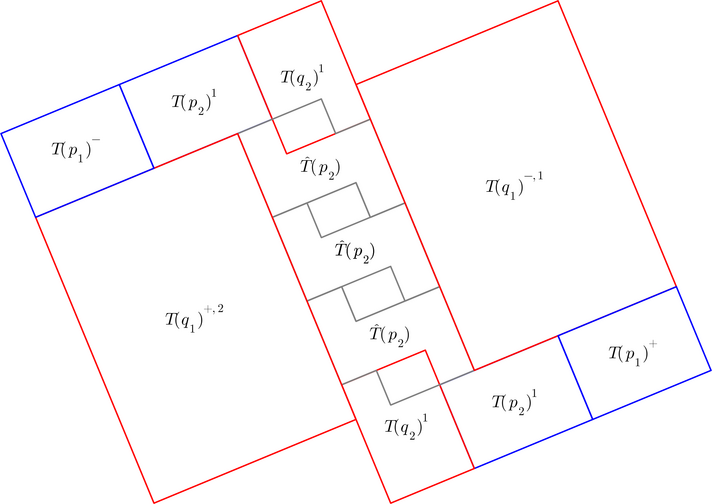} \quad
	\includegraphics[scale=0.29]{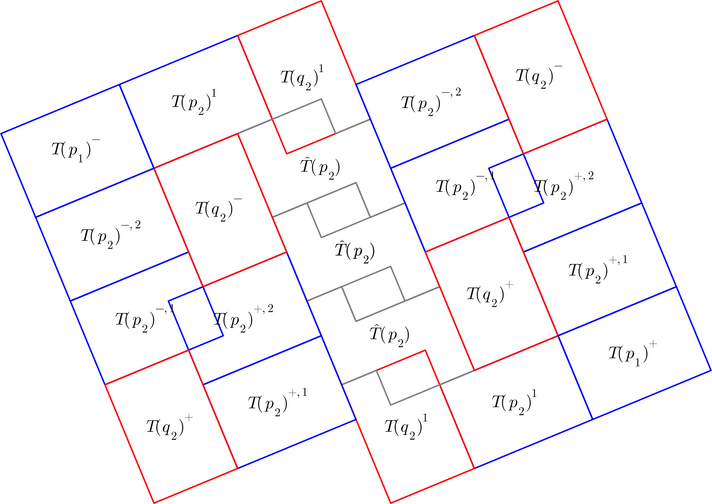} \\
	\includegraphics[scale=0.29]{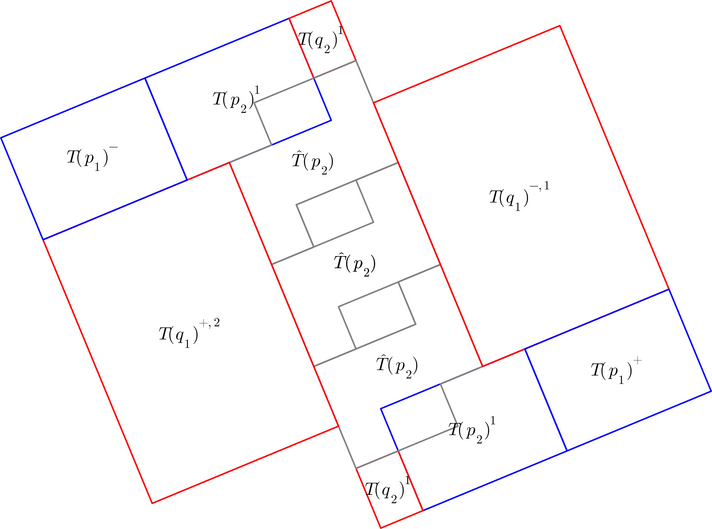} \quad
	\includegraphics[scale=0.29]{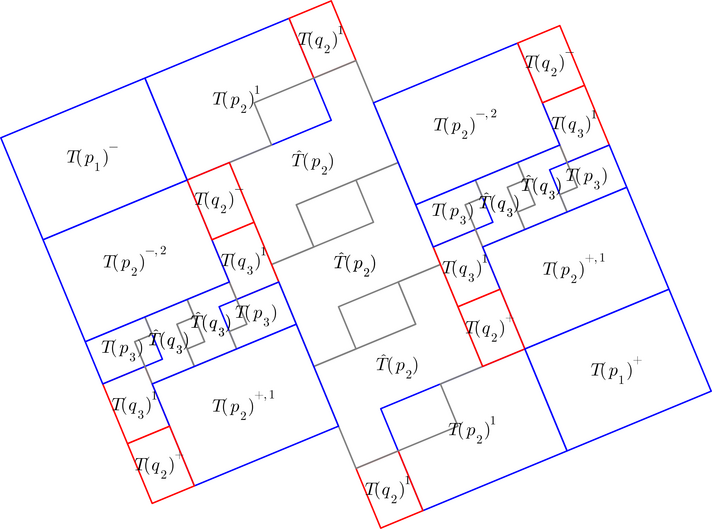}
	\caption{Decomposition and then double decomposition of corrected $\hat{T}(q_0)$ in the even-first orientation. The first row is $w = *13$ and the second is $w = *23$}\label{fig:alt_double_decomp_correction}
\end{figure}

We are now ready to fully define the recursion. Roughly, the full recursion proceeds by taking the weak recursion 
and filling in the interior. The difficulty occurs whenever there is overlap that is not on a boundary string. Whenever the doubled odometers
overlap on a common ancestor or do not overlap at all, the doubled odometers may be taken to be usual standard odometers. 
Otherwise, the overlap is {\it corrected} by a pair of complementary $\pm$ $L$-corrections.

\begin{definition}
	A weak standard (resp. alternate) tile odometer is a standard (resp. alternate) tile odometer
	if it is one of the base cases or in the decomposition \eqref{eq:std_odometer_decomposition}  (resp. \eqref{eq:alt_odometer_decomposition}), 
	weak subodometers are replaced by respective tile subodometers. Further, depending on the recursion word and parity, each weak doubled subodometer is either two standard tile subodometers or a standard tile subodometer and an odometer which respects an $L$-correction. 

	Specifically, let $n/d$ be a child in $\q_{w}$ and let $(p_1, q_1)$ denote the parent Farey pair. In the standard case, if $n/d$ is odd, then in \eqref{eq:std_odometer_decomposition}, each doubled term is replaced by
	\[
	d(o)_{q_1}^{\pm} \to
	\begin{cases}
	o_{q_1}^{\pm,1} \cup o_{q_1}^{\pm,2} &\mbox{ if $w = *\mbox{\{1 or 2\}}$} \\
	o_{q_1}^{\pm} \cup o^{L,\pm}_{q_1} &\mbox{ if $w = *3$}
	\end{cases}
	\]
	where $o_{q_1, \pm i}$ are standard tile odometers for $T(q_1)^{\pm, i}$. In the even-first orientation, 
	$o^{\pm}(q_1)$ is a tile odometer for $T(q_1)^{+,1}$ or $T(q_1)^{-,2}$ and $o^{L,\pm}_{q_1}$ is a partial odometer with decomposition
	that respects the $L$-corrections for $T(q_1)^{-,1}$ or $T(q_1)^{+,2}$. In the odd-first orientation, 1 and 2 are flipped.
	If $n/d$ is even, the decomposition is defined by rotating the decomposition for $\mathcal{R}(n/d)$. 
	
	Similarly, in the alternate case, if $n/d$ is even, then each weak shifted doubled term is replaced in the exact same way 
	as the standard odd case except the $L$ correction is taken to be the elongated $L$ correction. Rotate to complete the definition.

\end{definition}
Before proving existence of tile odometers, we prove that $L$-corrected odometers exist, assuming existence up to a
certain depth. An important tool in the remaining proofs will be the {\it double decomposition}, a decomposition of all or select subtiles
in the standard or alternate decomposition. In the figures, grandparent subtiles are indicated with a 2 subscript. 

\begin{lemma} \label{lemma:l_correct_existence}
	Suppose tile odometers exist for all $n/d \in \mathcal{T}_m$, for $m \leq m_0$, some $m_0 \geq 1$, 
	then partial odometers which respect the $L$-correction $o^{L,\pm}_{n/d}$ exist for all $n/d \in \mathcal{T}_{m_0+1}$ and respect
	the appropriate zero-one boundary strings from Lemma \ref{lemma:l_correc_properties}.
\end{lemma} 

\begin{proof}
	We show that the $L$-corrected odometer, $o^{L,\pm}_{n/d}$ exists by showing each pair of overlapping subodometers in the $L$-correction
	is compatible. The decomposition of $o^{L,\pm}_{n/d}$ consists of lattice adjacent tile odometers and one of two possible new types of intersection seen in Figure \ref{fig:l_corrections_overlaps}. 
	Since we have shown lattice adjacent odometers to be compatible (and that there are no gaps between lattice adjacent odometers), 
	it suffices deal with the new intersection. This can be dealt with by the double decomposition see Figure \ref{fig:l_corrections_dd}. In particular, in the double decomposition, every pair 
	of interfaces between the triple are part of or are a complete stacked zero-one boundary string. 
	In case the triple comes from the base cases, compatibility is a consequence of Proposition \ref{prop:base_cases}. 
	
	To see that  $o^{L,\pm}_{n/d}$ coincides with the standard odometer on the appropriate zero-one boundary strings, we induct on $k \geq 0$ 
	from Definitions \eqref{def:even_l_correction} or \eqref{def:odd_l_correction} as in the proof of Lemma \ref{lemma:l_correc_properties}.   
	%A straightforward induction on $k$ together with a double decomposition (from Definitions \eqref{def:even_l_correction} or %\eqref{def:odd_l_correction}) verifies that the appropriate zero-one boundary strings match that of the standard odometer.
\end{proof}

We finally prove existence of tile odometers. Figures \ref{fig:std_double_decomp_no_correction}, \ref{fig:std_double_decomp_correction},\ref{fig:alt_double_decomp_no_correction},\ref{fig:alt_double_decomp_correction}
will be a visual aid throughout the proof. 
\begin{lemma} \label{lemma:tile-odometer-existence}
	Tile odometers exist. 
\end{lemma}

\begin{proof}
	We proceed by induction. Start by using Lemma \ref{lemma:l_correct_existence} to see that tile odometers are indeed weak tile odometers.
	It remains to verify the internal odometer decomposition is well-defined. That is, we check that 
	each pair of overlapping subodometers is pairwise consistent. As compatibility is affine invariant, 
	and the decompositions are, up to affine factors, rotationally invariant, it suffices to show compatibility
	for either $n/d$ or $\mathcal{R}(n/d)$.  Let $\q_w$ denote the quadruple for which $n/d$ is a child.
	We split the remainder of proof into cases. In each case, we use the double decomposition displayed in the indicated Figure
	to show that the subodometers either overlap on stacked boundary strings or on grandparent tiles. Also, assume the even-first orientation, otherwise flip the subsequent arguments.

{\it Case 1: Odd standard odometer: $w = *\mbox{\{1 or 2\}}$; Figure \ref{fig:std_double_decomp_no_correction}} \\
The interfaces between every non-overlapping tile consist of stacked zero-one boundary strings therefore
those odometers are compatible. If $w = *1$, the tiles $T(q_1)^{+,2}$ and $T(q_1)^{-,1}$ overlap, by the double decomposition, 
exactly on $T(p_2)^{+, 1/2}$, and therefore those subodometers are compatible. Similarly, if $w=*2$, $T(p_1)^-$ and $T(p_1)^+$
overlap on $T(p_2)^{+, 1/2}$ and so those subodometers are compatible. 

{\it Case 2: Odd standard odometer: $w = *3$; Figure \ref{fig:std_double_decomp_correction}} \\
In this case, due to the $L$-correction, the only overlaps are that of lattice adjacent $\hat{T}(p_2)$ (subsubtiles). The rest of the interfaces are part of or are stacked zero-one boundary strings.  By previous arguments, 
the $(T(q_1)^{+,1}, T(p_1)^-)$ and $T(q_1)^{-,2}$ interfaces are parts of stacked boundary strings. 
The interface between $T(p_1)^{\pm}$ and $T^{L, \pm}(q_1)$ corresponds to full stacked vertical boundary strings -- the interfaces are that of lattice adjacent
standard tiles $T(p_1)$. 

 It remains to analyze the interface between $T^{L, \pm}(q_1)$ and $T(q_1){+,1}, T(q_1)^{-,2}$. By Lemma \ref{lemma:l_correct_existence}
 $(T^{L, +}(q_1), T(q_1)^{+,1})$ each intersect on a $w_v(q_1)$ stacked boundary string. By induction on 
 $k$ in the recursion word for the $L$ correction, similar to the proof of Lemma \ref{lemma:l_correct_existence}, 
 the interface between $(T^{L,+}(q_1), T(q_1)^{-,2})$ and $(T^{L,-}(q_1), T(q_1)^{+,1})$ consists of a sequence
 of lattice adjacent standard odometers for $T(p_1)^+$ followed by a $(q_2, p_2)$ interface which is an almost palindrome by
 Lemma \ref{lemma:almost_palindrome}.

{\it Case 3: Even alternate odometer: $w = *\mbox{\{1 or 2\}}$; Figure \ref{fig:alt_double_decomp_no_correction}} \\
The overlap argument is similar to Case 1. If $w = *1$,  $T(q_1)^{+,1}$ and $T(q_1)^{-,2}$ overlap $\hat{T}(p_1)$ on 
$T(p_2)^{\pm, 1/2}$. The interfaces between  $(T(p_1)^{+}, T(q_1)^{-,1})$, $(T(p_1)^{-}, T(q_1)^{+,2})$, $(T(p_1)^{+}, T(q_1)^{+,1})$, $(T(p_1)^{-}, T(q_1)^{-,2})$ are part of stacked zero-one boundary strings by the same argument as given previously in the weak standard case. 

%The interface between $(T(p_1)^{+}, T(q_1)^{-,1})$ and  $(T(p_1)^{-}, T(q_1)^{+,2})$
%are part of stacked horizontal zero-one-boundary strings. The interface between $(T(p_1)^{+}, T(q_1)^{+,1})$ and  $(T(p_1)^{-}, T(q_1)^{-,2})$
%are part of stacked vertical zero-one boundary strings.

The new interfaces are 
\[
(T(q_1)^{\pm,1}, T(q_1)^{\pm,2}, \hat{T}(p_1))
\]
and 
\[
(T(p_1), \hat{T}(p_1)).
\]
In the latter case, by the double decomposition, $(T(p_1), \hat{T}(p_1))$ intersect on the boundary of lattice adjacent standard tile odometers for $T(q_2)$. 
In the former case, the interface between $(T(q_1)^{+,1}, T(q_1)^{+,2}, \hat{T}(p_1))$ is part of a stacked vertical zero-one boundary string. Indeed, augment the vertical boundary string of $T(q_1)^{+,2}$ by concatenating $w_v(p_1)$ at the start. Then, since $\hat{T}(p_1)$ 
is a $w_v(p_1)$ pseudo-square, the augmented interface is exactly $w_v(p_1) \cup w_v(q_1)$ and $\rev(w_v(q_1)) \cup \rev(w_v(p_1))$. 
This is exactly stacked boundary string for $p_0$ and therefore, the odometers which intersect on it are compatible. 
The argument for $w = *2$ is identical.

{\it Case 4: Even alternate odometer: $w = *3$; Figure \ref{fig:alt_double_decomp_correction}} \\
This is almost identical to Case 2. The only difference is in the overlaps $(T^{L, +}(q_1), T(q_1)^{+,2})$
and $(T^{L,-}(q_1), T^(q_1)^{-,1})$. In this case, we need to augment the vertical boundary strings 
as in Case 3 for $T(q_1)$. Once augmented, those interfaces then become exactly stacked vertical boundary string for $p_0$.
\end{proof}

\section{Global odometers} \label{sec:global_odometers}

We now observe that both standard and alternate tile odometers can be extended to global odometers
with the correct growth. 

\begin{lemma} \label{lemma:global_odometers}
	For every reduced fraction $0 < n/d < 1$, there are two functions, $g_{n/d}, \hat{g}_{n/d}$ on $\Z^2$ whose
	restriction to a standard or alternate tile are alternate and standard tile odometers
	for which the periodicity condition \eqref{eq:periodicity} holds and for which 
	\begin{equation} \label{eq:periodicity_growth}
		x \to f(x) - \frac{1}{2} x^T M(n,d) x - b^T x 
	\end{equation}
	is $L'(n/d)$-periodic for some $b \in \R^2$, for each $f \in \{g_{n/d},\hat{g}_{n/d}\}$.
\end{lemma}
\begin{proof}
	Given that we have proved standard and alternate tile odometers which are lattice adjacent
	are compatible, the proof is identical to Lemma 10.1 in \cite{levine2017apollonian}. 	
\end{proof}

It remains to check that the functions which we have constructed are recurrent, which we do by induction. 
We first check that the constructed functions are indeed integer superharmonic. 

\begin{lemma}
		For each $0 < n/d < 1$, $g_{n/d}$ and $\hat{g}_{n/d}$ are integer superharmonic. 
\end{lemma}
\begin{proof}
	Let $s \in  \{ \Delta g_{n/d}+1, \Delta \hat{g}_{n/d}+1\}$ and proceed by induction. By Lemma \ref{lemma:rot_recurrence_invariance} and Proposition \ref{prop:base_cases}
	it suffices to take $n/d$ odd and suppose $T(n/d)$ and $\hat{T}(n/d)$ have double decompositions.

	{\it Case 1: Standard odometer.} 
	
	If $x$ lies in the intersection of two neighboring tiles, $T(n/d) \cap T(n'/d')$ then, by Lemma \ref{lemma:standard_tile_properties}, it must be contained 
	within a zero-one stacked boundary string. In this case $s(x) = 1$ by the explicit formulae in Section \ref{sec:zero-one_strings}. 
	Otherwise, $x$ is in the interior of $T(n/d)$. If $x$ is in the interior of a subtile in the double decomposition, we conclude by induction. 
	Otherwise, by considering the cases in Figures  \ref{fig:std_double_decomp_no_correction}
	and \ref{fig:std_double_decomp_correction} as in the proof of Lemma \ref{lemma:tile-odometer-existence},
	$x$ either lies within a zero-one stacked boundary string or in the interior of an ancestor tile. In the latter case we can use induction and in the former $s(x) = 1$. 
	
	{\it Case 2: Alternate odometer.}
	
	If $x$ is in the interior of $\hat{T}(n/d)$, the argument is similar to Step 1;  as in Lemma \ref{lemma:tile-odometer-existence} check the cases in Figures \ref{fig:alt_double_decomp_correction} and Figure \ref{fig:alt_double_decomp_no_correction} to see that $x$ must be on a stacked zero-one boundary string or within the interior of a subtile. If $x \in \partial^{-}\hat{T}(n/d)$, then, by Lemma \ref{lemma:alternate_tile_properties}, $x$ is within a stacked zero-one boundary string or in the interior of a subtile. 
\end{proof}

We conclude by checking recurrence.  

\begin{lemma}
	For each $0 < n/d < 1$, $g_{n/d}$ and $\hat{g}_{n/d}$ are recurrent. 
\end{lemma}
\begin{proof}
	Let $s \in  \{ \Delta g_{n/d}+1, \Delta \hat{g}_{n/d}+1\}$ and suppose the claim is true for all Farey quadruples $\q_{|w|}$ with 
	$|w| \leq n$. By Lemma \ref{lemma:rot_recurrence_invariance} and Proposition \ref{prop:base_cases}, it suffices to check $w = 3^k w'$, for $k \geq 0$ and $|w'| \geq 1$, where, if $k = 0$, $w'[1] = 1$. 
	
	Therefore, $T(n/d)$ and $\hat{T}(n/d)$ have double decompositions and each of the interfaces in the tiling of $T(n/d)$ consist of
	stacked $q-p$ boundary strings where $(q,p)$ depend on the first letter of $w'$. 
	If $w'[1] = 1$, $(q,p)$ is the Farey child of $\q_{3^k}$. Otherwise, $p$ is the odd child in $\q_{3^k}$
	and $q$ is the even child in $\q_{3^{k-1}}$.  In either case, the explicit forms of the odometers 
	and their Laplacians are given in Section \ref{sec:degenerate_cases}. 
	
	Let $T \in \{T(n/d), \hat{T}(n/d)\}$ and write $T^i, s^i$ for the subtiles and ancestor Laplacians in the double decomposition. 
		
	The inductive proof starts as in the proof of Lemma \ref{lemma:staircase_recurrent}: suppose for sake of contradiction there is an induced subgraph of the $F$-lattice, $H$, which is forbidden for $s$. 
	Let $c^{0} = -\infty$ and for $j \geq 1$, let 
	\begin{equation}
	\begin{aligned}
	c^{j} &= \min\{ x_1 > c^{j-1} : x \in H\}  \\ 
	V^{j} &= \{ x \in H : x_1 = c^j\}.
	\end{aligned}
	\end{equation}
	In words, sets of possibly disjoint vertical lines enumerated from left to right. Since $H$ is forbidden, it is nonempty, hence
	$V^{1}$ exists.

	We prove the following inductive hypotheses by induction on $j \geq 1$ for all translations of $T$ and all ancestor tiles $\mathrm{T}_a := \{ T(p), T(q), \hat{T}(p), \hat{T}(q)\}$. 
	
	\begin{enumerate}
		\item The hypotheses listed in Lemmas \ref{lemma:staircase_recurrent}, \ref{lemma:alt-staircase-recurrent}, \ref{lemma:doubled-staircase-recurrent}, \ref{lemma:doubled-alt-staircase-recurrent}
		for $T(p), \hat{T}(p), T(q)$ and $\hat{T}(q)$ respectively. 
		\item If for some subtile $T^i$, $V^j \cap T^i \neq \emptyset$ then $V^j \cap \partial^{-}T^i = \emptyset$. 
	\end{enumerate}
	In fact, we suppose, by induction, that the inductive hypotheses are satisfied for every parent.  
	Indeed, the cases $\q_{\tilde{w}}$ for $\tilde{w} = 3^k*\{1 \mbox{ or } 2^{k'}\}$ for $k \geq 0, k' \geq 1$ may be checked directly using the explicit formulae in Section \ref{sec:degenerate_cases}
	following the outline in Lemma \ref{lemma:staircase_recurrent}. Thus, we may suppose that the boundary of each subtile $T^i$, including grandparent subtiles, 
	consists of a $q-p$ boundary string and that the inductive hypotheses hold for each subtile. 

	Hypothesis (2) implies the existence of a forbidden subconfiguration strictly contained in some $T^i$, contradicting inductive recurrence. Thus it remains to verify the hypotheses. 
	
	{\it Proof of (1) and (2).} \\
	Suppose $y \in V^j \cap T'$ for some $T' \in \mathrm{T}_a$. By construction, $y \in T' \cap T^i$ for some subtile $T^i$ of $T$ in the double decomposition. 
	If $V^j \cap \partial^{-}T^i = \emptyset$, then we may conclude by induction. 
	
	Otherwise, if there is $y \in V^j \cap \partial^{-}T^i$, then, $y$ must be contained in a stacked $q-p$ boundary string. However, the fixed interfaces in Lemma \ref{lemma:almost_palindrome_offset} allow 
	for the arguments of Lemmas \ref{lemma:staircase_recurrent} and Lemma \ref{lemma:doubled-staircase-recurrent} to be repeated, resulting in a contradiction. 
\end{proof}

	\bibliography{farey_recursion}

	\appendix
	
	\section{Table of Odometer Patterns}
	
The table below displays a Farey quadruple $(p_0,q_0, p_1,q_1) = (\mathcal{C}(p_1, q_1), p_1,q_1)$ and the Laplacian of the odd child's standard and alternate tile odometers. We only draw the Laplacian of $p_0$ since the Laplacian of any odd $(\frac{n}{d})$ is the rotated Laplacian of even $(\frac{d-n}{d+n})$. All quadruples with $14 \leq \det(L'(p_0)) \leq 1000$ are displayed. 

\setlength{\tabcolsep}{2pt}
\begin{longtable}[c]{rcc}
	\hline
	$(p_0,q_0,p_1,q_1)$  \hspace{1 em} & standard tile odometer \hspace{1 em} &  alternate tile odometer
	\ \\
	\hline
	\input{app_small.tex}

\end{longtable}

\end{document}

%% file: defs.tex
\newcommand{\R}{\mathbb{R}}
\newcommand{\C}{\mathbb{C}}
\newcommand{\Z}{\mathbb{Z}}

\newcommand{\Q}{\mathbb{Q}}
\newcommand{\I}{\mathbf{i}}  %bounding box

\newcommand{\q}{\mathbf{q}} % farey quadruple

\newcommand{\rev}{\mathop{\bf rev}}

\newcommand{\eg}{{\it e.g.}}
\newcommand{\ie}{{\it i.e.}}

\newtheorem{theorem}{Theorem}[section]
\newtheorem{prop}{Proposition}[section]

\newtheorem{lemma}[theorem]{Lemma}
\newtheorem{remark}{Remark}
\newtheorem{definition}{Definition}

%% file: tikz/flattice_2.tikz
\begin{tikzpicture}
\fill[black] (-2,-2) circle (0.05 cm);
\draw[thick,->] (-2.925,-2) -- (-2.075,-2);
\draw[thick,->] (-1.075,-2) -- (-1.925,-2);
\fill[black] (-2,-1) circle (0.05 cm);
\draw[thick,->] (-2,-1.925) -- (-2,-1.075);
\draw[thick,->] (-2,-0.075) -- (-2,-0.925);
\fill[black] (-2,0) circle (0.05 cm);
\draw[thick,->] (-2.925,0) -- (-2.075,0);
\draw[thick,->] (-1.075,0) -- (-1.925,0);
\fill[black] (-2,1) circle (0.05 cm);
\draw[thick,->] (-2,0.075) -- (-2,0.925);
\draw[thick,->] (-2,1.925) -- (-2,1.075);
\fill[black] (-2,2) circle (0.05 cm);
\draw[thick,->] (-2.925,2) -- (-2.075,2);
\draw[thick,->] (-1.075,2) -- (-1.925,2);
\fill[black] (-1,-2) circle (0.05 cm);
\draw[thick,->] (-1,-2.925) -- (-1,-2.075);
\draw[thick,->] (-1,-1.075) -- (-1,-1.925);
\fill[black] (-1,-1) circle (0.05 cm);
\draw[thick,->] (-1.925,-1) -- (-1.075,-1);
\draw[thick,->] (-0.075,-1) -- (-0.925,-1);
\fill[black] (-1,0) circle (0.05 cm);
\draw[thick,->] (-1,-0.925) -- (-1,-0.075);
\draw[thick,->] (-1,0.925) -- (-1,0.075);
\fill[black] (-1,1) circle (0.05 cm);
\draw[thick,->] (-1.925,1) -- (-1.075,1);
\draw[thick,->] (-0.075,1) -- (-0.925,1);
\fill[black] (-1,2) circle (0.05 cm);
\draw[thick,->] (-1,1.075) -- (-1,1.925);
\draw[thick,->] (-1,2.925) -- (-1,2.075);
\fill[black] (0,-2) circle (0.05 cm);
\draw[thick,->] (-0.925,-2) -- (-0.075,-2);
\draw[thick,->] (0.925,-2) -- (0.075,-2);
\fill[black] (0,-1) circle (0.05 cm);
\draw[thick,->] (0,-1.925) -- (0,-1.075);
\draw[thick,->] (0,-0.075) -- (0,-0.925);
\fill[black] (0,0) circle (0.05 cm);
\draw[thick,->] (-0.925,0) -- (-0.075,0);
\draw[thick,->] (0.925,0) -- (0.075,0);
\fill[black] (0,1) circle (0.05 cm);
\draw[thick,->] (0,0.075) -- (0,0.925);
\draw[thick,->] (0,1.925) -- (0,1.075);
\fill[black] (0,2) circle (0.05 cm);
\draw[thick,->] (-0.925,2) -- (-0.075,2);
\draw[thick,->] (0.925,2) -- (0.075,2);
\fill[black] (1,-2) circle (0.05 cm);
\draw[thick,->] (1,-2.925) -- (1,-2.075);
\draw[thick,->] (1,-1.075) -- (1,-1.925);
\fill[black] (1,-1) circle (0.05 cm);
\draw[thick,->] (0.075,-1) -- (0.925,-1);
\draw[thick,->] (1.925,-1) -- (1.075,-1);
\fill[black] (1,0) circle (0.05 cm);
\draw[thick,->] (1,-0.925) -- (1,-0.075);
\draw[thick,->] (1,0.925) -- (1,0.075);
\fill[black] (1,1) circle (0.05 cm);
\draw[thick,->] (0.075,1) -- (0.925,1);
\draw[thick,->] (1.925,1) -- (1.075,1);
\fill[black] (1,2) circle (0.05 cm);
\draw[thick,->] (1,1.075) -- (1,1.925);
\draw[thick,->] (1,2.925) -- (1,2.075);
\fill[black] (2,-2) circle (0.05 cm);
\draw[thick,->] (1.075,-2) -- (1.925,-2);
\draw[thick,->] (2.925,-2) -- (2.075,-2);
\fill[black] (2,-1) circle (0.05 cm);
\draw[thick,->] (2,-1.925) -- (2,-1.075);
\draw[thick,->] (2,-0.075) -- (2,-0.925);
\fill[black] (2,0) circle (0.05 cm);
\draw[thick,->] (1.075,0) -- (1.925,0);
\draw[thick,->] (2.925,0) -- (2.075,0);
\fill[black] (2,1) circle (0.05 cm);
\draw[thick,->] (2,0.075) -- (2,0.925);
\draw[thick,->] (2,1.925) -- (2,1.075);
\fill[black] (2,2) circle (0.05 cm);
\draw[thick,->] (1.075,2) -- (1.925,2);
\draw[thick,->] (2.925,2) -- (2.075,2);
\end{tikzpicture}

%% file: tikz/flattice_packing.tikz
\begin{tikzpicture}[fill=white,fill opacity=.75]
\draw (0,0) circle (1);
\draw (1,1) circle (1);
\draw (-1,1) circle (1);
\draw (-1,-1) circle (1);
\draw (1,-1) circle (1);

\draw (2,0) circle (1);
\draw (0,2) circle (1);
\draw (-2,0) circle (1);
\draw (0,-2) circle (1);
\end{tikzpicture}

%% file: tikz/flattice_3.tikz
\begin{tikzpicture}[scale=0.5]
\fill[black] (-2,-2) circle (0.05 cm);
\draw[thick,->] (-2,-2.925) -- (-2,-2.075);
\draw[thick,->] (-2,-1.075) -- (-2,-1.925);
\fill[black] (-2,-1) circle (0.05 cm);
\draw[thick,->] (-2.925,-1) -- (-2.075,-1);
\draw[thick,->] (-1.075,-1) -- (-1.925,-1);
\fill[black] (-2,0) circle (0.05 cm);
\draw[thick,->] (-2,-0.925) -- (-2,-0.075);
\draw[thick,->] (-2,0.925) -- (-2,0.075);
\fill[black] (-2,1) circle (0.05 cm);
\draw[thick,->] (-2,0.075) -- (-2,0.925);
\draw[thick,->] (-2,1.925) -- (-2,1.075);
\fill[black] (-2,2) circle (0.05 cm);
\draw[thick,->] (-2.925,2) -- (-2.075,2);
\draw[thick,->] (-1.075,2) -- (-1.925,2);
\fill[black] (-1,-2) circle (0.05 cm);
\draw[thick,->] (-1.925,-2) -- (-1.075,-2);
\draw[thick,->] (-0.075,-2) -- (-0.925,-2);
\fill[black] (-1,-1) circle (0.05 cm);
\draw[thick,->] (-1,-1.925) -- (-1,-1.075);
\draw[thick,->] (-1,-0.075) -- (-1,-0.925);
\fill[black] (-1,0) circle (0.05 cm);
\draw[thick,->] (-1,-0.925) -- (-1,-0.075);
\draw[thick,->] (-1,0.925) -- (-1,0.075);
\fill[black] (-1,1) circle (0.05 cm);
\draw[thick,->] (-1.925,1) -- (-1.075,1);
\draw[thick,->] (-0.075,1) -- (-0.925,1);
\fill[black] (-1,2) circle (0.05 cm);
\draw[thick,->] (-1,1.075) -- (-1,1.925);
\draw[thick,->] (-1,2.925) -- (-1,2.075);
\fill[black] (0,-2) circle (0.05 cm);
\draw[thick,->] (0,-2.925) -- (0,-2.075);
\draw[thick,->] (0,-1.075) -- (0,-1.925);
\fill[black] (0,-1) circle (0.05 cm);
\draw[thick,->] (0,-1.925) -- (0,-1.075);
\draw[thick,->] (0,-0.075) -- (0,-0.925);
\fill[black] (0,0) circle (0.05 cm);
\draw[thick,->] (-0.925,0) -- (-0.075,0);
\draw[thick,->] (0.925,0) -- (0.075,0);
\fill[black] (0,1) circle (0.05 cm);
\draw[thick,->] (0,0.075) -- (0,0.925);
\draw[thick,->] (0,1.925) -- (0,1.075);
\fill[black] (0,2) circle (0.05 cm);
\draw[thick,->] (0,1.075) -- (0,1.925);
\draw[thick,->] (0,2.925) -- (0,2.075);
\fill[black] (1,-2) circle (0.05 cm);
\draw[thick,->] (1,-2.925) -- (1,-2.075);
\draw[thick,->] (1,-1.075) -- (1,-1.925);
\fill[black] (1,-1) circle (0.05 cm);
\draw[thick,->] (0.075,-1) -- (0.925,-1);
\draw[thick,->] (1.925,-1) -- (1.075,-1);
\fill[black] (1,0) circle (0.05 cm);
\draw[thick,->] (1,-0.925) -- (1,-0.075);
\draw[thick,->] (1,0.925) -- (1,0.075);
\fill[black] (1,1) circle (0.05 cm);
\draw[thick,->] (1,0.075) -- (1,0.925);
\draw[thick,->] (1,1.925) -- (1,1.075);
\fill[black] (1,2) circle (0.05 cm);
\draw[thick,->] (0.075,2) -- (0.925,2);
\draw[thick,->] (1.925,2) -- (1.075,2);
\fill[black] (2,-2) circle (0.05 cm);
\draw[thick,->] (1.075,-2) -- (1.925,-2);
\draw[thick,->] (2.925,-2) -- (2.075,-2);
\fill[black] (2,-1) circle (0.05 cm);
\draw[thick,->] (2,-1.925) -- (2,-1.075);
\draw[thick,->] (2,-0.075) -- (2,-0.925);
\fill[black] (2,0) circle (0.05 cm);
\draw[thick,->] (2,-0.925) -- (2,-0.075);
\draw[thick,->] (2,0.925) -- (2,0.075);
\fill[black] (2,1) circle (0.05 cm);
\draw[thick,->] (1.075,1) -- (1.925,1);
\draw[thick,->] (2.925,1) -- (2.075,1);
\fill[black] (2,2) circle (0.05 cm);
\draw[thick,->] (2,1.075) -- (2,1.925);
\draw[thick,->] (2,2.925) -- (2,2.075);
\end{tikzpicture}

%% file: tikz/flattice_4.tikz
\begin{tikzpicture}[scale=0.5]
\fill[black] (-2,-2) circle (0.05 cm);
\draw[thick,->] (-2.925,-2) -- (-2.075,-2);
\draw[thick,->] (-1.075,-2) -- (-1.925,-2);
\fill[black] (-2,-1) circle (0.05 cm);
\draw[thick,->] (-2,-1.925) -- (-2,-1.075);
\draw[thick,->] (-2,-0.075) -- (-2,-0.925);
\fill[black] (-2,0) circle (0.05 cm);
\draw[thick,->] (-2,-0.925) -- (-2,-0.075);
\draw[thick,->] (-2,0.925) -- (-2,0.075);
\fill[black] (-2,1) circle (0.05 cm);
\draw[thick,->] (-2,0.075) -- (-2,0.925);
\draw[thick,->] (-2,1.925) -- (-2,1.075);
\fill[black] (-2,2) circle (0.05 cm);
\draw[thick,->] (-2.925,2) -- (-2.075,2);
\draw[thick,->] (-1.075,2) -- (-1.925,2);
\fill[black] (-1,-2) circle (0.05 cm);
\draw[thick,->] (-1,-2.925) -- (-1,-2.075);
\draw[thick,->] (-1,-1.075) -- (-1,-1.925);
\fill[black] (-1,-1) circle (0.05 cm);
\draw[thick,->] (-1,-1.925) -- (-1,-1.075);
\draw[thick,->] (-1,-0.075) -- (-1,-0.925);
\fill[black] (-1,0) circle (0.05 cm);
\draw[thick,->] (-1,-0.925) -- (-1,-0.075);
\draw[thick,->] (-1,0.925) -- (-1,0.075);
\fill[black] (-1,1) circle (0.05 cm);
\draw[thick,->] (-1.925,1) -- (-1.075,1);
\draw[thick,->] (-0.075,1) -- (-0.925,1);
\fill[black] (-1,2) circle (0.05 cm);
\draw[thick,->] (-1,1.075) -- (-1,1.925);
\draw[thick,->] (-1,2.925) -- (-1,2.075);
\fill[black] (0,-2) circle (0.05 cm);
\draw[thick,->] (0,-2.925) -- (0,-2.075);
\draw[thick,->] (0,-1.075) -- (0,-1.925);
\fill[black] (0,-1) circle (0.05 cm);
\draw[thick,->] (0,-1.925) -- (0,-1.075);
\draw[thick,->] (0,-0.075) -- (0,-0.925);
\fill[black] (0,0) circle (0.05 cm);
\draw[thick,->] (-0.925,0) -- (-0.075,0);
\draw[thick,->] (0.925,0) -- (0.075,0);
\fill[black] (0,1) circle (0.05 cm);
\draw[thick,->] (0,0.075) -- (0,0.925);
\draw[thick,->] (0,1.925) -- (0,1.075);
\fill[black] (0,2) circle (0.05 cm);
\draw[thick,->] (0,1.075) -- (0,1.925);
\draw[thick,->] (0,2.925) -- (0,2.075);
\fill[black] (1,-2) circle (0.05 cm);
\draw[thick,->] (1,-2.925) -- (1,-2.075);
\draw[thick,->] (1,-1.075) -- (1,-1.925);
\fill[black] (1,-1) circle (0.05 cm);
\draw[thick,->] (0.075,-1) -- (0.925,-1);
\draw[thick,->] (1.925,-1) -- (1.075,-1);
\fill[black] (1,0) circle (0.05 cm);
\draw[thick,->] (1,-0.925) -- (1,-0.075);
\draw[thick,->] (1,0.925) -- (1,0.075);
\fill[black] (1,1) circle (0.05 cm);
\draw[thick,->] (1,0.075) -- (1,0.925);
\draw[thick,->] (1,1.925) -- (1,1.075);
\fill[black] (1,2) circle (0.05 cm);
\draw[thick,->] (1,1.075) -- (1,1.925);
\draw[thick,->] (1,2.925) -- (1,2.075);
\fill[black] (2,-2) circle (0.05 cm);
\draw[thick,->] (1.075,-2) -- (1.925,-2);
\draw[thick,->] (2.925,-2) -- (2.075,-2);
\fill[black] (2,-1) circle (0.05 cm);
\draw[thick,->] (2,-1.925) -- (2,-1.075);
\draw[thick,->] (2,-0.075) -- (2,-0.925);
\fill[black] (2,0) circle (0.05 cm);
\draw[thick,->] (2,-0.925) -- (2,-0.075);
\draw[thick,->] (2,0.925) -- (2,0.075);
\fill[black] (2,1) circle (0.05 cm);
\draw[thick,->] (2,0.075) -- (2,0.925);
\draw[thick,->] (2,1.925) -- (2,1.075);
\fill[black] (2,2) circle (0.05 cm);
\draw[thick,->] (1.075,2) -- (1.925,2);
\draw[thick,->] (2.925,2) -- (2.075,2);
\end{tikzpicture}

%% file: tikz/flattice_5.tikz
\begin{tikzpicture}[scale=0.5]
\fill[black] (-2,-2) circle (0.05 cm);
\draw[thick,->] (-2,-2.925) -- (-2,-2.075);
\draw[thick,->] (-2,-1.075) -- (-2,-1.925);
\fill[black] (-2,-1) circle (0.05 cm);
\draw[thick,->] (-2,-1.925) -- (-2,-1.075);
\draw[thick,->] (-2,-0.075) -- (-2,-0.925);
\fill[black] (-2,0) circle (0.05 cm);
\draw[thick,->] (-2,-0.925) -- (-2,-0.075);
\draw[thick,->] (-2,0.925) -- (-2,0.075);
\fill[black] (-2,1) circle (0.05 cm);
\draw[thick,->] (-2,0.075) -- (-2,0.925);
\draw[thick,->] (-2,1.925) -- (-2,1.075);
\fill[black] (-2,2) circle (0.05 cm);
\draw[thick,->] (-2.925,2) -- (-2.075,2);
\draw[thick,->] (-1.075,2) -- (-1.925,2);
\fill[black] (-1,-2) circle (0.05 cm);
\draw[thick,->] (-1,-2.925) -- (-1,-2.075);
\draw[thick,->] (-1,-1.075) -- (-1,-1.925);
\fill[black] (-1,-1) circle (0.05 cm);
\draw[thick,->] (-1,-1.925) -- (-1,-1.075);
\draw[thick,->] (-1,-0.075) -- (-1,-0.925);
\fill[black] (-1,0) circle (0.05 cm);
\draw[thick,->] (-1,-0.925) -- (-1,-0.075);
\draw[thick,->] (-1,0.925) -- (-1,0.075);
\fill[black] (-1,1) circle (0.05 cm);
\draw[thick,->] (-1.925,1) -- (-1.075,1);
\draw[thick,->] (-0.075,1) -- (-0.925,1);
\fill[black] (-1,2) circle (0.05 cm);
\draw[thick,->] (-1,1.075) -- (-1,1.925);
\draw[thick,->] (-1,2.925) -- (-1,2.075);
\fill[black] (0,-2) circle (0.05 cm);
\draw[thick,->] (0,-2.925) -- (0,-2.075);
\draw[thick,->] (0,-1.075) -- (0,-1.925);
\fill[black] (0,-1) circle (0.05 cm);
\draw[thick,->] (0,-1.925) -- (0,-1.075);
\draw[thick,->] (0,-0.075) -- (0,-0.925);
\fill[black] (0,0) circle (0.05 cm);
\draw[thick,->] (-0.925,0) -- (-0.075,0);
\draw[thick,->] (0.925,0) -- (0.075,0);
\fill[black] (0,1) circle (0.05 cm);
\draw[thick,->] (0,0.075) -- (0,0.925);
\draw[thick,->] (0,1.925) -- (0,1.075);
\fill[black] (0,2) circle (0.05 cm);
\draw[thick,->] (0,1.075) -- (0,1.925);
\draw[thick,->] (0,2.925) -- (0,2.075);
\fill[black] (1,-2) circle (0.05 cm);
\draw[thick,->] (1,-2.925) -- (1,-2.075);
\draw[thick,->] (1,-1.075) -- (1,-1.925);
\fill[black] (1,-1) circle (0.05 cm);
\draw[thick,->] (0.075,-1) -- (0.925,-1);
\draw[thick,->] (1.925,-1) -- (1.075,-1);
\fill[black] (1,0) circle (0.05 cm);
\draw[thick,->] (1,-0.925) -- (1,-0.075);
\draw[thick,->] (1,0.925) -- (1,0.075);
\fill[black] (1,1) circle (0.05 cm);
\draw[thick,->] (1,0.075) -- (1,0.925);
\draw[thick,->] (1,1.925) -- (1,1.075);
\fill[black] (1,2) circle (0.05 cm);
\draw[thick,->] (1,1.075) -- (1,1.925);
\draw[thick,->] (1,2.925) -- (1,2.075);
\fill[black] (2,-2) circle (0.05 cm);
\draw[thick,->] (1.075,-2) -- (1.925,-2);
\draw[thick,->] (2.925,-2) -- (2.075,-2);
\fill[black] (2,-1) circle (0.05 cm);
\draw[thick,->] (2,-1.925) -- (2,-1.075);
\draw[thick,->] (2,-0.075) -- (2,-0.925);
\fill[black] (2,0) circle (0.05 cm);
\draw[thick,->] (2,-0.925) -- (2,-0.075);
\draw[thick,->] (2,0.925) -- (2,0.075);
\fill[black] (2,1) circle (0.05 cm);
\draw[thick,->] (2,0.075) -- (2,0.925);
\draw[thick,->] (2,1.925) -- (2,1.075);
\fill[black] (2,2) circle (0.05 cm);
\draw[thick,->] (2,1.075) -- (2,1.925);
\draw[thick,->] (2,2.925) -- (2,2.075);
\end{tikzpicture}

%% file: app_small.tex
$(1/2, 1/3, 0/1, 1/1)$&\begin{tabular}{r}\includegraphics[scale=3.0]{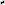}\end{tabular}&\begin{tabular}{r}\includegraphics[scale=3.0]{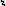}\end{tabular} \\ $(2/3, 3/5, 1/2, 1/1)$&\begin{tabular}{r}\includegraphics[scale=3.0]{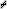}\end{tabular}&\begin{tabular}{r}\includegraphics[scale=3.0]{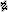}\end{tabular} \\ $(1/4, 1/5, 0/1, 1/3)$&\begin{tabular}{r}\includegraphics[scale=3.0]{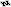}\end{tabular}&\begin{tabular}{r}\includegraphics[scale=3.0]{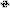}\end{tabular} \\ $(3/4, 5/7, 2/3, 1/1)$&\begin{tabular}{r}\includegraphics[scale=3.0]{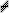}\end{tabular}&\begin{tabular}{r}\includegraphics[scale=3.0]{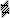}\end{tabular} \\ $(2/5, 3/7, 1/2, 1/3)$&\begin{tabular}{r}\includegraphics[scale=3.0]{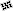}\end{tabular}&\begin{tabular}{r}\includegraphics[scale=3.0]{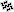}\end{tabular} \\ $(1/6, 1/7, 0/1, 1/5)$&\begin{tabular}{r}\includegraphics[scale=3.0]{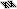}\end{tabular}&\begin{tabular}{r}\includegraphics[scale=3.0]{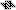}\end{tabular} \\ $(4/5, 7/9, 3/4, 1/1)$&\begin{tabular}{r}\includegraphics[scale=3.0]{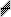}\end{tabular}&\begin{tabular}{r}\includegraphics[scale=3.0]{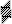}\end{tabular} \\ $(5/6, 9/11, 4/5, 1/1)$&\begin{tabular}{r}\includegraphics[scale=3.0]{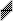}\end{tabular}&\begin{tabular}{r}\includegraphics[scale=3.0]{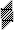}\end{tabular} \\ $(2/7, 3/11, 1/4, 1/3)$&\begin{tabular}{r}\includegraphics[scale=3.0]{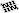}\end{tabular}&\begin{tabular}{r}\includegraphics[scale=3.0]{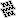}\end{tabular} \\ $(1/8, 1/9, 0/1, 1/7)$&\begin{tabular}{r}\includegraphics[scale=3.0]{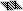}\end{tabular}&\begin{tabular}{r}\includegraphics[scale=3.0]{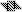}\end{tabular} \\ $(4/7, 5/9, 1/2, 3/5)$&\begin{tabular}{r}\includegraphics[scale=3.0]{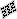}\end{tabular}&\begin{tabular}{r}\includegraphics[scale=3.0]{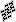}\end{tabular} \\ $(6/7, 11/13, 5/6, 1/1)$&\begin{tabular}{r}\includegraphics[scale=3.0]{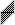}\end{tabular}&\begin{tabular}{r}\includegraphics[scale=3.0]{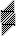}\end{tabular} \\ $(3/8, 5/13, 2/5, 1/3)$&\begin{tabular}{r}\includegraphics[scale=3.0]{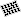}\end{tabular}&\begin{tabular}{r}\includegraphics[scale=3.0]{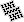}\end{tabular} \\ $(2/9, 3/13, 1/4, 1/5)$&\begin{tabular}{r}\includegraphics[scale=3.0]{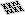}\end{tabular}&\begin{tabular}{r}\includegraphics[scale=3.0]{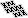}\end{tabular} \\ $(5/8, 7/11, 2/3, 3/5)$&\begin{tabular}{r}\includegraphics[scale=3.0]{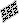}\end{tabular}&\begin{tabular}{r}\includegraphics[scale=3.0]{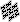}\end{tabular} \\ $(1/10, 1/11, 0/1, 1/9)$&\begin{tabular}{r}\includegraphics[scale=3.0]{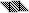}\end{tabular}&\begin{tabular}{r}\includegraphics[scale=3.0]{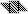}\end{tabular} \\ $(7/8, 13/15, 6/7, 1/1)$&\begin{tabular}{r}\includegraphics[scale=3.0]{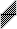}\end{tabular}&\begin{tabular}{r}\includegraphics[scale=3.0]{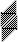}\end{tabular} \\ $(4/9, 5/11, 1/2, 3/7)$&\begin{tabular}{r}\includegraphics[scale=3.0]{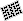}\end{tabular}&\begin{tabular}{r}\includegraphics[scale=3.0]{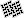}\end{tabular} \\ $(3/10, 5/17, 2/7, 1/3)$&\begin{tabular}{r}\includegraphics[scale=3.0]{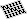}\end{tabular}&\begin{tabular}{r}\includegraphics[scale=3.0]{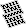}\end{tabular} \\ $(2/11, 3/17, 1/6, 1/5)$&\begin{tabular}{r}\includegraphics[scale=3.0]{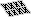}\end{tabular}&\begin{tabular}{r}\includegraphics[scale=3.0]{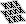}\end{tabular} \\ $(1/12, 1/13, 0/1, 1/11)$&\begin{tabular}{r}\includegraphics[scale=3.0]{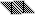}\end{tabular}&\begin{tabular}{r}\includegraphics[scale=3.0]{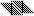}\end{tabular} \\ $(7/10, 9/13, 2/3, 5/7)$&\begin{tabular}{r}\includegraphics[scale=3.0]{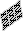}\end{tabular}&\begin{tabular}{r}\includegraphics[scale=3.0]{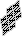}\end{tabular} \\ $(4/11, 7/19, 3/8, 1/3)$&\begin{tabular}{r}\includegraphics[scale=3.0]{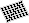}\end{tabular}&\begin{tabular}{r}\includegraphics[scale=3.0]{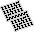}\end{tabular} \\ $(6/11, 7/13, 1/2, 5/9)$&\begin{tabular}{r}\includegraphics[scale=3.0]{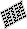}\end{tabular}&\begin{tabular}{r}\includegraphics[scale=3.0]{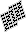}\end{tabular} \\ $(2/13, 3/19, 1/6, 1/7)$&\begin{tabular}{r}\includegraphics[scale=3.0]{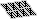}\end{tabular}&\begin{tabular}{r}\includegraphics[scale=3.0]{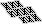}\end{tabular} \\ $(1/14, 1/15, 0/1, 1/13)$&\begin{tabular}{r}\includegraphics[scale=3.0]{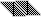}\end{tabular}&\begin{tabular}{r}\includegraphics[scale=3.0]{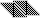}\end{tabular} \\ $(8/11, 11/15, 3/4, 5/7)$&\begin{tabular}{r}\includegraphics[scale=3.0]{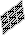}\end{tabular}&\begin{tabular}{r}\includegraphics[scale=3.0]{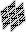}\end{tabular} \\ $(5/12, 7/17, 2/5, 3/7)$&\begin{tabular}{r}\includegraphics[scale=3.0]{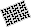}\end{tabular}&\begin{tabular}{r}\includegraphics[scale=3.0]{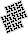}\end{tabular} \\ $(4/13, 7/23, 3/10, 1/3)$&\begin{tabular}{r}\includegraphics[scale=3.0]{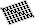}\end{tabular}&\begin{tabular}{r}\includegraphics[scale=3.0]{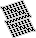}\end{tabular} \\ $(7/12, 11/19, 4/7, 3/5)$&\begin{tabular}{r}\includegraphics[scale=3.0]{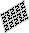}\end{tabular}&\begin{tabular}{r}\includegraphics[scale=3.0]{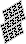}\end{tabular} \\ $(3/14, 5/23, 2/9, 1/5)$&\begin{tabular}{r}\includegraphics[scale=3.0]{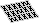}\end{tabular}&\begin{tabular}{r}\includegraphics[scale=3.0]{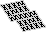}\end{tabular} \\ $(2/15, 3/23, 1/8, 1/7)$&\begin{tabular}{r}\includegraphics[scale=3.0]{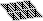}\end{tabular}&\begin{tabular}{r}\includegraphics[scale=3.0]{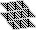}\end{tabular} \\ $(6/13, 7/15, 1/2, 5/11)$&\begin{tabular}{r}\includegraphics[scale=3.0]{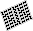}\end{tabular}&\begin{tabular}{r}\includegraphics[scale=3.0]{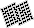}\end{tabular} \\ $(5/14, 9/25, 4/11, 1/3)$&\begin{tabular}{r}\includegraphics[scale=3.0]{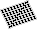}\end{tabular}&\begin{tabular}{r}\includegraphics[scale=3.0]{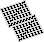}\end{tabular} \\ $(8/13, 13/21, 5/8, 3/5)$&\begin{tabular}{r}\includegraphics[scale=3.0]{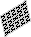}\end{tabular}&\begin{tabular}{r}\includegraphics[scale=3.0]{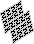}\end{tabular} \\ $(4/15, 5/19, 1/4, 3/11)$&\begin{tabular}{r}\includegraphics[scale=3.0]{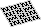}\end{tabular}&\begin{tabular}{r}\includegraphics[scale=3.0]{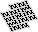}\end{tabular} \\ $(10/13, 13/17, 3/4, 7/9)$&\begin{tabular}{r}\includegraphics[scale=3.0]{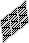}\end{tabular}&\begin{tabular}{r}\includegraphics[scale=3.0]{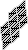}\end{tabular} \\ $(3/16, 5/27, 2/11, 1/5)$&\begin{tabular}{r}\includegraphics[scale=3.0]{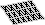}\end{tabular}&\begin{tabular}{r}\includegraphics[scale=3.0]{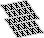}\end{tabular} \\ $(2/17, 3/25, 1/8, 1/9)$&\begin{tabular}{r}\includegraphics[scale=3.0]{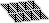}\end{tabular}&\begin{tabular}{r}\includegraphics[scale=3.0]{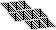}\end{tabular} \\ $(9/14, 11/17, 2/3, 7/11)$&\begin{tabular}{r}\includegraphics[scale=3.0]{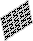}\end{tabular}&\begin{tabular}{r}\includegraphics[scale=3.0]{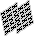}\end{tabular} \\ $(11/14, 15/19, 4/5, 7/9)$&\begin{tabular}{r}\includegraphics[scale=3.0]{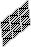}\end{tabular}&\begin{tabular}{r}\includegraphics[scale=3.0]{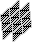}\end{tabular} \\ $(5/16, 9/29, 4/13, 1/3)$&\begin{tabular}{r}\includegraphics[scale=3.0]{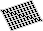}\end{tabular}&\begin{tabular}{r}\includegraphics[scale=3.0]{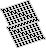}\end{tabular} \\ 